\documentclass[11pt]{amsart}

\usepackage[a4paper,margin=30mm]{geometry}
\usepackage{amsmath,amssymb,amsthm,mathtools}
\usepackage{mathrsfs}
\usepackage{enumitem}
\usepackage{tabularx}
\newcolumntype{Y}{>{\raggedright\arraybackslash}X}
\usepackage{tikz}
\usepackage{graphicx}
\usetikzlibrary{arrows.meta,calc,positioning,decorations.pathreplacing}
\tikzset{every node/.append style={execute at begin node={\hyphenpenalty=10000\exhyphenpenalty=10000}}}
\usepackage[hidelinks]{hyperref}

\usepackage{aliascnt}

\theoremstyle{plain}
\newtheorem{theorem}{Theorem}[section]
\newaliascnt{principle}{theorem}
\newtheorem{principle}[principle]{Principle}
\aliascntresetthe{principle}
\newaliascnt{conjecture}{theorem}

\aliascntresetthe{conjecture}
\newaliascnt{problem}{theorem}
\newtheorem{problem}[problem]{Problem}
\aliascntresetthe{problem}
\newaliascnt{proposition}{theorem}
\newtheorem{proposition}[proposition]{Proposition}
\aliascntresetthe{proposition}
\newaliascnt{lemma}{theorem}
\newtheorem{lemma}[lemma]{Lemma}
\aliascntresetthe{lemma}
\newaliascnt{corollary}{theorem}
\newtheorem{corollary}[corollary]{Corollary}
\aliascntresetthe{corollary}

\theoremstyle{definition}
\newaliascnt{definition}{theorem}
\newtheorem{definition}[definition]{Definition}
\aliascntresetthe{definition}
\newaliascnt{construction}{theorem}
\newtheorem{construction}[construction]{Construction}
\aliascntresetthe{construction}
\newaliascnt{example}{theorem}
\newtheorem{example}[example]{Example}
\aliascntresetthe{example}
\newaliascnt{notation}{theorem}
\newtheorem{notation}[notation]{Notation}
\aliascntresetthe{notation}
\newaliascnt{convention}{theorem}
\newtheorem{convention}[convention]{Convention}
\aliascntresetthe{convention}

\theoremstyle{remark}
\newaliascnt{remark}{theorem}
\newtheorem{remark}[remark]{Remark}
\aliascntresetthe{remark}

\usepackage[nameinlink,capitalize]{cleveref}
\crefname{convention}{Convention}{Conventions}
\Crefname{convention}{Convention}{Conventions}
\crefname{theorem}{Theorem}{Theorems}
\Crefname{theorem}{Theorem}{Theorems}
\crefname{proposition}{Proposition}{Propositions}
\Crefname{proposition}{Proposition}{Propositions}
\crefname{principle}{Principle}{Principles}
\Crefname{principle}{Principle}{Principles}
\crefname{problem}{Problem}{Problems}
\Crefname{problem}{Problem}{Problems}
\crefname{definition}{Definition}{Definitions}
\Crefname{definition}{Definition}{Definitions}
\crefname{construction}{Construction}{Constructions}
\Crefname{construction}{Construction}{Constructions}
\crefname{example}{Example}{Examples}
\Crefname{example}{Example}{Examples}
\crefname{corollary}{Corollary}{Corollaries}
\Crefname{corollary}{Corollary}{Corollaries}
\crefname{lemma}{Lemma}{Lemmas}
\Crefname{lemma}{Lemma}{Lemmas}
\crefname{remark}{Remark}{Remarks}
\Crefname{remark}{Remark}{Remarks}
\crefname{notation}{Notation}{Notations}
\Crefname{notation}{Notation}{Notations}
\crefname{conjecture}{Conjecture}{Conjectures}
\Crefname{conjecture}{Conjecture}{Conjectures}

\newcommand{\R}{\mathbb R}
\newcommand{\Z}{\mathbb Z}

\newcommand{\calS}{\mathcal S}
\newcommand{\calT}{\mathcal T}
\newcommand{\calN}{\mathcal N}
\newcommand{\calZ}{\mathcal Z}
\newcommand{\calX}{\mathcal X}
\newcommand{\calG}{\mathcal G}
\newcommand{\calH}{\mathcal H}

\newcommand{\Thi}{\operatorname{Thi}}
\newcommand{\Len}{\operatorname{Len}}
\newcommand{\Area}{\operatorname{Area}}
\newcommand{\Rop}{\operatorname{Rop}}
\newcommand{\Isom}{\operatorname{Isom}}
\newcommand{\Sim}{\operatorname{Sim}}
\newcommand{\DoF}{\operatorname{DoF}}

\newcommand{\CRad}{\operatorname{CRad}}
\newcommand{\slope}{\operatorname{slope}}
\newcommand{\into}{\hookrightarrow}
\newcommand{\type}{\operatorname{type}}
\newcommand{\ES}{\mathcal{ES}}
\newcommand{\Aut}{\operatorname{Aut}}
\newcommand{\Mod}{\operatorname{Mod}}
\newcommand{\code}{\operatorname{code}}

\title[Ropelength-Filtered Essential Surface Spaces]
{Ropelength-Filtered Essential Surface Spaces}

\author{Makoto Ozawa}

\address{Department of Natural Sciences, Faculty of Arts and Sciences, Komazawa University, Tokyo, Japan}
\email{w3c@komazawa-u.ac.jp}


\hypersetup{
  pdftitle={Ropelength-Filtered Essential Surface Spaces},
  pdfauthor={Makoto Ozawa},
  pdfsubject={Geometric finiteness for essential surfaces in knot exteriors},
  pdfkeywords={essential surface, knot exterior, geometric finiteness,
    positive reach, ropelength, finite geometric encoding, boundary slope,
    essential-surface complex, Kakimizu complex, disk complex, topological index}
}

\begin{document}

\begin{abstract}
We develop a filtered geometric framework for essential surfaces in knot
exteriors.  Let \(\gamma\) be a unit-thickness representative of a knot type
\(K\) with \(\Len(\gamma)\leq\Lambda\), and let \(F\) be a properly embedded
essential surface in the exterior of a fixed-radius tube about \(\gamma\), with
\(\Area(F)\leq\Delta\) and relative thickness at least \(\tau\), formulated
through Federer reach and controlled boundary collars.  We prove that this
bounded-geometry pair space contains only finitely many pair-isotopy classes,
and that equality of explicitly bounded canonical layered codes at resolution
\(\varepsilon\leq c\min\{1,\tau\}\) implies ambient pair-isotopy.  In a fixed
exterior \(E\), the surface systems visible in a geometric window form finite
subcomplexes \(\ES_{\Delta,\tau}(E)\) that exhaust the essential-surface
complex; isometries act levelwise and \(C^{1,1}\) self-diffeomorphisms act
with controlled reindexing.  For a fixed two-sided surface, compressing disks
are filtered in the same way, giving finite geometric witnesses for
compressibility and weak reducibility and recovering the index-one
characterization in Bachman's topological index theory.  On the peripheral
torus, every visible numerical boundary slope lies in an explicit writhe
window, so that \(|r|\leq C_{\mathrm{BS}}\Lambda^{4/3}+w(\Delta,\tau)\).
Together with finite Reidemeister certificates, the faithful codes give two
independent finite recognition mechanisms, one for the knot and one for the
carried essential-surface type.  The framework is a smooth, triangulation-free
analogue of the finiteness philosophy of normal surface theory; it does not
assert that a knot exterior has only finitely many essential surfaces without
geometric bounds.
\end{abstract}

\maketitle
\section{Introduction}
\label{sec:introduction}

Essential surfaces are topological objects, whereas ropelength, area, and
reach are geometric quantities.  In the absence of geometric bounds, these
worlds are far apart: a fixed knot exterior may carry infinitely many
essential surfaces, and the classical finite descriptions of such surfaces
are obtained by imposing auxiliary combinatorial structure, most notably a
triangulation and normal coordinates, within the classical Haken--Waldhausen
framework \cite{Haken,Waldhausen}.  The purpose of this paper is to show
that a different source of finiteness is available.  Suitable bounds in smooth
geometry constrain the topology directly, and sufficiently fine finite
geometric data recover that topology exactly.

This point of view leads to a geometric finiteness theory for essential
surfaces in knot exteriors.  Its basic objects are pairs \((\gamma,F)\), where
\(\gamma\) is a thick representative of a knot and \(F\) is an essential
surface in the exterior of a fixed-radius tube about \(\gamma\).  Its central
message is not merely that a finite approximation exists.  Rather, within
every bounded geometric window, only finitely many pair-isotopy types occur,
a sufficiently fine finite code separates all of them, and in a fixed
exterior these bounded slices form finite stages of the full
essential-surface complex.  The last observation permits the symmetry and
rigidity questions for essential surfaces to be studied one finite geometric
scale at a time.

\subsection*{The geometric window}
Let \(K\) be a knot type in \(S^3\).  After normalizing knot thickness to one,
we consider representatives
\[
  \gamma\in Y^{=}_\Lambda(K)
  =
  \{\gamma\in K\mid \Thi(\gamma)=1,\ \Len(\gamma)\leq\Lambda\}
  /\Isom^+(\R^3).
\]
We write \(Y^{\ge}_\Lambda(K)\) for the corresponding closed sublevel with
\(\Thi(\gamma)\ge1\) in place of \(\Thi(\gamma)=1\).  The latter is the
compact normalized model of the geometric filtration, and the two spaces are
not interchanged when attainment is asserted.  Pair spaces keep the actual
geometry until admissible components or pair-isotopy types are explicitly
taken, and the tube radius defining the exterior remains fixed.
For a fixed radius \(\rho_0\in(0,1)\), taken throughout to be
\(\rho_0=\tfrac12\), set
\[
  E(\gamma)=S^3\setminus\operatorname{int}N_{\rho_0}(\gamma).
\]
The strict inequality \(\rho_0<1\) ensures that the tube is embedded even
when the knot has unit thickness.  In \(E(\gamma)\) we study properly embedded
surfaces satisfying
\[
  \Area(F)\leq\Delta,
  \qquad
  \Thi(F)\geq\tau.
\]
Here relative surface thickness is expressed by positive reach, together with
uniform graphical boundary charts, an embedded boundary collar, and
peripheral clearance.  Thus, throughout this paper, \emph{bounded geometry}
means the simultaneous control of
\[
  \Len(\gamma),\qquad \Area(F),\qquad \Thi(F),
\]
with the knot thickness normalized to one.  These conditions define the
three-parameter pair space
\(\calZ_{\Lambda,\Delta,\tau}(K)\), and its essential subspace
\(\calZ^{\mathrm{ess}}_{\Lambda,\Delta,\tau}(K)\).

The roles of the three parameters are complementary.  The length bound
controls the positional freedom of the knot, the area bound controls the
amount of surface, and the positive lower thickness bound controls curvature,
sheet separation, boundary collars, and compactness.  The last condition is
the mechanism that turns metric bounds into topological finiteness.

\subsection*{Three guiding principles}
The theory is organized by three complementary principles.

\begin{principle}[Geometry constrains topology]
\label{prin:geometry-constrains-topology}
Bounded geometric complexity forces finite topological complexity.  At fixed
\((\Lambda,\Delta,\tau)\), only finitely many pair-isotopy types of essential
knot--surface pairs can occur, and the same geometric bounds impose
quantitative restrictions on topological data such as boundary slopes.
\end{principle}

\begin{principle}[Finite geometry determines topology]
\label{prin:finite-geometry-determines-topology}
For resolution sufficiently fine relative to the knot and surface thickness
scales, finite geometric data determine the ambient pair-isotopy type.  Thus
all topological types occurring in a bounded geometric window can be recovered
from a finite collection of finite codes.
\end{principle}

\begin{principle}[Finite geometry organizes rigidity]
\label{prin:finite-geometry-organizes-rigidity}
For a fixed geometric exterior, the essential surfaces visible in a bounded
area--thickness window form a finite marked complex.  These finite complexes
exhaust the full essential-surface complex, and exterior symmetries act on the
resulting filtered system, levelwise for isometries and with controlled
reindexing for general smooth mapping classes.  Thus geometric image,
invisible kernel, intrinsic reconstruction, and the economy of decorations can
be investigated at finite stages.
\end{principle}

The first principle is realized through compactness and peripheral rigidity;
the second through sampling and reconstruction; and the third through finite
visible complexes and controlled functoriality.  Together they give the
architecture in \Cref{fig:theory-architecture}.

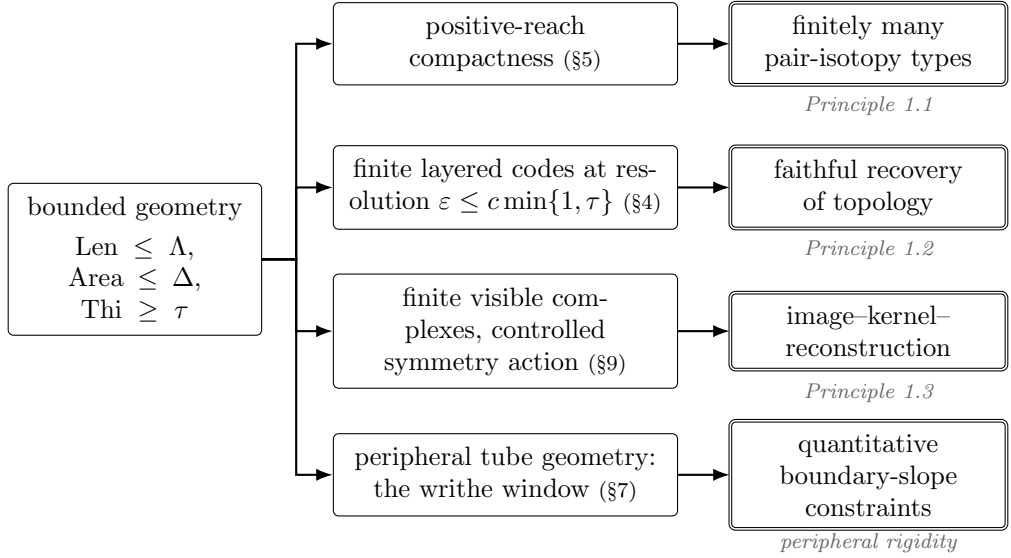
\begin{figure}[t]
\centering
\begin{tikzpicture}[x=1cm,y=1cm,
  box/.style={draw, rounded corners=2pt, align=center, inner sep=5pt,
              text width=42mm, font=\small},
  result/.style={draw, double, rounded corners=2pt, align=center, inner sep=5pt,
                 text width=33mm, font=\small},
  tag/.style={font=\scriptsize\itshape, text=black!60, align=center},
  arrow/.style={-{Latex[length=2.2mm]}, thick}]
  \node[box, text width=30mm] (budgets) at (1.55,2.85)
    {bounded geometry\\[1mm]
     \(\Len\leq\Lambda\), \(\Area\leq\Delta\),\\ \(\Thi\geq\tau\)};
  \node[box] (compact) at (6.45,5.7)
    {positive-reach compactness
     {\scriptsize(\S\ref{sec:bounded-complexity})}};
  \node[result] (finite) at (11.25,5.7) {finitely many\\pair-isotopy types};
  \node[tag] at (11.25,4.88) {\Cref{prin:geometry-constrains-topology}};
  \node[box] (codes) at (6.45,3.8)
    {finite layered codes at resolution
     \(\varepsilon\leq c\min\{1,\tau\}\)
     {\scriptsize(\S\ref{sec:dof})}};
  \node[result] (faithful) at (11.25,3.8) {faithful recovery\\of topology};
  \node[tag] at (11.25,2.98) {\Cref{prin:finite-geometry-determines-topology}};
  \node[box] (visible) at (6.45,1.9)
    {finite visible complexes, controlled symmetry action
     {\scriptsize(\S\ref{sec:filtered-rigidity})}};
  \node[result] (rigidity) at (11.25,1.9) {image--kernel--\\reconstruction};
  \node[tag] at (11.25,1.08) {\Cref{prin:finite-geometry-organizes-rigidity}};
  \node[box] (disks) at (6.45,0.0)
    {finite visible disk complexes and compression signatures
     {\scriptsize(\S\ref{sec:essential-intersections-compression-order})}};
  \node[result] (indexsig) at (11.25,0.0)
    {side births and\\weak-reduction merge};
  \node[box] (writhe) at (6.45,-1.9)
    {peripheral tube geometry: the writhe window
     {\scriptsize(\S\ref{sec:ideal-merge-pair-persistence})}};
  \node[result] (slopes) at (11.25,-1.9)
    {quantitative boundary-slope constraints};
  \node[tag] at (11.25,-2.80) {peripheral rigidity};
  \foreach \m in {compact,codes,visible,disks,writhe}{
    \draw[arrow] (budgets.east) -- ++(0.45,0) |- (\m.west);
  }
  \draw[arrow] (compact) -- (finite);
  \draw[arrow] (codes) -- (faithful);
  \draw[arrow] (visible) -- (rigidity);
  \draw[arrow] (disks) -- (indexsig);
  \draw[arrow] (writhe) -- (slopes);
\end{tikzpicture}
\caption{The architecture of the theory.  One set of geometric budgets feeds
five parallel mechanisms; each row names the mechanism, the section where it
is developed, and its output.  Bounded geometry produces smooth finiteness,
finite encodability, finite visible essential-surface complexes, finite
filtered compression data, and quantitative peripheral constraints.  The first
two rows implement
\Cref{prin:geometry-constrains-topology,prin:finite-geometry-determines-topology};
the third implements \Cref{prin:finite-geometry-organizes-rigidity}; the
fourth records the compression-theoretic extension to a fixed surface; and the
writhe window makes the passage from geometry to peripheral topology
quantitative.}
\label{fig:theory-architecture}
\end{figure}

\subsection*{Relation with classical three-manifold theories}
\label{subsec:classical-theories}

The present framework is best viewed as a bounded-geometry filtration placed
on top of several classical theories of surfaces and decompositions in
three-manifolds.  It does not replace those theories.  Instead, it asks a
common geometric question about the objects they produce:
\begin{quote}
\emph{At what length--area--thickness cost does a topological object become
visible, and when is it recoverable from finite geometric data?}
\end{quote}
The relation has three logically different levels.

\smallskip
\noindent
\emph{Direct combinatorial interfaces.}
Normal surface theory and the Jaco--Oertel algorithm encode essential surfaces
in a fixed triangulated manifold by integral normal coordinates
\cite{Haken,JacoOertel}.  Jaco--Rubinstein efficient triangulations and
crushing simplify the ambient triangulation
\cite{JacoRubinstein0Efficient}.  Floyd--Oertel and Oertel branched surfaces
provide finite carriers for incompressible surfaces
\cite{FloydOertel,OertelBranched}, while special spines give dual
combinatorial models for three-manifolds \cite{MatveevBook}.  The present
source of finiteness is different: positive-reach compactness in a moving
smooth exterior.  After a bounded-geometry triangulation is chosen, however,
the metric window can be compared with a finite region of normal-coordinate
or branch-weight space, and the layered code can be dualized to a decorated
cell structure or spine.  These are implementation and verification
interfaces; the smooth finiteness and reconstruction proofs themselves remain
triangulation-free.

\smallskip
\noindent
\emph{Filtered hierarchy and splitting theories.}
Gabai's sutured-manifold theory and Scharlemann's generalized Thurston norm
organize taut decompositions and hierarchies
\cite{GabaiFoliationsI,ScharlemannSuturedNorms}.  Scharlemann--Thompson
generalized Heegaard splittings alternate incompressible thin levels with
compressible thick levels \cite{ScharlemannThompson}, and Hempel distance
measures the separation of the two disk sets in the curve complex
\cite{HempelCurveComplex}.  Bachman's disk-complex formulation identifies
incompressible surfaces with topological index zero and strongly irreducible
surfaces with topological index one \cite{BachmanTopologicalIndex}.  Johnson's
complex of surfaces packages compressions into a thin-position framework
\cite{JohnsonHeegaardIsotopy}.  The essential surfaces treated here are
natural candidates for geometrically filtered decomposing surfaces and thin
levels.  The compression section below adds a finite geometric exhaustion of
the disk complex of an individual fixed surface.  A complete bounded-geometry
theory of generalized Heegaard splittings would still have to retain all
thick and thin levels simultaneously, compression-body incidence, and
compatibility of their disk data.

\smallskip
\noindent
\emph{Transverse and Cerf-theoretic extensions.}
The Gordon--Luecke graph method extracts information from transverse
intersection graphs of punctured surfaces, especially in Dehn filling
problems \cite{GordonCombinatorial,GordonLueckeToroidal}.  The
Rubinstein--Scharlemann graphic is the Cerf discriminant associated to a pair
of sweepouts \cite{RubinsteinScharlemann,RubinsteinScharlemannBounded}.  The
writhe window below constrains an individual visible boundary slope, and the
transition complexes below record finite-resolution changes of geometric
codes, but neither construction is an intersection graph or a
Rubinstein--Scharlemann graphic.  To reach those theories one must enlarge the
framework to controlled transverse systems and generic one- and two-parameter
families.

The comparison is summarized in
\Cref{tab:classical-geometric-comparison}.

\begin{table}[t]
\centering
\renewcommand{\arraystretch}{1.14}
\begin{tabularx}{\textwidth}{@{}>{\raggedright\arraybackslash}p{0.24\textwidth}YY@{}}
\hline
Classical framework & Classical organizing data & Role of bounded geometry here \\
\hline
Normal surfaces and efficient triangulations
& normal coordinates, matching equations, crushing
& conditional finite coordinate windows and algorithmic verification \\
Branched surfaces
& finite carriers with integral branch weights
& a proposed geometric cutoff on the carried weight semigroup \\
Special spines
& dual polyhedra and local spine moves
& a proposed decorated dual form of layered pair codes \\
Sutured manifolds
& taut decompositions and ordered hierarchies
& area--thickness costs and visibility levels for decomposing surfaces \\
Generalized Heegaard splittings and Hempel distance
& thin and thick levels, compression bodies, disk-set distance
& individual fixed surfaces have finite visible disk-complex filtrations;
  whole generalized splittings still require an enlarged space \\
Gordon--Luecke intersection graphs
& graphs of transverse intersection and Scharlemann-cycle combinatorics
& the writhe window gives individual slope constraints; a joint theory needs
  stratified transverse thickness \\
Rubinstein--Scharlemann graphics
& Cerf discriminants for pairs of sweepouts
& transition complexes are finite-resolution shadows, not replacements for
  the graphic \\
\hline
\end{tabularx}
\caption{Classical theories supply coordinates, carriers, hierarchies,
splittings, or transverse-intersection structures.  The present framework
supplies geometric visibility and finite recovery.  The branched-surface,
spine, hierarchy, splitting, intersection-graph, and graphic extensions are
programmatic unless explicitly proved later.}
\label{tab:classical-geometric-comparison}
\end{table}

Thus the phrase \emph{triangulation-free} refers to the analytic proofs of
smooth finiteness and faithful reconstruction.  It is not a claim that normal
surfaces, efficient triangulations, branched surfaces, or spines are
unnecessary for computation.  On the contrary, these classical structures
provide natural discrete verification layers for the finite geometric search
spaces constructed here.

\subsection*{Direct precedents and scope of the contribution}
The compactness--isotopy mechanism used below has a direct predecessor in
Durumeric's \(C^1\)-compactness and isotopy-finiteness theorem for closed
\(C^{1,1}\) submanifolds of uniformly positive normal injectivity radius
\cite{DurumericCompactness}; related isotopy-finiteness results follow from
geometric curvature-energy bounds \cite{KolasinskiStrzeleckiVonDerMosel}.
The contribution of the present finiteness theorem is the relative formulation
for properly embedded surfaces with quantitative boundary collars and its
uniform extension to a jointly moving knot--surface pair.  Isotopic
reconstruction, including reconstruction of surfaces with boundary, also has
substantial precedents \cite{AbeEtAl}; here the additional point is the uniform
reconstruction of the knot, the proper surface, and their relative peripheral
position from one layered code.  Likewise, surface complexes built from
isotopy classes of essential or incompressible surfaces are already established
objects \cite{SchultensSurfaceComplex,ZhangGuo,CharitosPapadoperakisTsapogas};
the construction below places an area--relative-thickness filtration on this
existing type of complex.  Disk complexes and their topological-index
interpretation are classical \cite{BachmanTopologicalIndex}, and Johnson's
complex of surfaces already organizes compression and thin position
\cite{JohnsonHeegaardIsotopy}.  The new compression-theoretic point here is
instead the finite geometric filtration of the disk complex of a fixed surface.
For comparison with other finiteness and counting
results for essential surfaces, see
\cite{DunfieldGaroufalidisRubinstein,PurcellTsvietkova,LackenbyTsvietkova}.
Finally, the boundary twisting diameter below is a determinant-distance variant
of the established boundary-slope diameter viewpoint
\cite{MattmanMaybrunRobinson,IchiharaSlopeLengths}.  Thus the claims of novelty
are confined to the relative moving-pair formulation, its filtered organization,
and the quantitative writhe-centered slope estimate, not to the underlying
compactness principle, reconstruction paradigm, or surface complex by itself.

\subsection*{Main results}
The main results are seven manifestations of the three guiding principles.

\begin{enumerate}[label=\textbf{\Alph*.},leftmargin=2.2em]
\item \emph{Smooth finiteness}
(\Cref{thm:smooth-finiteness-fixed-exterior,thm:smooth-finiteness-pair-space}).
As a relative extension of the positive-thickness compactness and
isotopy-finiteness mechanism of \cite{DurumericCompactness}, in a fixed thick
exterior, essential surfaces with
\(\Area(F)\leq\Delta\) and \(\Thi(F)\geq\tau\) occupy only finitely many
isotopy classes.  Allowing the unit-thickness knot representative to vary with
\(\Len(\gamma)\leq\Lambda\) still gives only finitely many ambient
pair-isotopy classes.  The same compactness gives attainment in the fixed,
ideal, and thickness-compactified visibility problems
(\Cref{cor:attainment-visibility,cor:unconditional-ideal-attainment}).

\item \emph{Explicitly bounded finite encoding}
(\Cref{thm:finite-resolution-finiteness-filtered-pairs,cor:explicit-encoded-bound}).
For every \(\varepsilon\leq c\min\{1,\tau\}\), all pairs in the filtered slice
are represented by an explicitly bounded finite collection of canonical
layered codes on a fixed ambient lattice.  The bound on the code space is
explicit; the canonical code of a given analytic pair exists but is not
claimed to be computable from arbitrary analytic input
(\Cref{rem:least-word-not-algorithmic}).

\item \emph{Faithful reconstruction} (\Cref{thm:faithfulness}).
At sufficiently fine resolution, equal codes imply ambient pair-isotopy.
Consequently, finite geometric information gives a separating certificate for
the pair-isotopy type within a bounded slice.  The code is not asserted to be
constant on a pair-isotopy class; rather, equal codes cannot occur in distinct
classes.

\item \emph{Finite recognition}
(\Cref{thm:surface-recognition,thm:two-unconditional-layers}).
Across knot types, every realized sufficiently fine pair code is
characteristic for its exact essential-surface type under the
orientation-preserving canonical-code convention; identifying codes under
ambient reflections yields recognition up to mirror.  Combined with the
finite Reidemeister certificates of
\cite{OzawaFiniteRecognition}, this gives two logically distinct finite
witnesses: one recognizes the underlying knot and the other the carried
surface type.

\item \emph{Filtered essential-surface rigidity}
(\Cref{thm:visible-complex-finiteness,thm:visible-complex-exhaustion,prop:controlled-visible-functoriality}).
For a fixed geometric exterior \(E\), an area--relative-thickness filtration
of the essential-surface complex gives finite visible subcomplexes
\(\ES_{\Delta,\tau}(E)\).  The
complexes are monotone and exhaust the full essential-surface complex.
Isometries act at each level, and meridian-preserving \(C^{1,1}\)
self-diffeomorphisms act with controlled changes of \((\Delta,\tau)\).  This
provides a finite-stage image--kernel--reconstruction framework while keeping
finite-code faithfulness distinct from automorphism realization.

\item \emph{The writhe window}
(\Cref{thm:writhe-window,cor:unconditional-slope-height}).
Every connected essential surface with nonempty non-meridional boundary,
visible at level \((\Lambda,\Delta,\tau)\), has numerical boundary slope within
\[
  \frac{\Delta}{2\pi\rho_0 c_{\mathrm{col}}(\tau)}
\]
of \(\operatorname{Wr}(\gamma)\).  Together with the Buck--Simon estimate,
this gives the unconditional bound
\[
  |r|\leq C_{\mathrm{BS}}\Lambda^{4/3}+w(\Delta,\tau).
\]
It also yields invisibility gaps and, for the torus knot \(T(p,q)\)
(\(\gcd(p,q)=1\)), a linear-in-\(pq\)
lower bound for the joint area of a Seifert surface and a cabling annulus
(\Cref{cor:unconditional-gap,cor:torus-joint-cost}).

\item \emph{Filtered compression and low-index signatures}
(\Cref{thm:visible-disk-complex-exhaustion,cor:compression-signature-classification}).
For every fixed compact two-sided surface \(H\) in a fixed geometric exterior,
disk area and relative disk thickness define finite visible subcomplexes
\(\mathcal D^{\mathrm{geom}}_{A,\sigma}(H)\) whose directed union is the
full disk complex \(\mathcal D(H)\).  Along any cofinal one-parameter chain,
the first appearances of disks on the positive and negative sides and the first
mixed-side edge define a compression signature.  Compressibility and weak
reducibility have finite witnesses.  Incompressibility is persistent absence of
both side births, while strong irreducibility is characterized by births on both
sides together with persistent absence of a mixed merge.  For Heegaard surfaces
this gives the filtered form of topological index one in Bachman's theory
(\Cref{cor:filtered-topological-index}).  No a priori finite completion stage
for these negative conditions is asserted.
\end{enumerate}

Results A and F implement \Cref{prin:geometry-constrains-topology}; Results B
and C implement \Cref{prin:finite-geometry-determines-topology}; Result D
shows how the second principle enters finite recognition; and Result E
implements \Cref{prin:finite-geometry-organizes-rigidity}.  Result G is a
compression-theoretic extension of the bounded-geometry philosophy from
essential surfaces to the disk complex of one fixed, possibly compressible,
surface.  Results A--C are compactness, sampling, and reconstruction
statements.  Result E is a filtered functoriality statement, Result F is a
peripheral rigidity statement, and Result G is a finite-exhaustion and
persistent-signature statement.  Result F is the only one whose proof uses the
connection one-form of the fixed-radius tube.

\subsection*{Scope and logical status}
Several distinctions are essential.  First, the word \emph{finite} refers to
bounded geometric slices and finite witnesses.  It does not assert that all
essential surfaces in a fixed exterior form a finite set; Haken sums and
infinite Kakimizu complexes remain possible after the bounds are removed.
Second, pair isotopy, admissible thick connectivity, equality of
finite-resolution codes, and geometric realization of automorphisms are
different relations and are kept separate.  In particular, the faithfulness
of an individual pair code does not imply that every automorphism of a finite
visible complex is induced by a homeomorphism of the exterior.  Third, the
transition complexes introduced later give finite-resolution
quotient persistence; they are not claimed to recover smooth admissible
components without the lifted component data, and they are not identified
with Rubinstein--Scharlemann graphics.  Fourth, the current relative-thickness
condition is a compactness condition for disjoint systems; transverse
intersection graphs require a different stratified condition.  Fifth, the
visible disk-complex filtration gives finite witnesses for the \emph{presence}
of a compression or weak reduction, but absence at one finite stage does not
certify incompressibility or strong irreducibility: those are persistent
nonappearance conditions unless an independent completion bound is available.
Finally, the algorithmic constructions identify finite search spaces and
faithful certificates, but do not by themselves supply sharp complexity bounds
for enumerating them.

The knot-recognition input from \cite{OzawaFiniteRecognition} is used only in
\Cref{sec:finite-recognition}.  None of the compactness, smooth finiteness,
finite encoding, reconstruction, or writhe-window theorems depends on that
input.  Persistence, trade-off functions, boundary-twisting diameters, and
Kakimizu filtrations are proved only to the extent explicitly stated; the
remaining claims are formulated as invariants or open problems.

\subsection*{Relation with companion work}
The framework belongs to a broader program on geometric filtrations of knot
space.  The knot-only ideal stratum and its admissible-component persistence
are developed in \cite{OzawaIdealStratum}; density and compression
factorizations in \cite{OzawaDensityCompression}; and finite Reidemeister
recognition in
\cite{OzawaFiniteRecognition}.  The image--kernel--reconstruction problem
for essential-surface complexes, scale-free density invariants, and
swept-area metrics are natural continuations of the program; they are
formulated here only to the extent that
\Cref{sec:filtered-rigidity} and the open problems state them, and the
present paper depends on no unpublished manuscript.  The paper is
self-contained in
its surface-side core.  Its new feature is the simultaneous treatment of a
moving thick knot and bounded-geometry essential surfaces in its exterior,
together with a finite geometric exhaustion of the topological object studied
in the rigidity problem.

\begin{table}[t]
\centering
\renewcommand{\arraystretch}{1.15}
\begin{tabularx}{\textwidth}{@{}lX@{}}
\hline
Symbol & Meaning \\
\hline
\(Y^{=}_\Lambda(K)\) & unit-thickness representatives of \(K\) with length \(\leq\Lambda\), modulo rigid motions \\
\(\calX_{\Lambda,u}(K)\) & projection-framed ropelength sublevel in direction \(u\) \\
\(\calG^{\mathrm{lift}}_{\Lambda,u}(K)\), \(\calH_{\Lambda,u}(K)\) & lifted Reidemeister multigraph and monotone diagram image \\
\(I(K)\) & ideal stratum \(Y^{=}_{\Rop(K)}(K)\) \\
\(\rho_0\), \(N_{\rho_0}(\gamma)\) & fixed tube radius \(\rho_0=\tfrac12\) and the closed normal \(\rho_0\)-tube \\
\(E(\gamma)\) & exterior \(S^3\setminus\operatorname{int}N_{\rho_0}(\gamma)\) \\
\(\calZ_{\Lambda,\Delta,\tau}(K)\) & length--area--thickness filtered pair space \\
\(\calZ^{\mathrm{ess}}_{\Lambda,\Delta,\tau}(K)\) & essential subspace of the pair space \\
\(\beta_{\mathfrak S}(K;\Delta,\tau)\) & visibility level of a surface type \(\mathfrak S\) \\
\(\operatorname{Tw}_{\partial}(K)\) & boundary twisting diameter \\
\(\Theta^{\partial}_{p,\ell}\), \(\operatorname{Tw}^{\mathrm{rate}}_{\partial}\) & finite-length peripheral twist and slope-rate diameter \\
\(\mathrm{MS}_{\Lambda,\Delta,\tau}(K)\) & visible taut Seifert/Kakimizu layer \\
\(\ES(E)\), \(\ES_{\Delta,\tau}(E)\) & full and bounded-geometry visible essential-surface complexes of a fixed exterior \\
\(\mathcal D(H)\), \(\mathcal D^{\mathrm{geom}}_{A,\sigma}(H)\) & disk complex of a fixed surface and its finite visible stages \\
\(\mathcal Q_{A,\sigma}(H)\) & two-sided compression-signature graph \\
\(\mathbf{ES}^{\mathrm{geom}}(E,\mu)\) & filtered system of finite visible complexes with meridian marking \\
\(\DoF_\varepsilon\) & effective finite-resolution degree of freedom \\
\(\operatorname{Wr}(\gamma)\), \(w(\Delta,\tau)\) & writhe and writhe-window half-width \\
\(\rho_D\), \(\CRad_D\) & density and compression-radius factors \\
\hline
\end{tabularx}
\caption{Principal notation.  The parameters \(\Lambda\), \(\Delta\),
\(\tau\), and \(\varepsilon\) denote length, area, surface-thickness, and
resolution scales.}
\label{tab:notation}
\end{table}

\subsection*{Organization of the paper}
\Cref{sec:ropelength-sublevel} recalls ropelength sublevel spaces.
\Cref{sec:pair-spaces} defines relative surface thickness and the filtered pair
space.  \Cref{sec:dof} proves explicitly bounded finite encoding, while
\Cref{sec:bounded-complexity} establishes bounded topological complexity,
positive-reach compactness, smooth finiteness, and faithful reconstruction.
\Cref{sec:ideal} treats ideal knots, ideal surfaces, ideal pairs, and the limit
\(\tau\to0\).  \Cref{sec:ideal-merge-pair-persistence} develops
admissible-component persistence and boundary-slope visibility, culminating in
the writhe window.  \Cref{sec:characteristic} treats characteristic systems and
normal-surface implementation; \Cref{sec:filtered-rigidity} constructs the
finite visible essential-surface complexes and their controlled symmetry
actions; \Cref{sec:essential-intersections-compression-order} constructs
finite visible disk complexes and low-index compression signatures for a fixed
surface and then studies compression and tubing; \Cref{sec:kakimizu} develops
taut, crosscap,
and Kakimizu layers, including the trefoil visibility window;
\Cref{sec:density-compression} develops density--compression factorizations;
and \Cref{sec:finite-recognition} establishes the two finite recognition
layers.  \Cref{sec:further} collects open problems and the wider scope of
the theory.  \Cref{sec:variational-interfaces} identifies the additional
hypotheses needed to connect the finiteness theory with critical-point and
degeneration arguments, and the final conclusion summarizes the passage from bounded
geometry to finite topology.

\section{Ropelength sublevel spaces}
\label{sec:ropelength-sublevel}

\subsection{Thick knot representatives}

Let \(\gamma:S^1\to \R^3\) be a \(C^{1,1}\) embedded curve.  We denote its
length by \(\Len(\gamma)\) and its thickness by \(\Thi(\gamma)\).  Thickness is
the radius of the largest embedded normal tube around \(\gamma\).  Equivalently,
it is controlled by curvature and doubly critical self-distance.

For a knot type \(K\), define its ropelength by
\[
  \Rop(K)
  =
  \inf_{\gamma\in K}
  \frac{\Len(\gamma)}{\Thi(\gamma)}.
\]
After scaling, we may impose \(\Thi(\gamma)=1\).  Then the ropelength
sublevel space is
\[
  Y^{=}_\Lambda(K)
  =
  \{\gamma\in K\mid \Thi(\gamma)=1,\ \Len(\gamma)\leq\Lambda\}/\Isom^+(\R^3).
\]
The ideal stratum is
\[
  I(K)=Y^{=}_{\Rop(K)}(K).
\]
By the existence theorem for ropelength minimizers \cite{CantarellaKusnerSullivan}, \(I(K)\) is nonempty.
Moreover, because \(\Rop(K)\) is the infimum of length among unit-thickness
representatives, \(Y^{=}_{\Rop(K)}(K)\) is precisely the set of unit-thickness
ropelength minimizers of \(K\).

\begin{remark}[Radius normalization of ropelength]
\label{rem:ropelength-normalization}
In this paper thickness is the \emph{radius} of the largest embedded normal
tube, following \cite{GonzalezMaddocks,CantarellaKusnerSullivan,LitherlandSimonDurumericRawdon}.
Part of the ropelength literature instead normalizes by the tube
\emph{diameter}; in particular Denne--Diao--Sullivan
\cite{DenneDiaoSullivan} take thickness to be the diameter, so their
numerical values are exactly half of the values in the present convention.
Every imported constant is converted accordingly: the
Denne--Diao--Sullivan universal bound \(15.66\) in diameter units reads
\(31.32\) here, while the numerically tightened trefoil length
\(32.7433864\) of \cite{BaranskaPieranskiPrzybylRawdon} is reported there
for unit-radius rope and is used unchanged.  All numerical ropelength
constants appearing in this paper are stated in the radius convention.
\end{remark}

\begin{remark}[Rigid-motion normalization]
The quotient by orientation-preserving Euclidean isometries is used only to
remove irrelevant rigid motions.  Equivalently, one may fix a concrete
normalization, for example barycenter zero together with a choice of principal
frame whenever the frame is non-degenerate.  A symmetric representative may
have a nontrivial stabilizer in the quotient, which can even be a continuous
group (the round circle has stabilizer containing \(SO(2)\)); this affects
none of the finite-resolution estimates below, and the ambiguity it creates
for encodings is resolved once and for all by the minimal-code convention of
\Cref{rem:code-well-defined}.  A normalized model may be used throughout if
one wants to avoid this quotient language completely.
\end{remark}

\subsection{Effective finite-resolution degree of freedom of \texorpdfstring{\(Y^{=}_\Lambda(K)\)}{Y Lambda(K)}}

Although \(Y^{=}_\Lambda(K)\) is not a finite-dimensional space in any naive
smooth sense, it has finite-resolution degree-of-freedom control.  The thickness bound prevents
arbitrarily sharp bending and arbitrarily close self-approach, while the
length bound limits the total amount of curve.

We formalize this by the effective degree of freedom introduced in
\Cref{sec:dof}.  At scale \(\varepsilon>0\), a representative
\(\gamma\in Y^{=}_\Lambda(K)\) can be approximated by a thickness-adapted
polygonal curve with
\[
  O(\Lambda/\varepsilon)
\]
edges.  Thus \(\Lambda\) restricts the positional freedom of \(\gamma\) at
finite resolution.

This control is the geometric reason why coarse projection data, coarse bridge
data, and coarse trunk data should be bounded at a fixed ropelength level.
Both structures are shown in \Cref{fig:ropelength-sublevel}.

\begin{figure}[t]
\centering
\begin{tikzpicture}[x=1cm,y=1cm,>=Latex,
  panel/.style={draw, rounded corners=2pt},
  title/.style={font=\footnotesize\bfseries},
  note/.style={font=\scriptsize, align=center}]
  \draw[panel] (0,-0.5) rectangle (6.8,4.35);
  \node[title] at (3.4,4.08) {the ropelength filtration};
  \draw[thin] (3.0,2.2) ellipse (2.6 and 1.25);
  \draw[thin] (2.05,2.2) ellipse (1.6 and 0.85);
  \draw[thin, fill=black!6] (1.25,2.2) ellipse (0.75 and 0.48);
  \fill (0.8,2.2) circle (1.5pt);
  \node[note] at (0.75,3.75) {$I(K)$};
  \draw[->, thin] (0.76,3.6) -- (0.8,2.3);
  \node[note] at (1.58,2.2) {$Y^{=}_{\Lambda_1}$};
  \node[note] at (3.05,2.2) {$Y^{=}_{\Lambda_2}$};
  \node[note] at (4.9,2.2) {$Y^{=}_{\Lambda_3}$};
  \draw[->] (0.35,0.35) -- (6.35,0.35)
    node[note, below left, xshift=2mm] {$\Lambda$};
  \foreach \x/\lab in {0.8/{$\Rop(K)$},2.0/{$\Lambda_1$},
                       3.65/{$\Lambda_2$},5.6/{$\Lambda_3$}}{
    \draw[thin] (\x,0.28) -- (\x,0.42);
    \node[note, below] at (\x,0.24) {\lab};
  }
  \draw[dashed, thin] (0.8,0.42) -- (0.8,2.16);
  \draw[dashed, thin] (2.0,0.42) -- (2.0,2.2);
  \draw[dashed, thin] (3.65,0.42) -- (3.65,2.2);
  \draw[dashed, thin] (5.6,0.42) -- (5.6,2.2);
  \draw[panel] (7.4,-0.5) rectangle (13.5,4.35);
  \node[title] at (10.45,4.08) {finite-resolution freedom};
  \draw[black!12, line width=9pt]
    (8.0,1.3) .. controls (9.0,2.9) and (10.2,0.6) .. (11.3,2.1)
    .. controls (12.0,3.05) and (12.6,2.6) .. (12.9,1.9);
  \draw[thin]
    (8.0,1.3) .. controls (9.0,2.9) and (10.2,0.6) .. (11.3,2.1)
    .. controls (12.0,3.05) and (12.6,2.6) .. (12.9,1.9);
  \draw[very thick]
    (8.0,1.3) -- (8.62,2.02) -- (9.35,2.05) -- (10.05,1.62)
    -- (10.78,1.55) -- (11.45,2.22) -- (12.12,2.62) -- (12.9,1.9);
  \foreach \p in {(8.0,1.3),(8.62,2.02),(9.35,2.05),(10.05,1.62),
                  (10.78,1.55),(11.45,2.22),(12.12,2.62),(12.9,1.9)}
    \fill \p circle (1.2pt);
  \draw[decorate, decoration={brace, amplitude=3.5pt}]
    (8.55,2.14) -- (9.32,2.17);
  \node[note, above] at (8.95,2.38) {$\asymp\varepsilon$};
  \node[note] at (10.45,3.55)
    {unit thickness forbids sharp bends\\and close returns};
  \node[note] at (10.45,0.15)
    {thickness-adapted polygon: $O(\Lambda/\varepsilon)$ edges};
\end{tikzpicture}
\caption{Ropelength sublevel spaces.  Left: the spaces
\(Y^{=}_{\Lambda}(K)\) form a nested filtration of the space of unit-thickness
representatives, beginning at the ideal stratum
\(I(K)=Y^{=}_{\Rop(K)}(K)\), which is nonempty by the existence of ropelength
minimizers.  Right: the reason \(Y^{=}_\Lambda(K)\) has finite-resolution
degree-of-freedom control: unit thickness excludes sharp bending and close
self-approach, so at scale \(\varepsilon\) a representative is captured by a
thickness-adapted polygon with \(O(\Lambda/\varepsilon)\) edges.}
\label{fig:ropelength-sublevel}
\end{figure}
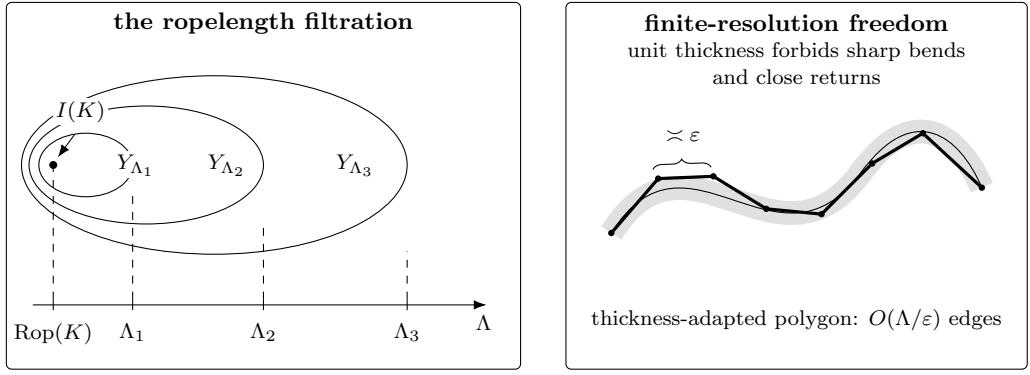

\section{Essential surface pair spaces}
\label{sec:pair-spaces}

\subsection{Ambient metric convention and surface thickness}

\begin{notation}[Ambient metric and tube-radius convention]
\label{not:ambient-metric}
Throughout the paper we identify \(S^3=\R^3\cup\{\infty\}\) and endow
\(\R^3\) with the Euclidean metric.  Every representative \(\gamma\) lies in
\(\R^3\), and every compact properly embedded surface in the exterior
\(E(\gamma)=S^3\setminus\operatorname{int}N_{\rho_0}(\gamma)\) is isotoped off
the point at infinity, so that it lies in \(\R^3\setminus\operatorname{int}
N_{\rho_0}(\gamma)\).  All lengths, areas, curvatures, reaches, and thicknesses
are measured in this Euclidean metric; the topology of \(S^3\) is used only to
speak of exteriors and of essential surfaces, never to measure a geometric
quantity.  Since \(E(\gamma)\setminus\{\infty\}\) is an unbounded subset of
\(\R^3\), compactness assertions for families of surfaces in \(E(\gamma)\) are
never taken for granted; they are proved in
\Cref{subsec:c11-compactness} from an anchoring lemma and a diameter estimate.

Here \(\rho_0\in(0,1)\) is a tube radius fixed once and for all; for
definiteness we take \(\rho_0=\tfrac12\), and every constant below that depends
on the tube geometry is allowed to depend on this fixed \(\rho_0\).  For a
unit-thickness \(C^{1,1}\) curve \(\gamma\), the normal exponential map
\((s,v)\mapsto\gamma(s)+v\), \(v\perp\gamma'(s)\), \(|v|\leq\rho\), is a
bi-Lipschitz homeomorphism onto the closed tube \(N_\rho(\gamma)\) for every
\(0<\rho<1=\Thi(\gamma)\), and \(\partial N_\rho(\gamma)\) is an embedded
\(C^{1,1}\) torus.  At the critical radius \(\rho=1\) this may fail: normal
disks of nearby strands may touch, and focal points may lie on the boundary,
so \(\partial N_1(\gamma)\) need not be an embedded surface.  For this reason
the exterior, the peripheral torus \(\partial E(\gamma)=\partial
N_{\rho_0}(\gamma)\), meridian--longitude representatives, boundary collars,
and all peripheral length computations are taken on the \(\rho_0\)-tube, never
on the critical radius-one tube.  This convention loses no generality for
compact surfaces and makes the quantities \(\Len\), \(\Area\), and \(\Thi\)
unambiguous.
\end{notation}

Let \(F\subset E(\gamma)\) be a properly embedded surface.  For the purposes of
this paper, a thickness condition for \(F\) should include both interior and
boundary control; see \Cref{fig:relative-thickness}.

\begin{definition}[Relative surface thickness]
\label{def:relative-surface-thickness}
Let \(F\subset E(\gamma)\) be a compact properly embedded \(C^{1,1}\)
surface, possibly with boundary.  We say that \(F\) has relative thickness at
least \(r\), and write \(\Thi(F)\geq r\), if the following bounded-geometry
conditions hold.  The constants
\[
  0<c_0<1,\qquad L_0\geq1
\]
are fixed once and for all and do not depend on \(F\), \(\gamma\), or \(r\).
\begin{enumerate}[label=(\roman*),leftmargin=2em]
\item \emph{Positive reach of the closed set.}  The closed subset
\(\overline F\subset\R^3\), including its boundary, has Federer reach at
least \(r\) in the Euclidean metric of \Cref{not:ambient-metric}: the
nearest-point projection to \(\overline F\) is single-valued on the open
\(r\)-neighbourhood of \(\overline F\) in \(\R^3\)
\cite{FedererReach}.  This metric-projection property is the definition;
see \Cref{rem:two-point-criterion} for the tangent-cone reformulation and for
a warning about surfaces with boundary.
\item \emph{Interior graphical control.}  For every interior point
\(x\in F\setminus\partial F\), after a rigid change of coordinates taking
\(x\) to the origin and \(T_xF\) to the horizontal plane, there is a
\(C^{1,1}\) function \(u\) on the disk of radius \(c_0r\) in \(T_xF\), with
\(u(0)=0\) and \(Du(0)=0\), whose \(C^1\)-norm and whose Lipschitz norm of
the first derivative are bounded by \(L_0\) after scaling by \(r^{-1}\),
such that
\[
  F\cap B(x,c_0r)\ \subseteq\ \operatorname{graph}(u),
  \qquad
  \operatorname{graph}\bigl(u|_{\text{disk of radius }c_0r/2}\bigr)
  \ \subseteq\ F .
\]
The clause is stated through these two inclusions, rather than through an
equality of \(F\cap B(x,c_0r)\) with a graph over a disk, so that it
restricts verbatim to smaller balls and scales.
\item \emph{Boundary half-chart control.}  For every boundary point
\(x\in\partial F\), the same statement holds with the disk replaced by a
half-disk in \(T_xF\), with \(\partial F\) corresponding to the boundary
diameter of the half-disk, and with the same uniform \(C^{1,1}\) bounds and
the same pair of inclusions at radii \(c_0r\) and \(c_0r/2\).
\item \emph{Quantitatively embedded boundary collar.}  Let
\[
  C_F:[0,c_0r]\times\partial F\longrightarrow F,
  \qquad C_F(0,p)=p,
\]
be the inward conormal collar map.  For points of the same boundary component
let \(d_{\partial F}\) be intrinsic arclength distance.  The map \(C_F\) is an
embedding and satisfies the following three estimates.
\begin{itemize}[leftmargin=2em]
\item[(iv-a)] \emph{Local upper bound.}  If \(p,q\) lie on the same component
of \(\partial F\), then
\[
  |C_F(s,p)-C_F(t,q)|
  \leq L_0\bigl(|s-t|+d_{\partial F}(p,q)\bigr).
\]
\item[(iv-b)] \emph{Truncated lower bound.}  If \(p,q\) lie on the same
component of \(\partial F\), then
\[
  |C_F(s,p)-C_F(t,q)|
  \geq L_0^{-1}\min\bigl\{r,\ |s-t|+d_{\partial F}(p,q)\bigr\}.
\]
\item[(iv-c)] \emph{Separation of distinct components.}  If \(p,q\) lie on
different components of \(\partial F\), then
\[
  |C_F(s,p)-C_F(t,q)|\geq c_0r .
\]
\end{itemize}
After scaling the collar parameter and boundary arclength by \(r^{-1}\), the
first derivative of \(C_F\) is bounded by \(L_0\) and is \(L_0\)-Lipschitz.
Thus the collar controls local bi-Lipschitz regularity through
(iv-a) and (iv-b) at small distances, forbids long-range returns of a
boundary component at scale \(r\) through the truncated lower bound
(iv-b), and keeps distinct boundary components a definite distance apart
through (iv-c).  Only the lower bounds are truncated; the upper bound
(iv-a) is a local, non-truncated estimate.  A truncated \emph{upper}
bound would force every boundary component to have extrinsic diameter at
most \(L_0r\), an unnatural restriction that moreover fails to be monotone
as \(r\) decreases; the present form imposes no such diameter bound and is
scale monotone by \Cref{lem:thickness-scale-monotonicity} below.
\item \emph{Peripheral clearance in collar coordinates.}  Writing
\(T=\partial E(\gamma)\), one has
\[
  \operatorname{dist}(C_F(s,p),T)\geq c_0s
  \qquad(0\leq s\leq c_0r),
\]
and
\[
  \operatorname{dist}\!\left(
    F\setminus C_F([0,c_0r/2]\times\partial F),T
  \right)\geq \frac{c_0^2}{4}r.
\]
For a closed surface the second inequality is required with \(F\) in place
of the displayed complement.  In particular
\(F\cap T=\partial F\), and no interior sheet of \(F\) can accumulate on the
peripheral torus.
\end{enumerate}
The graphical clauses control curvature, the reach clause controls global
sheet separation, the collar estimates (iv-a)--(iv-c) make the boundary
stratum compact in a fixed parameter domain, and the two clearance
inequalities separate the collar and deep-interior regimes.  Near the boundary, the first
clearance inequality is a quantitative transversality condition: the inward
surface direction has normal component at least \(c_0\) relative to the
peripheral torus.  A compact properly embedded \(C^{1,1}\) surface meeting the
torus transversely can be changed by an arbitrarily small collar isotopy so
that these conditions hold at some positive scale.

For a disconnected surface, all five clauses are imposed on the union
\(\overline F\).  In particular, clause~(i) implies that distinct components
have mutual distance at least \(2r\), so the normal \(r\)-tube of the union
is embedded.

The numerical relative thickness is
\[
  \Thi(F)
  :=
  \max\{r>0\mid \text{the five clauses above hold at scale }r\},
\]
with value \(0\) if no positive scale is admissible.  The maximum exists:
the clauses are monotone and closed under increasing limits of the scale for a
fixed compact \(C^{1,1}\) surface.  Thus the notation \(\Thi(F)\geq r\)
used above agrees with the ordinary numerical inequality.

The convention is monotone in the scale: this is not built into the
wording of the clauses but is proved once and for all in
\Cref{lem:thickness-scale-monotonicity} below, and it is used whenever a
compactness argument is first carried out at an arbitrary scale \(r'<r\)
and then \(r'\nearrow r\), and whenever the filtration parameter \(\tau\)
is decreased.
\end{definition}

\begin{lemma}[Scale monotonicity of relative thickness]
\label{lem:thickness-scale-monotonicity}
The convention of
\Cref{def:relative-surface-thickness} is monotone in the scale: if the five
clauses hold at scale \(r\), then they hold, with the same universal
constants \(c_0\) and \(L_0\), at every scale \(0<r'\leq r\).  In particular
\(\Thi(F)\geq r\)
implies \(\Thi(F)\geq r'\) for every \(0<r'\leq r\).
\end{lemma}

\begin{proof}
Fix \(0<r'\leq r\) and verify the clauses one by one.

\emph{Clause (i).}  Federer reach is a single number attached to
\(\overline F\); \(\operatorname{reach}(\overline F)\geq r\geq r'\) is
immediate.

\emph{Clauses (ii) and (iii).}  Let \(x\in F\setminus\partial F\) and let
\(u\) be the \(C^{1,1}\) function on the disk of radius \(c_0r\) in
\(T_xF\), with \(u(0)=0\) and \(Du(0)=0\), provided by clause (ii) at scale
\(r\); its unscaled bounds are \(\operatorname{Lip}(Du)\leq L_0/r\) and
\(|Du|\leq L_0\).  Take, at scale \(r'\), the restriction
\(u'=u|_{\text{disk of radius }c_0r'}\).  The two inclusions restrict
verbatim: a point of \(F\cap B(x,c_0r')\) lies in \(F\cap B(x,c_0r)\), hence
on \(\operatorname{graph}(u)\), and its horizontal projection has norm at
most \(c_0r'\), so it lies on \(\operatorname{graph}(u')\); and
\(\operatorname{graph}(u'|_{\text{disk of radius }c_0r'/2})\subseteq
\operatorname{graph}(u|_{\text{disk of radius }c_0r/2})\subseteq F\).
The Taylor estimates \(|Du(y)|\leq (L_0/r)|y|\leq L_0c_0r'/r\leq L_0\) and
\(|u(y)|\leq (L_0/2r)|y|^2\leq (L_0c_0^2/2)r'\) on the disk of radius
\(c_0r'\) show that, after scaling by
\((r')^{-1}\), the \(C^1\)-norm of \(u'\) is bounded
by \(L_0\), and its first derivative is
\((L_0r'/r)\)-Lipschitz, hence \(L_0\)-Lipschitz, in the scaled variables.
Thus clause (ii) holds at scale \(r'\) with the same constants.
The boundary half-chart clause (iii) is verified by the same restriction
argument applied to half-disks.

\emph{Clause (iv).}  The collar map at scale \(r'\) is the restriction of
\(C_F\) to \([0,c_0r']\times\partial F\), hence an embedding.  The local
upper bound (iv-a) contains no \(r\) and restricts verbatim.  For the
truncated lower bound (iv-b), monotonicity of the truncation gives
\[
  |C_F(s,p)-C_F(t,q)|
  \geq L_0^{-1}\min\{r,\ |s-t|+d_{\partial F}(p,q)\}
  \geq L_0^{-1}\min\{r',\ |s-t|+d_{\partial F}(p,q)\},
\]
which is the required estimate at scale \(r'\).  For (iv-c),
\(|C_F(s,p)-C_F(t,q)|\geq c_0r\geq c_0r'\).  The scaled derivative bounds
follow as in clauses (ii)--(iii).

\emph{Clause (v).}  The first clearance inequality
\(\operatorname{dist}(C_F(s,p),T)\geq c_0s\) holds on the smaller parameter
domain \(0\leq s\leq c_0r'\) by restriction.  For the second, decompose
\[
  F\setminus C_F\bigl([0,c_0r'/2]\times\partial F\bigr)
  =
  \Bigl(F\setminus C_F\bigl([0,c_0r/2]\times\partial F\bigr)\Bigr)
  \cup
  C_F\bigl([c_0r'/2,c_0r/2]\times\partial F\bigr).
\]
On the first piece the clause at scale \(r\) gives distance at least
\(c_0^2r/4\geq c_0^2r'/4\) from \(T\); on the second piece the first
clearance inequality gives distance at least
\(c_0\cdot c_0r'/2=c_0^2r'/2\geq c_0^2r'/4\).  Hence the second clearance
inequality holds at scale \(r'\).  For a closed surface the same estimate is
required with \(F\) in place of the complement and is inherited directly
from the scale-\(r\) clause.
\end{proof}

\begin{remark}[Why the truncation is one-sided]
\label{rem:one-sided-truncation}
The asymmetry between (iv-a) and (iv-b) is essential.  Taking \(s=t=0\) in a
truncated two-sided estimate would give
\(|p-q|\leq L_0\min\{r,d_{\partial F}(p,q)\}\leq L_0r\) for all boundary
points \(p,q\) of one component, forcing each boundary component to have
extrinsic diameter at most \(L_0r\); moreover the resulting condition would
become \emph{stronger}, not weaker, as \(r\) decreases, so no scale
monotonicity of the form of \Cref{lem:thickness-scale-monotonicity} could
hold.  With the truncation applied only to the lower bound, large boundary
components are allowed, distant returns of one component and distinct
components are still separated at a definite scale, and the convention is
scale monotone.
\end{remark}

\begin{figure}[t]
\centering
\begin{tikzpicture}[>=Latex, line join=round]
  \fill[black!7]
    (0,1.55) .. controls (2.7,2.15) and (5.3,1.05) .. (8.1,1.55)
    -- (8.1,0.45)
    .. controls (5.3,-0.05) and (2.7,1.05) .. (0,0.45) -- cycle;
  \draw[thick] (8.1,-1.7) -- (8.1,2.5);
  \node[above right, font=\scriptsize] at (8.05,2.3) {$\partial E(\gamma)$};
  \draw[very thick]
    (0,1.0) .. controls (2.7,1.6) and (5.3,0.5) .. (8.1,1.0);
  \node[font=\small] at (1.05,1.95) {$F$};
  \draw[<->] (4.05,1.36) -- (4.05,0.62);
  \node[right, font=\scriptsize] at (4.12,0.99) {$\tau$};
  \coordinate (xpoint) at (2.01094,1.18281);
  \fill (xpoint) circle (1.3pt);
  \node[above, font=\scriptsize] at ($(xpoint)+(0,0.08)$) {$x$};
  \draw[thin]
    ($(xpoint)+(-0.85,0.014)$) -- ($(xpoint)+(0.85,-0.014)$);
  \node[above right, font=\scriptsize] at ($(xpoint)+(0.72,0.02)$)
    {$T_xF\cap P$};
  \draw[dashed, thin] (xpoint) circle (0.60);
  \node[font=\scriptsize, align=center] at (0.95,0.08)
    {chart ball\\[-1mm]$B(x,c_0\tau)$};
  \draw[->, thin] (1.32,0.34) -- (1.60,0.72);
  \fill (8.1,1.0) circle (1.5pt);
  \node[right, font=\scriptsize] at (8.16,1.02) {$\partial F$};
  \draw[decorate, decoration={brace, amplitude=4pt, mirror}, thick]
    (6.7,0.83) -- (8.05,0.95);
  \node[below, font=\scriptsize, align=center] at (6.9,0.58)
    {collar\\[-1mm]$[0,c_0\tau]\times\partial F$};
  \draw[very thick] (0.5,-1.15) .. controls (3.0,-0.75) and (5.2,-1.2) .. (7.0,-0.95);
  \node[font=\scriptsize] at (5.0,-1.5) {another sheet of $F$};
  \draw[<->] (5.1,0.90) -- (5.1,-1.05);
  \node[fill=white, inner sep=1pt, font=\scriptsize] at (5.1,-0.05) {$\ge 2\tau$};
  \draw[<->, thin] (7.08,-0.95) -- (8.1,-0.95);
  \node[fill=white, inner sep=1pt, font=\scriptsize] at (7.58,-0.68)
    {$\ge\tfrac{c_0^2}{4}\tau$};
  \node[font=\scriptsize, align=center] at (7.42,-1.42)
    {peripheral\\[-1mm]clearance};
\end{tikzpicture}
\caption{Relative surface thickness
(\Cref{def:relative-surface-thickness}), in cross-section.  Reach at least
\(\tau\) means the normal tube of radius \(\tau\) about \(\overline F\) is
embedded (shaded).  This simultaneously bounds the curvature of each sheet, so
that \(F\) is locally a graph over \(T_xF\) within the dashed chart ball
\(B(x,c_0\tau)\).  In the drawing,
\(P\) denotes the cross-sectional plane and the thin line is
\(T_xF\cap P\).  Positive reach also keeps distinct sheets separated;
and, through the embedded collar \([0,c_0\tau]\times\partial F\) along
\(\partial F\subset\partial E(\gamma)\), controls the boundary stratum.  The
peripheral clearance clause additionally keeps interior sheets a definite
distance from \(\partial E(\gamma)\).  These
are the only consequences of relative thickness used in the compactness,
finiteness, and reconstruction theorems.}
\label{fig:relative-thickness}
\end{figure}
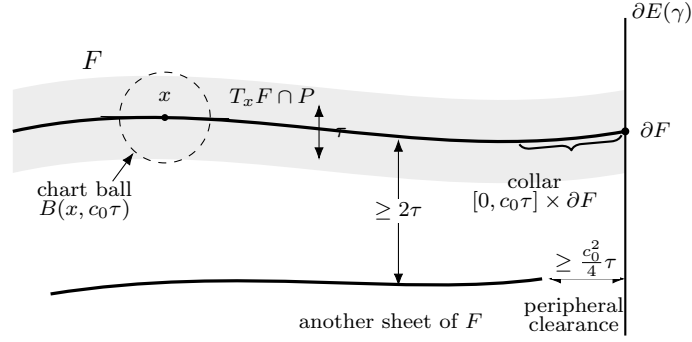

\begin{remark}[Bounded-geometry hypothesis]
Relative surface thickness is a bounded-geometry hypothesis, not a topological
property of the surface class.  Stating it through positive reach incorporates
both local curvature control and sheet separation, including the boundary
collars, and avoids treating the interior, boundary, and separation conditions
as separate pieces of terminology.
\end{remark}

\begin{remark}[Surface thickness scale]
Unlike the knot representative, the surface is not naturally normalized to
have thickness one after the normalization \(\Thi(\gamma)=1\).  The same
ambient scaling fixes the knot thickness and simultaneously scales the
surface.  Thus a surface in \(E(\gamma)\) should be assigned its own positive
thickness scale \(\tau\).

For this reason, the basic filtered pair space below uses the condition
\(\Thi(F)\geq\tau\), where \(\tau>0\) is an additional parameter.  The case
\(\tau=1\) is a convenient bounded-geometry subcase, but it should not be
interpreted as automatic for all essential surface representatives.  For
finite-resolution arguments one works at scales
\[
  0<\varepsilon\ll \tau.
\]
\end{remark}

\begin{remark}[Non-emptiness and richness]
The condition \(\Thi(F)\geq\tau\) selects bounded-geometry representatives of
surface classes.  It should not be interpreted as saying that every
essential surface automatically satisfies a prescribed lower bound
\(\tau\) in a fixed unit-thickness knot exterior.  For a fixed smooth compact
properly embedded representative, after smoothing and putting it in general
bounded-geometry position, one often obtains positive relative thickness for
some \(\tau>0\).  The value of \(\tau\), however, may depend on the chosen
representative and on the ambient normalization.
\end{remark}

\subsection{The filtered pair space}

\begin{definition}[Three equivalence relations for pairs]
Let \((\gamma_0,F_0)\) and \((\gamma_1,F_1)\) be pairs of the same type.
We use three related, but distinct, equivalence relations.

\begin{enumerate}[label=(\roman*),leftmargin=2em]
\item \emph{Pair isotopy.}  A pair isotopy is an ambient isotopy
\(h_t:S^3\to S^3\), \(t\in[0,1]\), such that
\[
  h_0=\operatorname{id},
  \qquad h_t(\gamma_0)=\gamma_t,
  \qquad h_t(F_0)=F_t\subset E(\gamma_t),
\]
and whose endpoint satisfies \(h_1(\gamma_0)=\gamma_1\) and
\(h_1(F_0)=F_1\).  This is the equivalence relation for topological classification, not for the underlying geometric configuration space.

\item \emph{Admissible thick deformation.}  At a fixed level
\((\Lambda,\Delta,\tau)\), an admissible thick deformation is a pair isotopy
through pairs \((\gamma_t,F_t)\), \(t\in[0,1]\), such that
\[
  \Thi(\gamma_t)=1,
  \qquad \Len(\gamma_t)\leq \Lambda,
  \qquad \Thi(F_t)\geq \tau,
  \qquad \Area(F_t)\leq \Delta
\]
for all \(t\), with the required topological condition, such as essentiality,
preserved throughout.  These deformations define the admissible components used
in persistence.

\item \emph{Finite-resolution equivalence.}  After a scale \(\varepsilon\) and
a finite encoding scheme have been fixed, two pairs are
\(\varepsilon\)-equivalent if their layered finite-resolution codes are
combinatorially isomorphic.  This relation is coarser than smooth data and is
used only for finite-resolution finiteness.
\end{enumerate}

Thus classification, filtered connectivity, and finite-resolution recognition
are deliberately separated.  The space \(\calZ_{\Lambda,\Delta,\tau}(K)\) is
a space of geometric representatives modulo rigid motions and source reparametrizations. Pair isotopy is applied only when forming the classification set. Admissible thick deformations define the filtered components of the geometric space.
\end{definition}

\begin{definition}[Ropelength--area--thickness filtered surface pair space]
For a knot type \(K\) and parameters \(\Lambda,\Delta,\tau>0\), define
\[
  \calZ_{\Lambda,\Delta,\tau}(K)
  =
  \left\{
  (\gamma,F)
  \ \middle|
  \begin{array}{l}
  \gamma\in Y^{=}_\Lambda(K),\quad \Thi(\gamma)=1,\quad \Len(\gamma)\leq\Lambda,\\
  F\subset E(\gamma)\ \text{a compact properly embedded surface}\\
  \quad\text{or finite surface system},\quad
  \Thi(F)\geq \tau,\quad \Area(F)\leq \Delta
  \end{array}
  \right\}/\!\sim .
\]
The relation \(\sim\) identifies rigid motions and source reparametrizations, with all stipulated orientations and labels retained. Both entries are interpreted as compatible geometric representatives before this simultaneous quotient. Admissible thick deformations define the admissible component relation; pair-isotopy classes are a further classification quotient.  The parameter
\(\Lambda\) controls the position of the knot representative, \(\Delta\)
controls the area level of the surface, and \(\tau\) records the
surface-side thickness scale.  When \(F\) is disconnected, \(\Area(F)\) denotes
the total area and \(\Thi(F)\) is the relative thickness of the union;
equivalently, the componentwise bounds hold and distinct components are
separated by at least twice the chosen scale.
The basic geometric configuration is illustrated in \Cref{fig:pair-window}.
\end{definition}

\begin{figure}[t]
\centering
\begin{tikzpicture}[x=1cm,y=1cm,>=Latex,
  panel/.style={draw, rounded corners=2pt},
  title/.style={font=\footnotesize\bfseries},
  note/.style={font=\scriptsize, align=center}]
  \draw[panel] (0,-0.45) rectangle (6.2,4.7);
  \node[title] at (3.1,4.42) {thick knot and fixed-radius tube};
  \begin{scope}[shift={(3.1,2.15)}, scale=0.50]
    \foreach \a in {90,210,330}{
      \draw[white, line width=4.8pt]
        plot[variable=\t, domain=\a:\a+140, samples=70, smooth]
        ({sin(\t)+2*sin(2*\t)}, {cos(\t)-2*cos(2*\t)});
      \draw[black, line width=3.4pt]
        plot[variable=\t, domain=\a:\a+140, samples=70, smooth]
        ({sin(\t)+2*sin(2*\t)}, {cos(\t)-2*cos(2*\t)});
      \draw[black!15, line width=2.2pt]
        plot[variable=\t, domain=\a:\a+140, samples=70, smooth]
        ({sin(\t)+2*sin(2*\t)}, {cos(\t)-2*cos(2*\t)});
    }
  \end{scope}
  \node[note, anchor=west] at (4.42,3.85) {$N_{\rho_0}(\gamma)$};
  \draw[->, thin] (4.55,3.70) -- (3.72,3.22);
  \node[note] at (1.0,4.05) {$E(\gamma)$};
  \node[note] at (3.1,-0.13)
    {$\Thi(\gamma)=1$,\quad $\Len(\gamma)\leq\Lambda$};
  \draw[panel] (6.7,-0.45) rectangle (13.4,4.7);
  \node[title] at (10.05,4.42) {cross-section of the pair $(\gamma,F)$};
  \draw[black!8, line width=10pt]
    (8.921,2.717) .. controls (10.3,2.5) and (10.3,1.6) .. (8.921,1.383);
  \foreach \yc in {3.1,1.0}{
    \draw[dashed, thin] (8.6,\yc) circle (1.02);
    \fill[black!15] (8.6,\yc) circle (0.5);
    \draw[thick] (8.6,\yc) circle (0.5);
    \fill (8.6,\yc) circle (1.2pt);
  }
  \node[note, below] at (8.6,3.05) {$\gamma$};
  \draw[->, thin] (8.6,3.1) -- (8.1,3.1);
  \node[note, above] at (8.35,3.1) {$\rho_0$};
  \draw[->, thin] (8.6,1.0) -- (7.70,0.60);
  \node[note] at (8.02,0.62) {$1$};
  \node[note, anchor=west] at (9.95,4.05) {$\partial E(\gamma)$};
  \draw[->, thin] (9.98,3.98) -- (8.97,3.47);
  \draw[very thick]
    (8.921,2.717) .. controls (10.3,2.5) and (10.3,1.6) .. (8.921,1.383);
  \fill (8.921,2.717) circle (1.2pt);
  \fill (8.921,1.383) circle (1.2pt);
  \node[note, anchor=west] at (10.12,2.62) {$F$};
  \node[note, anchor=west] at (9.05,1.05) {$\partial F$};
  \draw[<->, thin] (9.955,2.05) -- (10.30,2.05);
  \node[note, anchor=west] at (10.30,2.05) {$\tau$};
  \node[note] at (12.35,3.90) {$E(\gamma)$};
  \node[note, fill=white, inner sep=1.5pt] at (10.05,-0.20)
    {$\Area(F)\leq\Delta$,\quad $\Thi(F)\geq\tau$};
\end{tikzpicture}
\caption{The objects of the filtered pair space
\(\calZ_{\Lambda,\Delta,\tau}(K)\), schematically.  Left: a unit-thickness
representative \(\gamma\) with \(\Len(\gamma)\leq\Lambda\) is replaced by the
embedded tube \(N_{\rho_0}(\gamma)\) of fixed radius \(\rho_0=\tfrac12\); the
exterior is \(E(\gamma)=S^3\setminus\operatorname{int}N_{\rho_0}(\gamma)\).
Right: in a cross-sectional plane, each strand of \(\gamma\) contributes a
solid disk of radius \(\rho_0\) (shaded) sitting inside the embedded normal
disk of radius \(1=\Thi(\gamma)\) (dashed); the essential surface \(F\) is
properly embedded in the exterior, with \(\partial F\) on the peripheral torus
\(\partial E(\gamma)\), total area at most \(\Delta\), and relative thickness
at least \(\tau\) (light band).  The three budgets
\(\Lambda,\Delta,\tau\) respectively control the position of the knot, the
amount of surface, and the curvature and sheet separation of the surface.}
\label{fig:pair-window}
\end{figure}
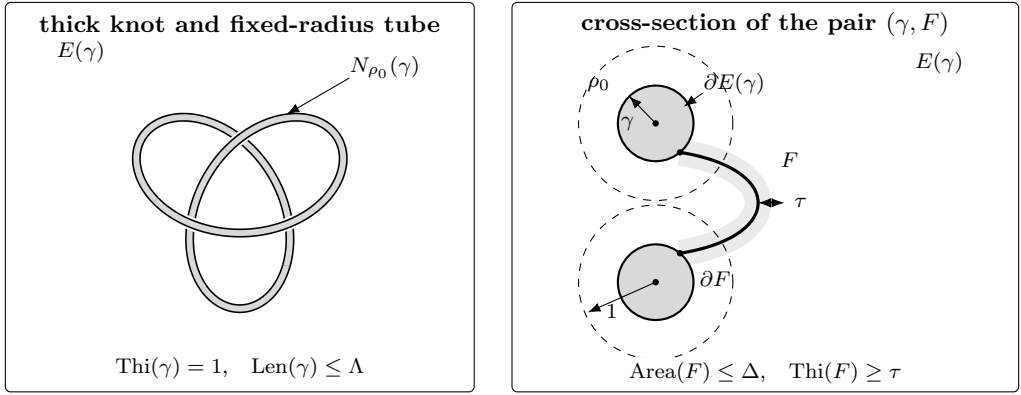

\begin{remark}[Relative thickness of a disconnected surface]
\label{rem:disconnected-thickness}
For a disconnected surface \(F=F_1\sqcup\cdots\sqcup F_m\), the condition
\(\Thi(F)\geq\tau\) is understood as a single positive-reach condition on the
union: \(\operatorname{reach}(\overline F)\geq\tau\) in the sense of
\Cref{def:relative-surface-thickness}.  A single reach bound on the union
encodes at once that each component has reach at least \(\tau\) and that distinct
components stay at mutual distance at least \(2\tau\); in particular the
boundary collars of the components are pairwise disjoint.  This is the reading
used in \Cref{lem:compactness-thick-surfaces} and throughout, and it is why the
component count and the compactness argument apply verbatim to finite surface
systems such as JSJ frontiers.
\end{remark}

\begin{definition}[Essential subspace]
The essential subspace
\[
  \calZ_{\Lambda,\Delta,\tau}^{\mathrm{ess}}(K)
  \subset
  \calZ_{\Lambda,\Delta,\tau}(K)
\]
is the subspace consisting of pairs \((\gamma,F)\) such that every component of
\(F\) is essential in \(E(\gamma)\), that is, incompressible (for two-sided
surfaces that are not spheres, equivalently \(\pi_1\)-injective, by the loop
theorem; one-sided surfaces are treated by the usual geometric-disk
convention),
boundary-incompressible, and not boundary-parallel.  For \(2\)-spheres we use
the standard convention: a \(2\)-sphere is essential only if it bounds no ball
in \(E(\gamma)\); since knot exteriors are irreducible, no component of an
essential system is ever a sphere.  Parallel duplicate components are not
allowed unless the chosen surface type explicitly keeps multiplicity.
Allowing finite surface systems is necessary for JSJ tori
and other characteristic systems; statements that require connectedness will
say so explicitly.
\end{definition}

If \(\Lambda\leq \Lambda'\), \(\Delta\leq \Delta'\), and \(\tau\geq \tau'\),
then, by the scale monotonicity of relative thickness
(\Cref{lem:thickness-scale-monotonicity}), which guarantees that
\(\Thi(F)\geq\tau\geq\tau'\) implies \(\Thi(F)\geq\tau'\), there are natural
inclusions
\[
  \calZ_{\Lambda,\Delta,\tau}(K)
  \into
  \calZ_{\Lambda',\Delta',\tau'}(K),
  \qquad
  \calZ_{\Lambda,\Delta,\tau}^{\mathrm{ess}}(K)
  \into
  \calZ_{\Lambda',\Delta',\tau'}^{\mathrm{ess}}(K).
\]
Thus \(\{\calZ_{\Lambda,\Delta,\tau}(K)\}\) and
\(\{\calZ_{\Lambda,\Delta,\tau}^{\mathrm{ess}}(K)\}\) form
three-parameter filtrations.  The direction of the \(\tau\)-parameter is
opposite to that of \(\Lambda\) and \(\Delta\): allowing smaller surface
thickness enlarges the space.

\begin{remark}[One-sided and coupled filtrations]
The coupled space \(\calZ_{\Lambda,\Delta,\tau}(K)\) contains useful
one-sided filtrations.  If a representative \(\gamma\in Y^{=}_\Lambda(K)\) is
fixed, one obtains the fixed-exterior surface filtration
\[
  \calZ_{\Delta,\tau}^{\mathrm{ess}}(\gamma)
  =
  \{F\subset E(\gamma)\mid
  F\text{ is essential},\ \Thi(F)\geq\tau,\ \Area(F)\leq\Delta\}/\!\sim .
\]
This studies the essential surface theory of one fixed geometric knot
exterior.  Conversely, one may fix a surface class and vary the knot
representative, obtaining a position filtration constrained by that surface
class.  The full space \(\calZ_{\Lambda,\Delta,\tau}(K)\) is the coupled object
in which both sides vary; it is the natural setting for the knot--surface
interaction studied here.
\end{remark}

\subsection{Surface types and visibility levels}

The pair space becomes more useful once one restricts the surface coordinate to
a specified topological or geometric type.  The following definition isolates
this common construction.

\begin{definition}[Filtered surface type]
A filtered surface type \(\mathfrak S\) for the knot type \(K\) is a rule which,
for each representative \(\gamma\in Y^{=}_\Lambda(K)\), selects a collection of
properly embedded surfaces in \(E(\gamma)\), invariant under pair isotopy.  We
write
\[
  \calZ_{\Lambda,\Delta,\tau}^{\mathfrak S}(K)
  \subset
  \calZ_{\Lambda,\Delta,\tau}^{\mathrm{ess}}(K)
\]
for the corresponding subspace.  Examples include essential annuli, essential
tori in the characteristic submanifold, surfaces of a fixed boundary slope,
minimal genus Seifert surfaces, fiber surfaces, and a fixed
isotopy class of an essential surface transported through pair isotopy.
\end{definition}

\begin{definition}[Surface visibility level]
\label{def:surface-visibility-level}
Let \(\mathfrak S\) be a filtered surface type and fix \(\Delta,\tau>0\).  The
\((\Delta,\tau)\)-visibility level of \(\mathfrak S\) is
\[
  \beta_{\mathfrak S}(K;\Delta,\tau)
  =
  \inf\left\{
  \Lambda
  \ \middle|
  \calZ_{\Lambda,\Delta,\tau}^{\mathfrak S}(K)\neq\varnothing
  \right\},
\]
with the convention that the infimum is \(\infty\) if the set is empty.  Dually,
for fixed \(\Lambda,\tau\), the area visibility level is
\[
  a_{\mathfrak S}(K;\Lambda,\tau)
  =
  \inf\left\{
  \Delta
  \ \middle|
  \calZ_{\Lambda,\Delta,\tau}^{\mathfrak S}(K)\neq\varnothing
  \right\}.
\]
The collection of functions
\[
  (\Delta,\tau)\longmapsto \beta_{\mathfrak S}(K;\Delta,\tau)
\]
is called the visibility spectrum of \(\mathfrak S\).
\end{definition}

\Cref{fig:visibility-diagram} displays the two visibility levels in one
picture: for fixed \(\tau\), the parameter pairs \((\Lambda,\Delta)\) at which
the type \(\mathfrak S\) is visible form a monotone region of the plane, and
\(\beta_{\mathfrak S}\) and \(a_{\mathfrak S}\) are the horizontal and
vertical entry coordinates of this region.

\begin{figure}[t]
\centering
\begin{tikzpicture}[x=1cm,y=1cm,>=Latex,
  note/.style={font=\scriptsize, align=center}]
  \fill[black!8]
    (2.6,4.4) .. controls (3.0,2.2) and (4.4,1.35) .. (8.6,1.05)
    -- (8.6,4.4) -- cycle;
  \draw[thick]
    (2.6,4.4) .. controls (3.0,2.2) and (4.4,1.35) .. (8.6,1.05);
  \draw[->] (0,0) -- (9.3,0) node[note, below left, xshift=2mm] {$\Lambda$};
  \draw[->] (0,0) -- (0,4.9) node[note, left] {$\Delta$};
  \draw[dashed, thin] (1.6,0) -- (1.6,4.4);
  \node[note, below] at (1.6,-0.05) {$\Rop(K)$};
  \node[note] at (6.6,3.15)
    {$\calZ^{\mathfrak S}_{\Lambda,\Delta,\tau}(K)\neq\varnothing$};
  \node[note, fill=white, inner sep=1.5pt] at (2.7,0.62)
    {$\calZ^{\mathfrak S}_{\Lambda,\Delta,\tau}(K)=\varnothing$};
  \node[note, anchor=west] at (4.75,4.05) {trade-off frontier};
  \draw[->, thin] (4.72,3.95) -- (3.30,3.35);
  \draw[dashed, thin] (0,2.0) -- (4.18,2.0);
  \draw[dashed, thin] (4.18,2.0) -- (4.18,0);
  \fill (4.18,2.0) circle (1.3pt);
  \node[note, left] at (0,2.0) {$\Delta^{*}$};
  \node[note, below] at (4.18,-0.05) {$\beta_{\mathfrak S}(K;\Delta^{*},\tau)$};
  \draw[dashed, thin] (6.4,0) -- (6.4,1.30);
  \draw[dashed, thin] (6.4,1.30) -- (0,1.30);
  \fill (6.4,1.30) circle (1.3pt);
  \node[note, below] at (6.4,-0.05) {$\Lambda^{*}$};
  \node[note, left] at (0,1.30) {$a_{\mathfrak S}(K;\Lambda^{*},\tau)$};
  \node[note, anchor=east] at (9.2,4.6) {(fixed $\tau$)};
\end{tikzpicture}
\caption{Visibility levels of a filtered surface type \(\mathfrak S\)
(\Cref{def:surface-visibility-level}), for a fixed thickness scale \(\tau\).
The shaded region consists of the budgets \((\Lambda,\Delta)\) at which
\(\mathfrak S\) admits a representative; by the filtration inclusions it is
closed under increasing \(\Lambda\) and \(\Delta\).  Entering the region
horizontally at area budget \(\Delta^{*}\) gives the ropelength visibility
level \(\beta_{\mathfrak S}\); entering vertically at length budget
\(\Lambda^{*}\) gives the area visibility level \(a_{\mathfrak S}\).  The
boundary curve is the knot--surface trade-off frontier studied in
\Cref{sec:ideal}: at a given area budget, its horizontal distance from the
line \(\Lambda=\Rop(K)\) measures the extra ropelength room needed before the
surface type becomes visible.}
\label{fig:visibility-diagram}
\end{figure}
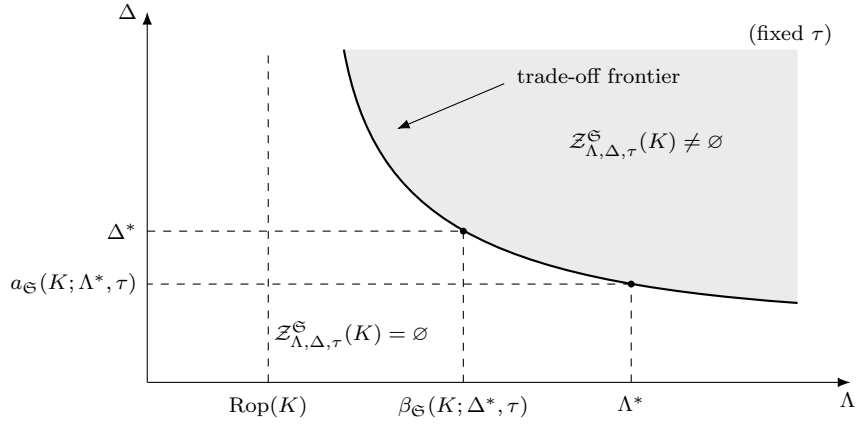

\begin{proposition}[Basic monotonicity of visibility]
\label{prop:visibility-monotonicity}
For every filtered surface type \(\mathfrak S\), the function
\(\beta_{\mathfrak S}(K;\Delta,\tau)\) is nonincreasing in \(\Delta\) and
nondecreasing in \(\tau\).  More precisely, if
\(\Delta\leq\Delta'\) and \(\tau\geq\tau'\), then
\[
  \beta_{\mathfrak S}(K;\Delta',\tau')
  \leq
  \beta_{\mathfrak S}(K;\Delta,\tau).
\]
The same monotonicity holds for the corresponding admissible-component birth
levels and finite-resolution visibility levels.
\end{proposition}

\begin{proof}
The inclusions of the filtered pair spaces give
\[
  \calZ_{\Lambda,\Delta,\tau}^{\mathfrak S}(K)
  \subset
  \calZ_{\Lambda,\Delta',\tau'}^{\mathfrak S}(K)
\]
whenever \(\Delta\leq\Delta'\) and \(\tau\geq\tau'\).  Taking the infimum over
all \(\Lambda\) for which the source is nonempty gives the stated inequality.
The component and finite-resolution versions follow from the same inclusion.
\end{proof}

\begin{remark}[Why visibility is not purely topological]
The exterior \(E(K)\) is fixed up to homeomorphism, so a topological class of
essential surfaces exists or does not exist independently of \(\Lambda\).  The
visibility level is therefore not a new existence invariant of the topological
exterior alone.  It becomes meaningful only after one imposes geometric
constraints on the representative \(\gamma\) and on the surface \(F\), namely
ropelength, area, and relative thickness.
\end{remark}

\subsection{Seifert and taut subfiltrations}

\begin{definition}[Seifert-surface subspace]
Let
\[
  \calS_{\Lambda,\Delta,\tau}^{\mathrm{ess}}(K)
  \subset
  \calZ_{\Lambda,\Delta,\tau}^{\mathrm{ess}}(K)
\]
be the subspace consisting of pairs \((\gamma,F)\) such that \(F\) is a
connected orientable Seifert surface for \(\gamma\).  Equivalently, the
boundary of \(F\) is the preferred longitude on \(\partial E(\gamma)\).
\end{definition}

For Seifert surfaces, the boundary slope is fixed.  The primary topological
quantity is genus.

\begin{definition}[Taut subspace]
Let
\[
  \calS_{\Lambda,\Delta,\tau}^{\mathrm{taut}}(K)
  \subset
  \calS_{\Lambda,\Delta,\tau}^{\mathrm{ess}}(K)
\]
be the subspace consisting of pairs \((\gamma,F)\) such that \(F\) is
Thurston norm minimizing in its relative homology class \cite{ThurstonNorm}.  For a knot in
\(S^3\), this is equivalent to \(F\) being a minimal genus Seifert surface.
\end{definition}

The taut subspace is the natural geometric refinement of the Kakimizu complex:
it records not only isotopy classes of minimal genus Seifert surfaces but also
their realization by thick representatives and bounded-area surfaces.

\section{Effective degrees of freedom}
\label{sec:dof}

The previous discussion suggests a unified definition of freedom.  Instead of
separating geometric freedom, positional freedom, and topological freedom at
the level of definitions, we define one effective finite-resolution degree of
freedom.  The usual invariants, such as crossing number, bridge number, trunk,
genus, boundary slope, and representativity, are then treated as observable
quantities controlled by this finite data.

\subsection{Classical combinatorial shadows}
\label{subsec:classical-combinatorial-shadows}

The geometric code is defined before any triangulation of the exterior is
chosen.  Nevertheless, a uniform subdivision of the ambient lattice into
tetrahedra turns a sufficiently transverse reconstructed pair into a
decorated triangulation.  The occupancy and tangent labels bound the number
of tetrahedra met by the curve and surface, and a further controlled local
subdivision permits normalization of the reconstructed surface.  Dually, the
same occupied cells determine a decorated polyhedral spine.  Thus normal
coordinates and spine data are classical combinatorial shadows of the finite
geometric code \cite{JacoRubinstein0Efficient,MatveevBook}.

No equivalence theorem with Pachner moves or Matveev--Piergallini moves is
needed for the faithfulness result.  The positive-reach proof compares pairs
directly in the moving smooth exterior.  The triangulation and spine shadows
instead provide possible verification layers after reconstruction: they can
be used to normalize the surface, apply classical essentiality algorithms,
and certify equivalence by familiar finite combinatorial moves.  This
separation between geometric reconstruction and combinatorial verification
will be used again in \Cref{sec:finite-recognition}.

\subsection{Finite-resolution freedom}

The term ``degree of freedom'' is used here in an effective, finite-resolution
sense.  A smooth knot or a smooth surface is an infinite-dimensional object.
Thus it would be misleading to claim that the thickness and length or area
bounds make the corresponding smooth space finite-dimensional.  What these
bounds do is different: they make the object finite at every fixed geometric
resolution.

Fix \(\varepsilon>0\).  The parameter \(\varepsilon\) is an observational
scale, or finite resolution.  Two features whose size or separation is much
smaller than \(\varepsilon\) are not resolved at this level.  If a curve has
thickness at least one, then it cannot oscillate meaningfully below the
thickness scale.  If, in addition, its length is at most \(\Lambda\), then at
resolution \(\varepsilon\) it can be encoded by a controlled number of
segments.  Similarly, if a surface has thickness at least \(\tau\) and area at
most \(\Delta\), then for \(0<\varepsilon\ll\tau\) it can be encoded by a
controlled number of two-dimensional cells.

Thus the effective degree of freedom is not an intrinsic dimension of the
smooth moduli space.  It is the minimal amount of finite data required to
describe the object at scale \(\varepsilon\).  This viewpoint is essential for
the finite recognition program: finite recognition should not be interpreted
as replacing smooth knot theory by a finite-dimensional theory, but as saying
that bounded-thickness, bounded-size representatives admit finite
scale-dependent encodings from which the relevant topological and positional
data can be read.

The same definition applies uniformly to curves, surfaces, and pairs.  For a
curve, the data are essentially the cells of a polygonal approximation.  For a
surface, the data are the triangles of a thickness-adapted triangulation.  For
a pair \((\gamma,F)\), the data must encode both objects simultaneously, including
their relative position in the exterior.  In this way, geometric data,
positional data, and topological data are not separated at the level of the
definition.  Instead, they are treated as observables extracted from a single
finite-resolution model.

\begin{definition}[Effective \(\varepsilon\)-degree of freedom]
Let \(X\) be a thick curve, a thick surface, or a thick pair.  The effective
\(\varepsilon\)-degree of freedom \(\DoF_\varepsilon(X)\) is defined using
finite-resolution covering data.

For a curve \(\gamma\), \(\DoF_\varepsilon(\gamma)\) is the minimum number of
arcs of length at most \(\varepsilon\) needed to cover \(\gamma\), equivalently
the minimum number of edges in an \(\varepsilon\)-polygonal model.

For a surface \(F\), \(\DoF_\varepsilon(F)\) is the minimum number of
bounded-geometry surface patches of intrinsic diameter at most
\(\varepsilon\), including half-patches near \(\partial F\), needed to cover
\(F\).  Equivalently, after choosing a thickness-adapted triangulation, it is
the number of triangles up to uniform multiplicative constants.

For a pair \((\gamma,F)\), \(\DoF_\varepsilon(\gamma,F)\) is computed using a
layered model consisting of an \(\varepsilon\)-model for \(\gamma\), an
\(\varepsilon\)-model for \(F\), and finite incidence data recording their
relative position.  A common simplicial subdivision is not required.
\end{definition}

\begin{remark}[Finite resolution versus thickness scale]
\label{rem:dof-not-dimension}
The resolution parameter \(\varepsilon\) is not an additional geometric
constraint on the pair.  It is the scale at which we sample or encode the
pair.  The thickness scale \(\tau\) is part of the geometric input, whereas
\(\varepsilon\) is chosen afterwards, usually with \(0<\varepsilon\leq c\min\{1,\tau\}\).
Thus \(\DoF_\varepsilon\) measures the amount of finite data needed to see a
bounded-geometry pair at scale \(\varepsilon\), not the dimension of the
smooth moduli space.
\end{remark}

\begin{remark}[Relation to covering numbers and \(\varepsilon\)-entropy]
\label{rem:dof-entropy}
The quantity \(\DoF_\varepsilon\) is not a new invariant but a bounded-geometry
covering number.  The minimal number of \(\varepsilon\)-cells needed to cover a
set is the covering number whose logarithm is the Kolmogorov--Tikhomirov
\(\varepsilon\)-entropy of the set \cite{KolmogorovTikhomirov}, and the growth
laws used here --- \(O(\Len(\gamma)/\varepsilon)\) for a curve and
\(O_\tau(\Area(F)/\varepsilon^2)\) for a surface --- are the standard covering
or box-counting estimates for a rectifiable curve and a bounded-area surface.
The role of positive reach is only to make these estimates uniform and to
bound the incidence layer, exactly as reach controls the passage from an
\(\varepsilon\)-sample to the topology of a submanifold in the reconstruction
theorem \cite{NiyogiSmaleWeinberger}.  Accordingly \(\DoF_\varepsilon(\gamma,F)\)
should be read as the size of a thickness-adapted \(\varepsilon\)-net for the
pair; the term ``degree of freedom'' refers to the amount of finite data
resolved at scale \(\varepsilon\), not to a dimension, as stressed in
\Cref{rem:dof-not-dimension}.
\end{remark}

\paragraph{Layered model for pairs.}
For a pair \((\gamma,F)\), the finite-resolution model is not required to be
a common simplicial subdivision of the curve and the surface.  Instead, we
adopt a \emph{layered cellular model} consisting of:
\begin{itemize}[leftmargin=2em]
\item a polygonal model for \(\gamma\), with edges of length at most
\(\varepsilon\);
\item a triangulated model for \(F\), with triangles of diameter at most
\(\varepsilon\);
\item finite incidence data recording their relative position in the exterior,
for example which curve edges and surface triangles lie in the same or
neighbouring ambient \(\varepsilon\)-cells, and which local separation, boundary,
or crossing type is visible at that resolution.
\end{itemize}
No common refinement at intersection points is required; the model keeps the
curve and the surface as separate layers.  The geometric part of the degree of
freedom is additive in the number of curve edges and surface triangles.  The
relative-position layer is bookkeeping over a finite alphabet.  If one records
all possible pairwise incidence slots, this bookkeeping has at most quadratic
size in the number of geometric cells; this affects the explicit number of
codes but not the linear-plus-area estimate for the geometric degrees of
freedom.

\begin{definition}[Layered finite-resolution encoding scheme]
A finite-resolution encoding scheme at scale \(\varepsilon\) is a fixed finite
alphabet and a fixed set of rules for encoding the following layered data.
\begin{enumerate}[label=(\roman*),leftmargin=2em]
\item \emph{Curve layer.}  The curve \(\gamma\) is sampled by arclength at
spacing at most \(\varepsilon/2\) and replaced by a polygonal model
\(\widehat\gamma\).  Its code records the ordered edge list, quantized tangent
or direction data, and local thickness-adapted type symbols.  For
\(\Len(\gamma)\leq\Lambda\), the number of edges is
\(O(\Lambda/\varepsilon)\).
\item \emph{Surface layer.}  The surface \(F\) is covered or triangulated by
bounded-geometry patches of intrinsic diameter at most \(\varepsilon\), with
half-patches near \(\partial F\).  The code records the cell adjacency,
boundary/collar type, and local sheet type.  For \(\Area(F)\leq\Delta\) and
\(\Thi(F)\geq\tau\), the number of surface cells is
\(O_\tau(\Delta/\varepsilon^2)\).
\item \emph{Relative-position layer.}  The code records coarse incidence data
between curve cells and surface cells, together with local crossing or
separation symbols visible at scale \(\varepsilon\).  The curve and surface
layers are not required to have a common subdivision.
\end{enumerate}
Optional labels, such as boundary slope data, may be added when a slope-height
cutoff or boundary-torus metric control has been imposed.  They are not part of
the basic finiteness theorem.

The scheme is fixed once and for all before applying the finiteness theorem.
The theorem below asserts finiteness of the resulting encoded
\(\varepsilon\)-types; it does not assert that the underlying smooth isotopy
classes are finite.
\end{definition}

\begin{remark}[Discrete admissible deformations]
An admissible thick deformation can also be sampled in time.  At scale
\(\varepsilon\), one obtains a finite sequence of layered codes and admissible
transitions between them.  Thus, after a finite encoding scheme is fixed,
admissible-component persistence may be approximated by a finite-state
transition system at each bounded level.  This is an implementation viewpoint,
not an additional equivalence relation in the smooth theory; the transition
complexes of \Cref{def:admissible-transition-complex} are defined by a fixed
sampling-free elementary transition rule, and
\Cref{lem:sampling-independence} shows that every sampled description
refines to it.
\end{remark}

\begin{construction}[A concrete cubical--triangulated encoding]
\label{con:concrete-encoding}
Fix \(0<\varepsilon\leq c\min\{1,\tau\}\), a curve-direction
quantization scale \(\delta_{\gamma}>0\), and a surface-direction
quantization scale \(\delta_{F}>0\).  The code space and the encoding map are
specified as follows.  All spatial labels are taken from one fixed lattice in \(\R^3\),
so the code of a pair does not change when the same pair is viewed at a larger
budget and codes arising at different ropelength levels can be compared
literally.

\emph{(a) Global lattice and bounded window.}  Fix once and for all the
half-open cubical lattice
\[
  \mathcal Q_\varepsilon
  =\bigl\{\varepsilon(k+[-\tfrac12,\tfrac12)^3)\mid k\in\Z^3\bigr\},
\]
anchored at the origin.  After rigid-motion normalization (the barycenter of
\(\gamma\) is at the origin), the essential pair lies inside the axis-parallel
cube \(Q\) of side
\[
  S(\Lambda,\Delta,\tau)
  =
  2\Bigl(\tfrac\Lambda2+1+C_{\mathrm{diam}}\tfrac\Delta\tau\Bigr),
\]
by the anchoring and diameter lemmas
(\Cref{lem:anchoring,lem:diameter-bound}; they are proved in
\Cref{sec:bounded-complexity} from the thickness convention alone and do not
depend on the present section).  Note that the window necessarily
depends on \((\Lambda,\Delta,\tau)\), not on \(\Lambda\) alone: without the
diameter control provided by \(\Thi(F)\geq\tau\), an essential surface of
large area could extend far outside any cube whose size depends only on the
curve.  Let \(\mathcal I_\varepsilon(\Lambda,\Delta,\tau)\subset\Z^3\) be the
finite set of indices of lattice cubes meeting \(Q\).  Its cardinality is at
most \(C(1+S/\varepsilon)^3\).  The ambient lattice is global; only the finite
set of labels that can occur is budget-dependent.

\emph{(b) Finite alphabets.}  Fix a \(\delta_{\gamma}\)-net
\(\mathcal D_{\gamma}\subset S^2\) for oriented knot-tangent directions, a
\(\delta_F\)-net \(\mathcal D_{\partial}\subset S^2\) for oriented boundary
tangents, and a \(\delta_F\)-net
\(\widehat{\mathcal D}_{F}\subset\mathbb{RP}^2\) of \emph{unoriented}
directions for the normal lines of the surface.  The latter convention is
necessary because the surfaces of this paper may be non-orientable: no
globally consistent unit normal need exist, whereas the unoriented normal line
is always well defined.  Also fix the finite
set \(\mathcal B\) of boundary-collar and local-sheet type symbols specified by
the bounded-geometry convention, and the finite set \(\mathcal J\) of coarse
incidence symbols.

\emph{(c) Curve layer.}  For an oriented and based arclength
parametrization, encode \(\gamma\) by the ordered cyclic list
\((q_1,u_1),\dots,(q_{N_\gamma},u_{N_\gamma})\), where
\(q_i\in\mathcal I_\varepsilon(\Lambda,\Delta,\tau)\) is the global lattice
cube containing the \(i\)-th point of an arclength \(\varepsilon/2\)-net and
\(u_i\in\mathcal D_{\gamma}\) is the quantized tangent direction there.  The final
canonicalization in part~(f) removes the choices of orientation, basepoint,
and admissible net.

\emph{(d) Surface and boundary layers.}  Encode \(F\) by an admissible
bounded-geometry triangulation with mesh at most \(\varepsilon\), fixed
shape-regularity constants, and the uniform cell bound of
\Cref{prop:basic-dof}, recording
for each triangle \(t\) the tuple
\[
  (q_t,n_t,b_t)\in
  \mathcal I_\varepsilon(\Lambda,\Delta,\tau)
  \times\widehat{\mathcal D}_{F}\times\mathcal B
\]
of the global lattice cube containing its barycenter, its quantized
\emph{unoriented} normal line, and its boundary-collar label, together with the
adjacency relation on the triangle set.  In addition, record each boundary
component as a cyclic list of boundary edges (or an arclength
\(\varepsilon/2\)-net on that component), including for every listed boundary
sample its global cube and its quantized tangent in
\(\mathcal D_{\partial}\).  This explicit boundary layer is
what supplies the boundary-closeness clause in \Cref{lem:code-proximity}.
Recording normal lines rather than normal vectors makes the layer well defined
for non-orientable surfaces; all later uses of the normal data depend only on
the tangent plane, hence only on the unoriented line.

\emph{(e) Relative-position layer.}  Record, for each ordered pair of a curve
edge and a surface triangle lying in equal or neighbouring global lattice
cubes, an incidence symbol from \(\mathcal J\).

\emph{(f) Canonicalization.}  Here an admissible discretization is required to
satisfy the uniform curve- and surface-cell bounds of
\Cref{prop:basic-dof}; arbitrary further refinements are not admissible.  The
raw construction involves harmless finite or compact choices: rigid-motion
normalization, the two orientations and cyclic
basepoints of the knot parametrization, admissible arclength nets, admissible
bounded-geometry triangulations, orderings of components and cells, and
boundary basepoints.  For every such choice the output is a word in the same
finite alphabet.  Fix once and for all a numbering of the countable global alphabet and
the induced shortlex order on finite words, and define the encoded
\(\varepsilon\)-type to be the least raw word over \emph{all} admissible
choices.  Thus the encoding is a function of the geometric pair, not of an
auxiliary triangulation or parametrization.  Equality of canonical codes
means that both pairs admit raw realizations with one and the same word, which
is exactly the form used in \Cref{lem:code-proximity}.

A code is therefore a word in the finite alphabet
\(\mathcal I_\varepsilon(\Lambda,\Delta,\tau)
\times\mathcal D_{\gamma}\times\mathcal D_{\partial}
\times\widehat{\mathcal D}_{F}\times\mathcal B\times\mathcal J\)
of explicitly bounded length, and the set
\(\mathcal C(\Lambda,\Delta,\tau,\varepsilon)\) of all possible codes is an
explicit finite set.  The alphabet size is \emph{not} a function of \(\tau\)
alone: the grid-index set has at most \(C(1+S/\varepsilon)^3\) elements.
However, for
\(\varepsilon\leq c\min\{1,\tau\}\) one has
\(S/\varepsilon\leq C(\Lambda/\varepsilon+\Delta/(\tau\varepsilon))
\leq C'(\Lambda/\varepsilon+\Delta/\varepsilon^2)\leq C'N\), where
\(N=N_{\Lambda,\Delta,\tau,\varepsilon}\) is the cell bound of
\Cref{prop:basic-dof}.  In the faithful regime used below we take
\(\delta_{\gamma}\leq\varepsilon\), so
\(|\mathcal D_{\gamma}|=O(\varepsilon^{-2})\).  For a nonempty slice the
unit-thickness curve has a universal positive length lower bound, hence
\(N\geq c/\varepsilon\), and this extra factor is polynomial in \(N\).
Thus the alphabet has at most \(C''N^5\) letters, with \(C''\) depending only
on the conventions and the fixed surface-direction quantization.  Consequently
a pair with at most \(N\) curve/surface cells has at most
\[
  \bigl(C''N^5\bigr)^{BN^2}
  \ \leq\
  (CN)^{BN^2}
\]
possible codes, with universal \(B,C\) after absorbing constants.  The
quadratic exponent is only a crude allowance for adjacency and incidence
relations; no sharpness is intended.  The three layers and the mechanism by
which the code recovers topology are summarized in
\Cref{fig:layered-geometric-code}.
\end{construction}

\begin{figure}[t]
\centering
\resizebox{0.98\textwidth}{!}{%
\begin{tikzpicture}[
  x=1cm,y=1cm,>=Latex,
  panel/.style={draw, rounded corners=2pt},
  title/.style={font=\footnotesize\bfseries, align=center},
  note/.style={font=\scriptsize, align=center},
  token/.style={draw, rounded corners=1.5pt, fill=white, inner sep=3pt,
                font=\scriptsize},
  codebox/.style={draw, rounded corners=2pt, align=center, inner sep=4pt,
                  font=\scriptsize},
  arrow/.style={-{Latex[length=2mm]}, thick}
]
  \draw[panel] (0,2.65) rectangle (4.05,5.55);
  \node[title] at (2.025,5.30) {curve layer};
  \draw[step=0.62, black!25, very thin] (0.30,3.08) grid (3.75,4.92);
  \node[note, anchor=north west] at (0.38,4.86)
    {global lattice $\mathcal Q_\varepsilon$};
  \draw[very thick]
    (0.50,3.38) -- (1.18,3.66) -- (1.90,3.50) --
    (2.58,4.08) -- (3.48,3.78);
  \foreach \pnt in {(0.50,3.38),(1.18,3.66),(1.90,3.50),(2.58,4.08),(3.48,3.78)}
    \fill \pnt circle (1.25pt);
  \node[note, anchor=north, fill=white, inner sep=1pt] at (1.18,3.56) {$q_i$};
  \draw[->, thin] (2.58,4.08) -- (3.18,4.25);
  \node[note, anchor=west] at (3.20,4.25) {$u_i$};
  \node[token] at (2.025,2.90) {$\{(q_i,u_i)\}_i$};

  \draw[panel] (4.55,2.65) rectangle (8.60,5.55);
  \node[title] at (6.575,5.24) {surface and boundary\\layer};
  \draw[step=0.62, black!25, very thin] (4.85,3.08) grid (8.30,4.92);
  \fill[black!8]
    (5.08,3.25) -- (5.50,4.72) -- (7.90,4.52) -- (8.08,3.28) -- cycle;
  \draw[thick]
    (5.08,3.25) -- (5.50,4.72) -- (7.90,4.52) -- (8.08,3.28) -- cycle;
  \draw[thin] (5.08,3.25) -- (6.42,3.98) -- (5.50,4.72);
  \draw[thin] (6.42,3.98) -- (7.90,4.52);
  \draw[thin] (6.42,3.98) -- (8.08,3.28);
  \draw[very thick] (5.08,3.25) -- (8.08,3.28);
  \fill (6.42,3.98) circle (1.15pt);
  \draw[<->, thin] (6.42,3.60) -- (6.42,4.38);
  \node[note, anchor=west] at (6.50,4.30) {$n_t$};
  \node[note, fill=white, inner sep=1pt] at (6.90,3.47) {boundary $b_t$};
  \node[token, font=\tiny] at (6.50,2.90)
    {$\{(q_t,n_t,b_t)\}_t$ + adjacency};

  \draw[panel] (9.10,2.65) rectangle (13.15,5.55);
  \node[title] at (11.125,5.30) {incidence layer};
  \draw[step=0.62, black!25, very thin] (9.40,3.08) grid (12.85,4.92);
  \fill[black!8] (10.30,3.25) -- (11.38,4.02) -- (12.16,3.28) -- cycle;
  \draw[thick] (10.30,3.25) -- (11.38,4.02) -- (12.16,3.28) -- cycle;
  \draw[very thick] (9.62,4.60) -- (10.53,4.24) -- (11.10,4.55);
  \fill (10.53,4.24) circle (1.15pt);
  \node[note, anchor=south east] at (10.02,4.44) {$e$};
  \node[note] at (11.28,3.48) {$t$};
  \draw[dashed, <->] (10.58,4.12) -- (11.10,3.77);
  \node[token] at (11.96,4.28) {$j(e,t)\in\mathcal J$};
  \node[token] at (11.125,2.90) {incidence symbols $\{j(e,t)\}$};

  \node[codebox, text width=106mm] (word) at (6.575,1.93)
    {same canonical layered code\\[-1mm]
     curve data $\mid$ surface/boundary data $\mid$ incidence data};
  \draw[arrow] (2.025,2.65) -- (2.025,2.38) -- (4.25,2.38) -- (4.25,2.16);
  \draw[arrow] (6.575,2.65) -- (6.575,2.16);
  \draw[arrow] (11.125,2.65) -- (11.125,2.38) -- (8.90,2.38) -- (8.90,2.16);

  \node[codebox, text width=43mm] (close) at (3.65,0.62)
    {positional and angular proximity\\
     $d_H=O(\varepsilon)$,\quad
     $\angle=O(\delta+\varepsilon/\tau)$};
  \node[codebox, double, text width=34mm] (isotopy) at (10.9,0.62)
    {ambient pair-isotopy};
  \draw[arrow] (word.south) -- ++(0,-0.20) -| (close.north);
  \draw[arrow] (close) -- node[above, note] {pair stability} (isotopy);
\end{tikzpicture}%
}
\caption{The layered finite-resolution encoding of a knot--surface pair.
The curve, surface with its explicit boundary layer, and relative-incidence
information are recorded separately in one fixed global cubical lattice.
Equality of sufficiently fine canonical codes forces positional and angular
proximity, and the pair-stability lemma then recovers the ambient
pair-isotopy type.  The drawing is schematic: the curve and surface layers do
not form a common subdivision.}
\label{fig:layered-geometric-code}
\end{figure}
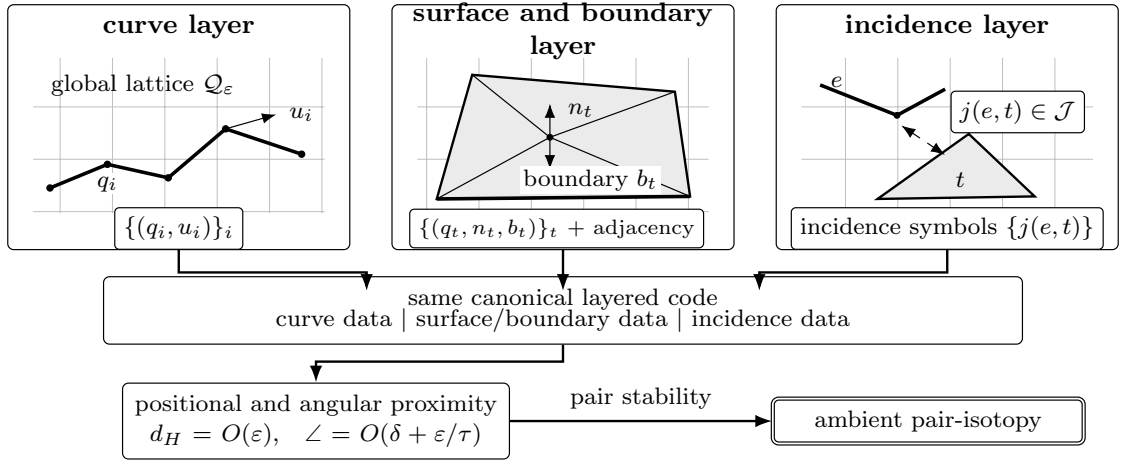

\begin{remark}[Finite code space versus an enumeration algorithm]
\label{rem:least-word-not-algorithmic}
The construction gives a finite, explicitly bounded code space and a
well-defined canonical code.  It does not by itself claim that the least word
can be computed --- efficiently or at all --- from arbitrary analytic input,
nor that every
formal word is realizable by an essential pair.  The least word is defined by
a minimization over all admissible normalizations and discretizations; this
determines it set-theoretically, but a terminating procedure computing it
would require a fixed computable input model for the pair (for example,
piecewise-polynomial or algebraic representatives with rational data)
together with certified geometric predicates, none of which is fixed here.
Accordingly, throughout the paper the words \emph{effective} and
\emph{explicit} refer to explicitly bounded finite search spaces and to the
existence of finite canonical certificates, never to an implemented or
asserted algorithm; where an algorithmic statement is intended, an input
model is stated.  The finiteness and
faithfulness theorems require only the existence of the canonical word; the
algorithmic realization problem is separated in \Cref{sec:finite-recognition}
and in the open problems.
\end{remark}

\begin{lemma}[Equal codes force geometric proximity]
\label{lem:code-proximity}
Let \((\gamma_0,F_0)\) and \((\gamma_1,F_1)\) be pairs with
\(\gamma_j\in Y^{=}_\Lambda(K)\), \(F_j\subset E(\gamma_j)\) essential,
\(\Thi(F_j)\geq\tau\), and \(\Area(F_j)\leq\Delta\), encoded by
\Cref{con:concrete-encoding} with the same code.  Then
\[
  d_H(\gamma_0,\gamma_1)\leq 4\,\varepsilon,
  \qquad
  d_H(\overline{F_0},\overline{F_1})\leq 4\,\varepsilon,
\]
corresponding boundary curves are \(4\varepsilon\)-close in \(\R^3\)
(the two pairs have their own peripheral tori, themselves
\(4\varepsilon\)-close).  Moreover, there is a universal encoding constant
\(C_{\mathrm{enc}}\geq1\) such that corresponding knot tangents differ in
angle by at most
\(C_{\mathrm{enc}}(\delta_{\gamma}+\varepsilon)\), while corresponding
surface tangent planes and boundary tangents differ in angle by at most
\[
  C_{\mathrm{enc}}\bigl(\delta_F+\varepsilon/\tau\bigr).
\]
\end{lemma}

\begin{proof}
The scheme records, for each object, one distinguished point per cell ---
the sample point of a curve edge, the \emph{barycenter} of a surface
triangle --- and the grid cube containing that distinguished point; this
convention matters, since a triangle of diameter \(\varepsilon\) may meet up
to eight cubes, and recording an unspecified cube would weaken the estimate.
By the canonicalization convention, equal encoded types admit raw
realizations with exactly the same word.  In those realizations the
\(\varepsilon/2\)-nets of \(\gamma_0\) and \(\gamma_1\) traverse the same
ordered list of global \(\varepsilon\)-cubes: the
\(i\)-th sample points lie in one cube, hence within \(\sqrt3\,\varepsilon\)
of each other, and every point of \(\gamma_j\) is within \(\varepsilon/2\) of
a sample point; so
\(d_H(\gamma_0,\gamma_1)\leq(1+\sqrt3)\varepsilon\leq4\varepsilon\).  For the
surfaces, every point of \(F_j\) lies in a triangle of diameter at most
\(\varepsilon\), hence within \(\varepsilon\) of that triangle's barycenter;
matched barycenters lie in one cube, within \(\sqrt3\,\varepsilon\) of each
other; so
\(d_H(\overline{F_0},\overline{F_1})\leq(2+\sqrt3)\varepsilon\leq4\varepsilon\).
The explicit boundary layer gives the same argument on every matched
boundary component: corresponding boundary samples lie in the same global
cube, every boundary point is within \(\varepsilon/2\) of a sample, and the
cyclic labels match, giving the stated Hausdorff bound and correspondence.
The angle statements hold because equal labels force equal quantized
directions in the nets \(\mathcal D_{\gamma}\),
\(\mathcal D_{\partial}\), and \(\widehat{\mathcal D}_{F}\).  The true knot
tangent differs from its quantization by at most \(\delta_{\gamma}\), with an
additional \(O(\varepsilon)\) variation across one curve cell.  Surface normal
lines and boundary tangents differ from their quantizations by at most
\(\delta_F\), with the uniform chart distortion
\(O(\varepsilon/\tau)\) from
\Cref{def:relative-surface-thickness}.  For the boundary tangents: a boundary curve is the intersection of the
surface with its peripheral torus, met at angle at least \(\arcsin c_0\) by
the clearance clause, so its tangent direction is determined, up to an error
controlled by that angle bound, by the tangent plane of the surface and the
tangent plane of the torus; both are matched to within
\(C_{\mathrm{enc}}(\delta_F+\varepsilon/\tau)\) by the explicit boundary
layer, the collar labels, and the curve layer.  Enlarging
\(C_{\mathrm{enc}}\) once gives all three stated angle estimates.
\end{proof}

\begin{remark}[Well-definedness and consistency of the code]
\label{rem:code-well-defined}
The canonicalization in \Cref{con:concrete-encoding} is essential.  A symmetric
representative may admit continuously many rigid normalizations (the round
circle has a stabilizer containing \(SO(2)\)); an unoriented closed curve has
no preferred orientation or basepoint; and a surface has no preferred
thickness-adapted triangulation or ordering of its components.  For a pair
\((\gamma,F)\), let \(\mathscr C_\varepsilon(\gamma,F)\) be the set of all raw
words obtained from every admissible normalization and every admissible
discretization choice listed in part~(f) of the construction.  This is a
nonempty subset of the finite set
\(\mathcal C(\Lambda,\Delta,\tau,\varepsilon)\).  Define
\[
  c_\varepsilon(\gamma,F)=\min\mathscr C_\varepsilon(\gamma,F)
\]
with respect to the fixed global shortlex order.  The minimum exists, is
invariant under orientation-preserving rigid motions, reparametrization, cell
renumbering, and changes of admissible triangulation, and is one of the raw
codes; hence all counting estimates are unchanged.

Because the ambient lattice is fixed globally, the canonical code is also
independent of the budget used to contain the pair.  Enlarging
\((\Lambda,\Delta)\) only enlarges the finite set of labels that are
\emph{available}; it does not relabel a pair already present.  This makes the
cross-level and cross-knot comparisons in \Cref{sec:finite-recognition}
literal.  Equivalently, one may retain the finite set
\(\mathscr C_\varepsilon(\gamma,F)\) itself as the code, but the least-word
convention is more economical.
\end{remark}

\begin{proposition}[Finite-resolution degree-of-freedom bound]
\label{prop:basic-dof}
Fix \(\tau>0\).  Assume that thickness is understood in the sense of positive
reach, or equivalently that the objects admit uniformly controlled local
graphical charts at the relevant thickness scale, including boundary-collar
charts near \(\partial F\).  Then, for \(0<\varepsilon\leq c\min\{1,\tau\}\), there
exists a universal constant \(C>0\), depending only on the chosen
finite-resolution model and the bounded-geometry conventions, such that
\[
  \DoF_\varepsilon(\gamma)
  \leq
  C\,\frac{\Len(\gamma)}{\varepsilon}
\]
for every unit-thickness curve \(\gamma\), and
\[
  \DoF_\varepsilon(F)
  \leq
  C\,\frac{\Area(F)}{\varepsilon^2}
\]
for every surface \(F\) with \(\Thi(F)\geq \tau\).  Consequently, if
\[
  (\gamma,F)\in \calZ_{\Lambda,\Delta,\tau}(K),
\]
then
\[
  \DoF_\varepsilon(\gamma,F)
  \leq
  C\left(
  \frac{\Lambda}{\varepsilon}
  +
  \frac{\Delta}{\varepsilon^2}
  \right).
\]
\end{proposition}

\begin{proof}
For curves, unit thickness gives curvature control at scale one.  Parametrize
\(\gamma\) by arclength and subdivide it into intervals of length at most
\(\varepsilon/2\).  Connecting consecutive subdivision points gives a
polygonal model with \(O(\Len(\gamma)/\varepsilon)\) edges, each of length at
most \(\varepsilon\).

For surfaces, the condition \(\Thi(F)\geq\tau\) gives uniformly controlled
graphical charts at scale \(\tau\).  Near the boundary we use the
boundary-collar charts included in the thickness convention.  For
\(0<\varepsilon\leq c\min\{1,\tau\}\), choose a maximal set of points on \(F\) whose
pairwise intrinsic distances are at least \(\varepsilon/2\).  The intrinsic
balls of radius \(\varepsilon/4\), or the corresponding half-balls near
\(\partial F\), have pairwise disjoint interiors.  Bounded geometry at scale
\(\tau\) gives a uniform lower area bound comparable to \(\varepsilon^2\) for
each such ball or half-ball.  Hence the number of points is bounded by
\[
  C'\frac{\Area(F)}{\varepsilon^2}.
\]
A greedy or Delaunay-type triangulation of this bounded-overlap net gives a
thickness-adapted triangulation with
\[
  C\frac{\Area(F)}{\varepsilon^2}
\]
triangles, after increasing \(C\) if necessary.

For the pair \((\gamma,F)\), we use the layered cellular model described
above.  A common simplicial refinement at intersection points is not required.
The geometric support of the model consists only of the curve edges and the
surface triangles, so its size is bounded by the sum of the curve and surface
contributions.  The relative-position layer is then recorded over these cells.
Because \(\varepsilon\leq c\tau\), positive reach and sheet separation imply a
uniform occupancy bound: only uniformly many curve edges or surface patches can
occur in a single ambient \(\varepsilon\)-cell or in its bounded neighbourhood.
Hence the local incidence symbols visible at resolution \(\varepsilon\) are
finite and uniformly controlled.  If desired, one may instead reserve a slot
for every ordered pair of geometric cells; this gives at most quadratic
bookkeeping, which is used below for the explicit counting of codes.  It does
not change the geometric degree-of-freedom estimate.  Substituting
\(\Len(\gamma)\leq\Lambda\) and \(\Area(F)\leq\Delta\) gives the desired
estimate.
\end{proof}

\begin{remark}
The constants are not meant to be sharp; they depend on the chosen
finite-resolution model and on the bounded-geometry conventions.  The
essential point is that \(\DoF_\varepsilon\) grows linearly in
\(\Len(\gamma)/\varepsilon\) for the curve part and linearly in
\(\Area(F)/\varepsilon^2\) for the surface part.  For background on positive
reach and curvature measures, see Federer \cite{FedererReach}; for
ropelength and thickness of knots, see
\cite{CantarellaKusnerSullivan,LitherlandSimonDurumericRawdon}.
\end{remark}

\begin{remark}[Dimensional homogeneity: the constants are universal]
\label{rem:dimensional-homogeneity}
Although intermediate covering arguments may introduce constants denoted
\(C_\tau\) or \(a_\tau\), \Cref{prop:basic-dof} is correctly stated with a
universal constant: the surface estimate is scale invariant, and the constants
of the encoded-type count may be taken independent of \(\tau\).  Indeed, rescaling the ambient
space by \(\tau^{-1}\) replaces a surface of reach at least \(\tau\) by one of
reach at least \(1\), replaces \(\Area(F)\) by \(\Area(F)/\tau^2\), and replaces
the resolution \(\varepsilon\) by \(\widetilde\varepsilon=\varepsilon/\tau\),
under which the surface-side admissibility constraint \(\varepsilon\leq c\tau\)
becomes the scale-free constraint \(\widetilde\varepsilon\leq c\).  The number of
\(\varepsilon\)-patches of \(F\) equals the number of
\(\widetilde\varepsilon\)-patches of the rescaled surface, namely
\[
  C\,\frac{\Area(F)/\tau^2}{(\varepsilon/\tau)^2}
  =
  C\,\frac{\Area(F)}{\varepsilon^2},
\]
with \(C\) depending only on the reach-\(1\) bounded-geometry convention and on
the fixed ratio bound \(c\), not on \(\tau\).  The curve estimate is already at
the unit-thickness scale, so its constant is universal for
\(\varepsilon\leq c\); this is the curve-side half of the combined constraint
\(\varepsilon\leq c\min\{1,\tau\}\), which cannot be dropped: for \(\tau\gg1\)
the condition \(\varepsilon\leq c\tau\) alone would allow \(\varepsilon>1\),
coarser than the unit thickness scale of the curve, and no faithful curve
sampling could be expected.  Consequently the degree-of-freedom bound may be
written in the scale-free form
\[
  \DoF_\varepsilon(\gamma,F)
  \leq
  C\left(\frac{\Lambda}{\varepsilon}+\frac{\Delta}{\varepsilon^2}\right),
  \qquad 0<\varepsilon\leq c\min\{1,\tau\},
\]
with \(C\) a universal constant, and the encoded-type count of
\Cref{cor:explicit-encoded-bound} may be written \((CN)^{BN^2}\) with
universal \(B,C\), the polynomial base reflecting the grid-cell labels.  The only role of \(\tau\) is to fix, through
\(\varepsilon\leq c\min\{1,\tau\}\),
the largest resolution at which the surface is faithfully sampled; once the
dimensionless budgets \(\Lambda/\varepsilon\) and \(\Delta/\varepsilon^2\) are
prescribed, no further \(\tau\)-dependence remains.
\end{remark}

\paragraph{Position of the finiteness theorem.}
The following theorem is the main proved bounded-geometry statement of the
paper.  It is not meant to replace Haken finiteness, normal surface theory, or
least-area theory.  Rather, it says that after a metric window and a resolution
are fixed, essential surface data carried by thick knot representatives can be
stored in a finite set of layered codes.  The proof is deliberately elementary
once positive reach and area bounds are assumed; the subsequent sections use
this finite code space as a bookkeeping device for surface-theoretic
filtrations.

\begin{theorem}[Finite-resolution finiteness of filtered knot--surface pairs]
\label{thm:finite-resolution-finiteness-filtered-pairs}
Let \(K\) be a knot type, and fix
\[
  \Lambda>0,\qquad \Delta>0,\qquad \tau>0.
\]
Let \(0<\varepsilon\leq c\min\{1,\tau\}\), where \(c>0\) is a fixed small universal
constant.  Consider pairs \((\gamma,F)\) such that
\[
  \gamma\in Y^{=}_\Lambda(K),
  \qquad
  F\subset E(\gamma)\text{ is essential},
\]
and
\[
  \Thi(F)\geq\tau,
  \qquad
  \Area(F)\leq\Delta.
\]
Fix once and for all a finite-resolution encoding scheme at scale
\(\varepsilon\).  Then only finitely many encoded \(\varepsilon\)-types of
such pairs occur.  More precisely, there exists a universal constant \(C>0\) such
that
\[
  \DoF_\varepsilon(\gamma,F)
  \leq
  C
  \left(
  \frac{\Lambda}{\varepsilon}
  +
  \frac{\Delta}{\varepsilon^2}
  \right).
\]
Consequently, the number of possible encoded pair types is bounded above by a
function depending only on
\[
  \Lambda,\Delta,\tau,\varepsilon
\]
and on the chosen finite-resolution encoding scheme.
\end{theorem}

\begin{proof}
We give a direct proof, including the finite-resolution degree-of-freedom
estimate, so that the finiteness statement does not depend logically on the
preceding proposition.

First fix the geometric setting.  Since \(\gamma\in Y^{=}_\Lambda(K)\), the curve
is a \(C^{1,1}\) embedded representative with
\[
  \Thi(\gamma)=1,
  \qquad
  \Len(\gamma)\leq \Lambda .
\]
The thickness condition gives the standard reach consequences: the curvature
of \(\gamma\) is uniformly bounded at scale one and the normal tube
\(N_{\rho_0}(\gamma)\) of the fixed radius \(\rho_0<1\) is embedded with
\(C^{1,1}\) boundary torus (\Cref{not:ambient-metric}).  The surface \(F\subset E(\gamma)\) is a compact properly embedded
essential surface with
\[
  \Thi(F)\geq \tau,
  \qquad
  \Area(F)\leq \Delta .
\]
Here the surface thickness is understood in the relative positive-reach sense
fixed earlier: both the interior sheets and boundary collars have controlled
\(C^{1,1}\) graphical charts at scale \(\tau\), with uniform sheet separation.
Only these bounded-geometry consequences are used below.

\smallskip
\noindent
\emph{Step 1: the curve layer.}
Parametrize \(\gamma\) by arclength and subdivide \([0,\Len(\gamma)]\) into
intervals of length at most \(\varepsilon/2\).  The number of subintervals is
at most
\[
  N_\gamma
  \leq
  \left\lceil \frac{2\Len(\gamma)}{\varepsilon}\right\rceil
  \leq
  \left\lceil \frac{2\Lambda}{\varepsilon}\right\rceil .
\]
Joining consecutive sample points gives a polygonal model
\(\widehat\gamma\).  Since the curvature is uniformly bounded, the deviation
between an arclength segment and its chord is bounded by a universal constant
times the square of the segment length.  For \(\varepsilon\) smaller than a
fixed universal scale this is certainly below the resolution level.  Thus the
curve layer of the code uses at most
\[
  C_1 \frac{\Lambda}{\varepsilon}
\]
curve cells, after increasing \(C_1\) to absorb ceilings and the chosen
quantization convention.

\smallskip
\noindent
\emph{Step 2: the surface layer.}
Choose a maximal subset \(\mathcal P\subset F\) whose points have pairwise
intrinsic distance at least \(\varepsilon/2\).  Near \(\partial F\) we use the
boundary-collar charts and intrinsic half-balls.  Since
\(0<\varepsilon\leq c\min\{1,\tau\}\), the intrinsic balls of radius \(\varepsilon/4\)
centered at the interior points of \(\mathcal P\), and the corresponding
half-balls near the boundary, have pairwise disjoint interiors.  The positive
reach and bounded-geometry assumptions give a uniform lower area bound
\[
  \Area\bigl(B_F(p,\varepsilon/4)\bigr)
  \geq
  a_\tau\varepsilon^2
\]
for every such ball or half-ball, where \(a_\tau>0\) depends only on the
bounded-geometry convention and on \(\tau\).  Hence
\[
  |\mathcal P|
  \leq
  \frac{\Area(F)}{a_\tau\varepsilon^2}
  \leq
  \frac{\Delta}{a_\tau\varepsilon^2}.
\]
By maximality, the intrinsic balls of radius \(\varepsilon/2\) centered at
\(\mathcal P\) cover \(F\).  A standard bounded-overlap Voronoi--Delaunay or
greedy triangulation construction, using the controlled local graphical charts
and boundary collars, produces a triangulated surface layer whose triangles
have intrinsic diameter at most a fixed multiple of \(\varepsilon\).  After
subdividing once more, if necessary, the mesh is at most \(\varepsilon\), and
the number of triangles is at most
\[
  C_2(\tau)\frac{\Delta}{\varepsilon^2} .
\]
The constant includes the bounded valence of the net, the boundary-collar
half-patch convention, and the chosen local quantization rule.

\smallskip
\noindent
\emph{Step 3: the layered pair model.}
The pair code is the combination of the polygonal curve layer, the triangulated
surface layer, and a relative-position layer.  We do not require a common
subdivision of \(\widehat\gamma\) and the triangulation of \(F\).  Instead the
relative-position layer records, at scale \(\varepsilon\), which curve edges
and surface triangles lie in neighbouring ambient cells, together with the
finite local symbols specified by the encoding scheme.

We use a fixed ambient finite-resolution atlas.  After the
rigid-motion normalization used in the definition of \(Y^{=}_\Lambda(K)\), the
essential pair lies in the cubical window of side
\(S(\Lambda,\Delta,\tau)\) of \Cref{con:concrete-encoding}, supplied by the
anchoring and diameter lemmas
(\Cref{lem:anchoring,lem:diameter-bound}); the window size depends on all
three budgets, not on \(\Lambda\) alone.  At scale \(\varepsilon\), each curve
edge and each bounded-geometry
surface triangle meets only uniformly boundedly many ambient cells.  The
relative-position layer records only the local configuration of cells that are
equal or adjacent in this atlas.  Since \(\varepsilon\leq c\tau\), the
positive-reach assumption prevents arbitrarily many independent sheets from
being packed into one such local cell cluster.  Thus the local incidence
alphabet is finite with uniformly bounded multiplicity.  If the approximation
of a disjoint smooth pair produces a coarse cell intersection, the code records
only the corresponding \(\varepsilon\)-visible relation; it is not meant to be a
common refinement of the smooth objects.  Consequently the number of geometric
curve/surface cells in the layered model is bounded by a constant multiple of
the sum of the curve and surface contributions.  Thus, for some \(C_\tau>0\),
\[
  \DoF_\varepsilon(\gamma,F)
  \leq
  C_\tau
  \left(
    \frac{\Lambda}{\varepsilon}
    +
    \frac{\Delta}{\varepsilon^2}
  \right).
\]
This is the asserted effective finite-resolution degree-of-freedom estimate.

\smallskip
\noindent
\emph{Step 4: finiteness of codes.}
Now fix the encoding scheme once and for all.  Its local data are drawn from a
finite alphabet: for example, grid-cell or chart labels, ordered curve-edge
labels, quantized tangent directions, triangle labels, quantized unoriented
normal lines, boundary-collar labels, adjacency symbols, and coarse
curve--surface incidence symbols.  Let
\[
  N
  =
  \left\lceil
  C_\tau
  \left(
    \frac{\Lambda}{\varepsilon}
    +
    \frac{\Delta}{\varepsilon^2}
  \right)
  \right\rceil .
\]
Every admissible pair in the theorem is represented by a layered model with at
most \(N\) curve/surface cells.  For a fixed number \(k\leq N\) of cells, there
are only finitely many choices of local symbols.  There are also only finitely
many adjacency and incidence patterns, since these can be encoded by entries
in a finite set of slots indexed by at most \(k^2\) ordered pairs of cells.
Thus, if \(A\) is the maximum number of choices allowed by the finite alphabet
for each local or incidence slot, the number of codes with \(k\) cells is
bounded above by a quantity of the form \(A^{B k^2}\), for some constant
\(B\) depending only on the scheme.  The quantity \(A\) may itself depend on
\((\Lambda,\Delta,\tau,\varepsilon)\) through the size of the ambient grid
--- for the concrete scheme of \Cref{con:concrete-encoding} one has
\(A\leq CN^5\) --- but it is finite once the budgets and the resolution are
fixed.  Consequently the total number of possible
codes is bounded by
\[
  \sum_{k=0}^{N} A^{B k^2},
\]
which is finite and depends only on
\(\Lambda,\Delta,\tau,\varepsilon\) and on the chosen encoding scheme.

The set of encoded \(\varepsilon\)-types arising from the pairs under
consideration is a subset of this finite set.  Hence only finitely many encoded
\(\varepsilon\)-types occur.
\end{proof}

\begin{corollary}[Explicit encoded-type bound for the concrete scheme]
\label{cor:explicit-encoded-bound}
For the encoding scheme of Construction~\ref{con:concrete-encoding}, set
\[
  N_{\Lambda,\Delta,\tau,\varepsilon}
  =
  \left\lceil
  C\left(
  \frac{\Lambda}{\varepsilon}+\frac{\Delta}{\varepsilon^2}
  \right)
  \right\rceil .
\]
Then the number of encoded \(\varepsilon\)-types of pairs in
\(\calZ_{\Lambda,\Delta,\tau}^{\mathrm{ess}}(K)\) is at most
\[
  \bigl(C\,N_{\Lambda,\Delta,\tau,\varepsilon}\bigr)^{\,B\,
  N_{\Lambda,\Delta,\tau,\varepsilon}^2},
\]
with \(B,C\) universal constants of the scheme.  We stress that the base of
the exponential cannot be taken to depend on \(\tau\) alone: the alphabet
contains the grid-cell labels
\(\mathcal I_\varepsilon=\{1,\dots,\lceil S/\varepsilon\rceil\}^3\), whose
number depends on \((\Lambda,\Delta,\tau,\varepsilon)\); the displayed form
absorbs this via \(\lceil S/\varepsilon\rceil\leq CN\)
(\Cref{con:concrete-encoding}).
In particular, after the geometric budget and the resolution are fixed, the
search space is not only finite abstractly but bounded by an explicit function
of \((\Lambda,\Delta,\tau,\varepsilon)\) and the chosen quantization constants.
\end{corollary}

\begin{proof}
By Proposition~\ref{prop:basic-dof}, every pair in the filtered space has a layered model
with at most \(N_{\Lambda,\Delta,\tau,\varepsilon}\) cells, after increasing the
constant in the definition of \(N\) if necessary.  The concrete scheme uses an
alphabet of at most \(CN^5\) letters
(\Cref{con:concrete-encoding}) and at most \(BN^2\) slots for cell labels,
adjacencies,
and coarse incidences.  Hence the number of possible words is bounded by
\((CN^5)^{BN^2}\leq(CN)^{5BN^2}\), and renaming \(5B\) as \(B\) gives the
stated form.
\end{proof}

\begin{corollary}[Finite-resolution finiteness over the ideal-knot stratum]
\label{cor:finite-resolution-finiteness-ideal-pairs}
Fix \(K\), \(\Delta>0\), \(\tau>0\), and \(0<\varepsilon\leq c\min\{1,\tau\}\).
Consider pairs \((\gamma,F)\) over ideal knot representatives satisfying
\[
  \gamma\in I(K),
  \qquad
  \Thi(F)\geq\tau,
  \qquad
  \Area(F)\leq\Delta.
\]
Then only finitely many encoded \(\varepsilon\)-types of such ideal pairs
occur.  More precisely,
\[
  \DoF_\varepsilon(\gamma,F)
  \leq
  C
  \left(
  \frac{\Rop(K)}{\varepsilon}
  +
  \frac{\Delta}{\varepsilon^2}
  \right).
\]
\end{corollary}

\begin{proof}
If \(\gamma\in I(K)\), then \(\Thi(\gamma)=1\) and
\[
  \Len(\gamma)=\Rop(K).
\]
The claim follows from
\Cref{thm:finite-resolution-finiteness-filtered-pairs} with
\(\Lambda=\Rop(K)\).
\end{proof}

\begin{remark}[A smooth finiteness lies behind the encoded count]
\label{rem:forward-smooth}
\Cref{thm:finite-resolution-finiteness-filtered-pairs} bounds the number
of encoded \(\varepsilon\)-types.  In \Cref{sec:bounded-complexity} we prove that
the underlying smooth object is already finite: the filtered pair space
\(\calZ^{\mathrm{ess}}_{\Lambda,\Delta,\tau}(K)\) has only finitely many
pair-isotopy classes (\Cref{thm:smooth-finiteness-pair-space}), and at fine
enough resolution the encoded type is a separating certificate for those
classes (\Cref{thm:faithfulness}).  Thus the explicit bound
\((CN)^{BN^2}\) of \Cref{cor:explicit-encoded-bound} is an explicit
set-theoretic upper bound for a genuinely finite count.  The compactness underlying this smooth
finiteness also makes the infima defining ideal surfaces and ideal pairs
attained rather than merely conditional, and makes the compactified
visibility level \(\overline\beta_{\mathfrak S}\) attained, with exact-slice
attainment over the ideal stratum; see
\Cref{lem:compactness-thick-surfaces,cor:attainment-visibility,cor:unconditional-ideal-attainment}.
\end{remark}

\section{Bounded topological complexity of thick surfaces}
\label{sec:bounded-complexity}

This short section supplies the topological input used throughout the rest of
the paper.  The finite-resolution theorem controls the number of encoded
geometric models; the estimates below explain why, under the same bounded
geometry hypotheses, those models cannot hide arbitrarily complicated surface
topology.  In this sense the section is the bridge from the metric parameters
\((\Delta,\tau)\) to the Haken-theoretic complexity of the visible surfaces.

We now record the elementary boundedness principle behind the filtered
surface theory.

\paragraph{Boundary curvature and collar estimates from relative thickness.}
Fix \(\tau>0\).  Under Definition~\ref{def:relative-surface-thickness}, there
are constants
\[
  c_{\mathrm{col}}(\tau)>0,
  \qquad
  \kappa_{\partial}(\tau)<\infty,
  \qquad
  K_{\mathrm{curv}}(\tau)<\infty,
  \qquad
  k_0(\tau)<\infty
\]
depending only on the relative-thickness convention and on \(\tau\), such that
any surface \(F\) with \(\Thi(F)\geq\tau\) has an embedded boundary collar with
\[
  \Area(\operatorname{Col}(\partial F))
  \geq
  c_{\mathrm{col}}(\tau)\,\Len(\partial F),
\]
each boundary component has ambient space-curve curvature at most
\(\kappa_{\partial}(\tau)\), and
\[
  |K_F|\leq K_{\mathrm{curv}}(\tau),
  \qquad
  |k_g|\leq k_0(\tau).
\]
Indeed, the interior and boundary half-chart clauses in
Definition~\ref{def:relative-surface-thickness} give a uniform bound for the
second fundamental form of \(F\) up to the boundary, of the form
\(\|A_F\|\leq C/\tau\).  They also give a uniform bound for the curvature of the
boundary curve inside the surface, namely \(|k_g|\leq C/\tau\).  For an
arclength-parametrized boundary component \(\alpha\), its ambient curvature
vector decomposes as
\[
  \nabla_s \alpha'(s)
  =
  k_g\,\nu_{\partial F} + II_F(\alpha'(s),\alpha'(s))\,n_F,
\]
where \(\nu_{\partial F}\) is the inward conormal in \(F\) and \(n_F\) is a
unit normal to \(F\).  Therefore
\[
  \kappa_{\mathrm{amb}}(\alpha)
  \leq
  |k_g|+\|A_F\|
  \leq
  C_\partial/\tau .
\]
This supplies the asserted space-curve curvature bound.  The embedded collar
clause gives a collar map \([0,c_0\tau]\times\partial F\to F\) with uniformly
controlled Jacobian.  After possibly reducing \(c_0\), its Jacobian is bounded
below by a positive constant independent of \(F\).  Hence the collar of width
comparable to \(\tau\) has area at least
\(c_{\mathrm{col}}(\tau)\Len(\partial F)\).  Finally, the bound on the second
fundamental form gives \(|K_F|\leq C/\tau^2\), and the boundary half-chart gives
the stated geodesic-curvature bound.

\begin{lemma}[Uniform reach of the boundary system]
\label{lem:boundary-system-reach}
There is a constant \(c_{\partial}>0\), depending only on the fixed
bounded-geometry convention, such that every surface with
\(\Thi(F)\geq r\) satisfies
\[
  \operatorname{reach}_{\R^3}(\partial F)
  \geq c_{\partial}\min\{r,\rho_0\}.
\]
Here \(\partial F\) denotes the union of all boundary components.  The same
conclusion, with a possibly smaller constant, holds for the normal injectivity
radius of \(\partial F\) as a curve system in the peripheral torus
\(\partial E(\gamma)\).
\end{lemma}

\begin{proof}
Set \(r_*=\min\{r,\rho_0\}\).  The boundary half-charts give
\[
  |\kappa_{\partial F}|\leq C/r_*
\]
almost everywhere.  Hence, if two points of one boundary component have
arclength separation at most \(a r_*\), with \(a>0\) universal and small,
the chord--tangent estimate for a \(C^{1,1}\) curve of curvature at most
\(C/r_*\) shows that the chord has a nonzero component in the initial tangent
direction.  Such a pair is therefore not doubly critical.

For pairs on one boundary component whose arclength separation is at least
\(a r_*\), the truncated lower bound (iv-b) of
\Cref{def:relative-surface-thickness}, evaluated at collar depth zero, gives
\[
  |p-q|\geq L_0^{-1}\min\{r,d_{\partial F}(p,q)\}
  \geq L_0^{-1}a r_*.
\]
For points on distinct boundary components the separation clause (iv-c)
gives directly \(|p-q|\geq c_0r\geq c_0r_*\).  Thus the doubly critical
self-distance of the entire boundary system is bounded below by a universal
multiple of \(r_*\), while its curvature radius is bounded below by another
universal multiple of \(r_*\).  The thickness characterization for closed
\(C^{1,1}\) curves, applied componentwise together with the mutual component
separation, proves the ambient reach estimate.

The peripheral torus has second fundamental form bounded in terms of the
fixed radius \(\rho_0\).  On the scale \(r_*\), ambient and intrinsic torus
chords and angles are uniformly comparable.  Decreasing
\(c_{\partial}\) therefore gives the asserted normal injectivity radius in the torus as well.
\end{proof}

\begin{proposition}[Geometric-topological boundedness]
\label{prop:geometric-topological-boundedness}
Fix \((\Lambda,\Delta,\tau)\).  Assume that \(F\) is connected, orientable or
not, with \(\Thi(F)\geq\tau\) in the sense of
Definition~\ref{def:relative-surface-thickness} and
\(\Area(F)\leq\Delta\).  Then the number of boundary components
\(|\partial F|\) and the absolute Euler characteristic \(|\chi(F)|\) are
bounded above by constants depending
only on \(\Delta\), \(\tau\), and the fixed relative-thickness convention.
Consequently the genus is bounded when \(F\) is orientable
(\(\chi=2-2g-|\partial F|\)) and the crosscap number is bounded when \(F\) is
non-orientable (\(\chi=2-k-|\partial F|\)).
\end{proposition}

\begin{proof}
By the boundary-curvature and collar estimates above, the relative-thickness
hypothesis supplies constants \(c_{\mathrm{col}}\), \(\kappa_{\partial}\),
\(K_{\mathrm{curv}}\), and \(k_0\), depending only on \(\tau\) and on the convention.  The
collar estimate gives
\[
  c_{\mathrm{col}}\Len(\partial F)
  \leq
  \Area(\operatorname{Col}(\partial F))
  \leq
  \Area(F)
  \leq
  \Delta .
\]
Hence
\[
  \Len(\partial F)
  \leq
  c_{\mathrm{col}}^{-1}\Delta .
\]

Each boundary component \(\alpha\) is a closed \(C^{1,1}\) space curve with
curvature bounded above by \(\kappa_{\partial}\).  By Fenchel's theorem for
closed space curves \cite{Fenchel1929,Fenchel1951}, applied to \(C^{1,1}\)
curves by smooth approximation and lower semicontinuity of total curvature,
\[
  \int_\alpha \kappa\,ds\geq 2\pi.
\]
Together with \(\kappa\leq\kappa_{\partial}\), this gives
\[
  \Len(\alpha)
  \geq
  \frac{2\pi}{\kappa_{\partial}}.
\]
Consequently
\[
  |\partial F|
  \leq
  \frac{\Len(\partial F)}{2\pi/\kappa_{\partial}}
  \leq
  \frac{\kappa_{\partial}}{2\pi c_{\mathrm{col}}}\Delta .
\]

Gauss--Bonnet gives
\[
  2\pi\chi(F)
  =
  \int_F K_F\,dA+
  \int_{\partial F} k_g\,ds .
\]
For the \(C^{1,1}\) surfaces of this paper, \(K_F\) and \(k_g\) are defined
almost everywhere with the stated \(L^\infty\) bounds, and the identity holds
by smooth approximation: mollifying the local graph functions preserves the
curvature bounds up to a factor, leaves \(\chi\) unchanged (the mollified
surface is diffeomorphic to \(F\)), and passes area and boundary length to the
limit; alternatively one may invoke the curvature-measure form of
Gauss--Bonnet for sets of positive reach \cite{FedererReach,RatajZahle}.
Using the curvature bounds and the boundary length estimate, we obtain
\[
  |\chi(F)|
  \leq
  \frac{1}{2\pi}\left(
  K_{\mathrm{curv}}\Area(F)+k_0\Len(\partial F)
  \right)
  \leq
  C(\tau)\Delta .
\]
Neither this estimate nor the collar and Fenchel steps used orientability.
The final statements follow from the surface classification:
\(\chi=2-2g-|\partial F|\) in the orientable case and
\(\chi=2-k-|\partial F|\) in the non-orientable case.
\end{proof}

The chain of estimates in the proof is summarized in
\Cref{fig:boundedness-chain}.

\begin{figure}[t]
\centering
\begin{tikzpicture}[x=1cm,y=1cm,
  box/.style={draw, rounded corners=2pt, align=center, inner sep=4pt,
              text width=46mm, font=\scriptsize},
  input/.style={draw, rounded corners=2pt, align=center, inner sep=5pt,
                text width=50mm, font=\small},
  result/.style={draw, double, rounded corners=2pt, align=center,
                 inner sep=4pt, text width=40mm, font=\small},
  arrow/.style={-{Latex[length=2.2mm]}, thick}]
  \node[input] (in) at (6.6,5.6)
    {$\Thi(F)\geq\tau$ \quad and \quad $\Area(F)\leq\Delta$};
  \node[box] (collar) at (3.3,4.0)
    {collar estimate:\\
     $c_{\mathrm{col}}(\tau)\Len(\partial F)
       \leq\Area(\operatorname{Col}(\partial F))\leq\Delta$};
  \node[box] (fenchel) at (3.3,2.5)
    {Fenchel: each boundary curve of curvature
     $\leq\kappa_{\partial}(\tau)$ has
     $\Len(\alpha)\geq 2\pi/\kappa_{\partial}$};
  \node[result] (bcount) at (3.3,1.0)
    {$|\partial F|\leq
      \dfrac{\kappa_{\partial}}{2\pi c_{\mathrm{col}}}\Delta$};
  \node[box] (curv) at (9.9,4.0)
    {curvature bounds from the graphical charts:\\
     $|K_F|\leq K_{\mathrm{curv}}(\tau)$,\quad $|k_g|\leq k_0(\tau)$};
  \node[box] (gb) at (9.9,2.5)
    {Gauss--Bonnet:\\
     $2\pi\chi(F)=\displaystyle\int_F K_F\,dA
       +\int_{\partial F}k_g\,ds$};
  \node[result] (chi) at (9.9,1.0)
    {$|\chi(F)|\leq C(\tau)\Delta$\\[-0.5mm]
     {\scriptsize hence bounded genus or crosscap number}};
  \draw[arrow] (in.south) -- ++(0,-0.25) -| (collar.north);
  \draw[arrow] (in.south) -- ++(0,-0.25) -| (curv.north);
  \draw[arrow] (collar) -- (fenchel);
  \draw[arrow] (fenchel) -- (bcount);
  \draw[arrow] (curv) -- (gb);
  \draw[arrow] (gb) -- (chi);
  \draw[arrow, dashed] (collar.east) -- (gb.west);
\end{tikzpicture}
\caption{How bounded geometry bounds topology
(\Cref{prop:geometric-topological-boundedness}).  Relative thickness turns
the area budget into a boundary-length budget through the collar estimate;
Fenchel's theorem prices each boundary component, bounding
\(|\partial F|\); and the curvature bounds feed Gauss--Bonnet, bounding
\(|\chi(F)|\).  The dashed arrow records that the boundary-length estimate
also enters the Gauss--Bonnet step through the geodesic-curvature term.
Orientability is never used, so genus and crosscap number are bounded alike.}
\label{fig:boundedness-chain}
\end{figure}
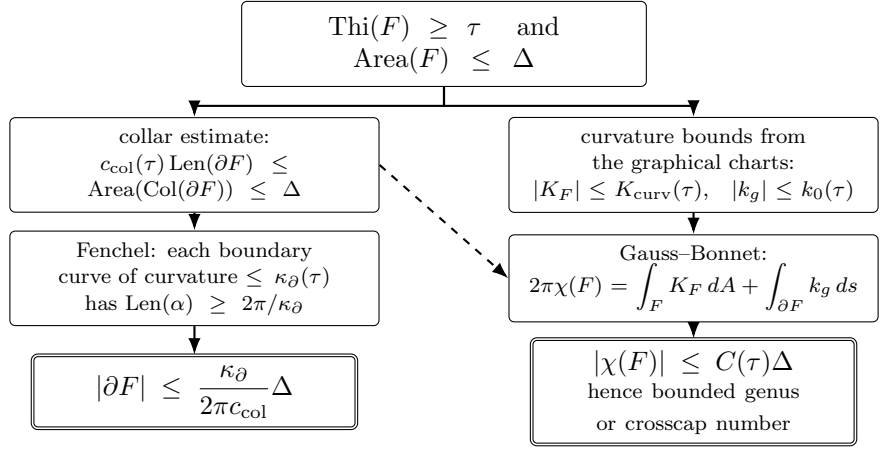

\begin{remark}
The proposition should be viewed as a boundedness statement, not as a
classification theorem.  It is stated for connected surfaces because the genus
formula is then clean.  For a finite surface system, the same estimate applies
componentwise; if the thickness convention gives a positive lower area bound
for each essential component at scale \(\tau\), then the number of components
is at most \(\Delta/a_{\min}(\tau)\).  Thus disconnected JSJ systems are
compatible with the framework after adding this component-count convention.
\end{remark}

\begin{remark}[Boundary slopes require boundary-torus geometry]
The proposition deliberately does not include a general bound on boundary
slope height.  Such a bound requires additional control of the meridian--longitude lattice in the induced metric on \(\partial E(\gamma)\).  Length
control of the core curve alone does not automatically control the longitude
geometry of the boundary torus.  Boundary slope estimates will therefore be
stated below only under explicit boundary-torus metric hypotheses.
\end{remark}

\begin{remark}[Role of essentiality]
Essentiality is not used in the proof of
Proposition~\ref{prop:geometric-topological-boundedness}.  The estimate is a geometric
consequence of thickness, area, boundary-collar control, and curvature control.  Essentiality enters later, when the controlled surfaces
are required to carry topological information about the knot exterior, such
as boundary slopes, JSJ pieces, taut Seifert surfaces, or characteristic
decomposition data.
\end{remark}

\begin{lemma}[Uniform area lower bound per component]
\label{lem:area-lower-bound}
Fix \(\tau>0\).  There is a constant \(a_{\min}(\tau)>0\), depending only on
\(\tau\) and the relative-thickness convention, such that every connected
component \(F'\) of a compact properly embedded surface with \(\Thi\geq\tau\)
satisfies \(\Area(F')\geq a_{\min}(\tau)\).
\end{lemma}

\begin{proof}
Suppose first that \(F'\) has nonempty boundary.  By the boundary curvature and
collar estimates preceding \Cref{prop:geometric-topological-boundedness}, each
boundary component is a closed \(C^{1,1}\) space curve of curvature at most
\(\kappa_\partial(\tau)\), hence of length at least
\(2\pi/\kappa_\partial(\tau)\) by Fenchel's theorem, and it carries an embedded
collar of area at least \(c_{\mathrm{col}}(\tau)\cdot 2\pi/\kappa_\partial(\tau)\).
If instead \(F'\) is closed, then reach at least \(\tau\) provides an embedded
two-sided normal \(\tau\)-tube, so \(F'\) contains a graphical disk of radius
\(\tfrac{\tau}{2}\) and \(\Area(F')\geq c\,\tau^2\) for a universal \(c>0\).  Set
\[
  a_{\min}(\tau)=\min\Bigl\{c_{\mathrm{col}}(\tau)\tfrac{2\pi}{\kappa_\partial(\tau)},\ c\,\tau^2\Bigr\}>0. \qedhere
\]
\end{proof}

\subsection{\texorpdfstring{\(C^{1,1}\)-compactness}{C1,1-compactness} of thick surface families}
\label{subsec:c11-compactness}

The relative-thickness hypothesis of
Definition~\ref{def:relative-surface-thickness} is not merely a bounded-geometry
normalization; it is a genuine compactness hypothesis.  This is the geometric
mechanism behind the finiteness statements of the paper.  It also upgrades
several infima below from conditional to attained, because positive reach, together with the uniform boundary charts and collars, is
precisely what makes the family precompact in locally graphical \(C^1\),
with a \(C^{1,1}\) limit.

\begin{definition}[Locally graphical convergence up to the boundary]
\label{def:graphical-convergence}
A sequence of compact \(C^{1,1}\) surfaces \(F_n\) converges
\emph{locally graphically in \(C^1\), uniformly up to the boundary}, to a
compact surface \(F\) if:
\begin{enumerate}[label=(\alph*),leftmargin=2em]
\item \(\overline F_n\to\overline F\) in Hausdorff distance;
\item on every sufficiently small interior chart of \(F\), the corresponding
pieces of \(F_n\) are graphs over the same fixed disk and their graph
functions converge in \(C^1\);
\item the analogous statement holds on fixed boundary half-disks; and
\item after relabelling the finitely many boundary components and using fixed
circle parameters, the boundary parametrizations and collar maps converge in
\(C^1\) on their fixed parameter domains.
\end{enumerate}
This is the meaning of \(C^1\)-convergence of surfaces used below.  It does
not presuppose a single global normal graph at the boundary.
\end{definition}

Throughout this subsection we fix a unit-thickness representative \(\gamma\)
of \(K\), normalized as in \Cref{sec:ropelength-sublevel} so that its
barycenter is the origin, and we work in \(\R^3\) with the Euclidean metric of
\Cref{not:ambient-metric}; surfaces in \(E(\gamma)\) are isotoped off
\(\infty\) and regarded as subsets of
\(E^\circ(\gamma)=\R^3\setminus\operatorname{int}N_{\rho_0}(\gamma)\).  Two
points require care and are treated by separate lemmas before the compactness
statement.

First, \(E^\circ(\gamma)\) is \emph{not} compact: the area and reach bounds
alone do not prevent a sequence of surfaces from escaping to infinity.  A
fixed round sphere of reach \(1\) and fixed area, translated farther and
farther from the knot, satisfies every metric bound yet has no convergent
subsequence.  Compactness therefore requires an anchoring hypothesis; for
essential surfaces anchoring is automatic
(\Cref{lem:anchoring}), and combined with a diameter bound
(\Cref{lem:diameter-bound}) it confines all competitors to one fixed compact
region.

Second, the passage of the reach bound to the limit must be handled with
Federer's theory for general closed sets, not with a tangent-plane
inequality.  We record the correct criterion and a warning.

\begin{remark}[Tangent cones, not tangent planes, characterize reach]
\label{rem:two-point-criterion}
For a closed set \(S\subset\R^n\) and \(a\in S\), let
\(\operatorname{Tan}(S,a)\) denote Federer's tangent cone.  Federer's two-point
criterion \cite[Theorem~4.18]{FedererReach} states that
\(\operatorname{reach}(S)\geq r\) if and only if
\[
  \operatorname{dist}\bigl(b-a,\ \operatorname{Tan}(S,a)\bigr)
  \leq
  \frac{|b-a|^2}{2r}
  \qquad\text{for all } a,b\in S .
\]
For a compact \(C^{1,1}\) surface \(F\) \emph{with boundary}, the tangent cone
at an interior point is the tangent plane \(T_xF\), but at a boundary point it
is only the tangent \emph{half-plane}.  Replacing the cone by the full tangent
plane destroys the criterion.  A flat example makes this vivid: for a planar
annulus \(A\subset\R^2\times\{0\}\) with small inner radius \(r_{\mathrm{in}}\),
every difference \(y-x\) of points of \(A\) lies in the plane, so
\(\operatorname{dist}(y-x,T_xA)=0\) for all \(x,y\), and the full-tangent-plane
inequality holds for every \(r\); yet
\(\operatorname{reach}(A)\leq r_{\mathrm{in}}\), as points on the axis near the
center of the hole have non-unique nearest points.  With the tangent
\emph{half}-plane at boundary points the criterion detects the hole, since
differences pointing across the hole leave the half-plane.  Accordingly, in
this paper the reach bound of \Cref{def:relative-surface-thickness} is always
understood through the metric-projection definition, and limits of reach
bounds are taken with \Cref{lem:reach-hausdorff-closed} below, never with a
tangent-plane two-point inequality.  For the general theory of sets and
manifolds of positive reach, including boundary behaviour, see
\cite{FedererReach,RatajZahle,LieutierWintraecken}.
\end{remark}

\begin{lemma}[Hausdorff closedness of the reach bound]
\label{lem:reach-hausdorff-closed}
Let \(r>0\) and let \(S_k\subset\R^n\) be closed sets with
\(\operatorname{reach}(S_k)\geq r\), converging to a closed set \(S\) locally
in the Hausdorff sense, meaning that for every \(R>0\) and \(\varepsilon>0\)
there is \(k_0\) with
\(S_k\cap \overline B(0,R)\subset N_\varepsilon(S)\) and
\(S\cap \overline B(0,R)\subset N_\varepsilon(S_k)\) for all \(k\geq k_0\),
where \(N_\varepsilon\) denotes the open \(\varepsilon\)-neighbourhood.  This
two-inclusion form is what uniform \(C^1\) convergence of compact surfaces
supplies.  Then \(\operatorname{reach}(S)\geq r\).
\end{lemma}

\begin{proof}
Suppose not.  Then there are \(x\notin S\) with
\(d=\operatorname{dist}(x,S)<r\) and two distinct nearest points
\(a,b\in S\), \(|x-a|=|x-b|=d\).  Fix
\(0<\delta<\min\bigl\{d,\tfrac12(r-d)\bigr\}\) and set
\[
  x_a=x+\delta\,\frac{a-x}{d},
  \qquad
  x_b=x+\delta\,\frac{b-x}{d}.
\]
We claim \(a\) is the unique nearest point of \(x_a\) in \(S\).  Indeed
\(\operatorname{dist}(x_a,S)\leq|x_a-a|=d-\delta\); and if \(q\in S\) satisfies
\(|x_a-q|\leq d-\delta\), then
\(|x-q|\leq|x-x_a|+|x_a-q|\leq d\), forcing \(|x-q|=d\) and equality in the
triangle inequality, so \(x_a\) lies on the segment \([x,q]\) and hence
\(q=a\).  Similarly \(b\) is the unique nearest point of \(x_b\).

Let \(\pi_k\) denote the nearest-point projection to \(S_k\), which is
single-valued on the open \(r\)-neighbourhood of \(S_k\) and, by Federer
\cite[Theorem~4.8(8)]{FedererReach}, Lipschitz with constant
\(L=r/(r-d')\) on \(\{y:\operatorname{dist}(y,S_k)\leq d'\}\) for each
\(d'<r\); we take \(d'=d-\tfrac\delta2\), admissible for large \(k\) since
\(\operatorname{dist}(x_a,S_k)\to\operatorname{dist}(x_a,S)=d-\delta\) by
Hausdorff convergence, and likewise for \(x_b\).

Every accumulation point \(q\) of the sequence \(\bigl(\pi_k(x_a)\bigr)_k\)
lies in \(S\) by Hausdorff convergence and satisfies
\(|x_a-q|=\lim_k\operatorname{dist}(x_a,S_k)=d-\delta\), so \(q\) is a nearest
point of \(x_a\) in \(S\) and hence \(q=a\) by the uniqueness claim.  Thus
\(\pi_k(x_a)\to a\), and likewise \(\pi_k(x_b)\to b\).  But then
\[
  |a-b|
  =\lim_k\bigl|\pi_k(x_a)-\pi_k(x_b)\bigr|
  \leq L\,|x_a-x_b|
  \leq \frac{r}{r-d}\cdot 2\delta ,
\]
which is smaller than \(|a-b|\) once \(\delta\) is chosen small.  This
contradiction proves the lemma.
\end{proof}

\begin{lemma}[Diameter bound for connected thick surfaces]
\label{lem:diameter-bound}
Fix \(\tau,\Delta>0\).  There is a constant \(C_{\mathrm{diam}}>0\), depending
only on the bounded-geometry convention of
\Cref{def:relative-surface-thickness}, such that every connected compact
properly embedded surface \(F'\) with \(\Thi(F')\geq\tau\) and
\(\Area(F')\leq\Delta\) satisfies
\[
  \operatorname{diam}_{\R^3}(F')
  \ \leq\
  \operatorname{diam}_{\mathrm{intr}}(F')
  \ \leq\
  C_{\mathrm{diam}}\,\frac{\Delta}{\tau}.
\]
\end{lemma}

\begin{proof}
Choose a maximal subset \(\mathcal P\subset F'\) whose points have pairwise
intrinsic distance at least \(c_0\tau\).  The intrinsic balls, or half-balls
near \(\partial F'\), of radius \(c_0\tau/2\) centered at \(\mathcal P\) have
pairwise disjoint interiors, and the graphical and collar clauses of
\Cref{def:relative-surface-thickness} give each of them area at least
\(a\,(c_0\tau)^2\) for a universal \(a>0\).  Hence
\(|\mathcal P|\leq\Delta/(a c_0^2\tau^2)\).  By maximality the intrinsic balls
of radius \(c_0\tau\) centered at \(\mathcal P\) cover \(F'\).  Since \(F'\) is
connected, any two of its points are joined by an intrinsic path, and a chain
of covering balls along this path shows that the intrinsic diameter is at most
\(2c_0\tau\,|\mathcal P|\leq (2/(ac_0))\,\Delta/\tau\).  The extrinsic diameter
is dominated by the intrinsic one.
\end{proof}

\begin{lemma}[Anchoring of essential surfaces]
\label{lem:anchoring}
Let \(\gamma\) be normalized with barycenter at the origin and
\(\Len(\gamma)\leq\Lambda\), and set
\(B_*=\overline B\bigl(0,\tfrac\Lambda2+1\bigr)\), a fixed compact ball
containing \(N_{\rho_0}(\gamma)\).  Then every properly embedded surface
\(F\subset E(\gamma)\) all of whose components are essential in \(E(\gamma)\)
satisfies: each component of \(F\) meets \(B_*\).  Consequently, if in
addition \(\Thi(F)\geq\tau\) and \(\Area(F)\leq\Delta\), then
\[
  F\subset K_0
  :=\overline B\Bigl(0,\ \tfrac\Lambda2+1+C_{\mathrm{diam}}\tfrac\Delta\tau\Bigr)
  \cap E^\circ(\gamma),
\]
a fixed compact subset of \(E^\circ(\gamma)\) depending only on
\((\Lambda,\Delta,\tau)\).
\end{lemma}

\begin{proof}
Since the barycenter of \(\gamma\) is the origin and
\(\operatorname{diam}(\gamma)\leq\Len(\gamma)/2\leq\Lambda/2\), the curve lies
in \(\overline B(0,\Lambda/2)\) and its \(\rho_0\)-tube in \(B_*\).  Let
\(F'\) be a component of \(F\).  If \(\partial F'\neq\varnothing\), then
\(\partial F'\subset\partial E(\gamma)=\partial N_{\rho_0}(\gamma)\subset
B_*\), and \(F'\) meets \(B_*\).

Suppose instead that \(F'\) is closed and disjoint from \(B_*\).  The
complement \(U=S^3\setminus B_*\) is an open round ball contained in
\(E(\gamma)\) (it is disjoint from \(N_{\rho_0}(\gamma)\)), so
\(F'\subset U\) with \(\pi_1(U)=1\).  Closed embedded surfaces in \(S^3\) are
orientable, hence two-sided.  If \(F'\) is not a sphere,
incompressibility of \(F'\) in \(E(\gamma)\) is equivalent, by the loop
theorem, to injectivity of
\(\pi_1(F')\to\pi_1(E(\gamma))\); this map factors through
\(\pi_1(U)=1\), a contradiction since \(\pi_1(F')\neq1\).  If \(F'\) is a
\(2\)-sphere, it separates \(S^3\) into two balls: the surfaces of this paper
are \(C^{1,1}\), hence \(C^1\) and locally flat, so the generalized
Schoenflies theorem applies (alternatively, smooth \(F'\) by a small isotopy
first).  The
connected set \(B_*\), which contains \(\gamma\) and is disjoint from \(F'\),
lies in one of the two balls, so the other is a ball in \(E(\gamma)\) bounded
by \(F'\); by the sphere convention in the definition of the essential
subspace, such a component is not essential.
This contradiction shows every essential component meets \(B_*\).

Finally, each component meets \(B_*\) and, by \Cref{lem:diameter-bound}, has
extrinsic diameter at most \(C_{\mathrm{diam}}\Delta/\tau\); hence \(F\) lies
in the stated ball, whose intersection with \(E^\circ(\gamma)\) is compact.
\end{proof}

\Cref{fig:anchoring} illustrates the two mechanisms of the lemma.

\begin{figure}[t]
\centering
\begin{tikzpicture}[x=1cm,y=1cm,>=Latex,
  note/.style={font=\scriptsize, align=center}]
  \draw[dashed, thin] (4.1,2.15) circle (2.95);
  \draw[thin] (4.1,2.15) circle (1.5);
  \node[note] at (2.15,4.15) {$B_*=\overline B(0,\tfrac\Lambda2+1)$};
  \draw[->, thin] (2.80,3.92) -- (3.30,3.32);
  \node[note] at (6.55,4.65) {$K_0$};
  \draw[->, thin] (6.45,4.50) -- (6.05,4.12);
  \draw[black!15, line width=5.5pt] (3.65,2.15) circle (0.52);
  \draw[thin] (3.65,2.15) circle (0.335);
  \draw[thin] (3.65,2.15) circle (0.705);
  \node[note] at (2.35,1.15) {$N_{\rho_0}(\gamma)$};
  \draw[->, thin] (2.62,1.35) -- (3.15,1.72);
  \draw[very thick]
    (4.36,2.32) .. controls (5.1,2.7) and (5.7,1.7) .. (6.35,2.3)
    .. controls (6.7,2.62) and (6.85,2.35) .. (6.95,2.20);
  \fill (4.36,2.32) circle (1.2pt);
  \node[note] at (5.6,3.0) {$F$};
  \draw[->, thin] (5.5,2.85) -- (5.35,2.48);
  \draw[<->, thin] (4.42,1.55) -- (6.95,1.55);
  \node[note, fill=white, inner sep=1pt] at (5.7,1.55)
    {$\leq C_{\mathrm{diam}}\Delta/\tau$};
  \draw[very thick, dashed] (9.65,3.35) ellipse (0.85 and 0.5);
  \node[note] at (9.65,4.35) {closed, disjoint from $B_*$};
  \node[note, align=center] at (9.95,2.30)
    {lies in a ball $\subset E(\gamma)$\\$\Rightarrow$ not essential};
  \node[note] at (1.0,0.1) {};
\end{tikzpicture}
\caption{Anchoring and confinement (\Cref{lem:anchoring}), schematically.
The tube \(N_{\rho_0}(\gamma)\) lies in the fixed ball \(B_*\).  A component
with boundary is anchored through \(\partial F\subset\partial E(\gamma)\); a
closed component disjoint from \(B_*\) (dashed) would lie in the
simply connected complement of \(B_*\), hence in a ball inside the exterior,
and could not be essential.  Once anchored, the diameter bound
\(\operatorname{diam}\leq C_{\mathrm{diam}}\Delta/\tau\) confines every
essential component to the fixed compact set \(K_0\) (dashed circle).  This
anchoring is what allows the compactness lemma to operate in the unbounded
exterior.}
\label{fig:anchoring}
\end{figure}
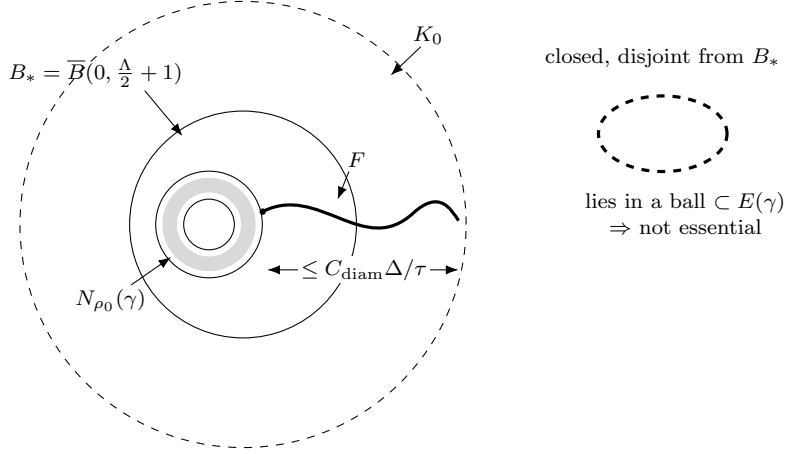

\begin{lemma}[Compactness of anchored thick surface families]
\label{lem:compactness-thick-surfaces}
Fix \(\gamma\), \(\tau>0\), and \(\Delta>0\).  Let \((F_n)_{n\geq1}\) be compact
properly embedded \(C^{1,1}\) surfaces in \(E(\gamma)\) with
\[
  \Thi(F_n)\geq\tau,
  \qquad
  \Area(F_n)\leq\Delta,
\]
and assume that the family is \emph{anchored}: there is a fixed compact set
\(B\subset\R^3\) met by every component of every \(F_n\).  Then a subsequence
converges locally graphically in \(C^1\), uniformly up to the boundary in the
sense of \Cref{def:graphical-convergence}, to a compact properly embedded
\(C^{1,1}\) surface \(F_\infty\subset E(\gamma)\) satisfying
\[
  \Thi(F_\infty)\geq\tau,
  \qquad
  \Area(F_\infty)\leq\Delta.
\]
The limit is not asserted to be essential.  When all components are essential,
the anchoring hypothesis follows from \Cref{lem:anchoring}.
\end{lemma}

\begin{proof}
By \Cref{lem:diameter-bound}, every component has extrinsic diameter at most
\(C_{\mathrm{diam}}\Delta/\tau\).  Since every component meets \(B\), all
surfaces lie in one compact set
\[
  K_0=\{x\in E^\circ(\gamma):
  \operatorname{dist}(x,B)\leq C_{\mathrm{diam}}\Delta/\tau\}.
\]
The topology and component-count estimates of
\Cref{prop:geometric-topological-boundedness,lem:area-lower-bound} allow us,
after passing to a subsequence, to fix the number and homeomorphism types of
the components and to label the boundary components consistently.

\emph{Step 1: a global Hausdorff limit.}
By Blaschke selection, a subsequence of the closed sets
\(\overline F_n\subset K_0\) converges in Hausdorff distance to a nonempty
compact set \(S\subset K_0\).  By
\Cref{lem:reach-hausdorff-closed},
\[
  \operatorname{reach}(S)\geq\tau.
\]
This global limit will force compatibility of all local chart limits; no
transition map will be declared constant along the sequence.

\emph{Step 2: fixed chart domains and \(C^1\) limits.}
Fix temporarily \(0<r<\tau\).  By the scale monotonicity of
\Cref{lem:thickness-scale-monotonicity}, every \(F_n\) satisfies the graphical
and collar estimates at scale \(r\).  Choose on each \(F_n\) a maximal
intrinsically \(c_0r/2\)-separated set, including a maximal such set on the
boundary.  The area lower bound for the corresponding disjoint graphical
half-balls gives a uniform bound \(M(r,\Delta)\) for the number of centers.
After passing to a further subsequence, the number and the interior/boundary
type of the centers are fixed, their positions converge, and their tangent
planes or tangent half-planes converge.

Use rigid motions converging to the limiting rigid motions to express every
chart over the corresponding limiting plane.  On the smaller fixed disk or
half-disk of radius \(c_0r/4\), its graph function \(u_{n,j}\) satisfies
\[
  \|u_{n,j}\|_{C^1}\leq C(r),
  \qquad
  \operatorname{Lip}(Du_{n,j})\leq C/r.
\]
Arzel\`a--Ascoli and a diagonal argument give
\[
  u_{n,j}\longrightarrow u_{\infty,j}
  \qquad\text{in }C^1
\]
for all chart indices.  Each limiting graph is a subset of \(S\).  Conversely,
if \(x\in S\), choose \(x_n\in F_n\) with \(x_n\to x\); a covering chart
contains each \(x_n\), and after fixing its index along a subsequence the
corresponding limit graph contains \(x\).  Therefore the limiting graphs
cover \(S\).  On an overlap two limiting graphs agree because both are equal
to the same subset of the unique Hausdorff limit \(S\).  This replaces the
incorrect assertion that continuously varying transition maps become
constant.  The functions \(u_{\infty,j}\) inherit the Lipschitz bound on their
first derivatives, so \(S\) is locally a \(C^{1,1}\) surface in the interior
and a \(C^{1,1}\) half-surface at its boundary.

\emph{Step 3: the boundary system and collar.}
By the boundary curvature, length, component-count, and reach estimates,
including \Cref{lem:boundary-system-reach}, the labelled boundary components
admit normalized arclength parametrizations
\[
  b_n:\Gamma\longrightarrow\partial F_n,
\]
where \(\Gamma\) is one fixed finite disjoint union of circles, such that,
after a subsequence, \(b_n\to b_\infty\) in \(C^1\).  The limit
\(B_\infty=b_\infty(\Gamma)\) lies on \(\partial E(\gamma)\), is embedded, and
is exactly the boundary stratum of \(S\) supplied by the boundary half-chart
limits.

Pull the collar maps back to the fixed domain by setting
\[
  \widehat C_n(s,z)=C_{F_n}(s,b_n(z)),
  \qquad (s,z)\in[0,c_0r]\times\Gamma.
\]
The scaled \(C^{1,1}\) estimates and Arzel\`a--Ascoli give
\(\widehat C_n\to\widehat C_\infty\) in \(C^1\), after another subsequence.
The collar estimates of clause (iv) of
\Cref{def:relative-surface-thickness} --- the local upper bound (iv-a), the
truncated lower bound (iv-b), and the separation (iv-c) of distinct
components --- are non-strict inequalities between continuous quantities and
are therefore closed under uniform convergence; in particular the limit map
\(\widehat C_\infty\) is injective by (iv-b) and (iv-c), hence an embedded
collar of \(B_\infty\) in \(S\).  The collar-coordinate clearance inequalities pass
directly to the limit:
\[
  \operatorname{dist}(\widehat C_\infty(s,z),\partial E(\gamma))
  \geq c_0s,
\]
and the part outside the half-collar remains at distance at least
\(c_0^2r/4\) from the peripheral torus.  Thus
\[
  S\cap\partial E(\gamma)=B_\infty=\partial S,
\]
so \(S\) is properly embedded in \(E(\gamma)\).

\emph{Step 4: preservation of thickness and area.}
The reach estimate at the full scale \(\tau\) was already obtained from the
global Hausdorff limit.  For every fixed \(r<\tau\), Steps 2 and 3 show that
the graphical, collar, and clearance clauses pass to the limit at scale
\(r\).  Since \(r<\tau\) is arbitrary and the defining inequalities are
closed as \(r\nearrow\tau\) --- the truncation \(\min\{r,\cdot\}\) in
(iv-b) and the thresholds \(c_0r\), \(c_0^2r/4\) depend continuously and
monotonically on \(r\) --- the clauses hold at scale \(\tau\) itself, so
\(\Thi(S)\geq\tau\).  Finally, a finite partition into the smaller graph
charts and the area formula give
\[
  \Area(S)\leq\liminf_{n\to\infty}\Area(F_n)\leq\Delta.
\]
Set \(F_\infty=S\).  The construction gives precisely the locally graphical
\(C^1\) convergence of \Cref{def:graphical-convergence}.
\end{proof}

\begin{corollary}[Moving-exterior compactness]
\label{cor:moving-exterior-compactness}
Let \(\gamma_j\to\gamma_\infty\) in \(C^1\), all
representatives of \(K\) with \(\Thi(\gamma_j)\geq1\) (exact unit thickness
is not required), barycenter at the origin, and length at most
\(\Lambda\), and let \(F_j\subset E(\gamma_j)\) be essential surfaces with
\(\Thi(F_j)\geq\tau\) and \(\Area(F_j)\leq\Delta\).  Then a subsequence of
\((F_j)\) converges in \(C^1\), uniformly up to the boundary --- in
particular the boundary curves \(\partial F_j\) converge in \(C^1\) --- to a
compact
properly embedded surface \(F_\infty\subset E(\gamma_\infty)\) with
\[
  \Thi(F_\infty)\geq\tau,
  \qquad
  \Area(F_\infty)\leq\Delta ,
  \qquad
  \partial F_\infty\subset\partial E(\gamma_\infty).
\]
We stress that the thickness of
\(F_\infty\) is obtained at the undistorted scale \(\tau\): no ambient
diffeomorphism is applied to the surfaces at any point.
\end{corollary}

\begin{proof}
The proof of \Cref{lem:compactness-thick-surfaces} applies with the fixed
exterior replaced by the moving ones.  The anchoring ball
\(B_*=\overline B(0,\Lambda/2+1)\) of \Cref{lem:anchoring} is uniform in
\(j\), because all \(\gamma_j\) are normalized with barycenter at the origin
and length at most \(\Lambda\); hence all \(F_j\) lie in one fixed compact
region \(K_0\).  The graphical charts are Euclidean and make no reference to
an exterior.  The boundary curves lie on the peripheral tori
\(\partial N_{\rho_0}(\gamma_j)\), which converge in \(C^1\), with uniform
\(C^{1,1}\) bounds, to \(\partial N_{\rho_0}(\gamma_\infty)\); the collar and
clearance clauses are stable under this convergence.  The limit \(F_\infty\)
lies in \(E(\gamma_\infty)\), since points at distance at least \(\rho_0\)
from \(\gamma_j\) have distance at least \(\rho_0\) from \(\gamma_\infty\) in
the limit; it has reach at least \(\tau\) by
\Cref{lem:reach-hausdorff-closed}; and it is properly embedded by the
clearance clause, exactly as in the fixed-exterior case.  The convergence is uniform up to the boundary, as in the fixed-exterior proof.
\end{proof}

\begin{remark}[Why reach, not merely bounded curvature]
\label{rem:reach-vs-curvature}
A bound on the second fundamental form alone does not give compactness: parallel
sheets could approach one another and the \(C^1\) limit would fail to be
embedded.  The reach lower bound simultaneously controls curvature and sheet
separation, so the limit is again an embedded \(C^{1,1}\) surface of reach at
least \(\tau\).  This is the exact place where relative thickness, rather than a
purely intrinsic bound, is used.
\end{remark}

The last ingredient of the smooth theory is a stability statement: two proper
pairs that are close in the \(C^1\) sense, \emph{including their peripheral
tori and their boundary curves}, are pair-isotopic through proper pairs.  We
do not construct the intermediate surface by blending an independently
interpolated collar with an independently interpolated interior.  Instead we
compose four genuine ambient isotopies: first the core and its fixed-radius
tube are moved, then the boundary curves are aligned on one fixed peripheral
torus, then the two collar germs are aligned while the torus is fixed, and
finally one surface is moved to the other as a single normal graph that
vanishes near the boundary.  The auxiliary lemmas below make each of these
steps quantitative and preserve properness automatically.

\begin{lemma}[Thickness stability of curve interpolation]
\label{lem:thickness-stability}
There are universal constants \(C_{\mathrm{th}}>0\) and
\(\varepsilon_0>0\) with the following property.  Let \(\eta_0\) be a closed
arclength-parametrized \(C^{1,1}\) curve with \(\Thi(\eta_0)\geq1\), and let
\(\eta_1\) be a closed \(C^{1,1}\) curve, parametrized over the same circle,
with
\[
  \|\eta_1-\eta_0\|_{C^1}\leq\varepsilon\leq\varepsilon_0,
  \qquad
  |\eta_1'|\in[1-\varepsilon,1+\varepsilon],
  \qquad
  |\eta_1''|\leq1+\varepsilon\ \text{a.e.}
\]
Then every interpolant \(\eta_t=(1-t)\eta_0+t\eta_1\), \(t\in[0,1]\), is an
embedded closed \(C^{1,1}\) curve with
\[
  \Thi(\eta_t)\ \geq\ 1-C_{\mathrm{th}}\,\varepsilon .
\]
\end{lemma}

\begin{proof}
Write \(T_t=\eta_t'/|\eta_t'|\), and recall the characterization
\(\Thi=\min\{1/\kappa_{\max},\ \tfrac12\operatorname{dcsd}\}\) of thickness by
maximal curvature and doubly critical self-distance
\cite{LitherlandSimonDurumericRawdon}; recall also that for a closed
\(C^{1,1}\) curve the thickness coincides with the Federer reach of the
image, so that \(\Thi(\eta_0)\geq1\) gives, by the tangent-cone criterion of
\Cref{rem:two-point-criterion} (tangent cones of a closed curve are the
tangent lines),
\begin{equation}
  \sin\angle\bigl(y-x,\ T_{x}\eta_0\bigr)\ \leq\ \frac{|y-x|}{2}
  \qquad\text{for all }x,y\in\eta_0.
  \tag{\(*\)}
\end{equation}

\emph{Step 1: speed and curvature.}  From
\(\eta_t'=(1-t)\eta_0'+t\eta_1'\) we get \(|\eta_t'|\in
[1-\varepsilon,1+\varepsilon]\) and \(|\eta_t''|\leq1+\varepsilon\) a.e., so
the curvature of \(\eta_t\) satisfies
\(\kappa_t\leq|\eta_t''|/|\eta_t'|^2\leq(1+\varepsilon)/(1-\varepsilon)^2
\leq1+4\varepsilon=:\kappa\) for \(\varepsilon\leq\tfrac18\).

\emph{Step 2: two classical consequences of a curvature bound.}  First, the
Schur-type chord bound: along any subarc of \(\eta_t\) of arclength
\(\ell\leq\pi/\kappa\), the chord has length at least
\((2/\kappa)\sin(\kappa\ell/2)\); this holds for \(C^{1,1}\) curves by smooth
approximation.  Second, chord--tangent positivity: if two parameters are
joined by an arc of length \(\ell<\pi/(2\kappa)\), then
\[
  \bigl\langle T_t(s),\ \eta_t(s')-\eta_t(s)\bigr\rangle
  =\int \bigl\langle T_t(s),\eta_t'(u)\bigr\rangle\,du
  \ \geq\ (1-\varepsilon)\,\ell\cos(\kappa\ell)\ >\ 0,
\]
since the tangent direction turns by at most \(\kappa\ell<\pi/2\) along the
arc.  In particular a \emph{doubly critical} pair of \(\eta_t\) --- one where
the chord is orthogonal to both tangents --- has both of its connecting arcs
of length at least \(\pi/(2\kappa)\geq\tfrac32(1-4\varepsilon)\geq\tfrac{7}5\).

\emph{Step 3: separated pairs of \(\eta_0\) have long chords.}  We claim that
any pair of points of \(\eta_0\) whose two connecting arcs both have length
at least \(1\) has chord at least \(2\sin\tfrac12>0.95\).  The chord length,
as a function on the compact set of such pairs, attains its minimum either at
an interior pair, where both partial derivatives of the squared distance
vanish, making the pair doubly critical, so the chord is at least
\(\operatorname{dcsd}(\eta_0)\geq2\Thi(\eta_0)\geq2\); or at a boundary pair,
where one arc has length exactly \(1\leq\pi\), and the Schur bound with
\(\kappa=1\) gives chord at least \(2\sin\tfrac12\).

\emph{Step 4: doubly critical pairs of \(\eta_t\).}  Let \((s,s')\) be doubly
critical for \(\eta_t\), with chord vector \(c_t=\eta_t(s')-\eta_t(s)\) and
\(D=|c_t|\).  Let \(x=\eta_0(s)\), \(y=\eta_0(s')\) be the corresponding
points of \(\eta_0\); then \(|(y-x)-c_t|\leq2\varepsilon\) and
\(|T_0(s)-T_t(s)|\leq C\varepsilon\).  By Step 2 the connecting arcs of the
pair exceed \(\tfrac75\), hence the corresponding arcs of \(\eta_0\) exceed
\(1\), and Step 3 gives \(|y-x|\geq0.95\).  Orthogonality
\(\langle T_t(s),c_t\rangle=0\) then yields
\[
  \bigl|\langle T_0(s),\,y-x\rangle\bigr|
  \leq
  \bigl|\langle T_0-T_t,\,c_t\rangle\bigr|
  +\bigl|\langle T_0,\,(y-x)-c_t\rangle\bigr|
  \leq C\varepsilon\,D+2\varepsilon ,
\]
so, dividing by \(|y-x|\geq0.95\) and using \(D\leq|y-x|+2\varepsilon\),
\[
  \cos\angle\bigl(y-x,\ T_x\eta_0\bigr)\ \leq\ C'\varepsilon,
  \qquad\text{hence}\qquad
  \sin\angle\bigl(y-x,\ T_x\eta_0\bigr)\ \geq\ 1-C'^2\varepsilon^2 .
\]
The criterion \((*)\) forces \(|y-x|\geq2(1-C'^2\varepsilon^2)\), and
therefore
\[
  D\ \geq\ |y-x|-2\varepsilon\ \geq\ 2-3\varepsilon
\]
for \(\varepsilon\) below a universal threshold.  Thus
\(\operatorname{dcsd}(\eta_t)\geq2-3\varepsilon\).

Combining Steps 1 and 4,
\(\Thi(\eta_t)\geq\min\{1-4\varepsilon,\,1-\tfrac32\varepsilon\}
\geq1-C_{\mathrm{th}}\varepsilon\).  Embeddedness follows from the same
Step~4 estimate: a self-intersection of \(\eta_t\) would be a doubly critical
pair with \(D=0\), which Step~4 excludes.
\end{proof}

\begin{lemma}[Periodic reparametrization of nearby thick curves]
\label{lem:periodic-reparametrization}
There are universal constants \(c_{\mathrm{par}},C_{\mathrm{par}}>0\) such
that the following holds.  Let \(\gamma_0,\gamma_1\) be closed
arclength-parametrized \(C^{1,1}\) curves with \(\Thi(\gamma_i)\geq1\).
If their images and tangent directions are \(\eta\)-close, with
\(\eta<c_{\mathrm{par}}\), then \(\gamma_1\) has an orientation-preserving
periodic reparametrization \(\eta_1\), defined on the parameter circle of
\(\gamma_0\), such that
\[
  \|\eta_1-\gamma_0\|_{C^1}\leq C_{\mathrm{par}}\eta,
  \qquad
  |\eta_1'|\in[1-C_{\mathrm{par}}\eta,1+C_{\mathrm{par}}\eta],
  \qquad
  |\eta_1''|\leq1+C_{\mathrm{par}}\eta
\]
almost everywhere.  The image of \(\eta_1\) is exactly the image of
\(\gamma_1\).
\end{lemma}

\begin{proof}
Nearest-point projection between the two curves is single-valued and
transverse for \(\eta\) small, and restricts to an orientation-preserving
bi-Lipschitz degree-one correspondence.  In arclength coordinates, lift its
inverse to a monotone map
\[
  \Theta:\R\longrightarrow\R,
  \qquad
  \Theta(s+L_0)=\Theta(s)+L_1,
\]
with \(\Theta'=1+O(\eta)\) almost everywhere; here \(L_i=\Len(\gamma_i)\),
and the same estimates give \(L_1/L_0=1+O(\eta)\).  Write
\[
  \Theta(s)=\lambda s+u(s),
  \qquad \lambda=L_1/L_0,
\]
where \(u\) is \(L_0\)-periodic.  Convolve \(u\) with a fixed smooth periodic
kernel of sufficiently small universal scale and put
\[
  \widetilde\Theta(s)=\lambda s+(u*\varphi)(s).
\]
Then
\(\widetilde\Theta(s+L_0)=\widetilde\Theta(s)+L_1\),
\(\widetilde\Theta'>0\), and
\[
  |\widetilde\Theta-\Theta|
  +|\widetilde\Theta'-1|
  +|\widetilde\Theta''|
  \leq C\eta.
\]
Thus \(\widetilde\Theta\) descends to an orientation-preserving degree-one
circle diffeomorphism.  Setting
\(\eta_1=\gamma_1\circ\widetilde\Theta\) preserves the image of \(\gamma_1\),
and the displayed estimates, together with
\(|\gamma_1''|\leq1\) almost everywhere, give the conclusion.
\end{proof}

\begin{lemma}[Parametric fixed-radius tube isotopy]
\label{lem:parametric-tube-isotopy}
Let \(\gamma_t\), \(0\leq t\leq1\), be a continuous \(C^1\)-isotopy of
closed \(C^{1,1}\) curves such that
\[
  \inf_t\Thi(\gamma_t)>\rho_0.
\]
Then there is an ambient isotopy \(H_t\) of \(S^3\) satisfying
\[
  H_t(\gamma_0)=\gamma_t,
  \qquad
  H_t(N_{\rho_0}(\gamma_0))=N_{\rho_0}(\gamma_t),
  \qquad
  H_t(E(\gamma_0))=E(\gamma_t).
\]
If the curve isotopy is \(C^1\)-small, \(H_t\) may be chosen \(C^1\)-small on
the fixed tube and supported in a slightly larger tube.
\end{lemma}

\begin{proof}
The strict reach inequality makes the fixed-radius normal disk bundles a
continuous \(C^1\) family of embedded solid tori.  Since the parameter
interval is contractible, these bundles can be trivialized continuously in
\(t\), producing an isotopy of solid-torus embeddings
\[
  e_t:S^1\times D^2_{\rho_0}\longrightarrow S^3
\]
whose cores are \(\gamma_t\) and whose images are exactly
\(N_{\rho_0}(\gamma_t)\).  The \(C^1\) isotopy extension theorem, in the local
triviality form of Palais \cite{PalaisRestriction}, extends
\(e_t\circ e_0^{-1}\) to an ambient isotopy; the locally flat topological
version is due to Edwards--Kirby \cite{EdwardsKirby}.  Local triviality also
shows that a sufficiently small tube isotopy has a \(C^1\)-small ambient
extension.  A cutoff in a slightly larger tube gives the support statement,
and the identities for the exteriors follow by taking complements.
\end{proof}

\begin{lemma}[Boundary alignment on a fixed peripheral torus]
\label{lem:boundary-alignment}
Fix a peripheral torus \(T=\partial E(\gamma)\).  There are
\(c_b,\delta_b>0\) such that the following holds.  Let
\(B_1\subset T\) be the boundary system of a surface of relative thickness
at least \(\tau\), and let \(B_0\subset T\) be an embedded \(C^1\) curve
system with the same number of components; \(B_0\) is \emph{not} assumed to
bound a thick surface, only to be an embedded \(C^1\) system --- in the
application it is the image of a thick boundary system under a
\(C^1\)-small ambient map, which controls positions and tangent directions
but no second-order data.  If \(B_0\) and \(B_1\) are
\(c_b\min\{\tau,\rho_0\}\)-close in position and
\(\delta_b\)-close in tangent direction, then there is a \(C^1\)-small
ambient isotopy \(K_t\), supported in a product neighbourhood of \(T\), with
\[
  K_t(T)=T,
  \qquad
  K_1(B_0)=B_1.
\]
The isotopy preserves each side of \(T\).
\end{lemma}

\begin{proof}
All curvature and separation control is taken on the \emph{thick} side
\(B_1\).  By \Cref{lem:boundary-system-reach}, applied to the surface of
relative thickness at least \(\tau\) whose boundary is \(B_1\), the system
\(B_1\) has normal injectivity radius in \(T\) at least
\(c_\partial\min\{\tau,\rho_0\}\), and its components have uniformly
disjoint intrinsic tubular neighbourhoods in \(T\) at that scale.  For
\(c_b\) small relative to \(c_\partial\), every point of \(B_0\) lies in
this tubular neighbourhood, and the intrinsic nearest-point projection
\(B_0\to B_1\) is well defined.  Each component of \(B_0\), being connected,
lies in the tubular annulus of a single component of \(B_1\); two-sided
position closeness forces every tubular annulus of a component of \(B_1\) to
contain at least one component of \(B_0\), and since the two systems have
the same number of components, the assignment is a bijection.  The
tangent-direction closeness makes the
projection an immersion, hence a covering map of circles on each component,
of some degree \(d\geq1\).  The degree is one for a topological reason:
the tubular annulus deformation-retracts onto its core component of
\(B_1\), so a degree-\(d\) covering exhibits the corresponding component of
\(B_0\) as an embedded circle in an open annulus freely homotopic to \(d\)
times the core; an embedded circle in an annulus is freely homotopic to
zero or \(\pm1\) times the core, and degree \(d\geq1\) excludes zero, so
\(d=1\).  Thus,
under this matching of components, each component of \(B_0\) is the intrinsic
normal graph in \(T\) of a unique small \(C^1\) section of the normal line
bundle of the corresponding component of \(B_1\).  Linear interpolation of
these sections to zero stays in the disjoint tubular neighbourhoods and
remains a disjoint embedded curve system at each time.  Extend the
resulting curve isotopy to a small isotopy of \(T\) by a cutoff vector
field in the intrinsic tubular neighbourhoods of the components of \(B_1\).
Extending this vector field constantly in the normal product coordinate of
\(T\), with another cutoff, gives the required ambient isotopy; only
\(C^1\) data of \(B_0\) enter, so \(K_t\) is \(C^1\)-small.  The normal
product coordinate is unchanged, so the two sides of \(T\) are preserved.
\end{proof}

\begin{lemma}[Relative collar alignment]
\label{lem:relative-collar-alignment}
Fix \(\tau>0\).  There are \(c_c,\delta_c>0\) with the following property.
Let \(G\subset E(\gamma)\) be a proper \(C^{1,1}\) surface of relative
thickness at least \(\tau\), let \(F_*\subset E(\gamma_*)\) be a proper
\(C^{1,1}\) surface of relative thickness at least \(\tau\) in a possibly
different unit-thickness exterior, and let \(\Phi\) be an ambient \(C^1\)
diffeomorphism, defined on a neighbourhood of \(F_*\), with
\[
  \sup|\Phi-\mathrm{id}|\leq c_c\tau,
  \qquad
  \sup\|D\Phi-I\|\leq c_c,
  \qquad
  F:=\Phi(F_*)\subset E(\gamma)
\]
proper, with the same boundary
\(\partial F=\partial G=B\subset T=\partial E(\gamma)\).  Assume in addition
that \(\Phi\) matches the two peripheral tori near the boundary and respects
their sides:
\[
  \Phi\bigl(\,U\cap\partial E(\gamma_*)\bigr)\subseteq\partial E(\gamma)
  \quad\text{for some neighbourhood \(U\) of }\partial F_*,
\]
and \(\Phi\) carries the exterior side of \(\partial E(\gamma_*)\) to the
exterior side of \(\partial E(\gamma)\) there.  This hypothesis is what
permits the transport of the peripheral clearance of \(F_*\) to \(F\) in
the proof; in the application in
\Cref{lem:pair-stability} it holds by construction.  The transported
surface \(F\) is only \(C^1\); no relative-thickness or reach hypothesis is
made for \(F\) itself.  If \(F\) and \(G\) are
\(c_c\tau\)-close in position and \(\delta_c\)-close in tangent planes in
their boundary collars, then an ambient \(C^1\)-small isotopy \(A_t\),
supported in those
collars and fixed pointwise on \(T\), carries \(F\) to a surface \(F'\) that
coincides with \(G\) on a smaller boundary collar.  The isotopy preserves
\(E(\gamma)\).
\end{lemma}

\begin{proof}
All second-order control is taken from the thick surfaces \(G\) and
\(F_*\); the transported surface \(F\) contributes only first-order data,
which are inherited from \(F_*\) up to the \(C^1\)-small distortion of
\(\Phi\).  The clearance inequality of clause (v) of
\Cref{def:relative-surface-thickness}, applied to \(F_*\) and to \(G\),
gives a uniform angle between each collar and the corresponding peripheral
torus: for \(F_*\) this angle is measured against
\(\partial E(\gamma_*)\).  By the torus-matching hypothesis, \(\Phi\)
carries \(\partial E(\gamma_*)\) to \(T=\partial E(\gamma)\) near
\(\partial F_*\), preserving sides; since \(\Phi\) is \(C^1\)-small, the
collar of \(F=\Phi(F_*)\)
meets \(T\) at an angle degraded by at most \(O(c_c)\), hence still
uniformly positive after decreasing \(c_c\).  Without the torus-matching
hypothesis this transport of clearance would fail, since the clearance of
\(F_*\) constrains its position relative to its own torus only.  Around one boundary component
choose product coordinates
\((z,u,v)\), with \(z\) along \(B\), such that
\[
  T=\{v=0\},
  \qquad
  G=\{u=0,\ v\geq0\}
\]
in a fixed collar box; the box dimensions depend only on the thickness
scale of \(G\).  Closeness and the uniform transversality just established
express \(F\) uniquely as
\[
  F=\{u=g(z,v),\ v\geq0\},
  \qquad g(z,0)=0,
\]
with small \(C^1\)-norm; only the \(C^1\) closeness of \(F\) to \(G\)
enters here, so \(g\) is \(C^1\) but need not be \(C^{1,1}\).  For a cutoff
\(\chi(v)\) equal to one near
\(v=0\), the shears
\[
  A_t(z,u,v)=(z,u-t\chi(v)g(z,v),v)
\]
are \(C^1\) embeddings when the constants are small.  They are the identity
on \(T\), preserve the coordinate \(v\geq0\), and carry the germ of \(F\) to
the germ of \(G\).  The boundary components of \(G\) have uniformly disjoint
collar boxes by \Cref{lem:boundary-system-reach} applied to \(G\), so the
local shears combine and may be cut off to the identity outside the
collars.
\end{proof}

\begin{lemma}[Relative normal graph after collar agreement]
\label{lem:interior-normal-graph}
There are universal constants \(c_g,\delta_g>0\) such that the following
holds.  Let \(F_*,G\subset\R^3\) be compact properly embedded \(C^{1,1}\)
surfaces with
\[
  \operatorname{reach}(\overline{F_*}),
  \operatorname{reach}(\overline G)\geq\tau.
\]
Let \(\Phi\) be an ambient \(C^1\) diffeomorphism, defined on a neighbourhood
of \(F_*\), with
\[
  \sup|\Phi-\mathrm{id}|\leq c_g\tau,
  \qquad
  \sup\|D\Phi-I\|\leq c_g,
\]
and put \(F=\Phi(F_*)\).  Suppose that
\[
  d_H(\overline F,\overline G)\leq c_g\tau,
\]
that corresponding tangent planes of \(F\) and \(G\) differ by at most
\(\delta_g\), and that \(F\) and \(G\) coincide on a neighbourhood of their
common boundary.  Then nearest-point projection
\[
  \pi_G:F\longrightarrow G
\]
is a \(C^1\) diffeomorphism, equal to the identity near the boundary.
Consequently \(F\) is the graph, in the normal line bundle of \(G\), of a
\(C^1\)-small section \(u\) that vanishes near \(\partial G\).  The family
\[
  F_s=\{x+s u(x):x\in G\},\qquad 0\leq s\leq1,
\]
is an isotopy through embedded surfaces fixed near the boundary.
No reach bound is assumed for the transported surface \(F\).
\end{lemma}

\begin{proof}
The reach of \(G\) makes \(\pi_G\) single-valued on \(F\).  The tangent-angle
bound and \(d_H(F,G)\ll\tau\) imply that its differential is nonsingular in
every graphical chart, so it is a local diffeomorphism.

We prove injectivity using the reach of the \emph{unmoved} source \(F_*\).
Suppose \(x=\Phi(x_*)\) and \(y=\Phi(y_*)\), with \(x\ne y\), lie on one
normal fiber of \(G\).  Then \(|x-y|\leq2c_g\tau\), and the chord \(y-x\) is
normal to \(G\).  Since \(T_xF\) is \(\delta_g\)-close to \(T_{\pi_G(x)}G\),
the chord makes angle \(O(\delta_g+c_g)\) with a normal to \(F\).  The
\(C^1\)-smallness of \(\Phi\) gives
\[
  |(y-x)-(y_*-x_*)|\leq c_g|y_*-x_*|
\]
and carries \(T_{x_*}F_*\) to within \(O(c_g)\) of \(T_xF\).  Hence
\(y_*-x_*\) makes angle \(O(\delta_g+c_g)\) with a normal to \(F_*\), while
\(|y_*-x_*|\leq3c_g\tau\).  Federer's tangent-cone inequality for
\(F_*\) gives
\[
  \operatorname{dist}(y_*-x_*,T_{x_*}F_*)
  \leq \frac{|y_*-x_*|^2}{2\tau}.
\]
The left-hand side is at least
\((1-O(\delta_g+c_g))|y_*-x_*|\), which is impossible for
\(c_g,\delta_g\) sufficiently small.  Therefore \(\pi_G\) is injective.

Its image is open by local invertibility and closed by compactness.  On every
component with boundary it contains a collar because the two surfaces agree
there.  For closed components, Hausdorff closeness and the separation of the
components of \(F_*\) and \(G\) at the reach scale match every component of
one with a component of the other.  Thus the image meets every component of
\(G\), and is all of \(G\).  The inverse gives a normal-bundle section \(u\),
equal to zero near the boundary.  Since \(|u|<\tau\), all graphs of \(su\)
lie in the embedded normal tube of \(G\); hence they are embedded and fixed
near the boundary.
\end{proof}

\begin{lemma}[Pair stability with moving peripheral tori]
\label{lem:pair-stability}
There are universal constants \(c_{\mathrm{st}}>0\) and
\(\delta_{\mathrm{st}}>0\), depending only on the fixed bounded-geometry
conventions, with the following property.  Fix \(\tau>0\), and let
\((\gamma_0,F_0)\) and \((\gamma_1,F_1)\) be pairs with
\(\Thi(\gamma_i)\geq1\) and with
\(F_i\subset E(\gamma_i)\) compact properly embedded surfaces satisfying
\(\Thi(F_i)\geq\tau\).  Suppose that, for some \(\eta,\delta>0\),
\[
  \eta\leq c_{\mathrm{st}}\min\{1,\tau\},
  \qquad
  \delta\leq\delta_{\mathrm{st}},
\]
and
\begin{enumerate}[label=(\alph*),leftmargin=2em]
\item the two cores are \(\eta\)-close in Hausdorff distance and in tangent
direction under nearest-point correspondence;
\item the two closed surface sets are \(\eta\)-close, and their tangent
planes are \(\delta\)-close at corresponding nearby points; and
\item after identification by the fixed-radius tube coordinates, the
boundary systems are \(\eta\)-close in position and \(\delta\)-close in
tangent direction.
\end{enumerate}
Then the pairs are ambient pair-isotopic through proper pairs.
\end{lemma}

\begin{proof}
We construct the isotopy as a concatenation of ambient isotopies.

\emph{Step 1: move the core together with its tube.}
Apply \Cref{lem:periodic-reparametrization} to reparametrize \(\gamma_1\) as
\(\eta_1\) on the parameter circle of \(\gamma_0\).  By
\Cref{lem:thickness-stability}, after reducing \(c_{\mathrm{st}}\), the linear
curve isotopy
\[
  \gamma_t=(1-t)\gamma_0+t\eta_1
\]
satisfies \(\Thi(\gamma_t)>\rho_0\) for every \(t\).  The parametric tube
isotopy lemma gives an ambient isotopy \(H_t\) with
\[
  H_t(N_{\rho_0}(\gamma_0))=N_{\rho_0}(\gamma_t).
\]
Set \(F'_0=H_1(F_0)\).  Then \(F'_0\subset E(\gamma_1)\) is proper, and the
smallness statement in \Cref{lem:parametric-tube-isotopy} preserves the
required \(C^1\)-closeness to \(F_1\).  We emphasize the bookkeeping used
from here on: the ambient isotopy \(H_1\) is only \(C^1\)-small, so the
transported surface \(F'_0\) is a proper \(C^1\) surface whose positions,
tangent planes, and clearance angles are close to those of the thick
surface \(F_0\), but which need be neither \(C^{1,1}\) nor of relative
thickness \(\geq\tau\) (\Cref{lem:reach-distortion}).  Accordingly, every
lemma applied below to a transported surface is stated for a
\(C^1\)-small image of a thick surface and takes its second-order control
from the unmoved thick source \(F_0\) or from the thick comparison surface
\(F_1\), never from a transported surface.

\emph{Step 2: align the boundary on the fixed torus.}
Both \(\partial F'_0\) and \(\partial F_1\) now lie on
\(T_1=\partial E(\gamma_1)\).  Apply
\Cref{lem:boundary-alignment} with the thick side
\(B_1=\partial F_1\) --- the boundary system of the surface \(F_1\) of
relative thickness at least \(\tau\), to which
\Cref{lem:boundary-system-reach} applies --- and with
\(B_0=\partial F'_0=H_1(\partial F_0)\), which is exactly the kind of
embedded \(C^1\) system, close in position and tangent direction, allowed
there.  A small ambient
isotopy \(K_t\), preserving \(T_1\) and its exterior side, carries
\(\partial F'_0\) to \(\partial F_1\).  Put
\(\widetilde F_0=K_1(F'_0)\).  Thus
\[
  \partial\widetilde F_0=\partial F_1.
\]

\emph{Step 3: align the collar germs.}
Apply \Cref{lem:relative-collar-alignment} with
\(G=F_1\), with thick source \(F_*=F_0\subset E(\gamma_0)\), and with the
\(C^1\)-small ambient diffeomorphism \(\Phi_2=K_1\circ H_1\), so that
\(F=\Phi_2(F_0)=\widetilde F_0\); the lemma requires no thickness of the
transported surface \(\widetilde F_0\), only of \(F_0\) and \(F_1\).  Its
torus-matching hypothesis holds by construction:
\(H_1\) carries the tube \(N_{\rho_0}(\gamma_0)\) onto
\(N_{\rho_0}(\gamma_1)\), hence \(\partial E(\gamma_0)\) onto
\(T_1=\partial E(\gamma_1)\) with the exterior sides corresponding
(\Cref{lem:parametric-tube-isotopy}), and \(K_1\) preserves \(T_1\) and its
sides (\Cref{lem:boundary-alignment}); so
\(\Phi_2(\partial E(\gamma_0))=T_1\) globally, not merely near
\(\partial F_0\).  The
resulting ambient isotopy is fixed on
\(T_1\), preserves \(E(\gamma_1)\), and produces a surface \(F''_0\) that
coincides with \(F_1\) on a smaller boundary collar.

\emph{Step 4: move the remaining interior as one normal graph.}
Let \(\Phi\) be the composition of the time-one maps in Steps 1--3, so
\(F''_0=\Phi(F_0)\).  The preceding ambient isotopies are \(C^1\)-small;
hence, after decreasing \(c_{\mathrm{st}}\) and
\(\delta_{\mathrm{st}}\), the hypotheses of
\Cref{lem:interior-normal-graph} hold with
\(F_*=F_0\), \(F=F''_0\), and \(G=F_1\).  Thus \(F''_0\) is a normal graph
over \(F_1\) of a section vanishing near the boundary.  Linear interpolation of that single section is an isotopy through
embedded proper surfaces fixed near \(T_1\).  By the isotopy extension theorem
for locally flat embeddings \cite{EdwardsKirby}, it extends to an ambient
isotopy supported away from the core and fixed near \(T_1\).

Concatenating the four ambient isotopies gives an ambient isotopy of \(S^3\)
that carries \((\gamma_0,F_0)\) to \((\gamma_1,F_1)\).  During Step 1 the
entire tube and exterior move together; during Steps 2--4 the core and its
peripheral torus are fixed and every isotopy preserves the exterior side.
Hence every intermediate pair is proper.
\end{proof}

\begin{remark}[What the stability lemma replaces]
\label{rem:stability-replaces}
All isotopy conclusions in this paper for pairs with different knot
representatives are routed through \Cref{lem:pair-stability}: the smooth
finiteness of the pair space, the attainment of visibility infima, and the
finite-resolution faithfulness theorem.  Two older shortcuts are deliberately
avoided.  Pulling surfaces back by a straightening diffeomorphism fails
because reach is not controlled under \(C^1\)-small maps
(\Cref{lem:reach-distortion}) and controlled \(C^{1,1}\) straightening maps
do not exist along \(C^{1,1}\) curves; and pushing boundaries into the
interior before interpolating fails because the intermediate surfaces would
not be properly embedded.  The revised proof keeps properness throughout by moving the entire
fixed-radius tube with the core, aligning the boundary and collar by ambient
isotopies in one fixed exterior, and using a single normal graph, zero near
the boundary, for the final interior motion.  No cutoff blend of two
independently constructed surface interpolations is used.

Moreover, the same discipline that rules out the first shortcut is enforced
inside the proof itself: since the ambient isotopies of Steps 1--3 are only
\(C^1\)-small, the transported surfaces \(F'_0\), \(\widetilde F_0\), and
\(F''_0\) carry no reach or relative-thickness hypotheses anywhere in the
argument.  \Cref{lem:boundary-alignment,lem:relative-collar-alignment,lem:interior-normal-graph}
are stated asymmetrically for exactly this reason: each takes a thick
unmoved source \(F_*\), its \(C^1\)-small image \(F=\Phi(F_*)\), and a thick
comparison surface \(G\), and draws all second-order control from \(F_*\)
and \(G\) alone.
\end{remark}

\subsection{Smooth finiteness}
\label{subsec:smooth-finiteness}

The compactness lemma turns the encoded finiteness of
\Cref{thm:finite-resolution-finiteness-filtered-pairs} into a genuinely smooth
statement: the number of isotopy classes is finite before any resolution is
chosen.  For closed submanifolds, this compactness--isotopy mechanism is due to
Durumeric \cite{DurumericCompactness}; the point here is its relative-boundary
form and the moving-pair extension below.

\begin{theorem}[Smooth finiteness in a fixed exterior]
\label{thm:smooth-finiteness-fixed-exterior}
Fix a unit-thickness representative \(\gamma\) of \(K\) and \(\tau,\Delta>0\).
Then the set of ambient isotopy classes in \(E(\gamma)\) of essential surfaces
\(F\) with \(\Thi(F)\geq\tau\) and \(\Area(F)\leq\Delta\) is finite.
\end{theorem}

\begin{proof}
If not, choose essential surfaces \(F_n\) with the stated bounds, pairwise
non-isotopic in \(E(\gamma)\).  The family is anchored by
\Cref{lem:anchoring}, so \Cref{lem:compactness-thick-surfaces} gives a
subsequence converging in \(C^1\), uniformly up to the boundary, to a proper
surface \(F_\infty\).  Hence for all large \(k,l\) the two pairs
\((\gamma,F_{n_k})\) and \((\gamma,F_{n_l})\) satisfy
\Cref{lem:pair-stability} with arbitrarily small positional and angular
errors.  They are therefore ambient isotopic in \(E(\gamma)\), contradicting
the choice.  Only mutual isotopy of the tail is used; the limit itself need
not be essential.
\end{proof}

For the pair-space version, the knot representative varies, and the naive
strategy --- straighten \(\gamma_n\) to a limit curve by an ambient
diffeomorphism and pull the surfaces back --- must be handled with care, for
two reasons.  First, \emph{reach is a second-order quantity and is not
continuous under \(C^1\)-small diffeomorphisms}: a \(C^1\)-small map with
merely bounded second derivatives can decrease the reach of a surface by a
definite factor.  The following lemma quantifies this, using the tangent-cone
criterion of \Cref{rem:two-point-criterion},
which is valid for arbitrary closed sets and hence for surfaces with
boundary.  Second, for \(C^{1,1}\) curves the straightening maps themselves
cannot be taken \(C^{1,1}\) with uniform bounds: normal frames along a
\(C^{1,1}\) curve are only Lipschitz, so exact tube-to-tube transport maps
have merely bounded, generally discontinuous, derivatives, and the lemma
below could not even be applied to them.  For both reasons, the proofs in
this subsection never pull surfaces back; they take limits with
\emph{moving exteriors}, where the reach passes to the limit undistorted.
The lemma is recorded because it explains precisely what is being avoided,
and because it is of independent use when a genuinely \(C^{1,1}\) change of
coordinates is available.

\begin{lemma}[Reach distortion under \(C^{1,1}\) diffeomorphisms]
\label{lem:reach-distortion}
Let \(\varphi\colon\R^3\to\R^3\) be a diffeomorphism with
\(\operatorname{Lip}(\varphi)\leq L\),
\(\operatorname{Lip}(\varphi^{-1})\leq L'\), and
\(\operatorname{Lip}(D\varphi)\leq M\).  Then every closed set
\(S\subset\R^3\) with \(\operatorname{reach}(S)\geq\tau\) satisfies
\[
  \operatorname{reach}\bigl(\varphi(S)\bigr)
  \ \geq\
  \frac{\tau}{L'^2\,(L+M\tau)}
  \ =:\ \tau'(\tau;L,L',M)\ >\ 0 .
\]
\end{lemma}

\begin{proof}
By Federer's tangent-cone criterion
(\Cref{rem:two-point-criterion}), it suffices to bound
\(\operatorname{dist}(\varphi(b)-\varphi(a),\operatorname{Tan}(\varphi(S),
\varphi(a)))\) for \(a,b\in S\).  Since \(\varphi\) is a diffeomorphism,
\(\operatorname{Tan}(\varphi(S),\varphi(a))=D\varphi_a\operatorname{Tan}(S,a)\).
Taylor's theorem with Lipschitz derivative gives
\(|\varphi(b)-\varphi(a)-D\varphi_a(b-a)|\leq\tfrac M2|b-a|^2\), and linearity
gives
\(\operatorname{dist}(D\varphi_a(b-a),D\varphi_a\operatorname{Tan}(S,a))
\leq L\operatorname{dist}(b-a,\operatorname{Tan}(S,a))
\leq \tfrac{L}{2\tau}|b-a|^2\), using the criterion for \(S\).  Adding, and
using \(|b-a|\leq L'|\varphi(b)-\varphi(a)|\),
\[
  \operatorname{dist}\bigl(\varphi(b)-\varphi(a),\
  \operatorname{Tan}(\varphi(S),\varphi(a))\bigr)
  \leq
  \Bigl(\frac M2+\frac{L}{2\tau}\Bigr)L'^2\,
  |\varphi(b)-\varphi(a)|^2
  =
  \frac{|\varphi(b)-\varphi(a)|^2}{2\tau'} . \qedhere
\]
\end{proof}

\begin{theorem}[Smooth finiteness of the filtered pair space]
\label{thm:smooth-finiteness-pair-space}
Fix a knot type \(K\) and \(\Lambda,\Delta,\tau>0\).  Then
\(\calZ^{\mathrm{ess}}_{\Lambda,\Delta,\tau}(K)\) consists of finitely many
pair-isotopy classes.
\end{theorem}

\begin{proof}
Suppose not, and choose pairs \((\gamma_n,F_n)\) representing pairwise distinct
pair-isotopy classes.  The set of representatives of \(K\) of thickness at least
\(1\) and length at most \(\Lambda\) is sequentially compact in the \(C^1\)
topology modulo \(\Isom^+(\R^3)\), by the Arzel\`a--Ascoli argument underlying
the ropelength existence theorem \cite{CantarellaKusnerSullivan}: thickness at
least one gives a uniform \(C^{1,1}\) bound, and the length bound controls the
domain.  After applying rigid motions and passing to a subsequence,
\(\gamma_n\to\gamma_\infty\) in \(C^1\), where \(\gamma_\infty\) is a
\(C^{1,1}\) representative of \(K\) with \(\Thi(\gamma_\infty)\geq1\) and
\(\Len(\gamma_\infty)\leq\Lambda\).

We do not pull the surfaces back to a fixed exterior: as explained before
\Cref{lem:reach-distortion}, a pull-back by \(C^1\)-small straightening maps
does not preserve the reach scale, and for \(C^{1,1}\) curves controlled
\(C^{1,1}\) straightening maps are not available.  Instead, by
\Cref{cor:moving-exterior-compactness}, a subsequence of \((F_n)\) converges
in \(C^1\) to a properly embedded surface
\(F_\infty\subset E(\gamma_\infty)\) with \(\Thi(F_\infty)\geq\tau\) and
\(\Area(F_\infty)\leq\Delta\).

Since \(\gamma_m,\gamma_n\to\gamma_\infty\) in \(C^1\) and
\(F_m,F_n\to F_\infty\) in \(C^1\) with converging tangent planes and, by
\Cref{cor:moving-exterior-compactness}, with boundary curves converging in
\(C^1\), for all
large \(m,n\) the two pairs \((\gamma_m,F_m)\) and \((\gamma_n,F_n)\) are
close in the sense of clauses (a)--(c) of
\Cref{lem:pair-stability}, with positional error at most
\(c_{\mathrm{st}}\min\{1,\tau\}\) and angular error at most
\(\delta_{\mathrm{st}}\).  The stability lemma
therefore provides a pair isotopy between them, through pairs that remain
proper in the moving exteriors at every stage; no thickness or area
constraint is imposed, or needed, along the way, since pair isotopy is an
equivalence of pairs, not an admissible thick deformation.  This contradicts
the choice of pairwise distinct classes and proves the theorem.
\end{proof}

\begin{remark}[The encoded theorem as an effective form]
\label{rem:encoded-is-effective}
\Cref{thm:smooth-finiteness-pair-space} is a purely smooth statement: it mentions
neither a resolution \(\varepsilon\) nor an encoding scheme.
\Cref{thm:finite-resolution-finiteness-filtered-pairs} is its effective
counterpart.  The smooth theorem asserts that the number of pair-isotopy classes
is finite; the encoded theorem bounds that number explicitly by
\((CN)^{BN^2}\) once a resolution and a finite encoding scheme are fixed
(\Cref{cor:explicit-encoded-bound}).  The reconstruction theorem below
(\Cref{thm:faithfulness}) shows that, at fine enough resolution, distinct
pair-isotopy classes receive distinct codes, so the explicit encoded bound is an
upper bound for the genuinely finite count of
\Cref{thm:smooth-finiteness-pair-space}.
\end{remark}

\begin{remark}[Consistency with the infiniteness of the full surface set]
\label{rem:slice-vs-whole}
\Cref{thm:smooth-finiteness-fixed-exterior,thm:smooth-finiteness-pair-space} do
not contradict the fact, emphasized throughout, that a fixed exterior may carry
infinitely many essential surfaces and an infinite Kakimizu complex.  The
finiteness holds for the slice cut out by \(\Thi\geq\tau\) and
\(\Area\leq\Delta\).  As \(\Delta\to\infty\) or \(\tau\to0\) the number of
classes may grow without bound; the infinite twisting families of the satellite
example escape every fixed \((\Delta,\tau)\)-slice by increasing area at fixed
thickness, or by requiring smaller and smaller thickness.  Thus the theorems
locate precisely which part of the essential-surface theory is finite.
\end{remark}

One point requires care in the statement.  The slice \(Y^{=}_\Lambda(K)\) is cut
out by the \emph{exact} normalization \(\Thi(\gamma)=1\), and this equality is
not preserved by \(C^1\) limits: a limit of unit-thickness curves satisfies
only \(\Thi(\gamma_\infty)\geq1\), and strict inequality can occur when the
features realizing the thickness (a curvature-one arc, or a doubly critical
pair at distance exactly \(2\)) disappear in the limit.  Rescaling
\(\gamma_\infty\) to unit thickness shrinks the ambient picture and with it
the surface, so \(\Thi(F)\geq\tau\) may fail after rescaling.  A limit
argument therefore does \emph{not} prove that the exact-slice level
\(\beta_{\mathfrak S}\) is attained; what it attains is a compactified
level, which we now define and which must be distinguished from
\(\beta_{\mathfrak S}\).  Write
\[
  Y^{\ge}_\Lambda(K)
  =\{\gamma\in K\mid\Thi(\gamma)\geq1,\ \Len(\gamma)\leq\Lambda\}
  /\Isom^+(\R^3)
  \ \supset\ Y^{=}_\Lambda(K),
\]
let \(\overline{\calZ}^{\mathfrak S}_{\Lambda,\Delta,\tau}(K)\) be the
corresponding compactified pair space, defined by the same surface
conditions over \(Y^{\ge}_\Lambda(K)\), and define the
\emph{compactified visibility level}
\[
  \overline\beta_{\mathfrak S}(K;\Delta,\tau)
  =
  \inf\bigl\{\Len(\gamma)\ \big|\
  (\gamma,F)\in\overline{\calZ}^{\mathfrak S}_{\Lambda,\Delta,\tau}(K)
  \text{ for some }\Lambda\bigr\}.
\]
Every \(\gamma\in Y^{\ge}_\Lambda(K)\) still has ropelength at most
\(\Lambda\), and since the exact slice is contained in the compactified one,
\[
  \overline\beta_{\mathfrak S}(K;\Delta,\tau)
  \ \leq\
  \beta_{\mathfrak S}(K;\Delta,\tau),
\]
with strict inequality not excluded in general: a pair with
\(\Thi(\gamma)>1\) may be shorter than every exact-slice pair.  For this
reason no claim is made that \(\beta_{\mathfrak S}\) itself is attained on
the compactified slice.

\begin{corollary}[Attainment of the compactified visibility level]
\label{cor:attainment-visibility}
Fix \(K\) and \(\Delta,\tau>0\), and let \(\mathfrak S\) be a filtered surface
type that is closed under pair isotopy and stable under \(C^1\)-limits of thick
representatives of a fixed topological type (for example a fixed isotopy class of
essential surface, a fixed boundary slope, a characteristic JSJ class, or the
taut class).  Whenever \(\overline\beta_{\mathfrak S}(K;\Delta,\tau)<\infty\), the following
hold.  In particular, this hypothesis is satisfied whenever
\(\beta_{\mathfrak S}(K;\Delta,\tau)<\infty\), since
\(\overline\beta_{\mathfrak S}\leq\beta_{\mathfrak S}\).
\begin{enumerate}[label=(\roman*),leftmargin=2em]
\item The compactified level is attained: there is a pair
\((\gamma,F)\) in the compactified pair space with \(\Thi(\gamma)\geq1\) and
\(\Len(\gamma)=\overline\beta_{\mathfrak S}(K;\Delta,\tau)\).
\item If some attaining pair satisfies \(\Thi(\gamma)=1\), then
\(\overline\beta_{\mathfrak S}=\beta_{\mathfrak S}\) and the exact-slice
level \(\beta_{\mathfrak S}\) is attained in
\(\calZ^{\mathfrak S}_{\beta,\Delta,\tau}(K)\).
\item At the ideal level this is automatic: one always has
\(\overline\beta_{\mathfrak S}\geq\Rop(K)\), because
\(\Len(\gamma)\geq\Thi(\gamma)\Rop(K)\geq\Rop(K)\) for every representative
with \(\Thi(\gamma)\geq1\); hence if
\(\beta_{\mathfrak S}(K;\Delta,\tau)=\Rop(K)\), then
\(\overline\beta_{\mathfrak S}=\beta_{\mathfrak S}=\Rop(K)\), every attaining
pair has \(\Thi(\gamma)=1\) (otherwise its ropelength would be strictly below
\(\Rop(K)\)), and the exact-slice level is attained.
\item A dilation by $s=\Thi(\gamma)^{-1}$ changes the fixed exterior-tube
radius $\rho_0$ to $s\rho_0$. It does not automatically produce a proper
surface in the prescribed exterior $E_{\rho_0}(s\gamma)$. Consequently,
exact-slice existence at altered surface parameters requires a separate
controlled collar adjustment and does not follow from scaling alone.

\end{enumerate}
The same statements hold for the compactified form of the area
visibility level \(a_{\mathfrak S}(K;\Lambda,\tau)\).
\end{corollary}

\begin{proof}
For (i), take a minimizing sequence for the \emph{compactified} level: pairs
\((\gamma_j,F_j)\) in the compactified space with
\(\Thi(\gamma_j)\geq1\) and
\(\Len(\gamma_j)\downarrow\overline\beta_{\mathfrak S}(K;\Delta,\tau)\).  The
class \(\{\Thi\geq1,\ \Len\leq\overline\beta+1\}\) is sequentially compact in
\(C^1\) modulo rigid motions, exactly as in the proof of
\Cref{thm:smooth-finiteness-pair-space}; the limit \(\gamma_\infty\)
satisfies \(\Thi(\gamma_\infty)\geq1\), because the curvature bound passes to
weak-\(*\) limits and the reach bound
\(\operatorname{reach}(\gamma_j)=\Thi(\gamma_j)\geq1\) --- for a closed
\(C^{1,1}\) curve the thickness, the normal injectivity radius, and the
Federer reach of the image coincide
\cite{LitherlandSimonDurumericRawdon,FedererReach} --- passes to Hausdorff
limits by \Cref{lem:reach-hausdorff-closed}, and
\(\Len(\gamma_\infty)=\lim_j\Len(\gamma_j)
=\overline\beta_{\mathfrak S}(K;\Delta,\tau)\), lengths being continuous
under \(C^1\) convergence.

We do \emph{not} pull the surfaces back by straightening
diffeomorphisms: by \Cref{lem:reach-distortion} a pull-back can lose a
definite amount of reach, which would only give \(\Thi(F_\infty)\geq\tau'<\tau\)
and would not prove attainment at the prescribed scale \(\tau\).  Instead,
\Cref{cor:moving-exterior-compactness}, whose hypotheses require only
\(\Thi(\gamma_j)\geq1\), provides a subsequence with
\(F_j\to F_\infty\subset E(\gamma_\infty)\) in \(C^1\), where \(F_\infty\) is
properly embedded with
\(\Area(F_\infty)\leq\Delta\) and, crucially,
\(\Thi(F_\infty)\geq\tau\) at the \emph{undistorted} scale: no
diffeomorphism has touched the surfaces.

To identify the type: for \(j\) large the pairs
\((\gamma_\infty,F_\infty)\) and \((\gamma_j,F_j)\) satisfy the positional and angular closeness hypotheses of
\Cref{lem:pair-stability}, including the boundary clause, since the boundary
curves converge in \(C^1\) by \Cref{cor:moving-exterior-compactness}; with both
errors as small as desired, the stability lemma joins them by a pair isotopy through proper
pairs (its thickness hypotheses, \(\Thi(\gamma)\geq1\) and
\(\Thi(F)\geq\tau\), hold for both pairs).  Hence
\(F_\infty\) is essential of type \(\mathfrak S\), and
\((\gamma_\infty,F_\infty)\) attains \(\overline\beta_{\mathfrak S}\), proving
(i).  Since the limit pair belongs to the same compactified class over which the
infimum is taken and its length equals the limiting infimum, it is a
minimizer.

For (ii): an attaining pair with \(\Thi(\gamma)=1\) lies in the exact slice,
so \(\beta_{\mathfrak S}\leq\Len(\gamma)=\overline\beta_{\mathfrak S}
\leq\beta_{\mathfrak S}\), forcing equality and exact attainment.  For (iii):
the inequality \(\Len\geq\Thi\cdot\Rop(K)\) is the definition of ropelength
as a scale-invariant infimum; the rest is stated in the corollary.  For (iv):
dilation sends the boundary of the radius-$\rho_0$ tube to that of the
radius-$s\rho_0$ tube. The fixed exterior condition is therefore not
preserved. The

area statement is
proved identically, minimizing \(\Area\) along the compact family in the
compactified slice.
\end{proof}

\subsection{Finite-resolution reconstruction}
\label{subsec:reconstruction}

The finiteness theorems above count pair-isotopy classes.  For the recognition
program of \Cref{sec:finite-recognition} one needs more: that the finite
\(\varepsilon\)-code determines the pair-isotopy class, so that distinct classes
are distinguished by their codes.  This is a sampling statement.  Its antecedent
for recovering homology from finite samples is the Niyogi--Smale--Weinberger
theorem \cite{NiyogiSmaleWeinberger}, and isotopic reconstruction of a fixed
positive-reach submanifold from a sufficiently dense sample is itself known:
Boissonnat--Ghosh \cite{BoissonnatGhosh} produce, from a dense enough point
sample of a closed submanifold with positive reach, a tangential Delaunay
complex ambient-isotopic to it, while reconstruction results for compact
surfaces with boundary were established by Abe--Bisceglio--Ferguson--Peters--%
Russell--Sakkalis \cite{AbeEtAl}.  The contribution here is therefore
\emph{not} isotopic reconstruction of a single manifold from a sample; what
is new is the setting and the joint statement: the exterior moves with the
knot representative, the surfaces are properly embedded with boundary on a
moving peripheral torus, the code records the knot, the surface, and their
relative position in one layered word, and equality of codes recovers the
\emph{pair}-isotopy class, uniformly over the whole filtered slice rather
than for one fixed submanifold.  Positive reach upgrades homology
recovery to an isotopy statement, but for surfaces \emph{with boundary} this
upgrade is not automatic, as the following remark explains.

\begin{remark}[No unconditional reach rigidity for surfaces with boundary]
\label{rem:no-boundary-reach-rigidity}
For \emph{closed} surfaces, positive reach makes the isotopy statement
elementary: if two closed surfaces of reach at least \(\tau\) are
\(c\tau\)-close in Hausdorff distance with tangent planes almost parallel at
nearby points, then nearest-point projection exhibits one as a small normal
graph over the other, and interpolation of the graph section gives a small
ambient isotopy.  For surfaces with boundary, however, Hausdorff closeness,
tangent-plane closeness, and even closeness of the boundaries do \emph{not}
imply that the nearest-point projection is surjective up to the boundary: a
model obstruction is a pair of flat half-planes with nearly identical tangent
planes, one shifted within its own plane in the direction transverse to the
boundary --- the region \(\{y\geq\epsilon\}\) viewed over
\(\{y\geq0\}\) --- for which the projection of one surface misses a
boundary collar of width \(\epsilon\) of the other while all three closeness
hypotheses hold at an
arbitrarily fine scale.  A correct statement must assume, in addition, that
the boundaries coincide and that the surfaces agree on a definite boundary
collar; that statement is exactly
\Cref{lem:interior-normal-graph}, which this paper uses instead.  All
isotopy conclusions below are therefore routed through the boundary and
collar alignment
\Cref{lem:boundary-alignment,lem:relative-collar-alignment} followed by
\Cref{lem:interior-normal-graph}, packaged as the pair-stability
\Cref{lem:pair-stability}; no free-standing reach-rigidity lemma for
surfaces with boundary is stated or used.
\end{remark}

\begin{theorem}[Finite-resolution faithfulness]
\label{thm:faithfulness}
Fix the concrete cubical--triangulated encoding scheme of
\Cref{con:concrete-encoding}.  Let \(C_{\mathrm{enc}}\) be the constant in
\Cref{lem:code-proximity} and put
\(C_*:=\max\{4,2C_{\mathrm{enc}}\}\).  Impose
\[
  \delta_{\gamma}\leq\varepsilon,
  \qquad
  \delta_F\leq\frac{\delta_{\mathrm{st}}}{8C_*}.
\]
There is a universal constant \(c_1\in(0,c]\) such that, for every
\(0<\varepsilon\leq c_1\min\{1,\tau\}\), any two pairs
\((\gamma_0,F_0),(\gamma_1,F_1)\in\calZ^{\mathrm{ess}}_{\Lambda,\Delta,\tau}(K)\)
with the same canonical encoded \(\varepsilon\)-type are pair-isotopic.  Consequently the
pair-isotopy class of a pair is a function of its encoded \(\varepsilon\)-type;
distinct pair-isotopy classes never share a code; and the boundary slope of each
boundary component and the essential intersection pattern with any fixed
bounded-geometry characteristic system are determined by the encoded
\(\varepsilon\)-type alone.
\end{theorem}

\begin{proof}
By the canonicalization convention of
\Cref{rem:code-well-defined}, equality of encoded \(\varepsilon\)-types means
that each pair admits an admissible normalization and discretization whose raw
word is the common least word.  Apply the corresponding orientation-preserving
rigid motions and choose those raw realizations.  Rigid motions preserve the
pair-isotopy class, so we may assume that the raw codes agree term by term.
By \Cref{lem:code-proximity}, the positional error is at most
\(4\varepsilon\), the knot-tangent error is at most
\(C_{\mathrm{enc}}(\delta_{\gamma}+\varepsilon)\), and the surface-plane and
boundary-tangent errors are at most
\(C_{\mathrm{enc}}(\delta_F+\varepsilon/\tau)\).  Set
\(\eta=C_*\varepsilon\).  Since \(\delta_{\gamma}\leq\varepsilon\), the first
two bounds are at most \(\eta\), as required in clause~(a) of
\Cref{lem:pair-stability}.  Both pairs also carry the full thickness data
\(\Thi(\gamma_i)=1\), \(\Thi(F_i)\geq\tau\), and
\(\Area(F_i)\leq\Delta\).
Choose \(c_1\) so that
\[
  C_*c_1\leq c_{\mathrm{st}},
  \qquad
  C_*c_1\leq\frac{\delta_{\mathrm{st}}}{4}.
\]
Then \(\eta\leq c_{\mathrm{st}}\min\{1,\tau\}\), while
\(C_{\mathrm{enc}}\varepsilon/\tau\leq C_*c_1\leq
\delta_{\mathrm{st}}/4\).  The imposed bound on \(\delta_F\) contributes at
most \(\delta_{\mathrm{st}}/8\), so the total surface angular error is below
\(\delta_{\mathrm{st}}\).  All choices are universal and independent of the
length and area budgets.  Hence
\Cref{lem:pair-stability} yields an ambient isotopy of \(S^3\) carrying
\((\gamma_0,F_0)\) to \((\gamma_1,F_1)\) through proper pairs.  We emphasize
what is being avoided: an argument that first straightens the curves by a
small ambient diffeomorphism \(\varphi\) and then compares \(\varphi(F_0)\)
with \(F_1\) by a closed-surface-style reach-rigidity argument
(\Cref{rem:no-boundary-reach-rigidity}) would be incomplete, because
\(\Thi(\varphi(F_0))\geq\tau\) does not follow from
\(\Thi(F_0)\geq\tau\) for a merely \(C^1\)-small \(\varphi\)
(\Cref{lem:reach-distortion}), because such a \(\varphi\) need not carry
the tube \(N_{\rho_0}(\gamma_0)\) onto \(N_{\rho_0}(\gamma_1)\), and because
for surfaces with boundary no such rigidity holds without prior boundary and
collar agreement.  The
stability lemma imposes thickness hypotheses only on the two given pairs and
matches the tubes explicitly, so neither issue arises.

Thus equal codes force pair-isotopy, so the pair-isotopy class factors through
the code; equivalently, by contraposition, pairs in distinct classes have
distinct codes.  Boundary slope and essential intersection pattern are
pair-isotopy invariants, hence are functions of the code.  The map is not a
bijection: one class may have many codes, corresponding to different geometric
positions; the content is that codes are fine enough to separate all classes.
\end{proof}

\begin{remark}[What faithfulness adds to the recognition program]
\label{rem:faithfulness-recognition}
Without \Cref{thm:faithfulness}, the finiteness theorems would only bound the
size of a search space.  Faithfulness makes the encoded data a complete
invariant on each bounded slice: the finite characteristic data of
\Cref{sec:finite-recognition} then genuinely determine the pair, not merely a
count.  The resolution requirement \(\varepsilon\leq c_1\min\{1,\tau\}\) is the precise
sense in which the sampling scale must be fine relative to the surface thickness
for reconstruction to hold.
\end{remark}

\subsection{Comparison with known finiteness and reconstruction theorems}
\label{subsec:comparison-finiteness}

It is worth situating the finiteness theorems above relative to the classical
finiteness results of three-dimensional topology, since they are finiteness
statements of a different kind, and relative to the manifold-reconstruction
literature in computational geometry, since the reconstruction ingredient is
not itself new.

On the reconstruction side, recovering a fixed submanifold up to ambient
isotopy from finite data under a positive-reach (or local feature size)
hypothesis is a developed subject: Niyogi--Smale--Weinberger
\cite{NiyogiSmaleWeinberger} recover the homotopy type from a dense sample,
and Boissonnat--Ghosh \cite{BoissonnatGhosh} construct from a sufficiently
dense sample a complex that is ambient-isotopic to the sampled manifold.
\Cref{thm:faithfulness} should accordingly not be read as a new
sampling-implies-isotopy principle.  Its content lies in the objects and the
uniformity: the sampled object is a knot--surface \emph{pair} whose two
constituents constrain each other; the surface is properly embedded with
boundary on the peripheral torus of a \emph{moving} exterior, where
reconstruction requires the boundary and collar alignment of
\Cref{lem:boundary-alignment,lem:relative-collar-alignment} rather than
closed-manifold normal-graph arguments
(\Cref{rem:no-boundary-reach-rigidity}); the code is a single canonical
combinatorial word on a fixed global lattice, shared by all pairs of a
filtered slice, rather than a sample of one fixed manifold; and equality of
words, not merely density of samples, is the hypothesis.  These are the
points at which the present theory goes beyond, and depends on more than,
manifold reconstruction; where the two overlap, the classical results are
used, not reproved.

Kneser and Haken finiteness \cite{Kneser,Haken} concerns a \emph{fixed}
triangulated three-manifold: a system of pairwise disjoint incompressible
surfaces, no two parallel, has boundedly many components, and the fundamental
normal surfaces in a fixed triangulation form a finite generating set for the
normal-surface semigroup.  These are finiteness statements about the
combinatorics of one triangulation.  They do not bound the number of isotopy
classes of essential surfaces: Haken sums \(F+nG\) already produce infinitely
many normal surfaces, and Kakimizu complexes are infinite.  The computational
side of this combinatorics is quantified by Hass--Lagarias--Pippenger
\cite{HassLagariasPippenger}, and effective versions of Haken-theoretic
certificates by Lackenby \cite{LackenbyCertification}.

\Cref{thm:smooth-finiteness-fixed-exterior,thm:smooth-finiteness-pair-space} are
finiteness statements of a complementary, metric kind.  They fix no
triangulation.  Instead they fix a geometric window: a lower bound \(\tau\) on
relative surface thickness and an upper bound \(\Delta\) on area, together with a
ropelength bound \(\Lambda\) on the representative.  Within this window the count
of isotopy classes is finite, even though the same exterior carries infinitely
many essential surfaces once the window is removed.  The mechanism is
compactness of the space of anchored bounded-geometry surfaces --- anchoring
being automatic for essential ones
(\Cref{lem:anchoring,lem:compactness-thick-surfaces}) --- not the
combinatorics of a normal coordinate system.  The two viewpoints meet in the normal-surface implementation
of \Cref{sec:finite-recognition}: a fixed bounded-geometry triangulation turns
the metric window into a bound on normal coordinates
(\Cref{prop:conditional-normal-coefficient-bound}), so that the metric
finiteness is realized inside the combinatorial one.  In this sense the present
theorems refine, rather than reprove, Haken finiteness: they select, among the
infinitely many normal surfaces of a fixed triangulation, the finitely many that
fit a prescribed area--thickness budget.

\section{Ideal knots, ideal surfaces, and ideal pairs}
\label{sec:ideal}

The central objects of the present framework are not only filtered pair
spaces, but also the ideal strata that appear at the bottom of the relevant
optimization problems.  The knot side and the surface side should be treated
in parallel.

\subsection{Ideal knots}

Recall from \Cref{sec:ropelength-sublevel} that the ropelength ideal stratum
\(I(K)=Y^{=}_{\Rop(K)}(K)\) is nonempty by the existence theorem for ropelength
minimizers \cite{CantarellaKusnerSullivan}, and that it consists precisely of
the unit-thickness ropelength minimizers of \(K\).  Thus an ideal knot is a
representative that is as short as possible after the thickness normalization
\(\Thi(\gamma)=1\).  The terminology follows the physical knot theory of ideal,
or tight, knots \cite{KatritchBednarNature,StasiakKatritchKauffman}.

\begin{proposition}[Universal ropelength lower bound]
\label{prop:positive-ropelength-lower-bound}
There is a universal constant \(c>0\) such that every unit-thickness
representative \(\gamma\) of every nontrivial knot type \(K\) satisfies
\[
  \Len(\gamma)\geq \Rop(K)\geq c.
\]
One may take \(c=31.32\).
\end{proposition}

\begin{proof}
For a unit-thickness representative, \(\Len(\gamma)=\Len(\gamma)/\Thi(\gamma)
\geq\Rop(K)\) by the definition of ropelength as an infimum.  The quadrisecant
estimates of Denne--Diao--Sullivan give the universal lower bound
\(15.66\) for every nontrivial knot type in their diameter normalization
\cite{DenneDiaoSullivan}, which is
\(\Rop(K)\geq31.32\) in the radius normalization fixed in
\Cref{rem:ropelength-normalization}.
\end{proof}

\subsection{Ideal surfaces in fixed exteriors}

The surface-side analogue is obtained by fixing a knot representative
\(\gamma\) and minimizing area among thick representatives of a fixed
essential surface class in \(E(\gamma)\).

\begin{definition}[\(\tau\)-ideal surface relative to a fixed exterior]
\label{def:tau-ideal-surface}
Let \(\gamma\) be a unit-thickness representative of \(K\), and let \([F]\) be
an isotopy or homotopy class of essential surfaces in \(E(\gamma)\).  For
\(\tau>0\), define
\[
  A_\tau([F];\gamma)
  =
  \inf\{
  \Area(F')\mid
  F'\in [F],\
  F'\subset E(\gamma),\
  F'\text{ is essential},\
  \Thi(F')\geq \tau
  \}.
\]
A surface \(F\in [F]\) is called a \(\tau\)-ideal surface relative to
\(\gamma\) if
\[
  \Thi(F)\geq \tau
  \quad\text{and}\quad
  \Area(F)=A_\tau([F];\gamma).
\]
The set of all such surfaces is denoted by
\[
  I_\tau([F];\gamma).
\]
\end{definition}

Thus an ideal surface is an area-minimizing thick representative of a fixed
essential surface class.  The thickness parameter \(\tau\) is part of the
definition: it records the scale at which the surface is required to be
visible as a bounded-geometry object.  If the infimum is not known to be
achieved, the set \(I_\tau([F];\gamma)\) may be empty and the ideal layer
should be interpreted as a limiting stratum.  In fact attainment is
automatic whenever the \(\tau\)-admissible class is nonempty
(\Cref{cor:unconditional-ideal-attainment} below), so emptiness of
\(I_\tau([F];\gamma)\) can occur only because no representative of \([F]\)
has relative thickness \(\geq\tau\) at all.

\begin{remark}[Existence versus least-area theory]
The existence theory for ideal knots is unconditional: by
Cantarella--Kusner--Sullivan, every knot or link type admits a ropelength
minimizer \cite{CantarellaKusnerSullivan}.  The surface side is governed by
classical least-area theory.  Freedman--Hass--Scott and Hass--Scott provide
least-area representatives for incompressible surfaces under suitable
geometric and boundary hypotheses \cite{FreedmanHassScott,HassScottLeastArea};
related existence and regularity results were developed by Meeks--Yau
\cite{MeeksYau}.  In the present notation, such a least-area representative
is an ideal surface in the classical sense.  If it is compact, embedded, and
has positive relative thickness, then it is a \(\tau\)-ideal surface for every
sufficiently small \(\tau>0\).  What is not automatic is existence for a
prescribed large value of \(\tau\).
\end{remark}

\begin{proposition}[Least-area representatives give small-\(\tau\) ideal surfaces]
\label{prop:least-area-gives-small-tau-ideal-surface}
Let \(\gamma\) be a unit-thickness representative of \(K\), and let \([F]\)
be an essential surface class in \(E(\gamma)\).  Suppose that \([F]\) admits a
least-area representative \(F_0\) in the appropriate relative sense, and that
\(F_0\) is a compact embedded surface with positive relative thickness.  Then
there exists \(\tau_0>0\) such that, for every \(0<\tau\leq\tau_0\),
\[
  F_0\in I_\tau([F];\gamma).
\]
In particular, the \(\tau\)-ideal surface layer is nonempty for all
sufficiently small \(\tau\).
\end{proposition}

\begin{proof}
Since \(F_0\) is compact, embedded, and has positive relative thickness, set
\[
  \tau_0=\Thi(F_0)>0.
\]
Then \(F_0\) is admissible for every \(0<\tau\leq\tau_0\) by
\Cref{lem:thickness-scale-monotonicity}.  Since \(F_0\)
minimizes area among all representatives of the class \([F]\), it also
minimizes area among the smaller class of representatives satisfying
\(\Thi(F')\geq\tau\).  Hence
\[
  \Area(F_0)=A_\tau([F];\gamma),
\]
and therefore \(F_0\in I_\tau([F];\gamma)\).
\end{proof}

\begin{remark}
The conclusion is deliberately a small-\(\tau\) statement.  The thickness
scale obtained from a least-area representative depends on the representative
and on the metric geometry of the knot exterior.  Thus least-area theory does
not imply nonemptiness of \(I_\tau([F];\gamma)\) for an arbitrarily prescribed
large value of \(\tau\).
\end{remark}

\subsection{Ideal surface layers over \texorpdfstring{\(Y^{=}_\Lambda(K)\)}{Y Lambda(K)}}

One can also allow the knot representative to vary while remaining inside a
ropelength level \(Y^{=}_\Lambda(K)\).  Define
\[
  A_\tau([F];\Lambda)
  =
  \inf_{\gamma\in Y^{=}_\Lambda(K)}
  A_\tau([F];\gamma).
\]
The corresponding ideal surface layer is
\[
  I_{\Lambda,\tau}([F];K)
  =
  \left\{
  (\gamma,F)
  \ \middle|
  \gamma\in Y^{=}_\Lambda(K),\
  F\in [F],\
  \Thi(F)\geq\tau,\
  \Area(F)=A_\tau([F];\Lambda)
  \right\}.
\]

When \(\gamma\) varies, the notation \([F]\) should be interpreted
topologically: the surface class is transported under pair isotopies, or
equivalently it represents a specified topological surface class in the knot
exterior of \(K\).  This convention is the same one used in the filtered pair
space \(\calZ_{\Lambda,\Delta,\tau}^{\mathrm{ess}}(K)\).

\subsection{Ideal pairs}

The simultaneous bottom layer is obtained by requiring both the knot and the
surface to be ideal.

\begin{definition}[\(\tau\)-ideal pair stratum]
\label{def:tau-ideal-pair}
Let \(K\) be a knot type and let \([F]\) be an essential surface class in the
exterior.  For \(\tau>0\), define
\[
  I_\tau(K,[F])
  =
  \left\{
  (\gamma,F)
  \ \middle|
  \gamma\in I(K),\
  F\in I_\tau([F];\gamma)
  \right\}.
\]
When nonempty, this stratum consists of pairs in which the knot representative
is ropelength-minimizing and the surface is area-minimizing among
\(\tau\)-thick representatives of its class.
\end{definition}

\begin{corollary}[Small-\(\tau\) ideal pairs]
\label{cor:small-tau-ideal-pairs}
Let \(\gamma\in I(K)\).  Suppose that an essential surface class
\([F]\subset E(\gamma)\) admits a compact embedded least-area representative
\(F_0\) with positive relative thickness.  Then, for every sufficiently small
\(\tau>0\),
\[
  (\gamma,F_0)\in I_\tau(K,[F]).
\]
In particular, the \(\tau\)-ideal pair stratum is nonempty for all sufficiently
small \(\tau\) for this chosen ideal knot exterior and surface class.
\end{corollary}

\begin{proof}
Apply Proposition~\ref{prop:least-area-gives-small-tau-ideal-surface} in the
exterior \(E(\gamma)\).  Since \(\gamma\in I(K)\), the resulting pair belongs
to \(I_\tau(K,[F])\).
\end{proof}

The ideal pair stratum may still be empty for a prescribed thickness scale
\(\tau\), for a prescribed surface class, or for a different ideal knot
representative.  This emptiness is meaningful: it indicates that a
ropelength-minimizing representative of \(K\) does not provide enough
geometric room for the class \([F]\) to be realized as a \(\tau\)-thick ideal
surface.  In such cases the relevant surface may appear only after the
ropelength level is increased, or after the surface thickness scale is
decreased.

\begin{proposition}[Conditional existence of ideal pairs at prescribed \(\tau\)]
\label{prop:conditional-existence-prescribed-tau}
Let \(K\) be a knot type and let \([F]\) be an essential surface class.  Assume
that there exists an ideal knot representative \(\gamma\in I(K)\) such that the
class \([F]\) in \(E(\gamma)\) admits at least one representative satisfying
\(\Thi(F)\geq\tau\).  Assume furthermore that the constrained class
\[
  \mathcal A_\tau([F];\gamma)
  =
  \{F'\in [F]\mid \Thi(F')\geq\tau\}
\]
is compact enough for the area functional to attain its infimum on it; for
example, this may be taken as an explicit hypothesis of compactness in the
chosen \(C^{1,1}\) topology.  Then the \(\tau\)-ideal pair stratum
\[
  I_\tau(K,[F])
\]
is nonempty.
\end{proposition}

\begin{proof}
Choose \(\gamma\in I(K)\) satisfying the hypotheses.  The admissible class
\(\mathcal A_\tau([F];\gamma)\) is nonempty.  By the stated compactness or
attainment hypothesis, there is a representative
\(F\in\mathcal A_\tau([F];\gamma)\) with
\[
  \Area(F)=\inf\{\Area(F')\mid F'\in\mathcal A_\tau([F];\gamma)\}.
\]
Thus \(F\in I_\tau([F];\gamma)\), and therefore \((\gamma,F)\in I_\tau(K,[F])\).
\end{proof}

\begin{remark}[Why the prescribed-\(\tau\) statement is conditional]
\label{rem:why-conditional}
The added attainment hypothesis is not a formal consequence of
Freedman--Hass--Scott or Hass--Scott.  Those least-area theorems give
unconstrained least-area representatives under their hypotheses; they do not
by themselves solve the nonconvex variational problem obtained by imposing a
positive lower bound \(\Thi(F)\geq\tau\).  Consequently the proposition is a
conditional statement at a prescribed surface-thickness scale.  What follows
unconditionally from an embedded least-area representative \(F_0\) of positive
relative thickness is the small-thickness conclusion of
\Cref{prop:least-area-gives-small-tau-ideal-surface}: if
\(\tau\leq\Thi(F_0)\), then \(F_0\) is \(\tau\)-admissible and realizes the same
area minimum among the larger unconstrained class.  A possible limiting
procedure as \(\tau\to0\) is therefore to study the nonempty layers
\(I_\tau([F];\gamma)\) for sufficiently small \(\tau\), while treating large
prescribed values of \(\tau\) as a separate constrained compactness problem.
\end{remark}

In fact the attainment hypothesis of
\Cref{prop:conditional-existence-prescribed-tau} is automatic: positive reach is
itself the compactness-providing condition, so the constrained minimization
always attains its infimum.

\begin{corollary}[Attainment for nonempty prescribed-thickness classes]
\label{cor:unconditional-ideal-attainment}
Let \(\gamma\) be a unit-thickness representative of \(K\) and \([F]\) an
essential surface class in \(E(\gamma)\).  If the \(\tau\)-admissible class
\[
  \mathcal A_\tau([F];\gamma)=\{F'\in[F]\mid \Thi(F')\geq\tau\}
\]
is nonempty, then the infimum \(A_\tau([F];\gamma)\) is attained, so
\(I_\tau([F];\gamma)\neq\varnothing\).  In particular, if \(\gamma\in I(K)\)
and \(\mathcal A_\tau([F];\gamma)\neq\varnothing\), then
\(I_\tau(K,[F])\neq\varnothing\).  This is an attainment statement, not an
existence statement: nonemptiness of the \(\tau\)-admissible class at the
prescribed \(\tau\) remains a hypothesis, and only the compactness
hypothesis of \Cref{prop:conditional-existence-prescribed-tau} is removed.
\end{corollary}

\begin{proof}
Let \(F_n\in\mathcal A_\tau([F];\gamma)\) be an area-minimizing sequence, so
\(\Area(F_n)\to A_\tau([F];\gamma)\leq\Area(F_1)<\infty\).  The surfaces are
essential, hence the family is anchored by \Cref{lem:anchoring}, and by
\Cref{lem:compactness-thick-surfaces} a subsequence converges in \(C^1\) to
\(F_\infty\) with \(\Thi(F_\infty)\geq\tau\) and
\(\Area(F_\infty)\leq A_\tau([F];\gamma)\).  The \(C^1\)-convergence,
together with the common thickness bound and the fixed peripheral torus,
places \(F_n\) and \(F_\infty\) under the fixed-core case of
\Cref{lem:pair-stability} for all large \(n\); hence they are ambient
isotopic.  Therefore \(F_\infty\in[F]\), \(F_\infty\) is
essential, and \(F_\infty\in\mathcal A_\tau([F];\gamma)\), so
\(\Area(F_\infty)\geq A_\tau([F];\gamma)\).  Therefore
\(\Area(F_\infty)=A_\tau([F];\gamma)\) and \(F_\infty\in I_\tau([F];\gamma)\).
If \(\gamma\in I(K)\) the resulting pair lies in \(I_\tau(K,[F])\).
\end{proof}

\begin{remark}[Positive reach supplies, rather than obstructs, compactness]
\label{rem:reach-gives-compactness}
This resolves the tension noted in \Cref{rem:why-conditional}.  Imposing
\(\Thi(F')\geq\tau\) does not make the variational problem intractable; on the
contrary, positive reach is exactly the condition that supplies
\(C^{1,1}\)-compactness (\Cref{lem:compactness-thick-surfaces}), so the
constrained minimization is better behaved than the unconstrained least-area
problem.  What is genuinely not automatic is the reverse implication --- that an
\emph{unconstrained} least-area representative attains a prescribed thickness
\(\tau\) --- which is a question about the intrinsic geometry of the exterior.
\end{remark}

\subsection{The classical limit \texorpdfstring{\(\tau\to0\)}{tau to 0}}
\label{subsec:classical-limit}

The thickness parameter \(\tau\) is a geometric regularization of the area
functional.  It is important that, as \(\tau\to0\), the filtered theory recovers
the classical least-area theory of essential surfaces, so that the
\(\tau\)-filtration is a genuine deformation of the classical picture rather than
a separate object.

Recall from \Cref{def:tau-ideal-surface} that \(A_\tau([F];\gamma)\) is
the infimum of area over essential representatives of \([F]\) with
\(\Thi(F')\geq\tau\).  Since a smaller thickness threshold enlarges the
admissible class, \(\tau\mapsto A_\tau([F];\gamma)\) is nondecreasing.  Write
\[
  A_{\mathrm{sm}}([F];\gamma)
  =
  \inf\{\Area(F')\mid F'\in[F]\text{ essential, embedded }C^{1,1},\
  \operatorname{reach}(F')>0\}
\]
for the least area among all positive-reach (equivalently, smoothly embedded
bounded-geometry) representatives.

\begin{proposition}[Classical limit of the area filtration]
\label{prop:classical-limit-area}
For every essential surface class \([F]\) in \(E(\gamma)\),
\[
  \lim_{\tau\to0^+}A_\tau([F];\gamma)
  =
  \inf_{\tau>0}A_\tau([F];\gamma)
  =
  A_{\mathrm{sm}}([F];\gamma).
\]
If moreover \([F]\) admits a compact embedded least-area representative \(F_0\)
of positive relative thickness, then \(A_{\mathrm{sm}}([F];\gamma)=\Area(F_0)\)
is the classical least area, and
\[
  A_\tau([F];\gamma)=\Area(F_0)
  \qquad\text{for all }0<\tau\leq\Thi(F_0),
\]
so the filtration is eventually constant as \(\tau\to0\) and attains the
classical value.
\end{proposition}

\begin{proof}
Monotonicity gives
\(\lim_{\tau\to0^+}A_\tau=\inf_{\tau>0}A_\tau=:L\).  Each admissible class in the
definition of \(A_\tau\) consists of positive-reach representatives, so
\(A_\tau\geq A_{\mathrm{sm}}\) for every \(\tau\), whence \(L\geq
A_{\mathrm{sm}}\).  Conversely, let \(F'\) be a positive-reach essential
representative and let \(\epsilon>0\).  By the final paragraph of \Cref{def:relative-surface-thickness}, an
arbitrarily small
collar adjustment replaces \(F'\) by an isotopic representative \(F''\) with
\(\Area(F'')\leq\Area(F')+\epsilon\) satisfying all five clauses of the
definition at some scale \(\tau'>0\); hence, by
\Cref{lem:thickness-scale-monotonicity}, \(F''\) is admissible for
every \(\tau\leq\tau'\), so \(A_\tau\leq\Area(F')+\epsilon\) for small
\(\tau\) and hence
\(L\leq\Area(F')+\epsilon\); taking the infimum over \(F'\) and
\(\epsilon\to0\) gives \(L\leq
A_{\mathrm{sm}}\).  Thus \(L=A_{\mathrm{sm}}\).

If a compact embedded least-area representative \(F_0\) has
\(\Thi(F_0)=\tau_0>0\), then by
\Cref{prop:least-area-gives-small-tau-ideal-surface} it is \(\tau\)-admissible
and area-minimizing for all \(\tau\leq\tau_0\), so \(A_\tau([F];\gamma)=\Area(F_0)\)
there; and \(A_{\mathrm{sm}}=\Area(F_0)\) since \(F_0\) already minimizes area
among all representatives.
\end{proof}

\begin{remark}[The visible Kakimizu layer exhausts the complex as \(\tau\to0\)]
\label{rem:kakimizu-exhaustion-tau}
The same monotonicity applies to visibility.  The notation
\(\mathrm{MS}_{\Lambda,\Delta,	au}(K)\) is defined in
\Cref{sec:kakimizu}.  By
\Cref{prop:visibility-monotonicity} the level
\(\beta_{\mathfrak S}(K;\Delta,\tau)\) is nondecreasing in \(\tau\), so
\(\lim_{\tau\to0^+}\beta_{\mathfrak S}(K;\Delta,\tau)\) is the thickness-free
birth level, the least ropelength at which a positive-reach representative of
\(\mathfrak S\) of area at most \(\Delta\) appears in some exterior over
\(Y^{=}_\Lambda(K)\).  Correspondingly the visible Kakimizu layer
\(\mathrm{MS}_{\Lambda,\Delta,\tau}(K)\) increases as \(\tau\to0\) and as
\(\Delta,\Lambda\to\infty\), exhausting the full complex \(\mathrm{MS}(K)\)
precisely when every minimal genus Seifert class has a positive-reach
representative realizable over \(Y^{=}_\Lambda(K)\) within area \(\Delta\).  Thus the
infinite classical objects are recovered as the \(\tau\to0\), \(\Delta\to\infty\)
limits of the finite filtered layers, and the finiteness of each layer is exactly
the assertion that these limits are approached one finite stratum at a time.
\end{remark}

\subsection{From ideal pairs to trade-off frontiers}

The trade-off frontier measures precisely this failure or success of
simultaneous ideality at a prescribed surface thickness scale.  Least-area
theory supplies ideal surfaces for sufficiently small \(\tau\) under the
regularity assumptions above, but it does not determine how much ropelength
room is needed at a fixed \(\tau\).  The frontier records this additional
ropelength cost.  For a fixed class \([F]\) and thickness scale \(\tau\), the
condition
\[
  \Delta_{\min}^{\tau}(\Rop(K);[F])<\infty
\]
means that the surface class is visible over the ideal knot stratum.  If the
corresponding ideal pair stratum is empty, then one asks how far \(\Lambda\)
must be increased beyond \(\Rop(K)\) before the class appears with controlled
area and thickness.

Thus the conceptual order is the following.
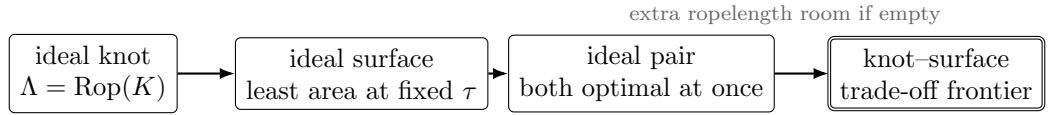
\begin{center}
\begin{tikzpicture}[x=1cm,y=1cm,
  stage/.style={draw, rounded corners=2pt, align=center, inner sep=4pt,
                font=\small},
  lbl/.style={font=\scriptsize, text=black!60, align=center},
  arrow/.style={-{Latex[length=2mm]}, thick}]
  \node[stage] (ik) at (0,0) {ideal knot\\$\Lambda=\Rop(K)$};
  \node[stage] (is) at (3.55,0) {ideal surface\\least area at fixed $\tau$};
  \node[stage] (ip) at (7.25,0) {ideal pair\\both optimal at once};
  \node[stage, double] (fr) at (11.15,0)
    {knot--surface\\trade-off frontier};
  \draw[arrow] (ik) -- (is);
  \draw[arrow] (is) -- (ip);
  \draw[arrow] (ip) -- (fr);
  \node[lbl] at (9.15,0.78) {extra ropelength room if empty};
\end{tikzpicture}
\end{center}
This is the main optimization viewpoint of the paper.  The geometric
mechanism behind a positive frontier is pictured in
\Cref{fig:tradeoff-corridor}.

\begin{figure}[t]
\centering
\begin{tikzpicture}[x=1cm,y=1cm,>=Latex,
  panel/.style={draw, rounded corners=2pt},
  title/.style={font=\footnotesize\bfseries},
  note/.style={font=\scriptsize, align=center}]
  \draw[panel] (0,-0.5) rectangle (6.4,4.3);
  \node[title] at (3.2,4.02) {ideal representative:
    $\Len=\Rop(K)$};
  \foreach \xc/\yc in {2.15/2.35, 4.25/2.35}{
    \draw[dashed, thin] (\xc,\yc) circle (1.02);
    \fill[black!15] (\xc,\yc) circle (0.5);
    \draw[thick] (\xc,\yc) circle (0.5);
    \fill (\xc,\yc) circle (1.1pt);
  }
  \draw[red!55!black, opacity=0.18, line width=40pt] (3.2,3.75) -- (3.2,0.95);
  \draw[very thick, dashed]
    (3.2,3.9) .. controls (3.2,3.3) and (3.2,2.9) .. (3.2,2.35)
    .. controls (3.2,1.8) and (3.2,1.4) .. (3.2,0.8);
  \node[note, fill=white, inner sep=1pt] at (3.2,0.35)
    {no room for the $\tau$-collar\\clearance of the sheet};
  \node[note] at (0.95,3.65) {$F$?};
  \draw[->, thin] (1.15,3.5) -- (2.95,3.35);
  \node[note] at (3.2,-0.25)
    {class invisible: ideal pair stratum empty};
  \draw[panel] (7.0,-0.5) rectangle (13.4,4.3);
  \node[title] at (10.2,4.02)
    {slack representative: $\Len=\Rop(K)+\delta$};
  \foreach \xc/\yc in {8.55/2.35, 11.85/2.35}{
    \draw[dashed, thin] (\xc,\yc) circle (1.02);
    \fill[black!15] (\xc,\yc) circle (0.5);
    \draw[thick] (\xc,\yc) circle (0.5);
    \fill (\xc,\yc) circle (1.1pt);
  }
  \draw[black!10, line width=40pt] (10.2,3.75) -- (10.2,0.95);
  \draw[very thick]
    (10.2,3.9) .. controls (10.2,3.0) .. (10.2,2.35)
    .. controls (10.2,1.7) .. (10.2,0.8);
  \draw[<->, thin] (10.2,1.30) -- (10.9,1.30);
  \node[note, anchor=west] at (10.92,1.30) {$\asymp\tau$};
  \node[note] at (9.35,3.65) {$F$};
  \draw[->, thin] (9.5,3.5) -- (10.05,3.35);
  \node[note] at (10.2,-0.25)
    {class visible: $\beta_{\mathfrak S}(K;\Delta,\tau)\leq\Rop(K)+\delta$};
\end{tikzpicture}
\caption{Why the knot--surface trade-off frontier can be positive, in
cross-section (schematic).  Left: over an ideal (length-minimizing)
representative, two strands of the unit-thickness tube may leave a corridor
narrower than the scale needed by a \(\tau\)-thick representative of a given
essential class, whose sheets require clearance comparable to \(\tau\)
(shaded band, which on the left overlaps the tubes); the ideal pair stratum of
that class is then empty.  Right: spending additional
ropelength \(\delta\) separates the strands, the corridor opens to width at
least comparable to \(\tau\), and the class becomes visible.  The least such
\(\delta\) is the additional ropelength room recorded by the frontier.}
\label{fig:tradeoff-corridor}
\end{figure}
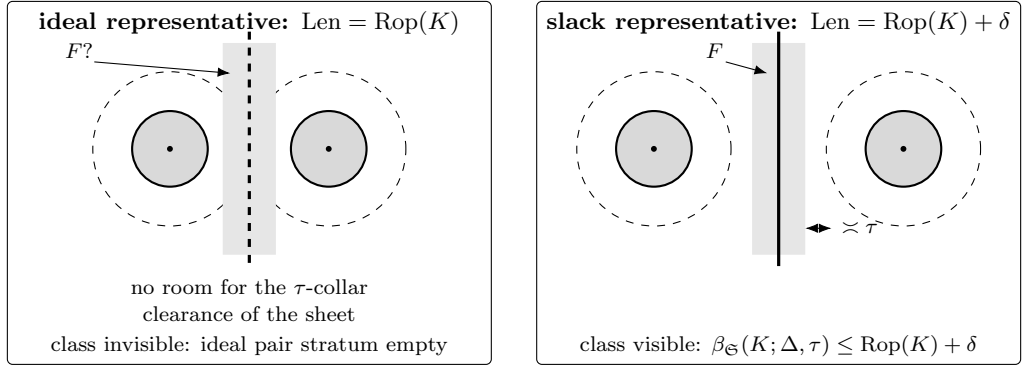

\begin{remark}[Programmatic statement: Ideal pair principle]
The essential surface theory of a knot type should be studied not only by
asking which essential surfaces exist in its exterior, but also by asking
whether they occur over ideal knot representatives as ideal surfaces.  When
they do not, the knot--surface trade-off frontier records the additional
ropelength room needed for their appearance.
\end{remark}

\begin{remark}[Closed and relative settings]
The classical least-area theorem is usually stated for incompressible
surfaces in suitable Riemannian 3-manifolds.  Knot exteriors have boundary,
and many surfaces considered in this paper are properly embedded.  In such
cases one should use the appropriate relative or boundary-controlled version,
or pass to a doubled or rounded setting when applicable.  The role of
least-area theory here is to justify the interpretation of the area
filtration as a genuine geometric filtration on essential surface classes,
rather than merely as a formal cutoff.
\end{remark}

\section{Ideal merge scales and admissible-component persistence}
\label{sec:ideal-merge-pair-persistence}

The preceding section defines ideal knots, ideal surfaces, and ideal pairs as
bottom layers of optimization problems.  The deformation-persistence viewpoint
of the ideal stratum of knot types adds a second layer of structure: after
bottom-layer objects are born, one asks when distinct admissible components
merge as the ropelength allowance is increased.  In the knot-only setting this
leads to admissible-component persistence, ideal merge scales, and the pure
merge Vietoris--Rips complex \cite{OzawaIdealStratum}.  We now record the
corresponding pair-space version.

\subsection{Two nearby graph constructions}
\label{subsec:intersection-graphs-and-graphics}

Two classical graph constructions are adjacent to, but distinct from, the
persistence objects of this section.  First, Gordon--Luecke type arguments
place graphs of intersection on punctured surfaces and use their faces,
parity rules, and Scharlemann cycles to constrain the distance between filling
slopes \cite{GordonCombinatorial,GordonLueckeToroidal}.  The boundary-slope
visibility and writhe-window results below instead constrain each visible
slope from the geometry of a single thick surface.  A combined theory would
therefore contain both an absolute geometric window for the slopes and a
relative combinatorial constraint coming from their intersection graphs.

Second, the Rubinstein--Scharlemann graphic is a discriminant in a
parameter square for two sweepouts, obtained by Cerf-theoretic analysis of
nongeneric tangencies \cite{RubinsteinScharlemann,RubinsteinScharlemannBounded}.
The transition complexes defined below have a more modest role: their vertices
are finite pair codes and their edges record elementary code changes under
admissible deformations.  They are finite-resolution quotients of deformation
data, not Cerf graphics.  Recovering either classical construction requires
additional structure: transverse-intersection data for the former and
sweepouts with controlled singular events for the latter.

\subsection{Admissible components of essential pair spaces}

Fix parameters \((\Lambda,\Delta,\tau)\).  The geometric configuration space
\(\calZ_{\Lambda,\Delta,\tau}^{\mathrm{ess}}(K)\) is taken modulo rigid motions and source reparametrizations, retaining all stipulated labels and orientations. It is not divided by pair isotopy. Use the local $C^1$ topology on representatives, including boundaries and collars, and the quotient topology. The map to pair-isotopy classes is constant on admissible components. Those components are defined by admissible thick deformations within the level; no identification with arbitrary topological path components is assumed.

\begin{definition}[Admissible components of pair spaces]
For fixed \((\Lambda,\Delta,\tau)\), let
\[
  \Pi^{\mathrm{ad}}_{\Lambda,\Delta,\tau}(K)
\]
be the set of equivalence classes of points of
\(\calZ_{\Lambda,\Delta,\tau}^{\mathrm{ess}}(K)\) under admissible thick
deformations through pairs \((\gamma_t,F_t)\) satisfying
\[
  \Thi(\gamma_t)=1,
  \qquad
  \Len(\gamma_t)\leq \Lambda,
  \qquad
  \Thi(F_t)\geq\tau,
  \qquad
  \Area(F_t)\leq\Delta,
\]
and with \(F_t\) remaining essential in \(E(\gamma_t)\) throughout the
deformation.
\end{definition}

If \(\Lambda\leq\Lambda'\) and \(\Delta\leq\Delta'\), then admissible components
at level \((\Lambda,\Delta,\tau)\) map naturally to admissible components at
level \((\Lambda',\Delta',\tau)\).  Thus, for fixed surface scale
\((\Delta,\tau)\), the family
\[
  \left\{
  \Pi^{\mathrm{ad}}_{\Lambda,\Delta,\tau}(K)
  \right\}_{\Lambda\geq\Rop(K)}
\]
records the zero-dimensional deformation persistence of the essential pair
space.

\begin{remark}[Comparison with the knot-only ideal stratum]
For the knot-only space \(Y^{=}_\Lambda(K)\), the first birth time of
admissible-component persistence is exactly \(\Rop(K)\).  In the present
pair-space setting the first birth time for a given surface constraint
\((\Delta,\tau)\) may be larger than \(\Rop(K)\), or may be infinite.  This is
precisely the phenomenon measured by the knot--surface trade-off frontier:
an ideal knot may exist before a prescribed thick essential surface layer is
visible.
\end{remark}

\subsection{Merge scales, forgetful maps, and tameness}

The pair theory is a decorated refinement of the knot-only ropelength
persistence: the forgetful map
\[
  \Phi_{\Lambda,\Delta,\tau}\colon
  \calZ_{\Lambda,\Delta,\tau}^{\mathrm{ess}}(K)
  \longrightarrow
  Y^{=}_\Lambda(K),
  \qquad
  (\gamma,F)\longmapsto \gamma,
\]
sends admissible thick deformations of pairs to admissible knot-only
deformations, since \(\gamma_t\) remains unit-thickness of length at most
\(\Lambda\); hence it induces a map
\(\Pi^{\mathrm{ad}}_{\Lambda,\Delta,\tau}(K)\to\pi_0^{\mathrm{ad}}Y^{=}_\Lambda(K)\)
on admissible components, compatible with the inclusions in all three
parameters.  Its fiber over a knot-only component records which
essential-surface structures can be transported along thick isotopies within
the same budget.

\begin{definition}[Ropelength merge scale and admissible-component persistence]
Fix \(\Delta,\tau>0\).  For admissible components \(C_1,C_2\) visible at some
ropelength levels, their ropelength merge scale is
\[
  m_{\Delta,\tau}(C_1,C_2)
  =
  \inf\left\{
  \Lambda
  \ \middle|
  C_1 \text{ and } C_2 \text{ have the same image in }
  \Pi^{\mathrm{ad}}_{\Lambda,\Delta,\tau}(K)
  \right\},
\]
with value \(\infty\) if no such \(\Lambda\) exists.  The persistence object
\(\{\Pi^{\mathrm{ad}}_{\Lambda,\Delta,\tau}(K)\}_{\Lambda\geq\Rop(K)}\) is
called the admissible-component persistence of the essential pair space; it
records births of admissible components and merge scales between them as
\(\Lambda\) increases.  For ideal pair components the bottom level is
\(\Lambda=\Rop(K)\), and \(m_{\Delta,\tau}\) measures the additional ropelength
room needed to pass between them.  \Cref{fig:merge-persistence} shows the
resulting merge-tree and barcode picture, together with the boundary-slope
obstruction of \Cref{prop:slope-obstruction-to-merging}.
\end{definition}

\begin{figure}[t]
\centering
\begin{tikzpicture}[x=1cm,y=1cm,>=Latex,
  note/.style={font=\scriptsize, align=center}]
  \draw[->] (1.2,0.30) -- (11.9,0.30)
    node[note, below left, xshift=3mm] {$\Lambda$};
  \draw[thin] (1.2,0.22) -- (1.2,0.38);
  \node[note, below] at (1.2,0.18) {$\Rop(K)$};
  \foreach \x in {2.0,3.2,4.4}{
    \draw[thin] (\x,0.22) -- (\x,0.38);
  }
  \draw[dashed, thin] (2.0,0.38) -- (2.0,2.85);
  \draw[dashed, thin] (3.2,0.38) -- (3.2,1.80);
  \draw[dashed, thin] (4.4,0.38) -- (4.4,1.05);
  \draw[dashed, thin] (6.0,0.38) -- (6.0,2.32);
  \node[note, below] at (3.2,0.18) {births};
  \node[note, below] at (6.0,0.18) {$m_{\Delta,\tau}(C_1,C_2)$};
  \draw[very thick] (2.0,2.85) .. controls (4.6,2.85) and (5.2,2.42) .. (6.0,2.32);
  \draw[very thick] (3.2,1.80) .. controls (4.6,1.80) and (5.2,2.22) .. (6.0,2.32);
  \draw[very thick, ->] (6.0,2.32) -- (11.6,2.32);
  \draw[very thick, ->] (4.4,1.05) -- (11.6,1.05);
  \fill (6.0,2.32) circle (1.5pt);
  \fill (2.0,2.85) circle (1.5pt);
  \fill (3.2,1.80) circle (1.5pt);
  \fill (4.4,1.05) circle (1.5pt);
  \node[note, anchor=east, fill=white, inner sep=1.5pt]
    at (1.88,2.85) {$C_1$ (slope $r$)};
  \node[note, anchor=east, fill=white, inner sep=1.5pt]
    at (3.08,1.80) {$C_2$ (slope $r$)};
  \node[note, anchor=east, fill=white, inner sep=1.5pt]
    at (4.28,1.05) {$C_3$ (slope $r'\neq r$)};
  \node[note, anchor=south] at (8.8,2.42) {merged admissible component};
  \node[note, anchor=south] at (8.8,1.15)
    {$m_{\Delta,\tau}(C_i,C_3)=\infty$ (slope obstruction)};
  \node[note, anchor=east] at (1.88,-0.85) {barcode};
  \draw[line width=1.8pt, ->] (2.0,-0.70) -- (11.6,-0.70);
  \draw[line width=1.8pt] (3.2,-1.05) -- (6.0,-1.05);
  \draw[line width=1.8pt, ->] (4.4,-1.40) -- (11.6,-1.40);
\end{tikzpicture}
\caption{Admissible-component persistence of the essential pair space at a
fixed surface budget \((\Delta,\tau)\).  As the ropelength allowance
\(\Lambda\) increases, admissible components of
\(\calZ^{\mathrm{ess}}_{\Lambda,\Delta,\tau}(K)\) are born at their visibility
levels and merge at their merge scales: here \(C_1\) and \(C_2\) carry the
same boundary slope \(r\) and merge at \(\Lambda=m_{\Delta,\tau}(C_1,C_2)\),
while \(C_3\) carries a different slope \(r'\) and, when admissible
deformations preserve the slope, can never merge with them
(\Cref{prop:slope-obstruction-to-merging}).  The bottom rows show the
associated zero-dimensional barcode; at fixed finite resolution
\(\varepsilon\), the finite-resolution quotient of this persistence is tame on
every bounded interval (\Cref{prop:persistence-tameness}).}
\label{fig:merge-persistence}
\end{figure}
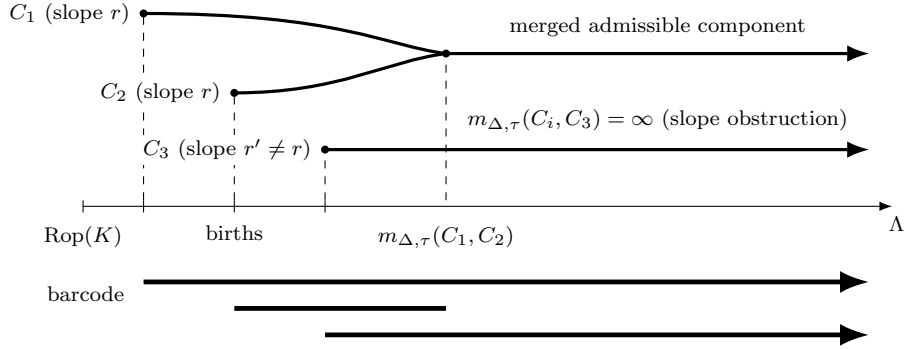

\begin{remark}[Finite-resolution persistence]
At the exact smooth level, the set of admissible components may be difficult
to control.  At fixed finite resolution \(\varepsilon\), however,
\Cref{thm:finite-resolution-finiteness-filtered-pairs} gives local
finiteness: only finitely many encoded \(\varepsilon\)-types are visible at
any bounded level.  Thus the finite-resolution \emph{quotient} of
admissible-component
persistence --- the code image, in the sense made precise in
\Cref{def:admissible-transition-complex} --- can be represented by finite
graphs or finite persistence modules
at each level, in the sense of persistent topology
\cite{EdelsbrunnerHarer,ZomorodianCarlsson}.
\end{remark}

\begin{definition}[Elementary transition rule and admissible transition
complex at scale \(\varepsilon\)]
\label{def:admissible-transition-complex}
Fix \(\Delta,\tau>0\), a resolution \(0<\varepsilon\leq c\min\{1,\tau\}\), and
the encoding scheme of \Cref{con:concrete-encoding}, with its
canonicalization convention (\Cref{rem:code-well-defined}); thus every pair
\((\gamma,F)\) at the given level has a single well-defined canonical
encoded \(\varepsilon\)-type, written
\(\code(\gamma,F)=c_\varepsilon(\gamma,F)\).  For an admissible thick deformation
\((\gamma_t,F_t)_{t\in[0,1]}\) at level \((\Lambda,\Delta,\tau)\), write
\(\code(t)=\code(\gamma_t,F_t)\) for its code trajectory.

\emph{Elementary transition rule.}  Two distinct realized codes \(c,c'\)
are \emph{elementary adjacent} at level \((\Lambda,\Delta,\tau)\) if there
exist an admissible thick deformation \((\gamma_t,F_t)_{t\in[0,1]}\) at that
level and a time \(t^*\in[0,1]\) such that every neighbourhood of \(t^*\) in
\([0,1]\) contains both a time \(t\) with \(\code(t)=c\) and a time \(t'\)
with \(\code(t')=c'\).  This rule is fixed once and for all; it refers only
to the continuum of deformation times and involves no time step, no bound on
the change per step, and no sampling rule.

For \(\Lambda\geq\Rop(K)\), let
\(\calN^{\varepsilon}_{\Lambda,\Delta,\tau}(K)\) be the simplicial complex
defined as follows.  Its vertices are the canonical encoded
\(\varepsilon\)-types
realized by pairs in \(\calZ^{\mathrm{ess}}_{\Lambda,\Delta,\tau}(K)\).  Two
vertices are joined by an edge if and only if they are elementary adjacent
at level \((\Lambda,\Delta,\tau)\).
Higher simplices are those of the flag (clique) complex on this graph.  Since
the vertex set and the elementary adjacencies are monotone in \(\Lambda\) and
in \(\Delta\) --- an admissible deformation at a lower level is one at every
higher level, and codes at different levels are compared literally on the
fixed lattice of \Cref{con:concrete-encoding} --- the complexes form an
increasing family in these parameters.

The relation between this complex and the smooth admissible components must
be stated with care.  Sending a pair to its code induces a natural
surjection
\[
  \Pi^{\mathrm{ad}}_{\Lambda,\Delta,\tau}(K)
  \ \longrightarrow\
  \pi_0\bigl(\calN^{\varepsilon}_{\Lambda,\Delta,\tau}(K)\bigr),
\]
but this surjection need not be injective: two pairs with the same code are
pair-isotopic by \Cref{thm:faithfulness}, yet they need not be joined by an
\emph{admissible thick deformation} --- one respecting
\(\Thi(\gamma_t)=1\), \(\Len(\gamma_t)\leq\Lambda\), \(\Thi(F_t)\geq\tau\),
\(\Area(F_t)\leq\Delta\) throughout --- so a single code, hence a single
vertex, may collect several distinct admissible components.  The complex
\(\calN^{\varepsilon}\) is therefore the \emph{finite-resolution quotient} of
admissible-component persistence, not an exact copy of it: it is the code
image of the smooth persistence, in precisely the same way the monotone
diagram--code image \(\calH^{\mathrm{surf}}\) of
\Cref{sec:finite-recognition} is the quotient of the lifted graph
\(\calG^{\mathrm{surf,lift}}\), whose vertices \((D,c,C)\) retain the fiber
component \(C\).  A lifted refinement of \(\calN^{\varepsilon}\), with
vertices the pairs \((c,C)\) of a realized code and an admissible component
of its code fiber, restores the exact persistence; but the set of such
vertices is not known to be finite at a fixed level, which is exactly why the
tameness statements below are proved for the quotient.  This complex is what
was informally called the ``nerve of
\(\Pi^{\mathrm{ad},\varepsilon}_{\Lambda,\Delta,\tau}(K)\)''; we use the
explicit transition complex to avoid ambiguity, and all persistence
statements below refer to it and are statements about the
finite-resolution quotient persistence.
\end{definition}

\begin{lemma}[Sampling independence of the transition complex]
\label{lem:sampling-independence}
Fix a level \((\Lambda,\Delta,\tau)\) and a resolution
\(0<\varepsilon\leq c\min\{1,\tau\}\).
\begin{enumerate}[label=(\roman*),leftmargin=2em]
\item \emph{Chain property.}  For every admissible thick deformation
\((\gamma_t,F_t)_{t\in[0,1]}\) at level \((\Lambda,\Delta,\tau)\), the codes
\(\code(0)\) and \(\code(1)\) are joined in
\(\calN^{\varepsilon}_{\Lambda,\Delta,\tau}(K)\) by a finite edge path all
of whose vertices are codes realized along the deformation.
\item \emph{Sampled steps refine to edge paths.}  For every time partition
\(0=t_0<t_1<\cdots<t_N=1\) of every admissible thick deformation, each
consecutive sampled pair \(\code(t_{i-1}),\code(t_i)\) is joined by a finite
edge path in \(\calN^{\varepsilon}_{\Lambda,\Delta,\tau}(K)\).  No claim is
made for an individual sampled description taken alone: one finite sampling
of one deformation may miss intermediate codes, and therefore need not see
every vertex or every elementary adjacency of the fixed complex.
\item \emph{Sampling-free identification of the components.}  Let
\(\sim\) be the equivalence relation on the realized codes generated by all
sampled steps of all admissible thick deformations at level
\((\Lambda,\Delta,\tau)\), that is, generated by the relations
\(\code(t_{i-1})\sim\code(t_i)\) over every such deformation and every time
partition.  Then \(\sim\) coincides with the partition of the vertex set
into path components of \(\calN^{\varepsilon}_{\Lambda,\Delta,\tau}(K)\).
Consequently the path components of the transition complex, the surjection
from admissible components, and the zero-dimensional persistence module and
barcode of \Cref{prop:persistence-tameness} are intrinsic to the fixed
elementary transition rule: no time step, admissible change bound, or
sampling rule enters \Cref{def:admissible-transition-complex}, and the
totality of sampled descriptions generates exactly the component partition
of the fixed complex.
\end{enumerate}
\end{lemma}

\begin{proof}
For (i), let \(\mathcal C\) denote the set of codes realized along the
deformation.  All pairs \((\gamma_t,F_t)\) lie at level
\((\Lambda,\Delta,\tau)\), so by
\Cref{thm:finite-resolution-finiteness-filtered-pairs} the set \(\mathcal C\)
is finite.  For \(c\in\mathcal C\) put
\(A_c=\{t\in[0,1]:\code(t)=c\}\); the sets \(A_c\) are nonempty and
partition \([0,1]\).  Let \(\mathcal C_0\subset\mathcal C\) be the set of
codes joined to \(\code(0)\) by a finite chain of elementary adjacencies
realized along this deformation, and let
\(U=\bigcup_{c\in\mathcal C_0}A_c\).  Suppose \(U\neq[0,1]\).  Both \(U\)
and \([0,1]\setminus U\) are nonempty, and since \([0,1]\) is connected they
cannot both be closed; hence there is a point
\(t^*\in\overline{U}\cap\overline{[0,1]\setminus U}\).  Choose sequences
\(t_n\to t^*\) with \(t_n\in U\) and \(s_n\to t^*\) with
\(s_n\notin U\).  Each \(t_n\) lies in some \(A_c\) with
\(c\in\mathcal C_0\) and each \(s_n\) in some \(A_{c'}\) with
\(c'\notin\mathcal C_0\); since \(\mathcal C\) is finite, after passing to
subsequences there are fixed codes \(c\in\mathcal C_0\) and
\(c'\notin\mathcal C_0\) realized at times accumulating at \(t^*\) from
both families.  Then every neighbourhood of \(t^*\) contains times realizing
\(c\) and times realizing \(c'\), so \(c\) and \(c'\) are elementary
adjacent, contradicting \(c'\notin\mathcal C_0\).  Hence \(U=[0,1]\); in
particular \(\code(1)\in\mathcal C_0\), which is the chain property.

For (ii), apply (i) to the restriction of the deformation to
\([t_{i-1},t_i]\), reparametrized to \([0,1]\); the restriction is again an
admissible thick deformation at the same level.

For (iii), compare the two relations in both directions.  In one direction,
by (ii) each generating relation
\(\code(t_{i-1})\sim\code(t_i)\) joins two codes lying in one path
component of the fixed complex, so the generated equivalence relation
\(\sim\) is contained in the path-component relation.  In the other
direction, every elementary adjacency of codes \(c,c'\), witnessed by a
deformation and times \(t_n\to t^*\), \(s_n\to t^*\) with
\(\code(t_n)=c\) and \(\code(s_n)=c'\), is realized as a single sampled
step by restricting the witness deformation to \([t_n,s_n]\) ---
admissible deformations are closed under restriction and monotone
reparametrization --- and sampling its endpoints; hence every edge of the
fixed complex is a generating relation of \(\sim\), and the path-component
relation, which is generated by the edges, is contained in \(\sim\).  The
two relations therefore coincide.  The remaining assertions follow, since
the edge set of the fixed complex is defined by the elementary transition
rule alone and the higher simplices are determined by the edge set through
the flag condition.
\end{proof}

\begin{remark}[Role of sampling after the fixed rule]
\label{rem:sampling-role}
Time sampling, as in the discrete-deformation remark of \Cref{sec:dof},
is thereby demoted to an implementation device: it produces finite-state
approximations whose transitions always refine to edge paths of the fixed
complex, and whose totality --- but not any individual finite sampling ---
generates exactly the component partition of that complex.  In particular the pathology of coarse sampling --- a single
sampled step passing between the endpoints of an arbitrary deformation and
thereby inserting an edge that depends on the sampling --- cannot arise,
because \Cref{def:admissible-transition-complex} never mentions samples:
an edge requires elementary adjacency, and a coarse sampled step certifies
only what \Cref{lem:sampling-independence} proves, namely an edge
\emph{path}.
\end{remark}

\begin{proposition}[Tameness and barcodes of the finite-resolution quotient
persistence]
\label{prop:persistence-tameness}
Fix \(\Delta,\tau>0\) and \(0<\varepsilon\leq c\min\{1,\tau\}\).  For each field
\(\mathbb F\) and each degree \(i\geq0\), consider the finite-resolution
quotient persistence module
\[
  \Lambda\longmapsto
  H_i\bigl(\calN^{\varepsilon}_{\Lambda,\Delta,\tau}(K);\mathbb F\bigr),
  \qquad \Lambda\geq\Rop(K).
\]
\begin{enumerate}[label=(\roman*),leftmargin=2em]
\item The module is pointwise finite-dimensional, and on every bounded
parameter interval \([\Rop(K),\Lambda_0]\) it changes at only finitely many
values of \(\Lambda\); in particular its restriction to any bounded interval
decomposes as a finite direct sum of interval modules and has a finite
barcode.
\item On the full ray \([\Rop(K),\infty)\) the set of critical values is
locally finite but need not be finite: new encoded types may keep becoming
visible as \(\Lambda\to\infty\).  The module is \(q\)-tame
\cite{ChazalDeSilvaGlisseOudot} and, being pointwise finite-dimensional,
decomposes by the structure theorem of Crawley-Boevey \cite{CrawleyBoevey}
into a direct sum of interval modules; this decomposition is in general
countable and locally finite, not finite.
\end{enumerate}
The same statements hold for the finite-resolution Kakimizu persistence
\(\mathrm{MS}^\varepsilon_{\Lambda,\Delta,\tau}(K)\).
\end{proposition}

\begin{proof}
Fix \(\Lambda_0\).  By
\Cref{thm:finite-resolution-finiteness-filtered-pairs} the set of encoded
\(\varepsilon\)-types visible at level \((\Lambda_0,\Delta,\tau)\) is finite,
so the complex \(\calN^{\varepsilon}_{\Lambda,\Delta,\tau}(K)\) has uniformly
boundedly many simplices for \(\Lambda\leq\Lambda_0\) (a flag complex on a
bounded vertex set is bounded); its homology is therefore finite-dimensional
at every \(\Lambda\).  As \(\Lambda\) increases through
\([\Rop(K),\Lambda_0]\), a simplex is added only when a new encoded type or a
new elementary adjacency (\Cref{def:admissible-transition-complex}) first
becomes visible, and by finiteness this
happens at only finitely many values of \(\Lambda\) in that interval.  Thus
the restricted module is constant between consecutive critical values and is
pointwise finite-dimensional, so it decomposes into finitely many interval
modules by the structure theorem \cite{CrawleyBoevey}; this proves (i).

For (ii), the critical values form a subset of \([\Rop(K),\infty)\) whose
intersection with every bounded interval is finite by (i), that is, a locally
finite set; nothing bounds their total number, since the vertex bound of
\Cref{cor:explicit-encoded-bound} grows with \(\Lambda\).  Pointwise finite
dimensionality gives \(q\)-tameness directly (every connecting map has finite
rank), and the Crawley-Boevey structure theorem applies to pointwise
finite-dimensional modules over the totally ordered set
\([\Rop(K),\infty)\) without any finiteness of critical values, yielding an
interval decomposition indexed by an at most countable set, locally finite by
(i).  The Kakimizu case is identical, using the vertex bound of
\Cref{cor:explicit-encoded-bound}.
\end{proof}

Merge scales add dynamic information to the static finiteness results:
boundary slopes become labels on admissible components, births of components
are generated by least-area representatives (\Cref{cor:unconditional-ideal-attainment}),
and the associated flag complexes and merge trees are the simplicial form of
this zero-dimensional persistence \cite{EdelsbrunnerHarer,OzawaIdealStratum}.
The basic obstruction is topological.

\begin{proposition}[Boundary slope obstruction to merging]
\label{prop:slope-obstruction-to-merging}
Suppose admissible deformations are required to preserve the boundary slope of
surfaces.  If two admissible components \(C_1\) and \(C_2\) are represented by
surfaces with different boundary slopes, then
\[
  m_{\Delta,\tau}(C_1,C_2)=\infty.
\]
\end{proposition}

\begin{proof}
Boundary slope is invariant under pair isotopy preserving the peripheral
structure.  A slope-preserving admissible thick deformation is, in particular,
a pair isotopy through surfaces of the same boundary slope.  Thus components
with distinct boundary slopes cannot map to the same admissible component at
any higher ropelength level.
\end{proof}

\subsection{Boundary-slope visibility spectrum}

Boundary slopes are a natural test case because they are topological labels of
essential surfaces, while visibility assigns geometric costs to those labels.
For a slope \(r\) on \(\partial E(K)\), let \(\mathfrak S_r\) denote the
filtered surface type consisting of essential surfaces with nonempty
boundary whose boundary components all have slope \(r\).  Define
\[
  \beta_r(K;\Delta,\tau)
  :=
  \beta_{\mathfrak S_r}(K;\Delta,\tau).
\]
The function
\[
  r\longmapsto \beta_r(K;\Delta,\tau)
\]
is the boundary-slope visibility spectrum at the surface budget
\((\Delta,\tau)\).

\begin{proposition}[Topological support of the slope spectrum]
\label{prop:boundary-slope-visibility-support}
For a fixed knot type \(K\), the finite values of
\(r\mapsto \beta_r(K;\Delta,\tau)\) occur only on the ordinary boundary slopes
of incompressible surfaces in \(E(K)\).  In particular, the support of the
boundary-slope visibility spectrum is contained in Hatcher's finite boundary
slope set.
\end{proposition}

\begin{proof}
If \(\beta_r(K;\Delta,\tau)<\infty\), then there is a pair \((\gamma,F)\) in
\(\calZ_{\Lambda,\Delta,\tau}^{\mathfrak S_r}(K)\) for some \(\Lambda\).  By
definition, \(F\) is an essential surface in the knot exterior with boundary
slope \(r\).  Thus \(r\) is an ordinary boundary slope of \(K\).  Hatcher's
boundary-slope theorem gives the finiteness of the set of such slopes
\cite{HatcherBoundarySlopes}.
\end{proof}

\begin{remark}[What the spectrum adds]
Hatcher finiteness says that only finitely many slopes can occur.  The
visibility spectrum asks a different question: among the slopes that occur,
which ones occur at low ropelength, low area, and positive relative surface
thickness?  Thus it refines the finite boundary-slope set by assigning a
geometric birth level to each slope.
\end{remark}

Thus the deformation-persistence viewpoint imported from the ideal stratum of
knot types turns the present theory from a static finite-resolution framework
into a dynamic theory of births, mergers, and completion scales for essential
surface layers.

\subsection{Boundary twisting diameter and ropelength-normalized slope spread}
\label{subsec:boundary-twisting-diameter}

The boundary-slope visibility spectrum records when individual slopes become
visible.  A complementary invariant measures how far apart the visible slopes
are on the boundary torus.  This is closer to a measure of boundary twisting
than the unordered slope set alone.  It is related to the established numerical
boundary-slope diameter and its comparisons with crossing number and peripheral
geometry \cite{MattmanMaybrunRobinson,IchiharaSlopeLengths}; the determinant
distance used here is a variant adapted to arbitrary rational slopes.

Let \(T=\partial E(K)\) be equipped with the standard meridian--longitude
basis \((\mu,\lambda)\).  A slope will be regarded as an unoriented primitive
class \(r=a\mu+b\lambda\), with \(\gcd(a,b)=1\), up to sign.  For two slopes
\(r=a\mu+b\lambda\) and \(r'=a'\mu+b'\lambda\), define their distance by
\[
  \Delta_T(r,r')=|ab'-a'b|.
\]
This is the minimal geometric intersection number of the two slope classes on
\(T\).

\begin{definition}[Boundary twisting diameter]
\label{def:boundary-twisting-diameter}
Let \(B(K)\) be the set of boundary slopes of connected essential surfaces in
\(E(K)\).  Define
\[
  \operatorname{Tw}_{\partial}(K)
  =
  \max\{\Delta_T(r,r')\mid r,r'\in B(K)\},
\]
with the convention that the maximum is \(0\) if \(|B(K)|\leq 1\).  We call
\(\operatorname{Tw}_{\partial}(K)\) the \emph{boundary twisting diameter} of
\(K\).
\end{definition}

\begin{remark}[Why intersection distance is the twisting measure]
If all relevant slopes are integral in the chosen meridian--longitude framing,
then \(\Delta_T(p/1,p'/1)=|p-p'|\), so
\(\operatorname{Tw}_{\partial}(K)\) agrees with the ordinary numerical diameter
\(\max B(K)-\min B(K)\).  For rational slopes, however,
\(\Delta_T\) is more intrinsic: it measures the actual intersection distance of
the two boundary directions on the peripheral torus.  Thus
\(\operatorname{Tw}_{\partial}(K)\) records the spread of boundary twisting
rather than only the spread of rational numbers.
\Cref{fig:slope-intersection} shows the intersection-number picture.
\end{remark}

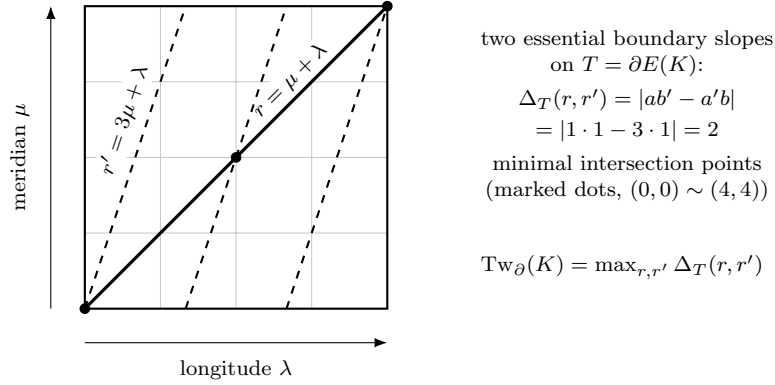
\begin{figure}[t]
\centering
\begin{tikzpicture}[x=1cm,y=1cm,>=Latex,
  note/.style={font=\scriptsize, align=center}]
  \draw[step=1.0, black!22, very thin] (0,0) grid (4.0,4.0);
  \draw[thick] (0,0) rectangle (4.0,4.0);
  \draw[->] (0,-0.45) -- (4.0,-0.45);
  \node[note, below] at (2.0,-0.55) {longitude $\lambda$};
  \draw[->] (-0.45,0) -- (-0.45,4.0);
  \node[note, rotate=90, above] at (-0.60,2.0) {meridian $\mu$};
  \draw[very thick] (0,0) -- (4.0,4.0);
  \node[note, fill=white, inner sep=1.5pt, rotate=45] at (2.72,3.02)
    {$r=\mu+\lambda$};
  \draw[thick, dashed] (0,0) -- (1.3333,4.0);
  \draw[thick, dashed] (1.3333,0) -- (2.6667,4.0);
  \draw[thick, dashed] (2.6667,0) -- (4.0,4.0);
  \node[note, fill=white, inner sep=1.5pt, rotate=71.6] at (0.52,2.50)
    {$r'=3\mu+\lambda$};
  \fill (0,0) circle (2pt);
  \fill (2.0,2.0) circle (2pt);
  \fill (4.0,4.0) circle (2pt);
  \node[note, anchor=north west] at (2.08,1.97) {};
  \node[note, align=center, anchor=west] at (5.1,2.55)
    {two essential boundary slopes\\on $T=\partial E(K)$:\\[1mm]
     $\Delta_T(r,r')=|ab'-a'b|$\\[0.5mm]
     $=|1\cdot1-3\cdot1|=2$\\[1mm]
     minimal intersection points\\(marked dots, $(0,0)\sim(4,4)$)};
  \node[note, align=center, anchor=west] at (5.1,0.55)
    {$\operatorname{Tw}_{\partial}(K)=
      \max_{r,r'}\Delta_T(r,r')$};
\end{tikzpicture}
\caption{The boundary twisting diameter measures intersection distance on the
peripheral torus.  In the flat picture of \(T=\partial E(K)\) (square with
opposite sides identified), the slopes \(r=\mu+\lambda\) and
\(r'=3\mu+\lambda\) intersect in \(\Delta_T(r,r')=|ab'-a'b|=2\) points,
marked by dots (the two corner dots are the same point of \(T\)).  The
twisting diameter \(\operatorname{Tw}_{\partial}(K)\) is the largest such
distance over pairs of essential boundary slopes; its filtered version
restricts to slopes visible at a bounded
\((\Lambda,\Delta,\tau)\)-budget.}
\label{fig:slope-intersection}
\end{figure}

\begin{proposition}[Finiteness of boundary twisting diameter]
\label{prop:boundary-twisting-finite}
For every knot \(K\subset S^3\), the invariant
\(\operatorname{Tw}_{\partial}(K)\) is finite.
\end{proposition}

\begin{proof}
By Hatcher's boundary-slope theorem, the set \(B(K)\) of boundary slopes of
incompressible surfaces in a knot exterior is finite
\cite{HatcherBoundarySlopes}.  The maximum of the finite set of distances
\(\Delta_T(r,r')\) is therefore finite.
\end{proof}

The filtered version is obtained by restricting to those slopes realized by
pairs at bounded ropelength, area, and surface thickness.

\begin{definition}[Filtered boundary twisting diameter]
\label{def:filtered-boundary-twisting-diameter}
For parameters \((\Lambda,\Delta,\tau)\), set
\[
  B_{\Lambda,\Delta,\tau}(K)
  =
  \{\,\slope(\partial F)\mid
  (\gamma,F)\in\calZ^{\mathrm{ess}}_{\Lambda,\Delta,\tau}(K),\
  F\text{ connected with }\partial F\neq\varnothing\,\}.
\]
Here \(\slope(\partial F)\) denotes the common boundary slope of the
components of \(\partial F\), which is well-defined for a connected essential
surface, so that \(B_{\Lambda,\Delta,\tau}(K)\subset B(K)\).
Define
\[
  \operatorname{Tw}_{\partial,\Lambda,\Delta,\tau}(K)
  =
  \max\{\Delta_T(r,r')\mid r,r'\in B_{\Lambda,\Delta,\tau}(K)\},
\]
again with value \(0\) if the visible slope set has at most one element.
\end{definition}

\begin{proposition}[Filtered boundary twisting is monotone]
\label{prop:filtered-boundary-twisting-monotonicity}
The quantity \(\operatorname{Tw}_{\partial,\Lambda,\Delta,\tau}(K)\) is
nondecreasing in \(\Lambda\) and \(\Delta\), and nonincreasing in \(\tau\).
Moreover,
\[
  \operatorname{Tw}_{\partial,\Lambda,\Delta,\tau}(K)
  \leq
  \operatorname{Tw}_{\partial}(K).
\]
\end{proposition}

\begin{proof}
The filtered inclusions
\(\calZ^{\mathrm{ess}}_{\Lambda,\Delta,\tau}(K)\subset
\calZ^{\mathrm{ess}}_{\Lambda',\Delta',\tau'}(K)\) for
\(\Lambda\leq\Lambda'\), \(\Delta\leq\Delta'\), and \(\tau\geq\tau'\) imply the
corresponding inclusion of visible slope sets.  Taking the maximum of
\(\Delta_T\) over a larger set cannot decrease the value.  The final inequality
follows from
\(B_{\Lambda,\Delta,\tau}(K)\subset B(K)\).
\end{proof}

\begin{definition}[Boundary-twisting visibility level]
\label{def:boundary-twisting-visibility-level}
For \(R\geq 0\) and fixed surface budget \((\Delta,\tau)\), define
\[
  \lambda^{\mathrm{tw}}_{\partial}(K;R,\Delta,\tau)
  =
  \inf\left\{
  \Lambda\ \middle|
  \operatorname{Tw}_{\partial,\Lambda,\Delta,\tau}(K)\geq R
  \right\}.
\]
Thus \(\lambda^{\mathrm{tw}}_{\partial}(K;R,\Delta,\tau)\) is the ropelength
level at which boundary-slope spread at least \(R\) first becomes visible
under the prescribed area and surface-thickness bounds.
\end{definition}

The numerical comparison with ropelength is then encoded by the following
scale-free quantities.

\begin{definition}[Ropelength-normalized boundary twisting density]
\label{def:boundary-twisting-density}
For \(\alpha>0\), define
\[
  \rho^{\mathrm{tw}}_{\partial,\alpha}(K)
  =
  \frac{\operatorname{Tw}_{\partial}(K)^\alpha}{\Rop(K)}.
\]
Equivalently, define the boundary-twisting compression ratio by
\[
  \kappa^{\mathrm{tw}}_{\partial,\alpha}(K)
  =
  \frac{\Rop(K)}{\operatorname{Tw}_{\partial}(K)^\alpha},
\]
whenever \(\operatorname{Tw}_{\partial}(K)>0\).
\end{definition}

\begin{problem}[Boundary twisting versus ropelength]
\label{prob:boundary-twisting-versus-ropelength}
Find natural classes of knots \(\mathcal K\) and exponents \(\alpha>0\) for
which there exists a constant \(C_{\mathcal K,\alpha}>0\) such that
\[
  \Rop(K)
  \geq
  C_{\mathcal K,\alpha}\,\operatorname{Tw}_{\partial}(K)^\alpha
\]
for all \(K\in\mathcal K\).  Equivalently, determine when
\(\rho^{\mathrm{tw}}_{\partial,\alpha}\) is uniformly bounded above on
\(\mathcal K\).
\end{problem}

\begin{remark}[Relation with tabulated boundary slopes]
For many Montesinos and two-bridge knots, boundary slopes are computable from
continued-fraction or edgepath data.  In such families, tables of boundary
slopes and the phenomenon of repeated slopes provide concrete test cases for
\Cref{prob:boundary-twisting-versus-ropelength}
\cite{DunfieldBoundarySlopesMontesinos,CurtisFranczakLeiserManheimer}.
The invariant \(\operatorname{Tw}_{\partial}\) deliberately forgets the number
and genera of surfaces realizing each slope; those refinements belong to the
decorated essential-surface complex, while the twisting diameter is a coarse
quantity designed for comparison with ropelength.
\end{remark}

\subsection{Finite-length peripheral twist and boundary slopes}
\label{subsec:finite-length-peripheral-twist}

The invariant \(\operatorname{Tw}_{\partial}(K)\) is a topological diameter of
boundary slopes.  We now record a finite-length differential-geometric version
which lives on the boundary of a thick tube.  This supplies a local mechanism
behind the slogan that large boundary-slope spread should force large
ropelength.

Let \(\gamma:S^1_L\to\mathbb R^3\) be an arclength-parametrized representative
with positive thickness, and let \((e_1,e_2)\) be an orthonormal frame
of the normal bundle of \(\gamma\), oriented so that \((\gamma',e_1,e_2)\) is
a positively oriented frame of \(\R^3\), that is, \(e_1\times e_2=\gamma'\).  For a tube of radius \(r<\Thi(\gamma)\), use
coordinates
\[
  X(s,\theta)
  =
  \gamma(s)+r\bigl(\cos\theta\,e_1(s)+\sin\theta\,e_2(s)\bigr).
\]
The meridians are the curves \(\theta\mapsto X(s,\theta)\).  The normal
connection one-form is
\[
  \Omega_\gamma(s)\,ds,
  \qquad
  \Omega_\gamma(s)=\langle e_1'(s),e_2(s)\rangle.
\]
Equivalently, the invariant angular form on the tube is
\[
  \alpha_\gamma=d\theta+\Omega_\gamma(s)\,ds.
\]
Changing the normal frame by a rotation changes \(d\theta\) and
\(\Omega_\gamma ds\) separately, but not their sum.

\begin{definition}[Preferred-longitude twist density]
\label{def:preferred-longitude-twist-density}
Let \(\lambda\) be a smooth preferred longitude on \(\partial N_r(\gamma)\),
written as a section
\[
  \lambda(s)=X(s,\theta_\lambda(s)).
\]
Its peripheral twist density is
\[
  \omega_\lambda(s)
  =
  \theta_\lambda'(s)+\Omega_\gamma(s).
\]
For \(0<\ell\leq L\) and \(1\leq p<\infty\), define the finite-length
\(L^p\)-twist amount by
\[
  \Theta^{\partial}_{p,\ell}(\gamma,\lambda)
  =
  \sup_{I\subset S^1_L,\ |I|=\ell}
  \left(
  \frac{1}{(2\pi)^p}\int_I |\omega_\lambda(s)|^p\,ds
  \right)^{1/p}.
\]
For \(p=\infty\), set
\[
  \Theta^{\partial}_{\infty,\ell}(\gamma,\lambda)
  =
  \frac{1}{2\pi}
  \sup_{s\in S^1_L}|\omega_\lambda(s)|.
\]
\end{definition}

The case \(p=1\) is the total absolute twist seen in a longitudinal window of
length \(\ell\).  The case \(p=\infty\) records the maximal local twist density.
For a preferred longitude, the signed total twist is related to the writhe of
\(\gamma\) by the classical Calugareanu--White--Fuller relation
\cite{Calugareanu1959,White1969,Fuller1971}; the finite-length quantities above
are deliberately absolute and local, so that cancellation between oppositely
oriented twisting regions is not allowed.

\begin{definition}[Ropelength-windowed peripheral twist]
\label{def:ropelength-windowed-peripheral-twist}
For a knot type \(K\), length bound \(\Lambda\), window length \(\ell\), and
\(1\leq p\leq\infty\), define
\[
  \Theta^{\partial}_{p,\ell}(K;\Lambda)
  =
  \inf
  \left\{
  \Theta^{\partial}_{p,\ell}(\gamma,\lambda)
  \ \middle|\
  \begin{array}{l}
  \gamma\in K,\ \Thi(\gamma)=1,\ \Len(\gamma)\leq\Lambda,\\
  \lambda\subset\partial N_{\rho_0}(\gamma)\text{ is a preferred longitude}
  \end{array}
  \right\},
\]
where representatives with \(\Len(\gamma)<\ell\) are omitted.  The knot-type
finite-length twist profile is the function
\[
  (p,\ell,\Lambda)
  \longmapsto
  \Theta^{\partial}_{p,\ell}(K;\Lambda).
\]
\end{definition}

We next relate finite-length twist to boundary slopes.  Fix a preferred
longitude \(\lambda(s)=X(s,\theta_\lambda(s))\).  If a curve \(c\) on the tube is
written on a longitudinal cover as
\[
  c(s)=X(s,\theta_c(s)),
\]
then its angular coordinate relative to the preferred longitude is
\[
  \phi_c(s)=\theta_c(s)-\theta_\lambda(s).
\]
The one-form
\[
  \eta_\lambda=d\phi=d\theta-\theta_\lambda'(s)\,ds
\]
measures meridional winding relative to the preferred longitude, while
\(\alpha_\gamma\) measures geometric rotation of the chosen curve relative to
parallel transport in the normal bundle.

\begin{definition}[Finite-length slope twist]
\label{def:finite-length-slope-twist}
Let \(c\) be a curve on \(\partial N_{\rho_0}(\gamma)\) which is a finite-sheeted graph
over the longitudinal parameter, and let \(J\) be an interval in its
longitudinal cover.  Define the relative meridional twist of \(c\) over \(J\) by
\[
  \nu_\lambda(c;J)
  =
  \frac{1}{2\pi}\int_J |\phi_c'(s)|\,ds.
\]
For a window length \(\ell>0\), define
\[
  \nu_{\lambda,\ell}(c)
  =
  \sup_{|J|=\ell}\nu_\lambda(c;J),
\]
where the supremum is taken over all longitudinal intervals of length \(\ell\)
in the relevant cover.
\end{definition}

\begin{proposition}[Boundary slope as total relative twist]
\label{prop:boundary-slope-total-relative-twist}
Let \(r=a\mu+b\lambda\) be a finite slope with \(b>0\), and let \(c\) be a
representative of this slope on \(\partial N_{\rho_0}(\gamma)\), written on the
\(b\)-fold longitudinal cover.  Then
\[
  \frac{1}{2\pi}\int_c \eta_\lambda
  =
  a.
\]
Consequently, for every \(0<\ell\leq bL\),
\[
  \nu_{\lambda,\ell}(c)
  \geq
  \frac{|a|\ell}{bL}.
\]
\end{proposition}

\begin{proof}
In the meridian--preferred-longitude coordinates determined by \(\lambda\), the
function \(\phi_c\) is precisely the meridional angle.  A curve in the class
\(a\mu+b\lambda\) winds \(a\) times in the meridional direction while running
\(b\) times longitudinally.  Hence the signed integral of \(d\phi\) over the
\(b\)-fold longitudinal cover is \(2\pi a\).  The second inequality follows by
averaging: the total variation of \(\phi_c\) is at least \(2\pi |a|\) over an
interval of total longitudinal length \(bL\), so some subinterval of length
\(\ell\) carries at least the corresponding average amount of variation.
\end{proof}

For two slopes, the same argument applied on a common longitudinal cover gives
the intersection distance.

\begin{corollary}[Finite-length form of boundary-slope distance]
\label{cor:finite-length-boundary-slope-distance}
Let
\[
  r=a\mu+b\lambda,
  \qquad
  r'=a'\mu+b'\lambda
\]
be finite slopes with \(b,b'>0\).  On the common \(bb'\)-fold longitudinal
cover, the relative meridional angle between representatives of slopes \(r\)
and \(r'\) has total degree
\[
  ab'-a'b.
\]
Therefore, for every \(0<\ell\leq bb'L\), any pair of representatives satisfies
an averaged finite-window lower bound of the form
\[
  \sup_{|J|=\ell}
  \frac{1}{2\pi}\int_J
  |(\phi_c-\phi_{c'})'(s)|\,ds
  \geq
  \frac{\Delta_T(r,r')\,\ell}{bb'L}.
\]
\end{corollary}

This motivates a denominator-normalized version of the boundary twisting
diameter, adapted to finite-length slope rates.

\begin{definition}[Boundary slope-rate diameter]
\label{def:boundary-slope-rate-diameter}
For finite slopes \(r=a\mu+b\lambda\) and \(r'=a'\mu+b'\lambda\) with
\(b,b'>0\), set
\[
  \Delta_T^{\mathrm{rate}}(r,r')
  =
  \frac{\Delta_T(r,r')}{bb'}.
\]
Define
\[
  \operatorname{Tw}^{\mathrm{rate}}_{\partial}(K)
  =
  \max\left\{
  \Delta_T^{\mathrm{rate}}(r,r')
  \ \middle|\
  r,r'\in B(K)\text{ are finite slopes}
  \right\}.
\]
When all relevant slopes are integral, one has
\[
  \operatorname{Tw}^{\mathrm{rate}}_{\partial}(K)
  =
  \operatorname{Tw}_{\partial}(K).
\]
\end{definition}

\begin{principle}[Finite-length twisting mechanism]
\label{prin:finite-length-twisting-mechanism}
Let \(\gamma\) be a unit-thickness representative of length \(L\).  If two
boundary slopes of essential surfaces differ by a large value of
\(\Delta_T(r,r')\), then any boundary representatives of these slopes on the
tube must accumulate a large amount of relative meridional rotation on some
finite longitudinal window.  Quantitatively, the finite-length boundary-slope distance estimate gives
the lower bound
\[
  \text{finite-window relative twisting}
  \ \gtrsim\
  \frac{\Delta_T(r,r')\,\ell}{bb'L}.
\]
Thus an a priori geometric upper bound on finite-window peripheral twisting
would imply a ropelength lower bound in terms of
\(\operatorname{Tw}^{\mathrm{rate}}_{\partial}(K)\), and in integral-slope
families in terms of \(\operatorname{Tw}_{\partial}(K)\).
\end{principle}

\begin{problem}[Finite-length boundary twisting versus ropelength]
\label{prob:finite-length-boundary-twisting-versus-ropelength}
Find geometric hypotheses, preferably consequences of unit thickness and
controlled essential-surface geometry, which give upper bounds for the
finite-window peripheral twisting quantities
\(\Theta^{\partial}_{p,\ell}(\gamma,\lambda)\) and
\(\nu_{\lambda,\ell}(c)\).  Under such hypotheses, determine constants
\(C>0\) and \(\alpha>0\) for which
\[
  \Rop(K)
  \geq
  C\,\operatorname{Tw}^{\mathrm{rate}}_{\partial}(K)^\alpha,
\]
or, in integral-slope families,
\[
  \Rop(K)
  \geq
  C\,\operatorname{Tw}_{\partial}(K)^\alpha.
\]
\end{problem}

\begin{remark}[How this refines the slope diameter]
The topological invariant \(\operatorname{Tw}_{\partial}(K)\) records the
largest peripheral slope distance.  The finite-length invariants
\(\Theta^{\partial}_{p,\ell}\) and \(\nu_{\lambda,\ell}\) record where that
peripheral rotation must occur along a thick representative.  Thus they turn a
static boundary-slope diameter into a local geometric constraint on a tube of
finite length.
\end{remark}

\subsection{A conditional boundary-slope height bound}
\label{subsec:conditional-slope-height}

We now close the twisting mechanism into a genuine, if conditional, inequality:
under explicit control of the peripheral geometry of
\(\partial N_{\rho_0}(\gamma)\), large boundary-slope height forces large
ropelength.  The unconditional core is a length lower bound for a slope curve on
the \(\rho_0\)-tube, obtained directly from the invariant angular form
\(\alpha_\gamma=d\theta+\Omega_\gamma\,ds\) of
\Cref{subsec:finite-length-peripheral-twist}.  Fix the normal frame so that its
zero section is the preferred (Seifert) longitude, and set
\[
  W(\gamma)=\frac{1}{2\pi}\oint_\gamma\Omega_\gamma(s)\,ds ,
\]
the total peripheral twist of the Seifert framing.  By the
C\u alug\u areanu--White--Fuller relation, \(W(\gamma)=-\operatorname{Wr}(\gamma)\),
where \(\operatorname{Wr}\) is the writhe
\cite{Calugareanu1959,White1969,Fuller1971}; the precise regularity and sign
conventions under which this relation is used are fixed in
\Cref{rem:cwf-conventions} below.  Let
\[
  \ell^{\mathrm{st}}_\lambda(\gamma)
  =
  \inf\left\{\frac{\Len(c)}{|b|}\ \middle|\
  c\subset\partial N_{\rho_0}(\gamma)\text{ of class }
  a\mu+b\lambda,\ a\in\Z,\ b\in\Z\setminus\{0\}\right\}
\]
be the peripheral stable longitudinal systole.  By definition, every curve
of class \(a\mu+b\lambda\), \(b\neq0\), satisfies
\(|b|\ell^{\mathrm{st}}_\lambda(\gamma)\leq\Len(c)\).

\begin{remark}[C\u alug\u areanu--White--Fuller for \(C^{1,1}\) curves, and sign
conventions]
\label{rem:cwf-conventions}
The representatives of this paper are \(C^{1,1}\), not smooth, so we record
once the conventions and the regularity under which the
C\u alug\u areanu--White--Fuller (CWF) relation is applied.

\emph{Conventions.}  Orient \(S^3=\R^3\cup\{\infty\}\) by the standard
right-handed orientation of \(\R^3\) and orient \(\gamma\).  On
\(\partial N_{\rho_0}(\gamma)\) the meridian \(\mu\) is oriented so that
\(\operatorname{lk}(\mu,\gamma)=+1\), and the preferred longitude \(\lambda\)
is the unique slope with \(\operatorname{lk}(\lambda,\gamma)=0\), oriented
parallel to \(\gamma\).  For a framing given by a unit normal field \(v\)
along \(\gamma\), the twist is
\(\operatorname{Tw}(\gamma,v)=\frac{1}{2\pi}\oint\langle v'(s),
\gamma'(s)\times v(s)\rangle\,ds\), the writhe
\(\operatorname{Wr}(\gamma)\) is the Gauss double integral, and CWF reads
\(\operatorname{lk}(\gamma,\gamma_v)=\operatorname{Tw}(\gamma,v)
+\operatorname{Wr}(\gamma)\), where \(\gamma_v\) is the push-off along \(v\).
With these signs, the Seifert framing has linking number zero, so its total
twist is \(W(\gamma)=-\operatorname{Wr}(\gamma)\), which is the form used in
\Cref{lem:meridional-length-bound}.

\emph{Regularity.}  For a unit-thickness \(C^{1,1}\) curve, the tangent
\(\gamma'\) is Lipschitz, so a rotation-minimizing (parallel) normal frame
\((e_1,e_2)\) exists with Lipschitz regularity and essentially bounded
connection form \(\Omega_\gamma\in L^\infty\); all framings used here differ
from it by a Lipschitz angle function, so twist densities are defined almost
everywhere and integrable.  The Gauss integral defining
\(\operatorname{Wr}(\gamma)\) converges absolutely, because unit thickness
separates the curve from itself and bounds the integrand near the diagonal by
the curvature bound \(\kappa\leq1\).  The CWF relation for this class follows
from the smooth case by approximation: mollifying \(\gamma\) gives smooth
curves \(\gamma_\delta\to\gamma\) in \(C^1\) with uniformly bounded curvature
and thickness bounded below, the integer \(\operatorname{lk}\) is constant for
small \(\delta\), and \(\operatorname{Tw}\) and \(\operatorname{Wr}\) converge
by dominated convergence; for \(\operatorname{Tw}\) the approximating framing
is obtained by mollifying \(v\), projecting to the normal bundle of
\(\gamma_\delta\), and renormalizing, so that \(v_\delta'\to v'\) almost
everywhere with a uniform \(L^\infty\) bound.  All statements below use CWF
only in this \(C^{1,1}\) form.
\end{remark}

\begin{lemma}[Meridional length bound on the \(\rho_0\)-tube]
\label{lem:meridional-length-bound}
Let \(\gamma\) be a unit-thickness representative and let \(c\subset\partial
N_{\rho_0}(\gamma)\) be a closed curve of Seifert slope \(r=a\mu+b\lambda\) with
\(\gcd(a,b)=1\).  Then
\[
  \Len(c)\ \geq\ 2\pi\rho_0\,\bigl|\,a+b\,W(\gamma)\,\bigr|
  \ =\ 2\pi\rho_0\,\bigl|\,a-b\,\operatorname{Wr}(\gamma)\,\bigr|.
\]
\end{lemma}

\begin{proof}
We stress that \(c\) is an arbitrary embedded curve in its slope class: it is
\emph{not} assumed to be a graph, or section, over the longitudinal
parameter, and the proof must not presume such a structure (a slope curve on
the torus can backtrack longitudinally, and for \(b=0\) no section exists).

The tube coordinates
\(X(s,\theta)=\gamma(s)+\rho_0\bigl(\cos\theta\,e_1(s)+\sin\theta\,e_2(s)\bigr)\)
give a bi-Lipschitz parametrization of \(\partial N_{\rho_0}(\gamma)\) by the
\((s,\theta)\)-torus, because \(\rho_0<1=\Thi(\gamma)\)
(\Cref{not:ambient-metric}).  Parametrize \(c\) by its own arclength
parameter \(t\) and take a Lipschitz lift \(t\mapsto(s(t),\theta(t))\).
From \(|X_\theta|=\rho_0\) and
\(X_s=(1-\rho_0\kappa_1\cos\theta-\rho_0\kappa_2\sin\theta)\gamma'
+\rho_0\Omega_\gamma\,u\), where \(u=X_\theta/\rho_0\) is the unit meridional
direction and \(\gamma'\perp u\), the velocity
\(c'=X_s\,s'+X_\theta\,\theta'\) has \(u\)-component
\(\rho_0\bigl(\Omega_\gamma(s(t))\,s'(t)+\theta'(t)\bigr)\).  Hence
\[
  \Len(c)=\int|c'|\,dt
  \geq\rho_0\int\bigl|\Omega_\gamma(s)\,s'+\theta'\bigr|\,dt
  \geq\rho_0\Bigl|\oint_c\alpha_\gamma\Bigr| ,
\]
where \(\alpha_\gamma=d\theta+\Omega_\gamma(s)\,ds\).  The form
\(\alpha_\gamma\) is closed --- \(\Omega_\gamma(s)\,ds\) is pulled back from
the \(s\)-circle --- so its integral over \(c\) depends only on the homology
class \(a\mu+b\lambda\):
\(\oint_c d\theta=2\pi a\), and
\(\oint_c\Omega_\gamma\,ds
=b\oint_\gamma\Omega_\gamma\,ds=2\pi bW(\gamma)\), since the
\(s\)-projection of \(c\) has degree \(b\).  Thus
\(\oint_c\alpha_\gamma=2\pi(a+bW(\gamma))\) for every representative of the
class, and the CWF relation in the form of \Cref{rem:cwf-conventions}
gives the second equality.
\end{proof}

\begin{theorem}[Conditional boundary-slope height bound]
\label{thm:conditional-slope-height}
Fix \(\Delta,\tau>0\) and \(\Lambda\geq\Rop(K)\).  Suppose the peripheral
geometry over \(Y^{=}_\Lambda(K)\) is controlled in the sense that there are
constants \(W_0(\Lambda)\) and \(\ell_0(\Lambda)>0\) with
\[
  |\operatorname{Wr}(\gamma)|\leq W_0(\Lambda),
  \qquad
  \ell^{\mathrm{st}}_\lambda(\gamma)\geq\ell_0(\Lambda)
  \qquad\text{for all }\gamma\in Y^{=}_\Lambda(K).
\]
Then every pair \((\gamma,F)\in\calZ^{\mathrm{ess}}_{\Lambda,\Delta,\tau}(K)\)
with \(F\) connected and \(\partial F\neq\varnothing\) of slope \(r\) satisfies
\[
  \Delta_T(r,\lambda)
  \ \leq\
  \frac{\Delta}{2\pi\rho_0\,c_{\mathrm{col}}(\tau)}
  \ +\
  \frac{\Delta\,W_0(\Lambda)}{c_{\mathrm{col}}(\tau)\,\ell_0(\Lambda)},
\]
where \(\lambda\) is the Seifert longitude and \(c_{\mathrm{col}}(\tau)\) is the
collar constant of \Cref{prop:geometric-topological-boundedness}.  Equivalently,
if some essential surface of slope \(r\) is visible at level
\((\Lambda,\Delta,\tau)\), then \(\Lambda\) must be large enough that the
right-hand side reaches \(\Delta_T(r,\lambda)\); this is a ropelength lower bound
in terms of boundary-slope height.
\end{theorem}

\begin{proof}
Write \(r=a\mu+b\lambda\); then \(\Delta_T(r,\lambda)=|a|\).  Each of the
\(|\partial F|\) boundary components has the common class \(r\), and by
\Cref{prop:geometric-topological-boundedness}, \(\Len(\partial F)\leq
c_{\mathrm{col}}(\tau)^{-1}\Delta\), so a single component \(c\) has
\(\Len(c)\leq c_{\mathrm{col}}(\tau)^{-1}\Delta\).
\Cref{lem:meridional-length-bound} gives
\(2\pi\rho_0|a-b\operatorname{Wr}(\gamma)|\leq
\Len(c)\), while the definition of stable longitudinal systole gives
\(|b|\,\ell^{\mathrm{st}}_\lambda(\gamma)\leq
\Len(c)\).  Therefore
\[
  |a|
  \leq
  |a-b\operatorname{Wr}(\gamma)|+|b|\,|\operatorname{Wr}(\gamma)|
  \leq
  \frac{\Len(c)}{2\pi\rho_0}+\frac{\Len(c)}{\ell^{\mathrm{st}}_\lambda(\gamma)}|\operatorname{Wr}(\gamma)| .
\]
Substituting the peripheral bounds and \(\Len(c)\leq
c_{\mathrm{col}}(\tau)^{-1}\Delta\) yields the claim.
\end{proof}

\begin{remark}[Instantiating the peripheral hypotheses]
\label{rem:instantiate-peripheral}
The writhe hypothesis always holds with an explicit exponent: for a
unit-thickness curve of length at most \(\Lambda\), the writhe is bounded in
absolute value by the average crossing number, since the two are the
direction-averages of the signed and unsigned crossing counts of the planar
projections, and the average crossing number is at most \(C_{\mathrm{BS}}\Lambda^{4/3}\) by
the thickness--crossing-number estimate of Buck--Simon \cite{BuckSimon}.  Thus
one may always take \(W_0(\Lambda)=C_{\mathrm{BS}}\Lambda^{4/3}\).  The stable longitudinal-systole
hypothesis, by contrast, is genuine.  On the \(\rho_0\)-tube the core-direction
component of an arbitrary slope curve has coefficient
\(1-\rho_0\kappa_N\geq1-\rho_0\kappa\); integrating over longitudinal degree
\(|b|\) gives
\[
  \ell^{\mathrm{st}}_\lambda(\gamma)\ \geq\ \Len(\gamma)-\rho_0\operatorname{TC}(\gamma),
\]
where \(\operatorname{TC}\) is the total curvature.  Unit thickness implies
\(\kappa\leq1\) almost everywhere, hence
\(\operatorname{TC}(\gamma)\leq\Len(\gamma)\), and therefore
\[
  \ell^{\mathrm{st}}_\lambda(\gamma)
  \geq (1-\rho_0)\Len(\gamma)
  =\tfrac12\Len(\gamma)
\]
for the convention \(\rho_0=\tfrac12\).  Thus a positive longitudinal
systole bound is automatic; only the stronger hypothesis that it grows
linearly with the upper budget \(\Lambda\) is additional.  Under that linear
hypothesis \(\ell_0(\Lambda)=c_2\Lambda\),
\Cref{thm:conditional-slope-height} gives
\(\Delta_T(r,\lambda)\leq C(\tau)\Delta(1+\Lambda^{1/3})\), so realizing an
intersection height \(H\) requires
\(\Lambda\gtrsim\bigl(H/(C(\tau)\Delta)-1\bigr)^{3}\).  The numerical-slope counterpart of
this bound requires no peripheral hypothesis at all; this is the content of the
writhe window proved next (\Cref{thm:writhe-window}).
\end{remark}

\subsection{The writhe window: an unconditional slope bound}
\label{subsec:writhe-window}

The conditional theorem above bounds the intersection-number height of a slope,
and its systole hypothesis is exactly what controls the denominator.  If slopes
are measured \emph{numerically} instead, the denominator divides out and every
hypothesis can be discharged.  The result is the strongest unconditional
statement of this section: boundary slopes of bounded-geometry essential
surfaces cluster around the writhe of the representative.

Throughout, a non-meridional slope is written in Seifert-framed
meridian--longitude coordinates as \(p\mu+q\lambda\) with \(\gcd(p,q)=1\) and
\(q\geq1\), and its numerical value is the rational number \(r=p/q\).  The
framing comparison and the resulting numerical window are illustrated in
\Cref{fig:writhe-window}.

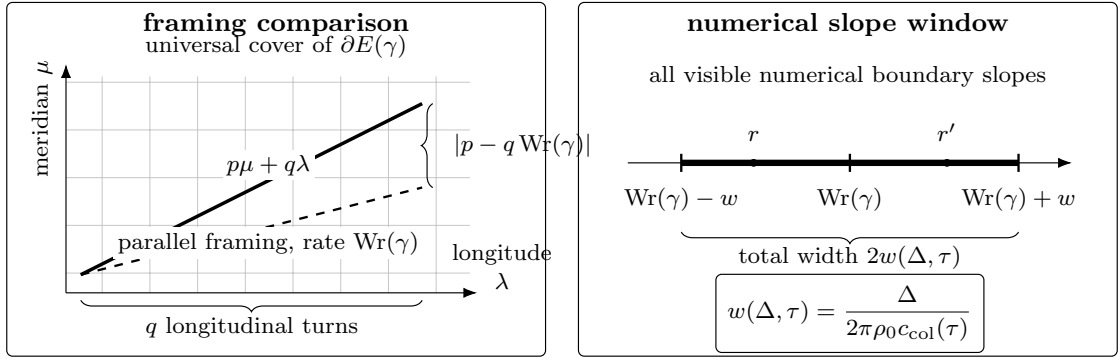
\begin{figure}[t]
\centering
\resizebox{0.98\textwidth}{!}{%
\begin{tikzpicture}[
  x=1cm,y=1cm,>=Latex,
  panel/.style={draw, rounded corners=2pt},
  title/.style={font=\footnotesize\bfseries},
  note/.style={font=\scriptsize, align=center}
]
  \draw[panel] (0,0.15) rectangle (6.55,4.45);
  \node[title] at (3.275,4.18) {framing comparison};
  \node[note] at (3.275,3.92) {universal cover of $\partial E(\gamma)$};
  \draw[step=0.58, black!22, very thin] (0.72,0.92) grid (5.63,3.55);
  \draw[->] (0.72,0.92) -- (5.72,0.92);
  \draw[->] (0.72,0.92) -- (0.72,3.70);
  \node[note, align=center] at (6.02,1.24) {longitude\\$\lambda$};
  \node[note, rotate=90] at (0.43,3.04) {meridian $\mu$};
  \coordinate (a) at (0.90,1.14);
  \coordinate (b) at (5.05,3.22);
  \coordinate (c) at (5.05,2.20);
  \draw[very thick] (a) -- (b);
  \draw[thick, dashed] (a) -- (c);
  \node[note, above, sloped, fill=white, inner sep=1.5pt] at ($(a)!0.55!(b)$)
    {$p\mu+q\lambda$};
  \node[note, below, sloped, fill=white, inner sep=1.5pt] at ($(a)!0.55!(c)$)
    {parallel framing, rate $\operatorname{Wr}(\gamma)$};
  \draw[decorate, decoration={brace, amplitude=4pt}]
    (5.18,2.20) -- (5.18,3.22);
  \node[note, anchor=west] at (5.30,2.71)
    {$|p-q\operatorname{Wr}(\gamma)|$};
  \draw[decorate, decoration={brace, amplitude=4pt, mirror}]
    (0.90,0.83) -- (5.05,0.83);
  \node[note] at (2.98,0.50) {$q$ longitudinal turns};

  \draw[panel] (6.95,0.15) rectangle (13.50,4.45);
  \node[title] at (10.225,4.18) {numerical slope window};
  \node[note] at (10.225,3.55) {all visible numerical boundary slopes};
  \draw[->] (7.55,2.50) -- (12.95,2.50);
  \draw[line width=2.2pt] (8.20,2.50) -- (12.30,2.50);
  \foreach \x in {8.20,10.25,12.30}
    \draw[thick] (\x,2.39) -- (\x,2.61);
  \node[note, below] at (8.20,2.33)
    {$\operatorname{Wr}(\gamma)-w$};
  \node[note, below] at (10.25,2.33)
    {$\operatorname{Wr}(\gamma)$};
  \node[note, below] at (12.30,2.33)
    {$\operatorname{Wr}(\gamma)+w$};
  \fill (9.08,2.50) circle (1.35pt);
  \fill (11.43,2.50) circle (1.35pt);
  \node[note, above] at (9.08,2.63) {$r$};
  \node[note, above] at (11.43,2.63) {$r'$};
  \draw[decorate, decoration={brace, amplitude=4pt, mirror}]
    (8.20,1.66) -- (12.30,1.66);
  \node[note] at (10.25,1.35) {total width $2w(\Delta,\tau)$};
  \node[draw, rounded corners=2pt, note, inner sep=4pt] at (10.25,0.67)
    {$\displaystyle
      w(\Delta,\tau)=
      \frac{\Delta}{2\pi\rho_0c_{\mathrm{col}}(\tau)}$};
\end{tikzpicture}%
}
\caption{The mechanism of the writhe window.  In the universal cover of the
peripheral torus (left), after \(q\) longitudinal turns the boundary slope
\(p\mu+q\lambda\) and the parallel framing are separated by
\(|p-q\operatorname{Wr}(\gamma)|\) meridional turns.  The collar--area
estimate bounds the length needed to create this separation.  After division
by \(q\), every visible numerical slope lies in the interval of half-width
\(w(\Delta,\tau)\) centered at \(\operatorname{Wr}(\gamma)\) (right).  The
picture displays the positive-separation case; reversing orientation changes
signs but not the absolute-value estimate.}
\label{fig:writhe-window}
\end{figure}

\begin{theorem}[Writhe window for visible boundary slopes]
\label{thm:writhe-window}
Let \((\gamma,F)\in\calZ^{\mathrm{ess}}_{\Lambda,\Delta,\tau}(K)\) with \(F\)
connected, \(\partial F\neq\varnothing\), and non-meridional boundary slope of
numerical value \(r=p/q\).  Then
\[
  \bigl|\,r-\operatorname{Wr}(\gamma)\,\bigr|
  \ \leq\
  \frac{\Area(F)}{2\pi\rho_0\,c_{\mathrm{col}}(\tau)\,q}
  \ \leq\
  \frac{\Delta}{2\pi\rho_0\,c_{\mathrm{col}}(\tau)}
  \ =:\ w(\Delta,\tau),
\]
where \(\rho_0\in(0,1)\) is the fixed tube radius of
\Cref{not:ambient-metric}.
Consequently, for a fixed representative \(\gamma\), all numerical boundary
slopes of connected essential surfaces visible at the \((\Delta,\tau)\)-budget
lie in the window
\(\bigl[\operatorname{Wr}(\gamma)-w,\ \operatorname{Wr}(\gamma)+w\bigr]\), and
any two such slopes \(r,r'\) satisfy
\[
  \Delta_T^{\mathrm{rate}}(r,r')=\Bigl|\frac{p}{q}-\frac{p'}{q'}\Bigr|
  \ \leq\ 2\,w(\Delta,\tau),
\]
independently of \(\Lambda\).
\end{theorem}

\begin{proof}
Let \(c\) be one component of \(\partial F\); it is a closed curve of Seifert
slope \(p\mu+q\lambda\) on \(\partial N_{\rho_0}(\gamma)\).
\Cref{lem:meridional-length-bound} gives
\(\Len(c)\geq2\pi\rho_0|p-q\operatorname{Wr}(\gamma)|\).  The collar estimate of
\Cref{prop:geometric-topological-boundedness} gives
\(c_{\mathrm{col}}(\tau)\Len(\partial F)\leq\Area(F)\), hence
\(\Len(c)\leq\Area(F)/c_{\mathrm{col}}(\tau)\).  Dividing by
\(2\pi\rho_0 q\geq2\pi\rho_0\)
gives the window; the diameter statement is the triangle inequality, and
\(|p/q-p'/q'|=|pq'-p'q|/(qq')=\Delta_T(r,r')/(qq')\) is the rate distance of
\Cref{def:boundary-slope-rate-diameter}.
\end{proof}

\begin{corollary}[Seifert surfaces on high-writhe representatives are expensive]
\label{cor:seifert-writhe-cost}
If \(F\subset E(\gamma)\) is a connected Seifert surface with
\(\Thi(F)\geq\tau\), then
\[
  \Area(F)\ \geq\ 2\pi\rho_0\,c_{\mathrm{col}}(\tau)\,\bigl|\operatorname{Wr}(\gamma)\bigr| .
\]
Thus a taut Seifert surface is visible at the budget \((\Delta,\tau)\) only over
representatives whose writhe satisfies
\(|\operatorname{Wr}(\gamma)|\leq w(\Delta,\tau)\).
\end{corollary}

\begin{proof}
A Seifert surface has numerical slope \(r=0\) with \(q=1\); apply the first
inequality of \Cref{thm:writhe-window}.  Essentiality is not needed here: the
proof of the window uses only the collar estimate of
\Cref{prop:geometric-topological-boundedness} and
\Cref{lem:meridional-length-bound}, neither of which involves essentiality,
so the conclusion holds for every connected Seifert surface of relative
thickness at least \(\tau\).
\end{proof}

\begin{corollary}[Unconditional slope-height--ropelength inequality]
\label{cor:unconditional-slope-height}
There is a universal constant \(C_{\mathrm{BS}}>0\) such that every visible non-meridional
numerical slope \(r\) as in \Cref{thm:writhe-window} satisfies
\[
  |r|\ \leq\ C_{\mathrm{BS}}\Lambda^{4/3}+w(\Delta,\tau).
\]
Hence, whenever \(|r|>w(\Delta,\tau)\), the visibility level of the slope-\(r\)
surface type obeys
\[
  \beta_{\mathfrak S_r}(K;\Delta,\tau)
  \ \geq\
  \left(\frac{|r|-w(\Delta,\tau)}{C_{\mathrm{BS}}}\right)^{3/4},
\]
and dually the area visibility level obeys
\(a_{\mathfrak S_r}(K;\Lambda,\tau)\geq
2\pi\rho_0 c_{\mathrm{col}}(\tau)\bigl(|r|-C_{\mathrm{BS}}\Lambda^{4/3}\bigr)_+\).
\end{corollary}

\begin{proof}
As in \Cref{rem:instantiate-peripheral},
\(|\operatorname{Wr}(\gamma)|\leq C_{\mathrm{BS}}\Lambda^{4/3}\) unconditionally
\cite{BuckSimon}; insert this into the window and solve for \(\Lambda\), or for
\(\Area(F)\) via the first inequality of \Cref{thm:writhe-window}.
\end{proof}

\begin{corollary}[Unconditional invisibility gap]
\label{cor:unconditional-gap}
Fix \((\Delta,\tau)\).  If \(K\) has a boundary slope of numerical value \(r\)
with
\[
  |r|\ >\ w(\Delta,\tau)+C_{\mathrm{BS}}\Rop(K)^{4/3},
\]
then \(\beta_{\mathfrak S_r}(K;\Delta,\tau)>\Rop(K)\): the slope-\(r\) surface
type is invisible over the ideal stratum at that budget, and the knot--surface
trade-off frontier of its class is nonconstant, unconditionally and with
explicit constants.
\end{corollary}

\begin{proof}
Immediate from \Cref{cor:unconditional-slope-height}.
\end{proof}

\begin{remark}[Geometric meaning, and relation with the conditional bound]
\label{rem:writhe-window-meaning}
The mechanism is a rigidity of framings: the Seifert framing of a thick
representative rotates relative to the tube at total rate
\(-\operatorname{Wr}(\gamma)\), by the C\u alug\u areanu--White--Fuller relation,
and a surface of bounded area cannot afford boundary longer than
\(\Area/c_{\mathrm{col}}(\tau)\), hence cannot wind meridionally at a rate far
from the framing rate.  In slogan form: \emph{cheap essential surfaces have
boundary slope near the writhe}.  The theorem complements Hatcher's finiteness
\cite{HatcherBoundarySlopes}: topologically the slope set is finite but its
numerical spread is unbounded over knot types, whereas geometrically the visible
slopes over one representative lie in an interval of explicit width
\(2w(\Delta,\tau)\), uniformly in \(\Lambda\).  Measured by intersection number
\(\Delta_T\) instead of numerically, the denominator \(q\) must be controlled
and the statement reverts to the conditional
\Cref{thm:conditional-slope-height}; for integral slopes (\(q=1\)) the two
coincide and the window bounds \(\Delta_T(r,\lambda)\) unconditionally.
\end{remark}

\begin{remark}[Comparison with Bennequin-type inequalities]
\label{rem:bennequin-comparison}
It is instructive to compare \Cref{cor:seifert-writhe-cost} with the Bennequin
inequality \cite{Bennequin}, which bounds the self-linking number of a
transverse representative by the negative Euler characteristic of a Seifert
surface.  Both statements bound a framing-type quantity of the knot
representative by the complexity of a spanning surface, but the quantities are
different in kind: Bennequin's inequality is contact-topological, bounds the
integer self-linking number, and is insensitive to the metric size of the
surface, whereas the writhe window is metric, bounds the real-valued writhe of
a \(C^{1,1}\) thick representative, and is insensitive to the topology of the
surface except through its area.  The two are linked by
\Cref{prop:geometric-topological-boundedness}: at relative thickness \(\tau\)
the area controls \(|\chi(F)|\), so the window may be read as a metric
counterpart in which area at fixed thickness plays the role of Euler
characteristic.  Neither statement implies the other.
\end{remark}

\begin{example}[The writhe window for torus knots]
\label{ex:writhe-window-torus}
Let \(K=T(p,q)\) with \(p,q\geq2\) and \(\gcd(p,q)=1\), so that \(T(p,q)\)
is a knot.  Two natural essential surface
types are an essential Seifert surface, of numerical slope \(0\), and a
cabling annulus, of integral slope \(pq\).  For every unit-thickness
representative \(\gamma\) and every \(\tau\)-thick realization in
\(E(\gamma)\), \Cref{thm:writhe-window} gives
\[
  \Area(F_0)\ \geq\ 2\pi\rho_0 c_{\mathrm{col}}(\tau)\,|\operatorname{Wr}(\gamma)|,
  \qquad
  \Area(A)\ \geq\ 2\pi\rho_0 c_{\mathrm{col}}(\tau)\,\bigl|pq-\operatorname{Wr}(\gamma)\bigr| .
\]
Thus the writhe of the representative arbitrates between the two cheapest
essential surfaces of the exterior: the Seifert surface is cheap only over
representatives of small writhe, which makes the cabling annulus expensive,
and conversely.  These two displayed inequalities are proved.  The following
sharpness discussion, by contrast, rests on numerical inputs that we cite but
do not prove.  Numerical computations of tight \(T(2,n)\) representatives
suggest that the trade-off is resolved strictly between the two slopes: the
computed writhe of the tight trefoil is \(\operatorname{Wr}\approx3.41\)
\cite{KatritchBednarNature}, and the computed ideal writhes of the \(T(2,n)\)
family are consistent with a quasi-quantization in units of \(4/7\), hence
with an asymptotically linear growth of order \(8n/7\)
\cite{StasiakKatritchKauffman}; these values lie between the Seifert slope
\(0\) and the annulus slope \(2n\).  If, as these computations suggest, the
writhes of tight \(T(2,n)\) representatives grow linearly with slope
\(8/7\), and if the cabling annulus of a tight representative can be realized
at some fixed relative thickness with area bounded by a constant multiple of
the representative's length --- a realization we expect but do not prove ---
then both sides of the window inequality for the annulus grow linearly in
\(n\), and the window would be sharp up to constants on this family.  Neither
numerical input has been established rigorously; sharpness of the window is
posed as \Cref{prob:writhe-window-sharpness}.
\end{example}

\begin{corollary}[Linear joint area cost for torus knots]
\label{cor:torus-joint-cost}
Let \(\gamma\) be any unit-thickness representative of \(T(p,q)\),
\(p,q\geq2\), \(\gcd(p,q)=1\),
and let \(F_0,A\subset E(\gamma)\) be an essential Seifert surface and a
cabling annulus, each of relative thickness at least \(\tau\).  Then
\[
  \Area(F_0)+\Area(A)\ \geq\ 2\pi\rho_0\,c_{\mathrm{col}}(\tau)\,pq .
\]
Since the standard closed-braid diagram gives the elementary bound
\(\mathrm{Cr}(T(p,q))\leq q(p-1)<pq\), the joint cost is at least
\(2\pi\rho_0 c_{\mathrm{col}}(\tau)\) times the crossing number.  In particular
the taut layer and the cabling-annulus layer of a torus knot are never
simultaneously visible below the area budget
\(\Delta=\pi\rho_0 c_{\mathrm{col}}(\tau)\,pq\), over any representative and at
any ropelength level.
\end{corollary}

\begin{proof}
Apply \Cref{thm:writhe-window} in the common exterior \(E(\gamma)\) to the
slopes \(0\) and \(pq\) (both of denominator \(1\)) and add:
\(|0-\operatorname{Wr}(\gamma)|+|pq-\operatorname{Wr}(\gamma)|\geq pq\).
\end{proof}

\begin{problem}[Sharpness of the writhe window]
\label{prob:writhe-window-sharpness}
Three sharpness questions remain.  First, is the linear dependence of the
window width \(w(\Delta,\tau)\) on \(\Delta\) optimal, that is, are there
families realizing slopes at distance comparable to
\(\Delta/(2\pi\rho_0 c_{\mathrm{col}}(\tau))\) from the writhe with area
\(\Delta\)?  The torus-knot family of \Cref{ex:writhe-window-torus} is a
natural candidate, conditionally on the numerically observed writhe
asymptotics and on a linear-area realization of the cabling annulus, neither
of which is proved.  Second, the composite bound
\(|r|\leq C_{\mathrm{BS}}\Lambda^{4/3}+w(\Delta,\tau)\) is sharp only if unit-thickness
representatives can attain writhe of order \(\Lambda^{4/3}\) while carrying a
\(\tau\)-thick essential surface of the corresponding slope; by
\Cref{cor:seifert-writhe-cost} such extremal writhe excludes cheap Seifert
surfaces, so extremal examples must have all their cheap slopes far from
zero.  Determine whether the exponent \(4/3\) is attained, or can be improved
for essential-surface-carrying representatives.  Third, determine the extremal
representatives: for fixed \((\Delta,\tau)\), which \(\gamma\) maximize the
number of boundary slopes visible in their window.
\end{problem}

\subsection{A dual area lower bound and a frontier gap}
\label{subsec:frontier-gap}

Read in the other direction, the length bound of
\Cref{lem:meridional-length-bound} bounds \emph{area from below} in terms of
boundary-slope height.  This is the first place where the trade-off frontier is
shown to be a genuinely nonconstant and positive function of the surface class,
rather than only bounded above by explicit realizations.

\begin{proposition}[Area lower bound from slope height]
\label{prop:area-lower-slope}
Let \((\gamma,F)\in\calZ^{\mathrm{ess}}_{\Lambda,\Delta,\tau}(K)\) with \(F\)
connected and \(\partial F\neq\varnothing\) of slope \(r=a\mu+b\lambda\).  Then
\[
  \Area(F)\ \geq\ 2\pi\rho_0\,c_{\mathrm{col}}(\tau)\,
  \bigl|\,a-b\operatorname{Wr}(\gamma)\,\bigr| ,
\]
where \(a-b\operatorname{Wr}(\gamma)\) is the parallel-frame meridional winding
of \(\partial F\).  Consequently, if \(|\operatorname{Wr}(\gamma)|\leq
W_0(\Lambda)\) on \(Y^{=}_\Lambda(K)\), the area visibility level of the slope \(r\)
satisfies
\[
  a_{\mathfrak S_r}(K;\Lambda,\tau)
  \ \geq\
  2\pi\rho_0\,c_{\mathrm{col}}(\tau)\,
  \bigl(\Delta_T(r,\lambda)-|b|\,W_0(\Lambda)\bigr)_+ ,
\]
where \((x)_+=\max\{x,0\}\).
\end{proposition}

\begin{proof}
The collar estimate of \Cref{prop:geometric-topological-boundedness} gives
\(\Area(F)\geq c_{\mathrm{col}}(\tau)\Len(\partial F)\geq
c_{\mathrm{col}}(\tau)\Len(c)\) for any single boundary component \(c\), and
\Cref{lem:meridional-length-bound} gives
\(\Len(c)\geq 2\pi\rho_0|a-b\operatorname{Wr}(\gamma)|\).  Since \(\Delta_T(r,\lambda)=|a|\) and
\(|a-b\operatorname{Wr}(\gamma)|\geq|a|-|b|\,|\operatorname{Wr}(\gamma)|\geq
|a|-|b|W_0(\Lambda)\), infimizing over \(\gamma\in Y^{=}_\Lambda(K)\) gives the
visibility bound.
\end{proof}

\begin{remark}[A nonconstant, unbounded frontier]
\label{rem:frontier-gap}
\Cref{prop:area-lower-slope} is the exact dual of
\Cref{thm:conditional-slope-height}: the theorem bounds slope height from above
given area and ropelength; the proposition bounds area from below given slope
height.  It shows unconditionally that the trade-off frontier
\(\Delta_{\min}^\tau(\Lambda;[F])\), defined in
\Cref{def:trade-off-frontier}, is bounded below by a positive quantity
whenever the parallel-frame meridional winding cannot be made to vanish over
\(Y^{=}_\Lambda(K)\) --- that is, whenever the slope \(a/b\) lies outside the band of
writhes achievable at ropelength \(\Lambda\).  In that regime visibility
genuinely costs area, and the area frontier is bounded away from zero at that level.

The gap grows without bound across families, in the following level-wise
sense.  Consider knot types \(K_n\) and slopes
\(r_n=a_n\mu+b_n\lambda\) of essential surfaces in \(E(K_n)\), with
\(\Delta_T(r_n,\lambda)=|a_n|\to\infty\) while \(|b_n|\) stays bounded.
Then at every fixed level \(\Lambda\) one has
\(a_{\mathfrak S_{r_n}}(K_n;\Lambda,\tau)\to\infty\).  Equivalently, at a
fixed area budget \(\Delta\) and a \emph{fixed} ropelength level
\(\Lambda\), the slope-\(r_n\) surfaces satisfying
\(2\pi\rho_0 c_{\mathrm{col}}(\tau)(|a_n|-|b_n|W_0(\Lambda))>\Delta\) are
invisible at that level, so \(\beta_{r_n}(K_n;\Delta,\tau)>\Lambda\).  It must be
stressed that this is a statement about one fixed level at a time, not about
all levels simultaneously: the achievable writhe band
\(W_0(\Lambda)=C_{\mathrm{BS}}\Lambda^{4/3}\) grows with \(\Lambda\), so a fixed slope
\(r_n\) that is invisible at one level may well become visible at a larger
one.  What the estimate yields uniformly in the family is the quantitative
lower bound
\(\beta_{r_n}(K_n;\Delta,\tau)\geq
\bigl((|a_n|/|b_n|-w(\Delta,\tau))/C_{\mathrm{BS}}\bigr)^{3/4}\to\infty\)
of \Cref{cor:unconditional-slope-height}, hence
finite-but-divergent visibility levels rather than outright invisibility.
This is a proven instance of the phenomenon the trade-off frontier was
designed to detect: an essential surface that is invisible within a prescribed
geometric budget and appears only once the budget is enlarged.  By
\Cref{cor:unconditional-slope-height} the writhe band is always available with
\(W_0(\Lambda)=C_{\mathrm{BS}}\Lambda^{4/3}\), so for numerical slope heights all of the
above holds unconditionally, with the explicit finite-gap criterion of
\Cref{cor:unconditional-gap}.  Whether a \emph{finite} gap
\(\Rop(K)<\beta_{\mathfrak S}(K;\Delta,\tau)<\infty\) occurs for a specific
low-crossing knot is then a computation of its boundary slopes against
\(C_{\mathrm{BS}}\Rop(K)^{4/3}\), and sharpening the exponent \(4/3\) is the quantitative
content of
\Cref{prob:finite-length-boundary-twisting-versus-ropelength}.
\end{remark}


\section{Characteristic essential surface systems}
\label{sec:characteristic}

\subsection{Classical carriers and decompositions}
\label{subsec:classical-carriers}

The characteristic systems in this section have several complementary
classical descriptions.  Once a triangulation is fixed, normal surface theory
represents candidate essential surfaces by integral solutions of the matching
equations, and the Jaco--Oertel algorithm tests for the existence of
incompressible and boundary-incompressible representatives
\cite{Haken,JacoOertel}.  Efficient and ideal triangulations, together with
crushing, reduce inessential normal phenomena and simplify the ambient input
\cite{JacoRubinstein0Efficient}.  In a different direction, Floyd--Oertel and
Oertel package incompressible surfaces into finitely many branched-surface
carriers with integral weights \cite{FloydOertel,OertelBranched}; special
spines provide dual finite descriptions of the pieces and their gluings
\cite{MatveevBook}.

The geometric filtration does not supply a new normal form or a new carrying
theorem.  Its role is to select a finite, faithfully recoverable portion of
these classical parameter spaces.  For the bounded-geometry triangulations
used below, \Cref{prop:conditional-normal-coefficient-bound} gives a
conditional normal-coefficient cutoff.  The analogous assertion for branch
weights is left as \Cref{prob:geometric-carrying}.  Similarly, dual decorated
spines are an implementation viewpoint rather than an ingredient of the JSJ
or reconstruction proofs.

\subsection{Classical JSJ and characteristic skeleton}

The first finite objects in the theory are purely topological.  Three levels
must be distinguished: the canonical torus JSJ system, the fuller
Jaco--Shalen--Johannson characteristic frontier, and separate prime or Conway
sphere systems.  None of them is a finite list of all essential surfaces.

\begin{proposition}[Canonical torus and characteristic systems]
\label{prop:finite-characteristic-surface-system}
For every nontrivial knot exterior $E(K)$ there are finite, knot-type
invariant collections
\[
  \mathcal T_{\mathrm{JSJ}}(K)
  \quad\text{and}\quad
  \mathcal E_{\mathrm{char}}(K),
\]
with the following properties.
\begin{enumerate}[label=(\roman*),leftmargin=2em]
\item $\mathcal T_{\mathrm{JSJ}}(K)$ is the set of isotopy classes of the
canonical JSJ tori.  Cutting along them gives pieces which are Seifert
fibered or atoroidal; after geometrization the non-Seifert pieces are
finite-volume hyperbolic.
\item $\mathcal E_{\mathrm{char}}(K)$ is the union of
$\mathcal T_{\mathrm{JSJ}}(K)$ with the isotopy classes of the essential
annuli and tori in the frontier of the characteristic submanifold.
It records the Seifert and $I$--bundle regions and determines the JSJ torus
system together with its characteristic incidence data.
\item Both systems are unique up to isotopy and are preserved by every
self-homeomorphism of the exterior.
\end{enumerate}
\end{proposition}

\begin{proof}
A nontrivial knot exterior is compact, orientable, irreducible, and has
incompressible torus boundary.  The assertions are the
Jaco--Shalen--Johannson characteristic-submanifold and JSJ theorems
\cite{JacoShalen,Johannson,NeumannSwarup}; the geometric description of the
atoroidal pieces follows from geometrization.  Finiteness and invariance are
part of the canonical decomposition theory.
\end{proof}

\begin{definition}[Rooted JSJ tree]
\label{def:rooted-jsj-tree}
The rooted JSJ tree $\Gamma_{\mathrm{JSJ}}(K)$ is the dual graph of the
splitting of $E(K)$ along $\mathcal T_{\mathrm{JSJ}}(K)$, rooted at the piece
meeting $\partial E(K)$ and labelled by the Seifert or hyperbolic type of each
piece and by peripheral incidence data.
\end{definition}

\begin{proposition}
\label{prop:jsj-dual-tree}
The graph $\Gamma_{\mathrm{JSJ}}(K)$ is a finite tree and is invariant, as a
rooted labelled graph, under symmetries of the knot.
\end{proposition}

\begin{proof}
Since $H_2(E(K);\mathbb Z)=0$, every JSJ torus separates.  Hence the dual
graph is a tree.  Invariance follows from the canonicality in
\Cref{prop:finite-characteristic-surface-system}.
\end{proof}

\Cref{fig:jsj-tree} illustrates the decomposition and its rooted dual tree.

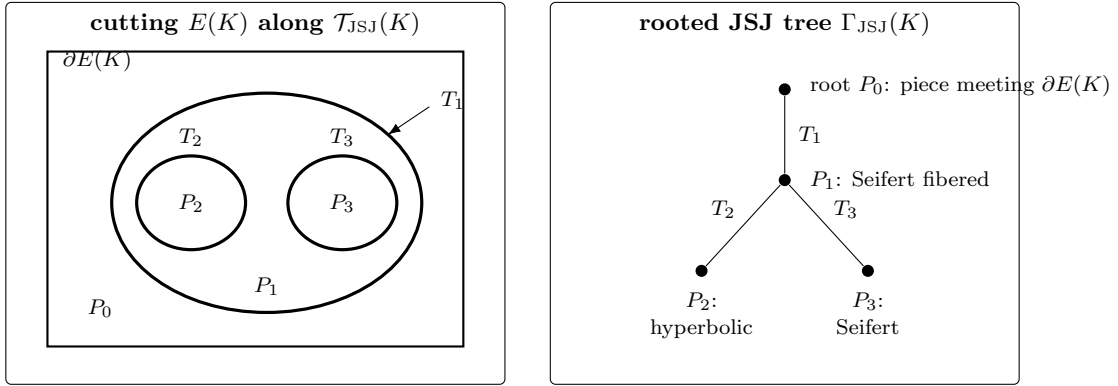
\begin{figure}[t]
\centering
\begin{tikzpicture}[x=1cm,y=1cm,>=Latex,
  panel/.style={draw, rounded corners=2pt},
  title/.style={font=\footnotesize\bfseries},
  note/.style={font=\scriptsize, align=center},
  vtx/.style={circle, fill, inner sep=1.6pt}]
  \draw[panel] (0,-0.45) rectangle (6.6,4.6);
  \node[title] at (3.3,4.32) {cutting $E(K)$ along
    $\mathcal T_{\mathrm{JSJ}}(K)$};
  \draw[thick] (0.55,0.05) rectangle (6.05,3.95);
  \node[note, anchor=south west] at (0.62,3.55) {$\partial E(K)$};
  \draw[very thick] (3.45,1.95) ellipse (2.05 and 1.45);
  \node[note] at (3.45,3.62) {};
  \node[note, anchor=west] at (5.62,3.30) {$T_1$};
  \draw[->, thin] (5.60,3.22) -- (5.05,2.85);
  \draw[very thick] (2.45,1.95) ellipse (0.72 and 0.62);
  \draw[very thick] (4.45,1.95) ellipse (0.72 and 0.62);
  \node[note] at (2.45,2.82) {$T_2$};
  \node[note] at (4.45,2.82) {$T_3$};
  \node[note] at (1.25,0.55) {$P_0$};
  \node[note] at (3.45,0.85) {$P_1$};
  \node[note] at (2.45,1.95) {$P_2$};
  \node[note] at (4.45,1.95) {$P_3$};
  \draw[panel] (7.2,-0.45) rectangle (13.4,4.6);
  \node[title] at (10.3,4.32) {rooted JSJ tree $\Gamma_{\mathrm{JSJ}}(K)$};
  \node[vtx] (p0) at (10.3,3.45) {};
  \node[vtx] (p1) at (10.3,2.25) {};
  \node[vtx] (p2) at (9.2,1.05) {};
  \node[vtx] (p3) at (11.4,1.05) {};
  \draw (p0) -- node[note, right=0.5mm] {$T_1$} (p1);
  \draw (p1) -- node[note, above left=-0.5mm] {$T_2$} (p2);
  \draw (p1) -- node[note, above right=-0.5mm] {$T_3$} (p3);
  \node[note, anchor=west] at (10.5,3.5)
    {root $P_0$: piece meeting $\partial E(K)$};
  \node[note, anchor=west] at (10.5,2.25) {$P_1$: Seifert fibered};
  \node[note, anchor=north] at (9.2,0.9) {$P_2$:\\hyperbolic};
  \node[note, anchor=north] at (11.4,0.9) {$P_3$:\\Seifert};
\end{tikzpicture}
\caption{The canonical torus decomposition and its rooted dual tree, for a
schematic satellite exterior.  Left: cutting \(E(K)\) along the JSJ tori
\(\mathcal T_{\mathrm{JSJ}}(K)=\{T_1,T_2,T_3\}\) produces pieces \(P_i\) that
are Seifert fibered or hyperbolic; every torus separates because
\(H_2(E(K))=0\).  Right: the dual graph is therefore a finite tree, rooted at
the piece meeting \(\partial E(K)\) and labelled by piece types and
peripheral incidence data (\Cref{def:rooted-jsj-tree,prop:jsj-dual-tree}).
The geometric filtration adds visibility levels to this classical finite
skeleton: it asks at which \((\Lambda,\Delta,\tau)\) the tori of the system
become visible as thick surfaces.}
\label{fig:jsj-tree}
\end{figure}

\begin{remark}[Companionship graphs]
The unlabelled tree alone does not determine the knot.  Budney enriches the
JSJ graph by companion-link labels and peripheral splicing data and obtains a
bijective encoding of links in $S^3$ by finite labelled acyclic graphs
\cite{BudneyJSJ}.  For knots the graph is naturally rooted.  Connected sum,
cabling, Whitehead doubling, and general satellite operations are all
instances of splicing.  Thus finite characteristic data should retain piece
labels and attachments, not only the set of visible tori.
\end{remark}

\begin{remark}[Mapping classes and the group-theoretic JSJ shadow]
Every knot symmetry acts on $\Gamma_{\mathrm{JSJ}}(K)$.  The kernel of this
action can still contain nontrivial patterned mapping classes of the pieces
and Dehn twists along frontier tori or annuli, as made precise by Johannson's
deformation theory \cite{Johannson}.  Van Kampen's theorem expresses
$\pi_1E(K)$ as a graph of groups with $\mathbb Z^2$ edge groups, and the
canonical splittings of Scott--Swarup recover a group-theoretic decomposition
closely related to the topological JSJ structure \cite{ScottSwarup}.  The
peripheral meridian remains essential when this algebraic skeleton is used to
recognize the knot rather than only its group.
\end{remark}

\begin{remark}[Prime and Conway layers are different]
A decomposing sphere gives a meridional essential annulus in the knot exterior
and belongs to the prime-decomposition layer, not automatically to the
frontier of the characteristic submanifold.  Likewise, an essential Conway
sphere gives an essential four-punctured sphere in the exterior.  The
Bonahon--Siebenmann splitting treats it as the Euclidean orbifold
$S^2(2,2,2,2)$ in the pair $(S^3,K)$
\cite{BonahonSiebenmannToric,BonahonSiebenmannKnots}.  These finite canonical
or auxiliary layers complement the ordinary torus JSJ system but should not
be identified with it.
\end{remark}

\begin{remark}[Two complementary finitenesses]
\label{rem:topological-skeleton-finite-resolution}
The finite systems $\mathcal T_{\mathrm{JSJ}}(K)$ and
$\mathcal E_{\mathrm{char}}(K)$ provide a topological skeleton of the knot
exterior, while
\Cref{thm:finite-resolution-finiteness-filtered-pairs} controls geometric
realizations at fixed parameters
\[
  (\Lambda,\Delta,\tau,\varepsilon).
\]
The former is classical topological finiteness of a canonical decomposition;
the latter is finite-resolution geometric finiteness of filtered knot--surface
pairs.
\end{remark}

\subsection{JSJ and characteristic visibility levels}

The canonical torus system and the fuller characteristic frontier have
separate geometric birth levels.  For a surface class $[F]$, let
$\mathfrak S_{[F]}$ denote the filtered surface type consisting of
representatives of that class.

\begin{definition}[JSJ and characteristic visibility]
For fixed $(\Delta,\tau)$, define
\[
  \beta_{\mathrm{JSJ}}(K;\Delta,\tau)
  =
  \max_{[T]\in\mathcal T_{\mathrm{JSJ}}(K)}
  \beta_{\mathfrak S_{[T]}}(K;\Delta,\tau)
\]
and
\[
  \beta_{\mathrm{char}}(K;\Delta,\tau)
  =
  \max_{[F]\in\mathcal E_{\mathrm{char}}(K)}
  \beta_{\mathfrak S_{[F]}}(K;\Delta,\tau).
\]
For an empty indexing set the corresponding maximum is defined to be
$\Rop(K)$.
\end{definition}

\begin{proposition}[Finite characteristic visibility under bounded realization]
\label{prop:finite-jsj-visibility-under-bounded-realization}
Assume that every class in $\mathcal E_{\mathrm{char}}(K)$ has a representative
in the exterior of some unit-thickness curve in $Y^{=}_\Lambda(K)$, with area at
most $\Delta$ and relative thickness at least $\tau$.  Then
\[
  \beta_{\mathrm{JSJ}}(K;\Delta,\tau)
  \leq
  \beta_{\mathrm{char}}(K;\Delta,\tau)
  \leq \Lambda.
\]
Consequently the rooted JSJ tree is visible at a finite level once its tori,
piece labels, and cutting incidence are included in the finite code.
\end{proposition}

\begin{proof}
Every class in either finite system contributes a nonempty filtered pair
space by hypothesis.  Hence each birth level is at most $\Lambda$, and the
first inequality follows from
$\mathcal T_{\mathrm{JSJ}}(K)\subset\mathcal E_{\mathrm{char}}(K)$.
The labelled tree is then recovered from the finite cutting and incidence
data.
\end{proof}

\begin{remark}[Geometric regimes]
For a hyperbolic knot both systems are empty, so the visibility level carries
no additional decomposition data.  For a torus knot the JSJ torus system is
empty because the whole exterior is Seifert fibered.  For cable and satellite
knots canonical tori give nontrivial JSJ levels.  For composite knots the
prime annuli form a separate low-area layer; tubing them produces the
swallow--follow tori adjacent to the composing-space pieces in the JSJ tree.
\end{remark}

\subsection{The finite essential surface principle as a finite-resolution principle}

We now formulate the finite essential surface principle in a deliberately
finite-resolution form.  It should not be read as a theorem asserting that all
essential surface isotopy classes in a knot exterior form a finite set.
Indeed, Kakimizu complexes may be infinite.

\begin{definition}[Finite characteristic data at scale \(\varepsilon\)]
Let \(K\) be a knot type and fix \((\Lambda,\Delta,\tau,\varepsilon)\).  A
finite characteristic data set at this scale consists of:
\begin{enumerate}[label=(\roman*),leftmargin=2em]
\item encoded \(\varepsilon\)-types of finitely many pairs
\((\gamma,F_i)\in\calZ_{\Lambda,\Delta,\tau}^{\mathrm{ess}}(K)\);
\item the peripheral meridian--longitude data on \(\partial E(\gamma)\),
recorded at the same resolution;
\item the rooted labelled JSJ tree and the finite decomposition data obtained
by cutting the encoded exterior along the visible surfaces \(F_i\), including
piece labels and peripheral attachment data;
\item labels specifying which visible surfaces are JSJ tori, characteristic-frontier
annuli or tori, prime-decomposition annuli, Conway surfaces, taut Seifert
surfaces, boundary-slope detectors, or other prescribed characteristic
surfaces.
\end{enumerate}
Two knot types have the same characteristic data at this scale if these finite
encoded objects agree up to the allowed finite-resolution moves and peripheral
identifications.
\end{definition}

\begin{remark}[Programmatic statement: Finite essential surface principle, finite-resolution form]
For every knot type \(K\), there should exist parameters
\((\Lambda,\Delta,\tau,\varepsilon)\) and a finite characteristic data set at
that scale which determines \(K\) among all knot types.  The geometric content
is the existence of such a detecting data set inside a bounded
length--area--thickness window; the finite-resolution finiteness theorem then
says that, once this window is fixed, only finitely many candidate codes have
to be searched.
\end{remark}

\begin{remark}[Difference from classical topological finiteness]
Classical Haken and JSJ theory supplies finite topological decompositions of a
fixed exterior.  The present principle is different in two ways.  First, it is
posed before the exterior is known: the data are carried by bounded-ropelength
representatives and must determine the knot type together with the peripheral
meridian.  Second, the data are constrained by explicit geometric budgets
\((\Lambda,\Delta,\tau)\) and by a resolution \(\varepsilon\).  Thus the
principle is not the tautological assertion that a triangulated exterior has a
finite topological description; it asks for a bounded-geometric detecting
window in which characteristic essential surfaces are visible.
\end{remark}

\subsection{Normal surface implementation viewpoint}

\begin{remark}[Classical algorithmic baseline]
Given a triangulation, normal-surface algorithms compute the prime and JSJ
decompositions and detect the relevant characteristic pieces
\cite{JacoOertel,JacoTollefson}.  This is an external topological algorithm.
The purpose of the present section is different: to understand how the same
finite skeleton becomes visible inside a bounded geometric window and how it
can be encoded at finite resolution.
\end{remark}

Let \(\calT\) be a bounded-geometry triangulation of \(E(K)\), chosen as part of
the finite-resolution model.  For instance, \(\calT\) may be obtained after
putting the exterior into a controlled cubical grid and then applying a fixed
subdivision rule.  The constants below depend on this choice; no canonical
triangulation is being asserted.  Normal surface theory gives a combinatorial
implementation of the finite-resolution framework: normal coordinates supply
finite codes, fundamental normal surfaces supply basic generators, and Haken
sums describe higher-complexity layers.  We use this viewpoint only as an
implementation guide, not as a claim that all essential surface isotopy classes
are represented by a fixed finite list.

\begin{remark}[Programmatic statement: Normal-surface implementation viewpoint]
Fundamental normal surfaces are combinatorial atoms for Haken sums.  In the
present geometric filtration, large Haken-sum coefficients should be interpreted
as large numbers of parallel normal sheets.  Area and thickness bounds then cut
off the high-coefficient part of the normal surface semigroup at finite
resolution.
\end{remark}

\begin{remark}[Fundamental surfaces and thickness]
A fundamental normal surface is indecomposable with respect to Haken sum.  This
is a combinatorial property, not a metric thickness condition.  Nevertheless,
if a surface contains a summand \(nG\), then the normal pieces of \(G\) appear
as \(n\) parallel families of sheets.  To realize the sum with
\(\Thi\geq\tau\), these sheets require transverse room proportional to
\(n\tau\), up to constants depending on the bounded geometry of \(\calT\);
see \Cref{fig:parallel-sheets}.
\end{remark}

\begin{figure}[t]
\centering
\begin{tikzpicture}[x=1cm,y=1cm,>=Latex,
  note/.style={font=\scriptsize, align=center}]
  \draw[thick] (0.6,0.3) -- (6.4,0.3) -- (3.1,4.3) -- cycle;
  \node[note] at (6.15,3.55)
    {tetrahedron of the\\bounded-geometry\\triangulation $\calT$};
  \foreach \s in {0.30,0.45,0.60,0.75}{
    \draw[very thick]
      ({0.6+\s*(3.1-0.6)},{0.3+\s*(4.3-0.3)}) --
      ({0.6+\s*(6.4-0.6)},0.3);
  }
  \node[note] at (1.15,3.55) {$n$ parallel copies\\of a normal disk};
  \draw[->, thin] (1.75,3.15) -- (3.05,2.45);
  \draw[<->, thin] (2.63,1.42) -- (3.05,1.72);
  \node[note, fill=white, inner sep=1pt] at (3.35,1.35) {$\geq c_0\tau$};
  \draw[<->, thin] (0.78,0.05) -- (3.02,3.72);
  \node[note, fill=white, inner sep=1pt, rotate=58.5] at (1.55,1.55)
    {width $\leq d_{\calT}$};
  \draw[->, thin] (5.50,1.42) -- (4.30,1.58);
  \node[note, anchor=west] at (5.55,1.40) {$\Area\geq a_0$ each};
  \node[note, draw, rounded corners=2pt, inner sep=4pt] at (9.9,1.0)
    {$\displaystyle n_i\ \leq\
      \min\Bigl\{\frac{\Delta}{a_0},\
      \frac{d_{\calT}}{c_0\tau}+1\Bigr\}$};
  \node[note, align=center] at (9.9,2.35)
    {area budget counts the disks;\\
     thickness budget spaces the sheets};
\end{tikzpicture}
\caption{The thickness interpretation of Haken sums
(\Cref{prop:conditional-normal-coefficient-bound}), in a cross-section of one
normal block.  A summand \(nG\) of a Haken sum appears as \(n\) parallel
copies of the same normal disk inside a tetrahedron of the fixed
bounded-geometry triangulation.  Realizing the sum with \(\Thi\geq\tau\)
forces consecutive sheets to be at transverse distance at least
\(c_0\tau\), while the block only offers width \(d_{\calT}\); and each disk
costs area at least \(a_0\) out of the budget \(\Delta\).  Together these cut
off the high-coefficient part of the normal-surface semigroup at finite
resolution.}
\label{fig:parallel-sheets}
\end{figure}
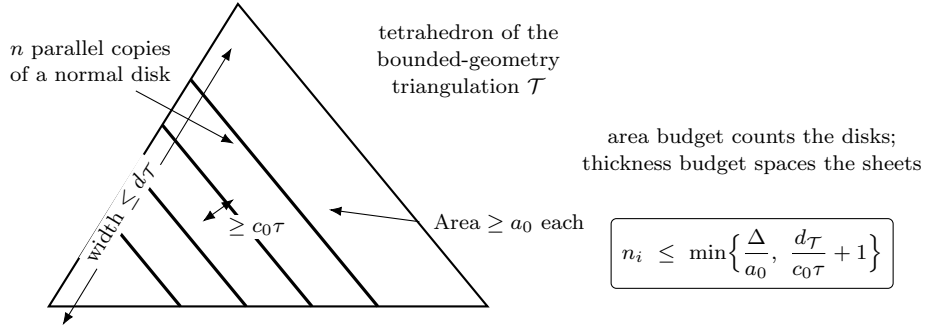

\begin{proposition}[Conditional normal coefficient bound]
\label{prop:conditional-normal-coefficient-bound}
Fix a bounded-geometry triangulation \(\calT\) of a knot exterior.  Suppose
that each normal disk type has area at least \(a_0>0\), each normal block has
transverse width at most \(d_{\calT}\), and \(\tau\)-separated parallel sheets
require transverse spacing at least \(c_0\tau\).  If a normal surface
\[
  H=\sum_i n_i F_i
\]
is realized with
\[
  \Area(H)\leq\Delta,
  \qquad
  \Thi(H)\geq\tau,
\]
then every coefficient contributing a repeated family of parallel sheets
satisfies the coarse estimate
\[
  n_i
  \leq
  \min\left\{
  \frac{\Delta}{a_0},
  \frac{d_{\calT}}{c_0\tau}+1
  \right\}.
\]
The constants depend on the chosen bounded-geometry realization of
\(\calT\).
\end{proposition}

\begin{proof}
The area estimate follows because each repeated normal disk contributes at
least \(a_0\) area.  The thickness estimate follows because \(n_i\) parallel
sheets require total transverse width at least comparable to \(n_i c_0\tau\),
while the available width of the relevant block is at most \(d_{\calT}\).
Combining these two estimates gives the stated bound, after absorbing endpoint
and smoothing conventions into the additive constant.
\end{proof}

\begin{problem}[Sharp normal coefficient bounds]
\label{prob:normal-coefficient-bounds}
For a fixed triangulation \(\calT\), determine explicit bounds on the normal
coordinates of a \(\tau\)-thick essential surface \(F\) with
\(\Area(F)\leq\Delta\).  Relate these coefficient bounds to
\[
  \chi(F),\qquad |\partial F|,
  \qquad \text{boundary slopes},
  \qquad \text{and admissible merge scales}.
\]
\end{problem}

\subsection{Behaviour under standard operations}
\label{subsec:standard-operations}

The framework is compatible with the standard structural operations on knots.
We record the connected-sum, symmetry, and link versions, since each provides
both a test of the theory and a source of low-complexity characteristic
surfaces.

\paragraph{Connected sum.}
If \(K=K_1\# K_2\), a connected-sum sphere \(S\) meets \(K\) in two points, so
\[
  A=S\cap E(K)
\]
is a properly embedded annulus with \(\partial A\) two meridians on
\(\partial E(K)\).

\begin{proposition}[Decomposing annulus of a connected sum]
\label{prop:connected-sum-annulus}
If \(K=K_1\# K_2\) with both factors nontrivial, then \(A=S\cap E(K)\) is an
essential annulus of meridional boundary slope and belongs to the separate
prime-decomposition layer.  The two associated swallow--follow tori are
incident to the composing-space vertex of the rooted JSJ tree.  Any
bounded-geometry
realization \((\gamma_0,A_0)\) gives
\[
  \beta_{\mathfrak S_A}(K;\Delta,\tau)\leq\Len(\gamma_0)
  \qquad\text{for }\Delta\geq\Area(A_0),\ \tau\leq\Thi(A_0),
\]
so the connected-sum structure is visible as a low-area annular layer.
\end{proposition}

\begin{proof}
For nontrivial factors the sphere \(S\) is essential in \((S^3,K)\), so its
restriction \(A\) is incompressible and not boundary-parallel, hence essential;
its two boundary curves are meridians.  This is the exterior form of prime
decomposition.  Tubing \(A\) along either complementary annulus in
\(\partial E(K)\setminus\partial A\) gives the associated swallow--follow
tori, which meet the composing-space vertex of the JSJ tree.  The visibility
bound is the definition of \(\beta\) applied to the realization
\((\gamma_0,A_0)\).
\end{proof}

\begin{remark}[Connected sum and the twisting diameter]
\label{rem:connected-sum-twisting}
The meridional slope of the decomposing annulus lies at intersection distance
one from every integral surface slope, so a connected sum always contributes the
meridian to the visible slope set.  Because ropelength is subadditive up to
controlled error under connected sum, the annular layer typically appears at a
ropelength level comparable to \(\Rop(K_1)+\Rop(K_2)\); making this precise is an
instance of the trade-off frontier for the class \([A]\).
\end{remark}

\paragraph{Symmetry-group equivariance.}
Let \(\operatorname{Sym}(K)\) be the symmetry group of \(K\), that is, the
mapping class group of the pair \((S^3,K)\), which acts on the exterior
\(E(K)\) by restriction.

\begin{proposition}[Equivariance of the visible structures]
\label{prop:symmetry-equivariance}
The group \(\operatorname{Sym}(K)\) acts on the set of isotopy classes of
essential surfaces in \(E(K)\), on the boundary-slope set \(B(K)\) preserving the
intersection distance \(\Delta_T\), on the canonical systems
\(\mathcal T_{\mathrm{JSJ}}(K)\) and
\(\mathcal E_{\mathrm{char}}(K)\), on the rooted labelled JSJ tree
\(\Gamma_{\mathrm{JSJ}}(K)\), and on the Kakimizu complex
\(\mathrm{MS}(K)\).  Consequently \(\operatorname{Rop}(K)\) and
\(\operatorname{Tw}_{\partial}(K)\) are \(\operatorname{Sym}(K)\)-invariant, and
the visibility spectrum is \(\operatorname{Sym}(K)\)-equivariant: for every
\(g\in\operatorname{Sym}(K)\) and surface type \(\mathfrak S\),
\[
  \beta_{g\mathfrak S}(K;\Delta,\tau)=\beta_{\mathfrak S}(K;\Delta,\tau)
\]
whenever \(g\) is realized by an isometry of \(\R^3\); in general
\(g\) permutes the surface types and hence the visibility spectrum as a labelled
structure.
\end{proposition}

\begin{proof}
A symmetry acts by a self-homeomorphism of \((S^3,K)\), hence permutes isotopy
classes of essential surfaces and preserves all topological structures built from
them, including \(B(K)\), \(\Delta_T\),
\(\mathcal T_{\mathrm{JSJ}}(K)\),
\(\mathcal E_{\mathrm{char}}(K)\), the rooted labelled JSJ tree, and
\(\mathrm{MS}(K)\); the numerical knot invariants \(\operatorname{Rop}\) and
\(\operatorname{Tw}_\partial\) are therefore fixed.  If \(g\) is realized by an
ambient Euclidean isometry \(\Phi\), then \(\Phi\) preserves length, area, and reach, so it
carries a pair realizing the infimum defining \(\beta_{\mathfrak S}\) to one
realizing \(\beta_{g\mathfrak S}\) with the same parameters, giving equality.
\end{proof}

\begin{remark}[Recognition is up to symmetry]
\label{rem:recognition-up-to-symmetry}
The finite characteristic data of \Cref{sec:finite-recognition} carry the
\(\operatorname{Sym}(K)\)-action, so a filtered-pair recognition statement
determines \(K\) up to the symmetry group.  This is the correct target: no
finite surface datum can distinguish a knot from its symmetric images, and
\(\operatorname{Sym}(K)\) is exactly the ambiguity that a complete invariant may
retain.
\end{remark}

\paragraph{Extension to links.}

\begin{remark}[Multi-component links]
\label{rem:links}
Everything in the paper extends to a link \(L\subset S^3\) with the expected
changes.  Ropelength minimizers exist for links \cite{CantarellaKusnerSullivan},
so \(Y^{=}_\Lambda(L)\) and its ideal stratum are defined verbatim; the exterior has
one peripheral torus per component, so boundary slopes, twisting diameters, and
peripheral twist densities are recorded componentwise, subject to the linking
constraints among components.  The Kakimizu complex is in fact defined for links
in its original form \cite{Kakimizu}.  The metric results transfer without
change: \Cref{prop:geometric-topological-boundedness} is a local
bounded-geometry statement insensitive to the number of components, and
\Cref{lem:compactness-thick-surfaces} applies since its anchoring hypothesis
is again automatic for essential surfaces in link exteriors, by the same
argument as \Cref{lem:anchoring}; thus the smooth finiteness,
faithfulness, and attainment theorems hold for links, with the surface budget
counting total area and the peripheral data indexed by the components.
\end{remark}


\section{Filtered essential-surface rigidity}
\label{sec:filtered-rigidity}

The preceding sections produce two kinds of finite data: finite lists of
bounded-geometry isotopy classes and faithful finite codes for individual
pairs.  We now assemble the fixed-exterior classes into finite simplicial
objects.  This makes precise a third role of geometric bounds: they provide a
finite exhaustion of the topological essential-surface complex and hence a
finite-stage setting for symmetry and rigidity questions.  The terminology
and the image--kernel--reconstruction viewpoint are developed to the extent
needed here in \Cref{def:filtered-rigidity-problems} and the surrounding
remarks; the rigidity problems themselves are stated as open questions.

Throughout this section, fix a unit-thickness representative \(\gamma\) and
write
\[
  E=E(\gamma)=S^3\setminus\operatorname{int}N_{\rho_0}(\gamma).
\]
The fixed tube determines a distinguished meridian slope \(\mu\) on
\(\partial E\).  All area, reach, collar, and clearance conditions are measured
in this fixed geometric exterior.

\begin{convention}[Compact ambient metric for this section]
\label{conv:compact-metric}
The convention of \Cref{not:ambient-metric} measures all quantities in the
Euclidean metric of \(\R^3=S^3\setminus\{\infty\}\).  The exterior \(E\) is
a compact subset of \(S^3\) containing \(\infty\), hence unbounded in that
metric.  For the fixed-window finiteness theorems of
\Cref{sec:bounded-complexity} this causes no difficulty, because anchoring
and the diameter bound confine all visible surfaces at one level to a fixed
compact subset of \(\R^3\).  In the present section, however, the objects of
interest are \emph{self-maps} of \(E\) --- isometries and meridian-preserving
diffeomorphisms --- for which the Euclidean convention is unnatural: it
gives the point \(\infty\) an artificial special role and makes ``isometry
of \(E\)'' a statement about an unbounded set.  For this section only, we
therefore fix once and for all a smooth compact Riemannian metric \(g\) on
\(S^3\) (for definiteness, the round metric of a fixed stereographic
identification \(S^3=\R^3\cup\{\infty\}\)) and define the visible complexes
\(\ES_{\Delta,\tau}(E)\) of this section using \(g\)-area, \(g\)-reach
(single-valuedness of the \(g\)-nearest-point projection at distance
\(\tau\)), and the \(g\)-collar and clearance conditions.  The compactness
and finiteness arguments of
\Cref{sec:bounded-complexity} carry over to this setting with only
simplifications and notational changes: the ambient \((S^3,g)\) is compact,
so Blaschke selection applies directly and no anchoring lemma or diameter
estimate is needed, while the chart, collar, and limit-reach arguments are
local and use only the fixed bounded geometry of \(g\), with Euclidean rigid
motions replaced by \(g\)-normal coordinates.  Consequently
\Cref{thm:visible-complex-finiteness,thm:visible-complex-exhaustion} below
hold for the \(g\)-visible complexes, which is how they are to be read.  The
direct limit \(\ES(E)\) is metric-free, but the individual levels
\(\ES_{\Delta,\tau}(E)\) of this section depend on \(g\), and no
identification with the Euclidean-window levels of the earlier sections is
asserted or used.  With this convention \(\Isom^{\pm}(E,\mu)\) is the group
of meridian-slope-preserving isometries of a compact Riemannian manifold
with boundary, the point \(\infty\) is an ordinary point of \(E\), and the
distortion quantities \(A_2(h)\), \(\operatorname{Lip}(h)\),
\(\operatorname{Lip}(h^{-1})\), \(\|\nabla Dh\|_{L^\infty}\) of
\Cref{prop:controlled-visible-functoriality} are measured with respect to
\(g\).
\end{convention}

\begin{remark}[Filtered complexes, graphics, and spines]
\label{rem:filtered-graphics-spines}
The visible complex records which disjoint essential systems occur inside a
geometric window and how exterior symmetries act on their isotopy classes.  It
does not record the Cerf discriminant of a pair of sweepouts and therefore is
not a Rubinstein--Scharlemann graphic.  Nor is it, by definition, a complex of
special spines and local spine moves.  These structures are complementary:
a fine pair code may be dualized to a decorated spine, while a controlled
two-parameter family may yield a graphic.  No equivalence among these three
objects is used in the image--kernel--reconstruction results below.
\end{remark}

\subsection{Finite visible essential-surface complexes}

Complexes whose vertices are isotopy classes of essential or incompressible
surfaces have been studied in several forms
\cite{SchultensSurfaceComplex,ZhangGuo,CharitosPapadoperakisTsapogas}.  The
underlying complex in the next definition is therefore not claimed as new;
the additional structure used here is its area--relative-thickness filtration.

\begin{definition}[Essential-surface complex]
\label{def:essential-surface-complex-geometric}
The essential-surface complex \(\ES(E)\) is the simplicial complex whose
vertices are ambient-isotopy classes \([F]\) of connected essential surfaces
in \(E\).  A finite set of distinct vertices spans a simplex if it admits a
simultaneously disjoint representative system.
\end{definition}

The simultaneity condition is part of the definition; no flag-complex
assertion is needed.  It is also compatible with the finite-system convention
used throughout this paper.

\begin{definition}[Finite visible essential-surface complex]
\label{def:finite-visible-complex}
For \(\Delta,\tau>0\), the \((\Delta,\tau)\)-visible essential-surface complex
\[
  \ES_{\Delta,\tau}(E)
  \subset \ES(E)
\]
consists of the simplices
\(\sigma=\{[F_0],\ldots,[F_k]\}\) for which there are simultaneously disjoint
representatives such that the union
\[
  F=F_0\sqcup\cdots\sqcup F_k
\]
satisfies
\[
  \Area(F)\leq\Delta,
  \qquad
  \Thi(F)\geq\tau.
\]
Here area is total area, and thickness is the relative thickness of the union
in the sense of \Cref{rem:disconnected-thickness}.  In particular, passing to
a subsystem preserves the bounds, so this is a simplicial subcomplex.

A \emph{decorated visible complex} is obtained by attaching any prescribed
combination of the following labels to the vertices and simplices: surface
homeomorphism type, boundary slope relative to \(\mu\), relative homology
class, taut or Kakimizu status, JSJ support, cutting data, and incidence with
the characteristic frontier.  We write
\(\ES^{\mathfrak D}_{\Delta,\tau}(E,\mu)\) when the decoration package is
\(\mathfrak D\).
\end{definition}

\begin{theorem}[Finiteness of every visible complex]
\label{thm:visible-complex-finiteness}
For every fixed geometric exterior \(E=E(\gamma)\) and every
\(\Delta,\tau>0\), the complex \(\ES_{\Delta,\tau}(E)\) is finite.  The same
holds for every decorated visible complex with finitely recorded labels.
Moreover, if \(\Delta\leq\Delta'\) and \(\tau\geq\tau'\), then there is a
natural simplicial inclusion
\[
  \ES_{\Delta,\tau}(E)
  \into
  \ES_{\Delta',\tau'}(E).
\]
\end{theorem}

\begin{proof}
Every vertex of \(\ES_{\Delta,\tau}(E)\) is represented by a connected
essential surface with area at most \(\Delta\) and relative thickness at least
\(\tau\), both measured in the compact metric of
\Cref{conv:compact-metric}.  By \Cref{thm:smooth-finiteness-fixed-exterior}
--- whose compactness proof applies in this compact ambient with only
simplifications, as noted in \Cref{conv:compact-metric} --- only finitely many
ambient-isotopy classes of such surfaces occur.  A simplicial complex with a
finite vertex set has only finitely many simplices.  Adding finitely recorded
labels does not change this conclusion.  The monotonicity follows directly
from weakening the area bound and, via
\Cref{lem:thickness-scale-monotonicity}, the thickness bound.
\end{proof}

\begin{theorem}[Finite exhaustion of the full essential-surface complex]
\label{thm:visible-complex-exhaustion}
The visible complexes form a directed exhaustion of \(\ES(E)\):
\[
  \ES(E)
  =
  \bigcup_{\Delta>0,\ \tau>0}
  \ES_{\Delta,\tau}(E)
  =
  \varinjlim_{\Delta\to\infty,\ \tau\downarrow0}
  \ES_{\Delta,\tau}(E).
\]
Equivalently, every finite simplex of \(\ES(E)\) occurs in some bounded
area--thickness window.  A countable exhaustion is obtained by restricting to
positive rational values of \(\Delta\) and \(\tau\).
\end{theorem}

\begin{proof}
Let \(\sigma=\{[F_0],\ldots,[F_k]\}\) be a simplex of \(\ES(E)\).  Choose a
smooth simultaneously disjoint proper representative system and set
\(F=F_0\sqcup\cdots\sqcup F_k\).  Its total area is finite.  After an
arbitrarily small proper isotopy near the boundary, the system satisfies the
uniform graphical collar and peripheral-clearance conventions of
\Cref{def:relative-surface-thickness}.  Compactness, embeddedness, and the
positive separation of the finitely many components then give
\(\Thi(F)>0\).  Thus, by the scale monotonicity of
\Cref{lem:thickness-scale-monotonicity},
\(\sigma\in\ES_{\Delta,\tau}(E)\) for every
\(\Delta\geq\Area(F)\) and every
\(0<\tau\leq\Thi(F)\).  The reverse inclusion is immediate from the
definition.
\end{proof}

\begin{remark}[What the exhaustion does and does not say]
The theorem does not assert that one bounded level contains every essential
surface, nor that the full complex is finite.  It says that every finite
configuration of simultaneously disjoint essential surfaces is detected at a
finite geometric stage.  Infinite Kakimizu complexes and Haken-sum families
therefore reappear only through passage to larger area or smaller thickness.
\Cref{fig:visible-exhaustion} illustrates the first stages of the exhaustion.
\end{remark}

\begin{figure}[t]
\centering
\begin{tikzpicture}[x=1cm,y=1cm,>=Latex,
  panel/.style={draw, rounded corners=2pt},
  note/.style={font=\scriptsize, align=center},
  vtx/.style={circle, fill, inner sep=1.3pt}]
  \draw[panel] (0,0) rectangle (3.4,3.1);
  \node[vtx] (a1) at (1.0,1.5) {};
  \node[vtx] (a2) at (2.4,1.5) {};
  \draw (a1) -- (a2);
  \node[note, below] at (1.0,1.38) {$[F_1]$};
  \node[note, below] at (2.4,1.38) {$[F_2]$};
  \node[note] at (1.7,-0.35) {$\ES_{\Delta_1,\tau_1}(E)$};
  \draw[panel] (4.2,0) rectangle (7.6,3.1);
  \fill[black!10] (5.0,1.5) -- (6.4,1.5) -- (5.7,2.55) -- cycle;
  \node[vtx] (b1) at (5.0,1.5) {};
  \node[vtx] (b2) at (6.4,1.5) {};
  \node[vtx] (b3) at (5.7,2.55) {};
  \node[vtx] (b4) at (7.0,2.45) {};
  \draw (b1) -- (b2) -- (b3) -- (b1);
  \node[note, below] at (5.0,1.38) {$[F_1]$};
  \node[note, below] at (6.4,1.38) {$[F_2]$};
  \node[note, above] at (5.7,2.66) {$[F_3]$};
  \node[note, above] at (7.0,2.56) {$[F_4]$};
  \node[note] at (5.9,-0.35) {$\ES_{\Delta_2,\tau_2}(E)$};
  \draw[panel] (8.4,0) rectangle (11.8,3.1);
  \fill[black!10] (9.2,1.5) -- (10.6,1.5) -- (9.9,2.55) -- cycle;
  \node[vtx] (c1) at (9.2,1.5) {};
  \node[vtx] (c2) at (10.6,1.5) {};
  \node[vtx] (c3) at (9.9,2.55) {};
  \node[vtx] (c4) at (11.2,2.45) {};
  \node[vtx] (c5) at (11.1,0.85) {};
  \draw (c1) -- (c2) -- (c3) -- (c1);
  \draw (c3) -- (c4);
  \draw (c2) -- (c5) -- (c4);
  \node[note, below] at (9.2,1.38) {$[F_1]$};
  \node[note, below] at (10.45,1.38) {$[F_2]$};
  \node[note, above] at (9.9,2.66) {$[F_3]$};
  \node[note, above] at (11.2,2.56) {$[F_4]$};
  \node[note, below] at (11.1,0.73) {$[F_5]$};
  \node[note] at (10.1,-0.35) {$\ES_{\Delta_3,\tau_3}(E)$};
  \node[note] at (3.8,1.55) {$\into$};
  \node[note] at (8.0,1.55) {$\into$};
  \node[note] at (12.65,1.55) {$\into\ \cdots\ \ES(E)$};
  \draw[->, thin] (0.2,-0.95) -- (11.6,-0.95);
  \node[note, below] at (5.9,-1.05)
    {$\Delta$ increases, $\tau$ decreases; each stage is finite};
\end{tikzpicture}
\caption{The finite exhaustion of the essential-surface complex of a fixed
geometric exterior \(E=E(\gamma)\)
(\Cref{thm:visible-complex-finiteness,thm:visible-complex-exhaustion}).
Vertices are ambient-isotopy classes of connected essential surfaces; a
simplex, such as the shaded triangle \(\{[F_1],[F_2],[F_3]\}\), records a
simultaneously disjoint system whose union satisfies the area and thickness
bounds of the current window.  Enlarging the window
(\(\Delta_1\leq\Delta_2\leq\Delta_3\), \(\tau_1\geq\tau_2\geq\tau_3\)) makes
new vertices and simplices visible, each finite stage includes into the next,
and the directed union over all windows is the full complex \(\ES(E)\), which
may be infinite.  Exterior symmetries act on this filtered system levelwise
for isometries and with controlled reindexing in general
(\Cref{prop:levelwise-isometric-action,prop:controlled-visible-functoriality}).}
\label{fig:visible-exhaustion}
\end{figure}
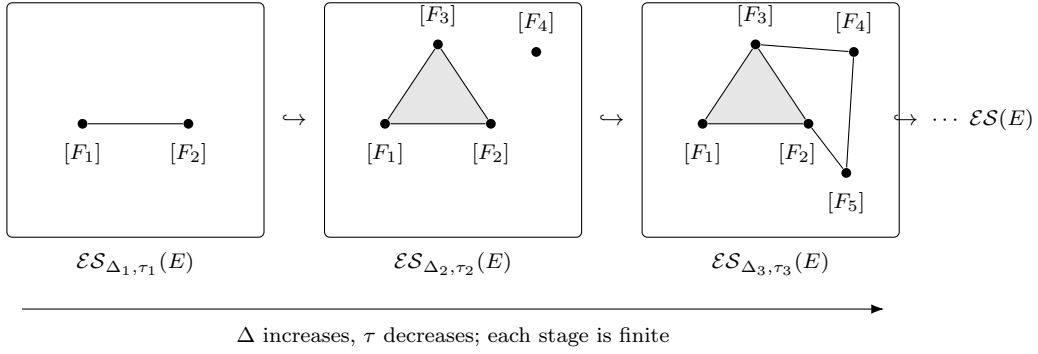

\subsection{Exact and controlled equivariance}

The finite complexes depend on the chosen metric representative of the
exterior.  Consequently, a general mapping class need not preserve one
\((\Delta,\tau)\)-level.  Isometries do preserve levels, while smooth
self-homeomorphisms act after a controlled change of parameters.

\begin{proposition}[Levelwise isometric action]
\label{prop:levelwise-isometric-action}
Let \(\Isom^{\pm}(E,\mu)\) denote the group of isometries of the fixed
geometric exterior which preserve the meridian slope.  For every
\(\Delta,\tau>0\), there is a natural homomorphism
\[
  \Isom^{\pm}(E,\mu)
  \longrightarrow
  \Aut\bigl(\ES_{\Delta,\tau}(E)\bigr),
  \qquad
  [F]\longmapsto[h(F)].
\]
The same statement holds for every decoration package preserved by the
isometries.
\end{proposition}

\begin{proof}
An isometry preserves area, reach, the boundary collar conditions, peripheral
clearance, essentiality, and simultaneous disjointness.  It therefore sends
every visible simplex to a visible simplex at the same level.
\end{proof}

\begin{proposition}[Controlled functoriality]
\label{prop:controlled-visible-functoriality}
Let \(h:E\to E\) be a meridian-preserving \(C^{1,1}\) diffeomorphism, extended
to a \(C^{1,1}\) diffeomorphism of a neighbourhood of \(E\).  All metric
quantities are taken with respect to the fixed compact metric \(g\) of
\Cref{conv:compact-metric}: \(\operatorname{Lip}(h)\) and
\(\operatorname{Lip}(h^{-1})\) are \(g\)-Lipschitz constants, the
second-order distortion of \(h\) is measured by
\(\|\nabla Dh\|_{L^\infty}\), the essential supremum of the covariant
derivative of \(Dh\) in \(g\), and
\[
  A_2(h)=\sup_{x\in E}\|\wedge^2Dh_x\|_g .
\]
There is a positive function
\(\vartheta_{h,g}:(0,\infty)\to(0,\infty)\), depending
only on the \(C^{1,1}\) distortion data
\(\operatorname{Lip}(h)\), \(\operatorname{Lip}(h^{-1})\),
\(\|\nabla Dh\|_{L^\infty}\), on the metric \(g\) --- through its curvature
bounds and a lower bound for its injectivity radius on a neighbourhood of
\(E\) --- and on the boundary
collar conventions, such that \(h\) induces a simplicial map
\[
  h_*:
  \ES_{\Delta,\tau}(E)
  \longrightarrow
  \ES_{A_2(h)\Delta,\,\vartheta_{h,g}(\tau)}(E).
\]
No closed formula for \(\vartheta_{h,g}\) is asserted: the explicit
Euclidean reach-distortion expression of \Cref{lem:reach-distortion} is
derived from the flat tangent-cone inequality and does not transfer
verbatim to a curved metric, where comparison constants coming from the
curvature and injectivity radius of \(g\) enter.  Applying the same
statement to
\(h^{-1}\) shows that the induced maps are mutually inverse after cofinal
changes of the parameters.
\end{proposition}

\begin{proof}
The area formula gives
\(\Area(h(F))\leq A_2(h)\Area(F)\).  For the reach part, work in
\(g\)-normal coordinates on balls of a fixed radius \(r_g>0\) determined by
the curvature bounds and injectivity radius of \(g\) on a neighbourhood of
\(E\).  On such a ball the metric coefficients are \(C^{1,1}\)-controlled
by the same curvature data, so the tangent-cone argument of
\Cref{lem:reach-distortion} applies with constants depending on \(g\), and
yields a positive lower bound
\(\vartheta_{h,g}(\tau)\) for the \(g\)-reach of the image surface,
depending only on the stated distortion data of \(h\), \(h^{-1}\), and
\(g\); the scale is further capped by \(r_g\), which is harmless since only
positivity and the stated dependence are claimed.  Since \(h\) and \(h^{-1}\) are bi-Lipschitz and
\(C^{1,1}\) on a compact neighbourhood of the boundary, graphical boundary
charts, collar widths, and peripheral clearances remain uniformly positive,
with constants depending only on the stated distortion data.  Essentiality,
simultaneous disjointness, and the meridian marking are topologically
preserved.  The assertion for the inverse gives cofinal invertibility.
\end{proof}

\begin{definition}[Filtered geometric essential-surface system]
\label{def:filtered-geometric-ES-system}
Write
\[
  \mathbf{ES}^{\mathrm{geom}}(E,\mu)
  =
  \bigl\{\ES_{\Delta,\tau}(E,\mu)\bigr\}_{\Delta>0,\tau>0}
\]
for the directed system under
\((\Delta,\tau)\preceq(\Delta',\tau')\) when
\(\Delta\leq\Delta'\) and \(\tau\geq\tau'\).  A \emph{cofinal automorphism}
of this system is an automorphism represented by compatible simplicial maps
after cofinal changes of the two parameters.  Its group is denoted
\[
  \Aut_{\mathrm{cof}}
  \bigl(\mathbf{ES}^{\mathrm{geom}}(E,\mu)\bigr).
\]
\end{definition}

\begin{corollary}[Controlled mapping-class action]
\label{cor:controlled-mapping-class-action}
Choosing \(C^{1,1}\) representatives gives a natural homomorphism
\[
  \Phi^{\mathrm{geom}}:
  \Mod^{\pm}(E,\mu)
  \longrightarrow
  \Aut_{\mathrm{cof}}
  \bigl(\mathbf{ES}^{\mathrm{geom}}(E,\mu)\bigr).
\]
On the direct limit, this is the usual action of the meridian-preserving
mapping class group on \(\ES(E)\).
\end{corollary}

\begin{proof}
By \Cref{prop:controlled-visible-functoriality}, a representative and its
inverse give mutually inverse cofinal maps.  Isotopic representatives induce
the same permutation of ambient-isotopy classes of essential surfaces, so the
map depends only on the mapping class.  Composition is respected because the
parameter distortions compose cofinally.  The final assertion follows from
\Cref{thm:visible-complex-exhaustion}.
\end{proof}

\subsection{Image, kernel, reconstruction, and economy}

The filtered action isolates four logically distinct questions.

\begin{definition}[Filtered rigidity problems]
\label{def:filtered-rigidity-problems}
For a chosen decoration package \(\mathfrak D\), consider the action on
\(\mathbf{ES}^{\mathrm{geom},\mathfrak D}(E,\mu)\).
\begin{enumerate}[label=(\roman*),leftmargin=2em]
\item The \emph{geometric image problem} asks which cofinal automorphisms are
induced by meridian-preserving mapping classes.
\item The \emph{invisible-kernel problem} asks which mapping classes act
trivially on every finite stage, equivalently on the filtered system.
\item The \emph{intrinsic reconstruction problem} asks whether the
characteristic frontier, its pieces, the meridian, and the gluing data can be
recovered from the filtered combinatorics and labels.
\item The \emph{economy problem} asks for the weakest natural decoration
package for which the preceding tasks can be solved.
\end{enumerate}
\end{definition}

\begin{remark}[Faithful pair codes are not automorphism rigidity]
\label{rem:faithfulness-versus-rigidity}
\Cref{thm:faithfulness} is an object-recognition theorem:
\[
  \text{equal sufficiently fine pair codes}
  \quad\Longrightarrow\quad
  \text{pair-isotopic realizations}.
\]
Filtered rigidity asks a different question:
\[
  \begin{gathered}
  \text{an automorphism of the collection of visible classes}\\
  \stackrel{?}{\Longrightarrow}\\
  \text{a homeomorphism of the exterior}.
  \end{gathered}
\]
The second implication does not follow formally from the first.  In
particular, raw lattice-code automorphisms may contain symmetries of the
encoding scheme which have no geometric meaning.  The natural rigidity object
is therefore the visible essential-surface complex, while the layered codes
serve as finite certificates for its vertices and incidence data.
\end{remark}

\begin{remark}[Reduced decorations and forgetful maps]
The finite characteristic data of \Cref{sec:characteristic} naturally produce
a hierarchy of forgetful maps
\[
  \ES^{\mathrm{full}}_{\Delta,\tau}(E,\mu)
  \longrightarrow
  \ES^{\mathrm{type,slope}}_{\Delta,\tau}(E,\mu)
  \longrightarrow
  \ES^{\mathrm{type}}_{\Delta,\tau}(E)
  \longrightarrow
  \ES_{\Delta,\tau}(E).
\]
Comparing the automorphism groups along these maps separates indispensable
labels from redundant ones.  Surface type may eliminate exchanges between
homeomorphically different surfaces, boundary slope may eliminate peripheral
false symmetries, and JSJ support or twist coordinates may be required to
detect gluing across characteristic annuli and tori.
\end{remark}

\subsection{Frontier-crossing compatibility}

The characteristic system gives a finite decomposition of the exterior, but
geometric realization is not purely piecewise.  A visible surface may cross
several JSJ pieces.  Even when an automorphism is realized on each piece, the
local realizations may differ by annular or toral twisting on the frontier and
may therefore fail to induce the prescribed action on crossing surfaces.
This is the \emph{frontier-crossing compatibility problem}.

For this reason, a frontier-aware finite characteristic data set should record,
in addition to the items of \Cref{sec:characteristic},
\begin{enumerate}[label=(\alph*),leftmargin=2em]
\item the JSJ pieces met by each visible surface;
\item the curves and slopes in which it meets the characteristic frontier;
\item the pairing of the cut surface pieces across each frontier component;
\item relative annular or toral twist coordinates whenever the gluing has a
twist ambiguity.
\end{enumerate}
These are finite data at each bounded level.  Determining when they are
intrinsically recoverable, and when they force piecewise realizations to glue,
is a rigidity problem rather than a consequence of smooth finiteness.

\begin{remark}[Fixed exterior versus moving exterior]
The complexes in this section belong to one fixed geometric exterior.  When
\(\gamma\) varies in \(Y^{=}_\Lambda(K)\), the correct global object is not a
single complex mixing vertices from different exteriors, but the family
\[
  \bigl\{\mathbf{ES}^{\mathrm{geom}}(E(\gamma),\mu_\gamma)
  \bigr\}_{\gamma\in Y^{=}_\Lambda(K)}
\]
together with pair-isotopy transport.  This separates fixed-exterior rigidity
from geometric persistence under motion of the knot representative.
\end{remark}


\section{Compression, tubing, and incompressibility}
\label{sec:essential-intersections-compression-order}

The finiteness and reconstruction theorems of
\Cref{sec:bounded-complexity,sec:finite-recognition} use essentiality only
through the anchoring lemma (\Cref{lem:anchoring}): positive reach and area
bound the topology and, for anchored families --- in particular for all
properly embedded surfaces every component of which has nonempty boundary,
each boundary sitting on the
peripheral torus, and for all essential surfaces --- pin down the isotopy type
whether or not the surface is incompressible.  Incompressibility and
boundary-incompressibility are detected by the absence of compressions;
essentiality additionally excludes boundary-parallel components and
inessential sphere or disk components
(\Cref{rem:compression-detects-essentiality}).  Thus the operation that certifies a visible surface as
essential, and that relates a compressible visible surface to its essential
reduction, is compression; its inverse, tubing, generally costs area and ambient
room.  This makes compression and tubing the basic transition mechanism of the
filtration.  We first construct a finite geometric filtration of the disk
complex of one fixed surface and extract its low-index compression signature.
We then return to the compression order on varying surface classes.  The
filtered viewpoint is not that compression gives an absolute order on
existence, but that it compares the geometric budgets needed to realize
surfaces related by compression and inverse tubing.

\subsection{Finite visible disk complexes}
\label{subsec:visible-disk-complexes}

We now pass from the compression order on surface classes to the compression
data of one fixed surface.  This is a local extension of the main
essential-surface framework: the surface below need not itself be essential.
Throughout this subsection, fix a geometric exterior \(E=E(\gamma)\) and a
compact, properly embedded, two-sided \(C^{1,1}\) surface \(H\subset E\).
A \emph{compressing disk} for \(H\) is an embedded disk
\(D\subset E\) such that
\[
  D\cap H=\partial D\subset\operatorname{int}H,
  \qquad D\cap\partial E=\varnothing,
\]
and \(\partial D\) is essential in \(H\).  Two compressing disks are
\emph{\(H\)-isotopic} if they are related by an ambient isotopy of \(E\) that
preserves \(H\) setwise.  This gives the usual disk-complex equivalence by
isotopy extension.

\begin{definition}[Relative thickness of a compressing-disk system]
\label{def:relative-disk-thickness}
Fix universal constants \(0<c_D<1\) and \(L_D\geq1\).  Let
\(\mathbf D=D_0\sqcup\cdots\sqcup D_k\) be a finite pairwise disjoint system
of compressing disks for \(H\).  We write
\[
  \Thi_H(\mathbf D)\geq\sigma
\]
if the following conditions hold.
\begin{enumerate}[label=(\roman*),leftmargin=2em]
\item The closed set \(\overline{\mathbf D}\subset\R^3\) has Federer reach at
least \(\sigma\), and each disk has the corresponding uniform
\(C^{1,1}\) graphical charts at scale \(c_D\sigma\) with Lipschitz constant
at most \(L_D\).
\item The boundary system \(\partial\mathbf D\subset H\) has intrinsic normal
injectivity radius in \(H\) at least \(c_D\sigma\).
\item Each boundary component has an embedded inward collar
\[
 C_D:[0,c_D\sigma]\times\partial\mathbf D\longrightarrow\mathbf D
\]
such that
\[
 \operatorname{dist}\bigl(C_D(s,x),H\bigr)\geq c_Ds
 \qquad(0\leq s\leq c_D\sigma).
\]
\item Away from that collar the disks have uniform clearance from both the
carrier and the peripheral boundary:
\[
 \operatorname{dist}\!\left(
 \mathbf D\setminus C_D([0,c_D\sigma/2]\times\partial\mathbf D),H
 \right)\geq c_D^2\sigma,
\]
\[
 \operatorname{dist}(\mathbf D,\partial E)\geq c_D^2\sigma.
\]
\end{enumerate}
For a single disk we write \(\Thi_H(D)\).  The constants are part of the
bounded-geometry convention and are not optimized.
\end{definition}

\begin{remark}[No loss in the directed limit]
\label{rem:disk-thickness-no-loss}
Every finite smooth pairwise disjoint compressing-disk system can be made,
by an arbitrarily small isotopy through such systems, transverse to \(H\)
along its boundary.  Compactness then gives a positive boundary angle,
positive intrinsic boundary injectivity radius, positive clearance away from
the boundary collars and from \(\partial E\), and positive reach of the
finite union.  Hence every simplex of the ordinary disk complex is visible
for some \((A,\sigma)\) below.  The relative disk-thickness condition is thus
a compactness device, not a topological restriction in the directed limit.
\end{remark}

\begin{definition}[Visible disk complex]
\label{def:visible-disk-complex}
The \emph{disk complex} \(\mathcal D(H)\) has vertices the \(H\)-isotopy
classes of compressing disks, with a simplex whenever the classes admit
pairwise disjoint representatives.  For \(A,\sigma>0\), the
\emph{\((A,\sigma)\)-visible disk complex}
\[
  \mathcal D^{\mathrm{geom}}_{A,\sigma}(H)
  \subset \mathcal D(H)
\]
is the subcomplex in which a simplex \([D_0],\ldots,[D_k]\) is present exactly
when it has pairwise disjoint representatives \(\mathbf D=D_0\sqcup\cdots
\sqcup D_k\) satisfying
\[
  \sum_{i=0}^k\Area(D_i)\leq A,
  \qquad
  \Thi_H(\mathbf D)\geq\sigma.
\]
The conditions in \Cref{def:relative-disk-thickness} are inherited by
subsystems, so this is indeed a simplicial subcomplex.  If \(H\) separates
\(E\), a transverse orientation labels the two sides \(+\) and \(-\); the
full and visible one-side subcomplexes are denoted \(\mathcal D^\pm(H)\)
and \(\mathcal D^{\mathrm{geom},\pm}_{A,\sigma}(H)\).
\end{definition}

\begin{lemma}[Relative stability for compression disks]
\label{lem:compression-disk-stability}
Fix \(H\), \(A>0\), and \(\sigma>0\).  There exist
\(\eta_D,\delta_D>0\) such that any two compressing disks
\(D_0,D_1\) with
\[
 \Area(D_i)\leq A,
 \qquad
 \Thi_H(D_i)\geq\sigma
\]
that are \(\eta_D\)-close in Hausdorff distance and \(\delta_D\)-close in
tangent direction are \(H\)-isotopic.
\end{lemma}

\begin{proof}
The proof is the fixed-carrier analogue of the boundary and collar alignment
used in \Cref{lem:boundary-alignment,lem:relative-collar-alignment,lem:interior-normal-graph}.  By clause (ii) of
\Cref{def:relative-disk-thickness}, sufficiently close boundary curves are
normal graphs over one another in the fixed surface \(H\).  Interpolating the
normal graph gives an isotopy in \(H\), and isotopy extension gives a small
ambient isotopy preserving \(H\) setwise that aligns the two boundaries.
Clause (iii) then aligns the disk collars without crossing \(H\).  After the
collars agree, clauses (i) and (iv), together with sufficiently small
positional and tangent errors, put the remaining interiors in a common normal
tube and express one as a single normal graph over the other.  Linear
interpolation of that graph stays embedded, is fixed on the aligned collar,
and stays disjoint from \(H\cup\partial E\).  A final application of isotopy
extension gives the required \(H\)-preserving ambient isotopy.
\end{proof}

\begin{theorem}[Finite visible disk complexes and exhaustion]
\label{thm:visible-disk-complex-exhaustion}
For every fixed two-sided \(C^{1,1}\) surface \(H\subset E\) and every
\(A,\sigma>0\), the complex
\(\mathcal D^{\mathrm{geom}}_{A,\sigma}(H)\) is finite.  Moreover,
\[
 A\leq A',\quad \sigma'\leq\sigma
 \quad\Longrightarrow\quad
 \mathcal D^{\mathrm{geom}}_{A,\sigma}(H)
 \subseteq
 \mathcal D^{\mathrm{geom}}_{A',\sigma'}(H),
\]
and
\[
 \mathcal D(H)
 =
 \bigcup_{A<\infty,\,\sigma>0}
 \mathcal D^{\mathrm{geom}}_{A,\sigma}(H),
\]
where the union is directed by increasing \(A\) and decreasing \(\sigma\).
The same statements hold for the one-side subcomplexes when \(H\) separates
\(E\).
\end{theorem}

\begin{proof}
Monotonicity is immediate from the definition.  For finiteness it suffices to
prove that only finitely many visible vertices occur.  Suppose instead that
\(D_n\) is a sequence of pairwise non-\(H\)-isotopic compressing disks with
\(\Area(D_n)\leq A\) and \(\Thi_H(D_n)\geq\sigma\).  Every disk is anchored on
the fixed compact carrier \(H\) through \(\partial D_n\subset H\).  The net
argument of \Cref{lem:diameter-bound}, using the graphical and collar clauses
of \Cref{def:relative-disk-thickness}, gives a uniform diameter bound
\(C A/\sigma\).  Hence all disks lie in one compact subset of
\(E\setminus\partial E\).  Repeating the Blaschke and graphical-chart argument of
\Cref{lem:compactness-thick-surfaces}, with the boundary constrained to the
fixed carrier \(H\) by \Cref{def:relative-disk-thickness}, gives a
\(C^1\)-convergent subsequence.  For large
\(k,l\), \Cref{lem:compression-disk-stability} makes \(D_{n_k}\) and
\(D_{n_l}\) \(H\)-isotopic, a contradiction.  Thus the vertex set is finite,
and a simplicial complex on a finite vertex set is finite.

For exhaustion, let \([D_0],\ldots,[D_k]\) be any simplex of
\(\mathcal D(H)\).  Choose smooth pairwise disjoint representatives and apply
\Cref{rem:disk-thickness-no-loss}.  Their total area is finite and the system
has \(\Thi_H(\mathbf D)>0\), so the simplex lies in some finite stage.
\end{proof}

\subsection{Compression births, mixed merges, and low topological index}
\label{subsec:compression-signatures}

The preceding filtration is monotone, so disk classes do not literally die as
\(A\) increases and \(\sigma\) decreases.  The useful birth--death language
comes instead from the two sides of a separating surface and the merging of
their connected components when a weak reducing pair becomes visible.

\begin{definition}[Compression-signature graph]
\label{def:compression-signature-graph}
Let \(H\) be connected and separating, with sides labelled \(+\) and \(-\).
For \(A,\sigma>0\), define \(\mathcal Q_{A,\sigma}(H)\) to be the graph with
at most two vertices \(v_+,v_-\): the vertex \(v_\pm\) is present exactly when
\(\mathcal D^{\mathrm{geom},\pm}_{A,\sigma}(H)\neq\varnothing\), and the edge
\(v_+v_-\) is present exactly when
\(\mathcal D^{\mathrm{geom}}_{A,\sigma}(H)\) contains an edge represented by
a disjoint opposite-side pair of compressing disks.  Thus the edge records a
visible weak reduction.
\end{definition}

Let
\[
 \chi:t\longmapsto(A(t),\sigma(t)),\qquad t\geq0,
\]
be any cofinal chain with \(A(t)\) nondecreasing and tending to infinity and
\(\sigma(t)\) nonincreasing and tending to zero.  Set
\[
 b_\pm(H;\chi)
 =\inf\{t:v_\pm\in\mathcal Q_{A(t),\sigma(t)}(H)\},
\]
and
\[
 m(H;\chi)
 =\inf\{t:v_+v_-\in\mathcal Q_{A(t),\sigma(t)}(H)\},
\]
with \(\inf\varnothing=\infty\).  The numbers \(b_\pm\) are the two
\emph{compression-birth times}, and \(m\) is the \emph{mixed merge time}.
Only finiteness versus infinity is topologically intrinsic; the finite values
depend on the chosen geometric chain.  Recall that a separating surface is
\emph{strongly irreducible} when it has compressing disks on both sides but
every opposite-side pair intersects, and is \emph{weakly reducible} when a
disjoint opposite-side pair exists.

\begin{corollary}[Filtered compression signature]
\label{cor:compression-signature-classification}
For every cofinal chain \(\chi\), the following hold.
\begin{enumerate}[label=(\roman*),leftmargin=2em]
\item \(H\) is incompressible if and only if
\[
 b_+(H;\chi)=b_-(H;\chi)=\infty.
\]
\item \(H\) is strongly irreducible if and only if
\[
 b_+(H;\chi)<\infty,\qquad
 b_-(H;\chi)<\infty,\qquad
 m(H;\chi)=\infty.
\]
\item \(H\) is weakly reducible if and only if
\[
 m(H;\chi)<\infty.
\]
\end{enumerate}
Consequently compressibility and weak reducibility admit finite geometric
witnesses.  Incompressibility and strong irreducibility, by contrast, are
persistent nonappearance statements for this filtration and are not claimed
to be decidable from an arbitrary single finite stage.
\end{corollary}

\begin{proof}
By \Cref{thm:visible-disk-complex-exhaustion}, a compressing disk on a given
side exists topologically if and only if one appears at some finite stage of
any cofinal chain.  Likewise, a weak reducing pair exists if and only if its
edge appears at some finite stage.  The three equivalences are therefore the
definitions of incompressible, strongly irreducible, and weakly reducible,
respectively.
\end{proof}

\begin{corollary}[Low topological index as a filtered signature]
\label{cor:filtered-topological-index}
In Bachman's topological index theory \cite{BachmanTopologicalIndex}, a surface
has topological index zero exactly when it is incompressible.  Hence index zero
is equivalent to the persistent absence of both compression births above.  If
\(H\) is a Heegaard surface, then it has topological index one exactly when it
is strongly irreducible; equivalently, along every cofinal chain both side
births occur at finite time but the mixed merge time is infinite.
\end{corollary}

\begin{remark}[Why this is not yet a finite decision procedure]
\label{rem:no-finite-negative-certificate}
The finite-stage theorem should not be strengthened without additional input.
If \(\mathcal D^{\mathrm{geom}}_{A,\sigma}(H)=\varnothing\) at one stage, a
compressing disk of larger area or smaller relative thickness may still
appear later.  Similarly, the absence of a mixed edge at one stage does not
certify strong irreducibility.  Thus positive events have finite witnesses,
whereas the negative cases require either the entire directed system or an
independent a priori completion bound.  Producing such a bound, uniformly in a
bounded surface window, is an effective problem rather than a consequence of
compactness alone.
\end{remark}

\begin{remark}[Relation with the complex of surfaces]
\label{rem:johnson-complex-surfaces}
Johnson's complex of surfaces has isotopy classes of strongly separating
surfaces as vertices and compressions as descending edges; its descending
links recover disk-complex data, and thin paths encode the classical
thin-position mechanism \cite{JohnsonHeegaardIsotopy}.  The construction here
is not a replacement for that complex.  It supplies a geometric filtration of
the disk complex attached to one fixed surface.  A genuinely filtered complex
of surfaces would have to control the geometry of the vertices and of all
compression edges simultaneously; such a construction is left open below.
\end{remark}

\subsection{Relation to sutured hierarchies and generalized Heegaard splittings}
\label{subsec:sutured-ghs-interface}

Compression and tubing place the present filtration next to two hierarchy
formalisms.  In a taut sutured-manifold hierarchy, incompressible decomposing
surfaces reduce the sutured manifold while retaining norm-minimizing
information \cite{GabaiFoliationsI,ScharlemannSuturedNorms}.  In a generalized
Heegaard splitting, incompressible thin levels alternate with compressible
thick levels and compression bodies \cite{ScharlemannThompson}.  The essential
surfaces considered here therefore model possible decomposing surfaces and
thin-level components, while the compression order records part of the local
transition mechanism.

The present pair space does not encode an entire hierarchy or generalized
splitting.  A geometric sutured hierarchy must retain the order of successive
decompositions and control the accumulated area and minimum thickness of all
decomposing surfaces.  A geometric generalized Heegaard splitting must
additionally retain all thick levels simultaneously, compression-body
incidence, and compatibility of the two disk sets whose curve-complex
separation is Hempel distance \cite{HempelCurveComplex}.  The visible disk
complexes above solve only the local fixed-surface part of this problem; no
finiteness or distance theorem for whole generalized splittings is asserted
here.

\begin{remark}[Compression detects essentiality]
\label{rem:compression-detects-essentiality}
Bounded geometry sees a surface but not, by itself, its essentiality: a visible
pair \((\gamma,F)\) may be compressible.  Minimality for the order
\(\geq_{\mathrm{cb}}\) below --- no compression or boundary-compression being
available --- is necessary for essentiality but not equivalent to it: a
boundary-parallel annulus and a \(2\)-sphere bounding a ball admit no
compressions, yet are inessential.  A surface is essential precisely when it
is \(\geq_{\mathrm{cb}}\)-minimal \emph{and}, in addition, has no
boundary-parallel component and no component that is an inessential sphere
or disk, as in the essential-subspace convention of
\Cref{sec:pair-spaces}.  Hence the reconstruction of
\Cref{thm:faithfulness}, which recovers the isotopy type of \(F\) from its
finite code, must be complemented by a compression analysis to determine
whether the recovered surface is essential.  For one fixed surface, the
visible disk complexes above provide finite geometric witnesses whenever a
compression exists.  What they do \emph{not} provide is an a priori finite
stage at which failure to see a disk certifies incompressibility
(\Cref{rem:no-finite-negative-certificate}).  In the surface-class filtration
a compressible visible surface lies above its essential reduction in
\(\geq_{\mathrm{cb}}\), and at no larger filtered budget under the controlled
hypotheses of \Cref{prop:controlled-compression-budget}.
\end{remark}

\begin{definition}[Compression order]
For properly embedded surfaces \(F,F'\subset E(K)\), write
\[
  F\geq_{\mathrm{cb}}F'
\]
if \(F'\) is obtained from \(F\) by a finite sequence of compressions and
boundary-compressions, followed by deletion of inessential sphere, disk, and
boundary-parallel components.  After passing to isotopy classes and using this
deletion convention, \(\geq_{\mathrm{cb}}\) is a partial order, since the
lexicographic complexity built from \(-\chi\), the number of components, and
boundary data strictly decreases along every nontrivial move.
\end{definition}

\begin{definition}[Birth region and birth frontier of a surface class]
\label{def:birth-region}
For fixed \(\tau>0\), the \emph{birth region} of a surface class \([F]\) is
the realizable set
\[
  B_\tau([F])=
  \left\{(\Lambda,\Delta)\;\middle|\;
  \begin{array}{l}
  \text{there exist }\gamma\in Y^{=}_\Lambda(K)\text{ and }F'\in[F]\subset E(\gamma)\\
  \text{with }\Thi(F')\geq\tau\text{ and }\Area(F')\leq\Delta
  \end{array}\right\}
  \subset(0,\infty)^2 .
\]
By definition \(B_\tau([F])\) is an upper set for the product order on the
\((\Lambda,\Delta)\)-plane: enlarging either budget preserves
realizability.  The \emph{birth frontier} \(b_\tau([F])\) is the set of
minimal elements of the closure of \(B_\tau([F])\) for the product order,
that is, its Pareto frontier.  The product order on the plane is not total,
and the two budgets can genuinely trade off against each other, so
\(B_\tau([F])\) need not have a least element: the birth data of a class is
in general a frontier curve, not a single birth point, and statements about
births are statements about the region \(B_\tau([F])\) or its frontier,
never about a single distinguished level.
\end{definition}

\begin{proposition}[Controlled compression decreases the filtered budget]
\label{prop:controlled-compression-budget}
Let \(\tau'\leq\tau\), and suppose the class \([F']\) is obtained from
\([F]\) by a compression or boundary-compression that is \emph{uniformly
realizable at thickness loss \(\tau\to\tau'\)}: for every representative
\(\gamma\in Y^{=}_\Lambda(K)\) and every representative \(F\in[F]\) in
\(E(\gamma)\) with \(\Thi(F)\geq\tau\), the compression can be realized in
\(E(\gamma)\) so that, after
smoothing and deleting inessential components, the resulting representative
\(F'\in[F']\) satisfies
\[
  \Area(F')\leq \Area(F),\qquad \Thi(F')\geq \tau' .
\]
The hypothesis quantifies over all realizing exteriors and representatives;
a compression realizable only over one exterior yields only the
corresponding budgets.  Then
\[
  B_\tau([F])\subseteq B_{\tau'}([F'])
\]
as upper sets in the \((\Lambda,\Delta)\)-plane: every budget that realizes
\([F]\) at thickness \(\tau\) realizes \([F']\) at thickness \(\tau'\).  We
deliberately state the conclusion as an inclusion of birth regions; a
comparison of the two birth frontiers would require a domination relation
between Pareto frontiers, which is not needed here.
\end{proposition}

\begin{proof}
Any representative of \([F]\) realizing a pair budget \((\Lambda,\Delta)\) gives,
by the controlled compression hypothesis, a representative of \([F']\) in the
same knot exterior with area at most \(\Delta\) and thickness at least
\(\tau'\).  Hence \((\Lambda,\Delta)\in B_{\tau'}([F'])\), which is the
inclusion of birth regions.
\end{proof}

\begin{remark}[Why the converse has a cost]
If \(F'\) is obtained from \(F\) by compression, then \(F\) is recovered from
\(F'\) by tubing.  Thus the existence of \(F'\) already gives the topological
possibility of producing \(F\).  What the filtration records is the extra
geometric cost of inserting the tube with the prescribed area and thickness.
A \(\tau\)-thick tube with core length \(\ell\) has area at least comparable to
\(\tau\ell\), and it also requires a corridor in the exterior wide enough to
accommodate a \(\tau\)-neighbourhood.  The asymmetry of the two directions is
illustrated in \Cref{fig:compression-tubing}.
\end{remark}

\begin{figure}[t]
\centering
\begin{tikzpicture}[x=1cm,y=1cm,>=Latex,
  note/.style={font=\scriptsize, align=center}]
  \draw[very thick]
    (0.4,3.45) .. controls (1.5,3.62) and (2.1,3.52) .. (2.4,3.10)
    .. controls (2.58,2.60) and (2.58,1.60) .. (2.4,1.10)
    .. controls (2.1,0.68) and (1.5,0.58) .. (0.4,0.75);
  \draw[very thick]
    (5.3,3.45) .. controls (4.2,3.62) and (3.6,3.52) .. (3.3,3.10)
    .. controls (3.12,2.60) and (3.12,1.60) .. (3.3,1.10)
    .. controls (3.6,0.68) and (4.2,0.58) .. (5.3,0.75);
  \draw[dashed] (2.85,2.10) ellipse (0.30 and 0.10);
  \node[note] at (3.72,2.42) {$D$};
  \draw[->, thin] (3.60,2.34) -- (3.18,2.16);
  \node[note, align=center] at (1.30,2.10) {$\tau$-thick\\tube};
  \draw[->, thin] (1.78,2.10) -- (2.46,2.10);
  \node[note] at (0.85,3.95) {$F$};
  \draw[very thick]
    (8.3,3.45) .. controls (9.4,3.62) and (10.1,3.52) .. (10.4,3.10)
    .. controls (10.62,2.72) and (11.18,2.72) .. (11.4,3.10)
    .. controls (11.7,3.52) and (12.4,3.62) .. (13.5,3.45);
  \draw[very thick]
    (8.3,0.75) .. controls (9.4,0.58) and (10.1,0.68) .. (10.4,1.10)
    .. controls (10.62,1.48) and (11.18,1.48) .. (11.4,1.10)
    .. controls (11.7,0.68) and (12.4,0.58) .. (13.5,0.75);
  \draw[dashed, thin] (10.9,2.82) -- (10.9,1.38);
  \fill (10.9,2.82) circle (1.1pt);
  \fill (10.9,1.38) circle (1.1pt);
  \node[note, anchor=west] at (11.0,2.10) {$a$};
  \draw[decorate, decoration={brace, amplitude=3.5pt}]
    (10.62,1.38) -- (10.62,2.82);
  \node[note, anchor=east] at (10.42,2.10) {$\ell(a)$};
  \node[note] at (8.75,3.95) {$F'$};
  \draw[->, thick] (5.85,2.65) -- (7.85,2.65);
  \node[note, above] at (6.85,2.72) {compression along $D$};
  \node[note, below] at (6.85,2.58) {area does not increase};
  \draw[->, thick] (7.85,1.35) -- (5.85,1.35);
  \node[note, above] at (6.85,1.42) {tubing along $a$};
  \node[note, below] at (6.85,1.28)
    {$\Delta_\tau(a)\gtrsim 2\pi\tau\,\ell(a)$};
\end{tikzpicture}
\caption{Compression and tubing as the transition mechanism of the
filtration, in cross-section.  Left: a visible surface \(F\) with a
\(\tau\)-thick tube; the dashed compression disk \(D\) spans the tube.
Compressing along \(D\) produces the surface \(F'\) (right) and never
increases the filtered budget
(\Cref{prop:controlled-compression-budget}).  The inverse operation, tubing
\(F'\) along an embedded arc \(a\), is topologically always available but
carries a geometric cost: a \(\tau\)-thick tube of core length \(\ell(a)\)
adds area at least comparable to \(2\pi\tau\,\ell(a)\) and needs a
\(\tau\)-wide corridor in the exterior, possibly forcing a larger ropelength
level.  The filtration records exactly this asymmetry between the two
directions.}
\label{fig:compression-tubing}
\end{figure}
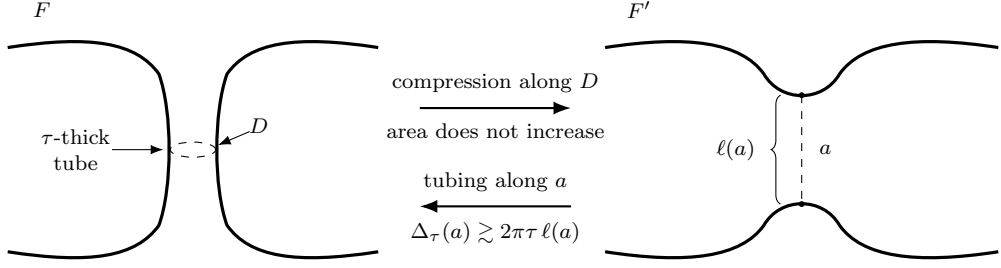

\begin{definition}[Tubing cost]
Let \(F\) be obtained from \(F'\) by attaching a tube along an embedded arc
\(a\) in the knot exterior, or by the corresponding relative tubing operation.
The \(\tau\)-tubing cost is the pair
\[
  \operatorname{Cost}_\tau(a)=\bigl(\Delta_\tau(a),\Lambda_\tau(a)\bigr),
\]
where \(\Delta_\tau(a)\) is the additional area required by the tube and
\(\Lambda_\tau(a)\) is the additional ropelength level needed, if any, to make a
\(\tau\)-thick corridor for the tube.  In model bounded-geometry situations one
expects
\[
  \Delta_\tau(a)\ \gtrsim\ 2\pi\tau\,\ell(a),
\]
up to smoothing and endpoint constants.
\end{definition}

\begin{problem}[Quantitative tubing cost]
\label{prob:tubing-cost}
Estimate \(\operatorname{Cost}_\tau(a)\) in terms of the length of the tubing
arc, available collar width, local injectivity radius of the complement, and
the area and thickness of \(F'\).  Such estimates would turn the
compression--tubing duality into explicit inequalities between birth levels.
\end{problem}

\section{Seifert surfaces, tautness, and ropelength-filtered Kakimizu persistence}
\label{sec:kakimizu}

\subsection{The sutured-manifold interpretation of taut visibility}
\label{subsec:kakimizu-sutured}

Cutting a knot exterior along a minimal-genus Seifert surface produces a taut
sutured manifold.  Consequently, a vertex in a visible Kakimizu layer may be
viewed as a geometrically controlled first decomposing surface in a Gabai
hierarchy \cite{GabaiFoliationsI}.  The area--thickness window measures the
cost of making this initial taut decomposition visible, whereas the Kakimizu
complex records disjointness among the resulting isotopy classes.

This interpretation does not imply that every bounded-geometry minimal-genus
Seifert surface extends to a sutured hierarchy with a uniform geometric
budget.  Later decomposing surfaces may require more area or less thickness.
A persistence invariant for entire taut hierarchies is therefore a genuine
extension of the first-stage Kakimizu persistence studied here, not a
consequence of it.

\begin{remark}[Diagrammatic sources of thick essential surfaces]
Regular diagrams provide a basic source of surfaces in the filtered pair
space.  Given a diagram \(D\) of \(K\), Seifert's algorithm produces a
Seifert surface \(F_D\).  This surface need not be incompressible or of
minimal genus.  However, after performing compressions, boundary-compressions,
and discarding inessential or boundary-parallel components, one obtains an
essential part \(F_D^{\mathrm{ess}}\) whenever such a part remains.

After smoothing and putting the surface in bounded-geometry position, this
essential part gives a positive-thickness representative for some
\(\tau>0\).  Thus regular projections, via Seifert's algorithm and essential
reduction, provide concrete diagrammatic sources of points in the essential
filtered surface space.
\end{remark}

\subsection{The taut Seifert layer}

For a knot in \(S^3\), a connected Seifert surface \(F\) satisfies
\[
  -\chi(F)=2g(F)-1.
\]
Thus Thurston norm minimization is equivalent to genus minimization.  The
taut subfiltration therefore isolates minimal genus Seifert surfaces by
definition, while the ropelength, area, and thickness parameters refine their
geometric realization.

\begin{definition}[Ropelength--area--thickness filtered Kakimizu layer]
Let \(\mathrm{MS}(K)\) denote the Kakimizu complex of \(K\)
\cite{Kakimizu}.  For
\(\Lambda,\Delta,\tau>0\), define
\[
  \mathrm{MS}_{\Lambda,\Delta,\tau}(K)
\]
to be the full subcomplex of \(\mathrm{MS}(K)\) spanned by isotopy classes of
minimal genus Seifert surfaces \(F\) for which there exists a representative
\(\gamma\in Y^{=}_\Lambda(K)\) such that
\[
  (\gamma,F)\in \calZ_{\Lambda,\Delta,\tau}^{\mathrm{ess}}(K).
\]
\end{definition}

Thus \(\mathrm{MS}_{\Lambda,\Delta,\tau}(K)\) is the visible full subcomplex of the Kakimizu
complex: its vertices are the minimal genus Seifert surface classes visible at
that level, and its simplices are the simplices of \(\mathrm{MS}(K)\) spanned
by those visible vertices.  It is the part of the Kakimizu complex visible at
ropelength level \(\Lambda\), surface area level \(\Delta\), and surface
thickness scale \(\tau\).

\subsection{Non-orientable spanning surfaces and the crosscap layer}
\label{subsec:nonorientable}

The Seifert layer uses orientable spanning surfaces.  Non-orientable spanning
surfaces are equally essential objects, and the whole bounded-geometry machinery
applies to them without change, because neither
\Cref{prop:geometric-topological-boundedness} nor
\Cref{lem:compactness-thick-surfaces} uses orientability.  Recording this layer
completes the spanning-surface side of the theory.

\begin{definition}[Spanning-surface subspace]
\label{def:spanning-subspace}
A spanning surface for \(\gamma\) is a compact connected properly embedded
surface \(F\subset E(\gamma)\), possibly non-orientable, whose boundary
\(\partial F\) is a single longitudinal curve on \(\partial E(\gamma)\).  Let
\[
  \calS^{\pm}_{\Lambda,\Delta,\tau}(K)\subset\calZ^{\mathrm{ess}}_{\Lambda,\Delta,\tau}(K)
\]
be the subspace of pairs \((\gamma,F)\) with \(F\) an essential spanning surface,
and let \(\calS^{-}_{\Lambda,\Delta,\tau}(K)\) be the further subspace on which
\(F\) is non-orientable.
\end{definition}

For an orientable spanning surface the boundary slope is the Seifert longitude.
For a non-orientable spanning surface \(F\) the boundary curve
\(\partial F\) has nonzero even integral slope; with the standard
Gordon--Litherland normalization it equals \(-e(F)/2\), where \(e(F)\) is
the normal Euler number.  The associated Gordon--Litherland form computes the
signature of \(K\)
\cite{GordonLitherland}.  The crosscap number \(\gamma_c(K)\), the minimal number
of crosscaps of a non-orientable spanning surface, is the non-orientable
analogue of genus \cite{Clark}.

\begin{proposition}[Boundedness and finiteness for spanning surfaces]
\label{prop:nonorientable-boundedness}
Fix \((\Lambda,\Delta,\tau)\).  For a non-orientable connected spanning surface
\(F\) with \(k\) crosscaps and one boundary component,
\(\chi(F)=1-k\), and \Cref{prop:geometric-topological-boundedness} gives
\[
  k\ \leq\ 1+C(\tau)\,\Delta .
\]
Consequently the crosscap numbers of visible spanning surfaces are bounded, and
\Cref{thm:smooth-finiteness-fixed-exterior,thm:smooth-finiteness-pair-space}
apply verbatim: at fixed \((\Lambda,\Delta,\tau)\) there are only finitely many
pair-isotopy classes of essential spanning surfaces, orientable or not.
\end{proposition}

\begin{proof}
The Euler-characteristic estimate of
\Cref{prop:geometric-topological-boundedness} is proved from thickness, area,
collar, and curvature control and uses the identity \(2\pi\chi(F)=\int_F
K_F+\int_{\partial F}k_g\); it is insensitive to orientability.  For a
non-orientable surface with \(k\) crosscaps and \(b\) boundary components
\(\chi(F)=2-k-b\), so with \(b=1\) one gets \(k=1-\chi(F)\leq1+C(\tau)\Delta\).
The finiteness theorems never used orientability either, so they hold for the
spanning-surface subspace.
\end{proof}

\begin{remark}[The crosscap and non-orientable taut layers]
\label{rem:crosscap-layer}
One may filter spanning surfaces by minimal complexity exactly as for Seifert
surfaces.  The non-orientable taut subspace consists of pairs \((\gamma,F)\) with
\(F\) of minimal crosscap number \(\gamma_c(K)\), and the disjointness relation
of Kakimizu defines a spanning-surface complex whose vertices are the
minimal-complexity spanning-surface classes and whose ropelength--area--thickness
filtration is defined as in
\Cref{def:relative-surface-thickness} and the definitions following it.  The
visibility levels, height
functions, and merge scales of the Seifert case carry over verbatim; the taut
Seifert layer is the orientable part, and the crosscap layer is the
non-orientable part, of a single spanning-surface filtration.  Because a
non-orientable spanning surface fixes a nonzero even boundary slope, these layers
occupy distinct boundary strata, so by
\Cref{prop:slope-obstruction-to-merging} they have infinite merge scale under
slope-preserving deformations.
\end{remark}

\subsection{Ropelength height and completion levels}

Fix \(\Delta,\tau>0\).  The ropelength parameter gives a natural height
function on the vertices of the Kakimizu complex.

\begin{definition}[Ropelength height on the Kakimizu complex]
For a vertex \([F]\in \mathrm{MS}(K)\), define
\[
  h_{\Delta,\tau}([F])
  =
  \inf\left\{
  \Lambda
  \mid
  [F]\in \mathrm{MS}_{\Lambda,\Delta,\tau}(K)
  \right\}.
\]
We call \(h_{\Delta,\tau}\) the ropelength height function on the Kakimizu
complex.
\end{definition}

The term ``height function'' is used in a combinatorial filtered sense.  It is
not a smooth Morse function.  Its sublevel complexes are precisely the
ropelength-filtered Kakimizu layers.

Several numerical thresholds are natural.

\begin{definition}[Kakimizu birth, connectivity, and completion levels]
For fixed \(\Delta,\tau>0\), define
\[
  \Lambda_{\mathrm{first}}^{\Delta,\tau}(K)
  =
  \inf\left\{
  \Lambda
  \mid
  \mathrm{MS}_{\Lambda,\Delta,\tau}(K)\neq\emptyset
  \right\},
\]
\[
  \Lambda_{\mathrm{conn}}^{\Delta,\tau}(K)
  =
  \inf\left\{
  \Lambda
  \middle|
  \begin{array}{l}
  \mathrm{MS}_{\Lambda',\Delta,\tau}(K)\text{ is nonempty and connected}\\
  \text{for every }\Lambda'\geq\Lambda
  \end{array}
  \right\},
\]
and
\[
  \Lambda_{\mathrm{MS}}^{\Delta,\tau}(K)
  =
  \inf\left\{
  \Lambda
  \mid
  \mathrm{MS}_{\Lambda,\Delta,\tau}(K)=\mathrm{MS}(K)
  \right\}.
\]
The last quantity is called the Kakimizu completion level.
\end{definition}

The empty complex is regarded as not connected.  These quantities measure,
respectively, when the first minimal genus Seifert surface becomes visible,
when the visible taut layer is connected at every subsequent level, and when
the full Kakimizu complex has appeared.  The eventual formulation is needed
because adding a new isolated vertex can temporarily destroy connectedness.  They satisfy
\[
  \Rop(K)
  \leq
  \Lambda_{\mathrm{first}}^{\Delta,\tau}(K)
  \leq
  \Lambda_{\mathrm{conn}}^{\Delta,\tau}(K)
  \leq
  \Lambda_{\mathrm{MS}}^{\Delta,\tau}(K),
\]
with the convention that some quantities may be infinite if the
\((\Delta,\tau)\)-surface constraint is too restrictive.
\Cref{fig:kakimizu-height} shows the three thresholds as levels of the
height function.

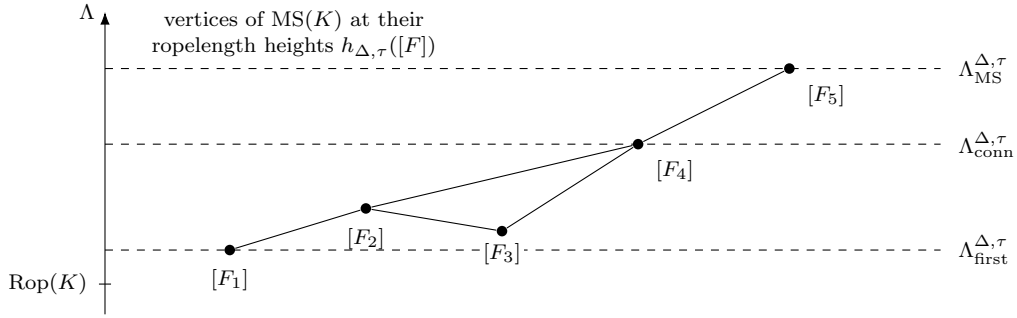
\begin{figure}[t]
\centering
\begin{tikzpicture}[x=1cm,y=1cm,>=Latex,
  note/.style={font=\scriptsize, align=center},
  vtx/.style={circle, fill, inner sep=1.4pt}]
  \draw[->] (0.55,0.35) -- (0.55,4.35) node[note, left] {$\Lambda$};
  \draw[thin] (0.47,0.75) -- (0.63,0.75);
  \node[note, left] at (0.45,0.75) {$\Rop(K)$};
  \draw[dashed, thin] (0.55,1.20) -- (11.6,1.20);
  \node[note, anchor=west] at (11.7,1.20)
    {$\Lambda_{\mathrm{first}}^{\Delta,\tau}$};
  \draw[dashed, thin] (0.55,2.60) -- (11.6,2.60);
  \node[note, anchor=west] at (11.7,2.60)
    {$\Lambda_{\mathrm{conn}}^{\Delta,\tau}$};
  \draw[dashed, thin] (0.55,3.60) -- (11.6,3.60);
  \node[note, anchor=west] at (11.7,3.60)
    {$\Lambda_{\mathrm{MS}}^{\Delta,\tau}$};
  \node[vtx] (v1) at (2.2,1.20) {};
  \node[vtx] (v2) at (4.0,1.75) {};
  \node[vtx] (v3) at (5.8,1.45) {};
  \node[vtx] (v4) at (7.6,2.60) {};
  \node[vtx] (v5) at (9.6,3.60) {};
  \draw (v1) -- (v2) -- (v3);
  \draw (v3) -- (v4) -- (v5);
  \draw (v2) -- (v4);
  \node[note, below] at (2.2,1.08) {$[F_1]$};
  \node[note, below] at (4.0,1.63) {$[F_2]$};
  \node[note, below, fill=white, inner sep=1pt] at (5.8,1.33) {$[F_3]$};
  \node[note, below right] at (7.7,2.52) {$[F_4]$};
  \node[note, below right] at (9.7,3.52) {$[F_5]$};
  \node[note, fill=white, inner sep=1.5pt] at (3.05,4.05)
    {vertices of $\mathrm{MS}(K)$ at their\\ropelength heights
     $h_{\Delta,\tau}([F])$};
\end{tikzpicture}
\caption{The ropelength height function on the Kakimizu complex, at a fixed
surface budget \((\Delta,\tau)\), drawn schematically with each taut vertex
placed at its height \(h_{\Delta,\tau}([F])\).  The sublevel complexes are
the filtered layers \(\mathrm{MS}_{\Lambda,\Delta,\tau}(K)\): the first
vertex appears at \(\Lambda_{\mathrm{first}}^{\Delta,\tau}\), the visible
layer is connected from \(\Lambda_{\mathrm{conn}}^{\Delta,\tau}\) onward, and the
full complex \(\mathrm{MS}(K)\) is present from the completion level
\(\Lambda_{\mathrm{MS}}^{\Delta,\tau}\) on.  Since the limit complex is
contractible, all homology of the intermediate layers is transient, so the
barcode of this persistence is an invariant of the geometric filtration
itself.}
\label{fig:kakimizu-height}
\end{figure}

The completion level has the following interpretation: it is the least
ropelength allowance needed to see the entire taut Seifert-surface theory of
\(K\) at the prescribed surface scale.

\subsection{Kakimizu persistence and its finite-resolution form}

The filtration \(\Lambda\mapsto\mathrm{MS}_{\Lambda,\Delta,\tau}(K)\), for
\(\Lambda\geq\Rop(K)\), defines a persistence module
\(H_i(\mathrm{MS}_{\Lambda,\Delta,\tau}(K))\), \(i\geq0\), in the sense of
persistent topology \cite{EdelsbrunnerHarer,ZomorodianCarlsson}: the
ropelength persistence of the Kakimizu layer.  It records how the taut Seifert-surface complex grows with the
ropelength allowance, from the first birth level
\(\Lambda_{\mathrm{first}}^{\Delta,\tau}\) to the completion level
\(\Lambda_{\mathrm{MS}}^{\Delta,\tau}\).  Two refinements are worth recording
briefly.  First, a merge-theoretic refinement: for vertices \([F_1],[F_2]\) of
a visible layer, the merge scale
\(m_{\Delta,\tau}^{\mathrm{MS}}([F_1],[F_2])\) of the corresponding admissible
pair components refines the combinatorial Kakimizu distance, since two classes
may be adjacent as surfaces yet require additional ropelength room to be
joined by an admissible thick deformation; the associated merge completion
level \(\Lambda_{\mathrm{merge}}^{\Delta,\tau}(K)\), the least level at which
all visible components have merged, may differ from
\(\Lambda_{\mathrm{MS}}^{\Delta,\tau}(K)\), because appearance and
deformation-connectivity are different events.  Second, a finite-resolution
form: identifying classes with the same \(\varepsilon\)-scale model gives the
layer \(\mathrm{MS}^{\varepsilon}_{\Lambda,\Delta,\tau}(K)\), whose vertex
count is bounded by the explicit encoded-type bound of
\Cref{cor:explicit-encoded-bound}, so its quotient persistence, in the sense
of \Cref{def:admissible-transition-complex}, has a well-defined
finite barcode on every bounded parameter interval, and a locally finite
interval decomposition on the full ray, by \Cref{prop:persistence-tameness}.

\begin{remark}[Contractibility and transient persistence]
The full Kakimizu complex is contractible \cite{PrzytyckiSchultens}.
Consequently, along any directed exhaustion by subcomplexes, every reduced
homology class represented at a finite stage vanishes at some later stage.
This concerns reduced homology, including reduced degree zero; ordinary
\(H_0\) retains the eventual connected component.  It does not assert a
common finite completion level for the entire filtration.  In particular,
exhaustion must be checked separately when \(\Delta,\tau\) are fixed and
only \(\Lambda\) increases.
\end{remark}

\subsection{Finite-stage null-homotopies for incompressible Seifert surfaces}
\label{subsec:incompressible-finite-stage}

Fix a nontrivial knot \(K\), a geometric exterior \(E=E(\gamma)\), and
an integer \(\ell\geq g(K)\).  Let \(\mathrm{IS}(K)\) be the complex of
ambient-isotopy classes of connected orientable incompressible Seifert
surfaces, with simplices given by simultaneous disjointness, and let
\(\mathrm{IS}_{\ell}(K)\) be its full subcomplex on genus at most \(\ell\).
These surfaces have one longitudinal boundary component and are essential
in the nontrivial knot exterior.  We regard these complexes as subcomplexes
of \(\ES(E)\), using the same ambient-isotopy convention, and set
\[
 X^{\ell}_{\Delta,\tau}(E)
 :=\ES_{\Delta,\tau}(E)\cap\mathrm{IS}_{\ell}(K).
\]
The intersection is taken on simplices: the total-area and union-thickness
requirements of \Cref{def:finite-visible-complex} are retained.  It is not
merely the full subcomplex on individually visible vertices.
By \Cref{thm:visible-complex-finiteness,thm:visible-complex-exhaustion},
these are finite subcomplexes and
\[
 \bigcup_{\Delta>0,\,\tau>0}X^{\ell}_{\Delta,\tau}(E)
 =\mathrm{IS}_{\ell}(K).
\]

Pan--You--Zhou--Chen \cite[Theorems A and C]{PanYouZhouChen}
prove contractibility of \(\mathrm{IS}(K)\) and of
\(\mathrm{IS}_{\ell}(K)\).  The consequence below uses only the
unfiltered and genus-truncated contractibility statements of that preprint;
it does not use its area-truncation assertion.  The finite-stage argument
is formal once contractibility and finite geometric exhaustion are given.

\begin{proposition}[Finite-stage null-homotopy]
\label{prop:incompressible-finite-stage-null}
For every nonempty \(X^{\ell}_{\Delta,\tau}(E)\), there exist
\(\Delta'\geq\Delta\) and \(0<\tau'\leq\tau\) such that the inclusion
\[
 |X^{\ell}_{\Delta,\tau}(E)|
 \longrightarrow |X^{\ell}_{\Delta',\tau'}(E)|
\]
is null-homotopic.  The same statement holds with
\(\mathrm{IS}_{\ell}(K)\) replaced throughout by \(\mathrm{IS}(K)\).
Here geometric realizations carry their usual CW topology.
\end{proposition}

\begin{proof}
Write \(X=X^{\ell}_{\Delta,\tau}(E)\) and
\(Y=\mathrm{IS}_{\ell}(K)\).  Contractibility of \(|Y|\) supplies a
null-homotopy of the inclusion \(|X|\to|Y|\).  Since \(X\) is finite,
\(|X|\times[0,1]\) is compact.  Its image lies in a finite subcomplex
\(L\subset Y\): a compact subset of a CW complex is contained in a
finite subcomplex \cite[Appendix, Proposition A.1]{HatcherAT}.
Enlarge \(L\) to contain \(X\) if necessary.  For each of its finitely
many simplices, \Cref{thm:visible-complex-exhaustion} supplies an
area--thickness window in which that simplex is visible.  Taking the maximum
of their area bounds and \(\Delta\), and the minimum of their positive
thickness bounds and \(\tau\), gives \(\Delta',\tau'\) with
\(L\subset X^{\ell}_{\Delta',\tau'}(E)\).  The same null-homotopy
therefore has image in this finite stage.  For \(\mathrm{IS}(K)\), use
its contractibility and the identical argument.
\end{proof}

\begin{corollary}[Transient reduced homology at finite geometric stages]
\label{cor:incompressible-transience}
For every nonempty stage, the enlarged window in
\Cref{prop:incompressible-finite-stage-null} can be chosen so that all maps
\[
 \widetilde H_q(X^{\ell}_{\Delta,\tau}(E);R)
 \longrightarrow
 \widetilde H_q(X^{\ell}_{\Delta',\tau'}(E);R),\qquad q\geq0,
\]
vanish simultaneously, for every coefficient ring \(R\).
In particular their directed limits are zero.  The same conclusions hold
for the untruncated incompressible Seifert-surface complex.
\end{corollary}
\begin{proof}
A null-homotopic map induces the zero map on reduced homology.
Every element of a directed limit is represented at some stage.
\end{proof}

\begin{remark}[Scope of the geometric conclusion]
Neither the initial nor the enlarged visible complex is asserted to be
contractible.  The proposition concerns their inclusion and gives no
explicit bound on \(\Delta'\) or lower bound on \(\tau'\) in terms of
the input parameters.  It supplies no uniform window for all stages and
no conclusion about a fixed-thickness ropelength ray.  For
\(\ell=g(K)\), the topological input is already the classical
contractibility of \(\mathrm{MS}(K)\); the new input extends this
consequence to non-minimal incompressible Seifert surfaces.
\end{remark}

\begin{remark}[Exchange operations and geometric windows]
The exchange lemma of \cite[Theorem 5.4]{PanYouZhouChen} controls the
sum of genera of two outputs and, when that sum is unchanged, strictly
reduces the sum of their relative least areas.  Its common-neighbour
property is the input to an abstract flag-complex contraction argument.
This does not provide a relative reach bound for the outputs or simultaneous
geometric control of a whole simplex.  A visible complex defined using total
area and union thickness need not be flag.
Moreover, the set used in the proof of \cite[Section 8.2]{PanYouZhouChen}
is the lexicographic initial segment
\[
 \{v:g(v)<\ell\}\ \cup\
 \{v:g(v)=\ell,\ A(v)\leq a_0\},
\]
where \(A(v)\) is the relative least area used there.  This is different
from imposing \(g(v)\leq\ell\) and \(A(v)\leq a_0\) simultaneously.
Neither set is the total-area--union-thickness window used here.
\end{remark}

\begin{problem}[Quantitative finite-stage contraction]
\label{prob:quantitative-incompressible-contraction}
Find geometric bounds for the enlarged window in
\Cref{prop:incompressible-finite-stage-null}.  In particular, determine
which additional hypotheses allow exchange and corner-rounding operations
to control total area, relative reach, boundary collars, and simultaneous
disjointness through a finite null-homotopy certificate.  Distinguish
existence of such a certificate from an effective procedure for finding it.
\end{problem}

\begin{problem}[Geometric Kakimizu persistence]
\label{prob:kakimizu-persistence}
Compute or estimate
\[
  h_{\Delta,\tau},\quad
  \Lambda_{\mathrm{first}}^{\Delta,\tau}(K),\quad
  \Lambda_{\mathrm{conn}}^{\Delta,\tau}(K),\quad
  \Lambda_{\mathrm{MS}}^{\Delta,\tau}(K)
\]
for natural classes of knots.  In particular, determine how these quantities
are related to fiberedness, satellite structure, and density--compression
profiles.
\end{problem}

\begin{example}[Concrete test cases and a trefoil window]
The following examples are meant as computable bounds rather than sharp
optimizations.

\emph{The trefoil.}  Let \(K=3_1\).  The trefoil is fibered and its minimal
genus Seifert surface is unique up to isotopy, so the Kakimizu complex is a
single vertex.  Let \(\mathfrak S_{\mathrm{taut}}\) denote this taut surface
class, and set
\[
  \beta_{\mathrm{taut}}(3_1;\Delta,\tau)
  :=
  \beta_{\mathfrak S_{\mathrm{taut}}}(3_1;\Delta,\tau).
\]
Denne--Diao--Sullivan's quadrisecant estimates give the rigorous lower
bound
\[
  31.32 < \Rop(3_1)
\]
in the radius normalization of \Cref{rem:ropelength-normalization} (their
published constant is \(15.66\) in diameter units)
\cite{DenneDiaoSullivan}.  Hence every visible taut-surface level satisfies
\[
  31.32
  <
  \Rop(3_1)
  \leq
  \beta_{\mathrm{taut}}(3_1;\Delta,\tau).
\]
Conversely, choose any explicit smooth thick trefoil representative
\(\gamma_0\) and its fiber surface \(F_0\subset E(\gamma_0)\), scale so that
\(\Thi(\gamma_0)=1\), and set
\[
  L_0=\Len(\gamma_0),\qquad A_0=\Area(F_0),\qquad
  \tau_0=\Thi(F_0)>0.
\]
Then for every \(\Delta\geq A_0\) and \(0<\tau\leq\tau_0\),
\[
  31.32
  <
  \beta_{\mathrm{taut}}(3_1;\Delta,\tau)
  \leq
  L_0.
\]
If \(\gamma_0\) is taken from a numerical tight-trefoil model with
\(L_0<32.75\), as in the polygonal tightening computations of
\cite{BaranskaPieranskiPrzybylRawdon}, this gives the concrete window
\[
  31.32
  <
  \beta_{\mathrm{taut}}(3_1;\Delta,\tau)
  <
  32.75
\]
for the corresponding measured surface budget \((A_0,\tau_0)\).  Since the
visible Kakimizu layer has one vertex in this case, the first, connected, and
completion levels agree whenever the taut surface is visible:
\[
  \Lambda_{\mathrm{first}}^{\Delta,\tau}(3_1)
  =
  \Lambda_{\mathrm{conn}}^{\Delta,\tau}(3_1)
  =
  \Lambda_{\mathrm{MS}}^{\Delta,\tau}(3_1)
  =
  \beta_{\mathrm{taut}}(3_1;\Delta,\tau).
\]
Thus even the first nontrivial example produces a genuine quantitative
visibility interval, rather than merely an existence statement.  Improving the
interval means either improving the ropelength bounds for the trefoil or
improving the simultaneous bounded-geometry realization of the trefoil and its
fiber surface.

\emph{Other fibered knots with one visible taut surface.}  Let \(K\) be a
fibered knot whose minimal genus Seifert surface is unique up to isotopy, such
as the figure-eight knot.  The same argument gives
\[
  \Rop(K)
  \leq
  \beta_{\mathrm{taut}}(K;\Delta,\tau)
  \leq
  L_0
\]
whenever \(\Delta\geq A_0\) and \(0<\tau\leq\tau_0\) for a chosen thick model
\((\gamma_0,F_0)\).  The Kakimizu thresholds again collapse to the same single
visible taut-surface level.

\emph{Torus knots.}  For the torus knot \(T(p,q)\), with
\(\gcd(p,q)=1\), the standard fiber has
genus \((p-1)(q-1)/2\).  A thickened closed braid representative and its fiber
surface give constants \(L_{p,q},A_{p,q},\tau_{p,q}>0\).  Hence
\[
  \beta_{\mathrm{taut}}(T(p,q);\Delta,\tau)
  \leq L_{p,q}
\]
for \(\Delta\geq A_{p,q}\) and \(\tau\leq\tau_{p,q}\).  The dependence of these
bounds on \(p\) and \(q\) is a concrete test of the density--compression
frontier: increasing braid complexity raises the length budget, while the
fiber genus and area raise the surface budget.

\emph{A satellite test.}  Let \(K\) be a cable or more general satellite knot
with companion torus \(T\).  If \((\gamma_0,T_0)\) is one bounded-geometry
realization of the companion torus, then
\[
  \beta_{\mathrm{JSJ}}(K;\Delta,\tau)
  \leq \Len(\gamma_0)
\]
for \(\Delta\geq\Area(T_0)\) and \(\tau\leq\Thi(T_0)\).  Twisting a Seifert
surface around the companion torus gives a family of Kakimizu vertices whose
normal coefficients, and hence their visible area at fixed thickness, grow
with the amount of twisting.  Consequently, fixed \((\Delta,\tau)\) sees only
a finite initial portion of this direction, while the full Kakimizu complex may
be infinite.  This is precisely the kind of phenomenon that the visible
Kakimizu filtration is designed to record.
\end{example}


\subsection{Knot--surface trade-off frontier}

The trade-off frontier continues the ideal-pair viewpoint introduced above.
It measures how the surface-side ideal area changes when one allows the knot
representative to move away from the ropelength ideal stratum.

The coupled filtration can also be viewed as an optimization device.  It asks
how short and thick a knot representative can be while an essential surface in
its exterior is simultaneously realized with small area and positive
thickness.  This leads to the following frontier.

\begin{definition}[Knot--surface trade-off frontier]
\label{def:trade-off-frontier}
Fix a knot type \(K\), an essential surface class \([F]\), and a surface
thickness scale \(\tau>0\).  Define
\[
  \Delta_{\min}^{\tau}(\Lambda;[F])
  =
  \inf\{\Area(F')\mid
  (\gamma,F')\in
  \calZ_{\Lambda,\infty,\tau}^{\mathrm{ess}}(K),\ F'\in [F]\}.
\]
The graph of \(\Delta_{\min}^{\tau}(\Lambda;[F])\), when the infimum is finite,
is called the knot--surface trade-off frontier for the class \([F]\).
\end{definition}

It is often useful to use the scale-invariant normalized version
\[
  \mathfrak F^\tau(\Lambda;[F])
  =
  \frac{\Delta_{\min}^{\tau}(\Lambda;[F])}{\tau^2}.
\]
This records the least normalized surface area cost at surface thickness
scale \(\tau\).

\begin{remark}[Programmatic statement: Trade-off principle]
The geometry of a knot exterior is reflected in the frontier between
ropelength efficiency of knot representatives and area--thickness efficiency
of essential surfaces in their exteriors.  Sharp changes in this frontier are
expected to detect significant topological features of the knot exterior.
\end{remark}

\subsection{Trade-off frontiers and density--compression profiles}

The knot--surface trade-off frontier is not merely an area-minimization
function.  It records how the ropelength budget of the knot and the
area--thickness budget of the surface are distributed between density and
compression.

Let \(D(\gamma)\) be a length-scale measuring the spatial size of the curve
\(\gamma\), such as its diameter or minimal enclosing radius.  Under the
normalization \(\Thi(\gamma)=1\), the curve-side factorization gives
\[
  \Len(\gamma)
  =
  \rho_D(\gamma)\,\CRad_D(\gamma).
\]
Thus the same length level \(\Lambda\) can arise from different geometric
profiles: a representative may be dense in a small region, or more diffuse in
a larger region.

Similarly, let \(D(F)\) be a length-scale measuring the spatial size of the
surface \(F\), such as its diameter or minimal enclosing radius.  Since
\(\Area(F)\) has length dimension two, the surface-side factorization takes
the form
\[
  \frac{\Area(F)}{\Thi(F)^2}
  =
  \rho_D^{\mathrm{surf}}(F)
  \left(\CRad_D^{\mathrm{surf}}(F)\right)^2.
\]
In the level \(\calZ_{\Lambda,\Delta,\tau}(K)\), this normalized area is
bounded by
\[
  \frac{\Area(F)}{\Thi(F)^2}
  \leq
  \frac{\Delta}{\tau^2}.
\]
Consequently, the frontier
\[
  \Delta_{\min}^{\tau}(\Lambda;[F])
\]
measures not only the least area at which the class \([F]\) appears, but also
how the available geometric room is divided between surface density and
surface compression, while the knot representative itself has its own
density--compression profile.

\begin{remark}[Programmatic statement: Density--compression interpretation of the frontier]
\label{rem:frontier-density-compression}
Sharp changes in the knot--surface trade-off frontier should reflect changes
in the density--compression profiles of knot representatives and essential
surfaces.  In this sense, the frontier is a geometric shadow of how positional
room in the knot exterior is converted into surface area, thickness, and
compression.
\end{remark}

\section{Geometric density, compression, and packing}
\label{sec:density-compression}

This section develops the density--compression viewpoint of the companion
paper \cite{OzawaDensityCompression} on the surface and pair sides.  The core
finiteness and recognition theorems do not depend on any optimized density
invariant; the point here is structural.  It explains why the
length--area--thickness budget is a packing constraint, builds the surface
analogue of the curve-side density, compression radius, and packing ratio, and
shows that positive reach makes the constrained least-area ideal, and the
area-truncated density and compression ideals, attained
(\Cref{prop:surface-attainment}) --- exactly where the curve side requires
approximation hypotheses.  For the untruncated ratio functionals an area
bound along a minimizing sequence remains a genuine hypothesis.

\subsection{The curve-side factorization}
\label{subsec:curve-density}

Let \(D(\gamma)\) be a \emph{size functional}: a Euclidean-invariant,
scale-covariant, positive function of an embedded curve, such as the diameter,
the minimal enclosing radius \(R_{\min}\), the radius of gyration, an
\(L^p\)-radial size, or a regularized convex-hull size
\cite{OzawaDensityCompression}.  Define the \(D\)-density, the \(D\)-compression
radius, and the \(D\)-packing ratio by
\[
  \rho_D(\gamma)=\frac{\Len(\gamma)}{D(\gamma)},
  \qquad
  \CRad_D(\gamma)=\frac{D(\gamma)}{\Thi(\gamma)},
  \qquad
  \operatorname{Pack}_D(\gamma)=\frac{\Thi(\gamma)}{D(\gamma)}=\frac{1}{\CRad_D(\gamma)}.
\]
At the representative level these cancel to the ropelength,
\[
  \Rop(\gamma)=\frac{\Len(\gamma)}{\Thi(\gamma)}=\rho_D(\gamma)\CRad_D(\gamma),
\]
so under \(\Thi(\gamma)=1\) the ropelength bound is the product constraint
\(\rho_D(\gamma)\CRad_D(\gamma)\leq\Lambda\).  After optimizing over the knot
type the identity becomes only an inequality,
\[
  \Rop(K)\geq\rho_D(K)\CRad_D(K),
  \qquad
  \rho_D(K)=\inf_{\gamma\in K}\rho_D(\gamma),
  \quad
  \CRad_D(K)=\inf_{\gamma\in K}\CRad_D(\gamma),
\]
and it is strict in general: the density, compression, and ropelength ideal
strata --- the minimizers of the three functionals --- need not share a common
minimizing sequence \cite{OzawaDensityCompression}.  The factorization alone
therefore produces no new ropelength bound; it isolates two factors whose
separate behaviour can be studied.  \Cref{fig:density-compression} shows the
two factors on the curve side and their surface analogues.

\begin{figure}[t]
\centering
\begin{tikzpicture}[x=1cm,y=1cm,>=Latex,
  panel/.style={draw, rounded corners=2pt},
  title/.style={font=\footnotesize\bfseries},
  note/.style={font=\scriptsize, align=center},
  fbox/.style={draw, rounded corners=2pt, align=center, inner sep=3.5pt,
               font=\scriptsize, fill=white}]
  \draw[panel] (0,-0.55) rectangle (6.3,4.55);
  \node[title] at (3.15,4.28) {curve side};
  \draw[dashed, thin] (2.55,2.45) circle (1.45);
  \begin{scope}[shift={(2.55,2.45)}, scale=0.4]
    \foreach \a in {90,210,330}{
      \draw[white, line width=6pt]
        plot[variable=\t, domain=\a:\a+140, samples=70, smooth]
        ({sin(\t)+2*sin(2*\t)}, {cos(\t)-2*cos(2*\t)});
      \draw[black!15, line width=4.6pt]
        plot[variable=\t, domain=\a:\a+140, samples=70, smooth]
        ({sin(\t)+2*sin(2*\t)}, {cos(\t)-2*cos(2*\t)});
      \draw[thin]
        plot[variable=\t, domain=\a:\a+140, samples=70, smooth]
        ({sin(\t)+2*sin(2*\t)}, {cos(\t)-2*cos(2*\t)});
    }
  \end{scope}
  \draw[->, thin] (2.55,2.45) -- (3.58,3.47);
  \node[note, fill=white, inner sep=1pt] at (3.35,3.05) {$D(\gamma)$};
  \node[note] at (5.35,3.9) {$\Thi=1$};
  \draw[->, thin] (5.0,3.75) -- (4.05,3.10);
  \node[fbox] at (3.15,0.10)
    {$\rho_D=\Len/D$,\quad $\CRad_D=D/\Thi$\\[0.5mm]
     $\Rop=\rho_D\,\CRad_D$};
  \draw[panel] (6.9,-0.55) rectangle (13.2,4.55);
  \node[title] at (10.05,4.28) {surface side};
  \draw[dashed, thin] (9.55,2.45) circle (1.45);
  \fill[black!12]
    (8.35,2.05) .. controls (8.9,2.85) and (9.4,1.75) .. (10.0,2.55)
    .. controls (10.45,3.12) and (10.75,2.35) .. (10.72,2.18)
    -- (10.66,1.92)
    .. controls (10.35,2.05) and (10.4,2.62) .. (10.05,2.25)
    .. controls (9.5,1.5) and (8.95,2.6) .. (8.42,1.83) -- cycle;
  \draw[thick]
    (8.35,2.05) .. controls (8.9,2.85) and (9.4,1.75) .. (10.0,2.55)
    .. controls (10.45,3.12) and (10.75,2.35) .. (10.72,2.18);
  \draw[thick]
    (8.42,1.83) .. controls (8.95,2.6) and (9.5,1.5) .. (10.05,2.25)
    .. controls (10.4,2.62) and (10.35,2.05) .. (10.66,1.92);
  \draw[->, thin] (9.55,2.45) -- (10.58,3.47);
  \node[note, fill=white, inner sep=1pt] at (10.38,3.05) {$D(F)$};
  \node[note] at (12.3,3.9) {$\Thi(F)\geq\tau$};
  \draw[->, thin] (11.9,3.72) -- (10.55,2.35);
  \node[fbox] at (10.05,0.10)
    {$\rho^{\mathrm{surf}}_D=\Area/D^2$,\quad
     $\CRad^{\mathrm{surf}}_D=D/\Thi$\\[0.5mm]
     $\Area/\Thi^2=\rho^{\mathrm{surf}}_D
       \bigl(\CRad^{\mathrm{surf}}_D\bigr)^2$};
\end{tikzpicture}
\caption{The density--compression factorization.  Left: a unit-thickness
representative inside its minimal enclosing size \(D(\gamma)\) (dashed).  The
density \(\rho_D\) measures how much length is packed per unit size, the
compression radius \(\CRad_D\) measures how large the configuration is
relative to its thickness, and their product is the ropelength.  Right: the
surface analogue, with area in the role of length and the squared size in the
role of the linear scale; the normalized area budget
\(\Area/\Thi^2\leq\Delta/\tau^2\) factors in the same way.  The ideal strata
of the three functionals need not coincide, which is why the factorization
isolates two separately studied factors rather than producing a new
ropelength bound.}
\label{fig:density-compression}
\end{figure}
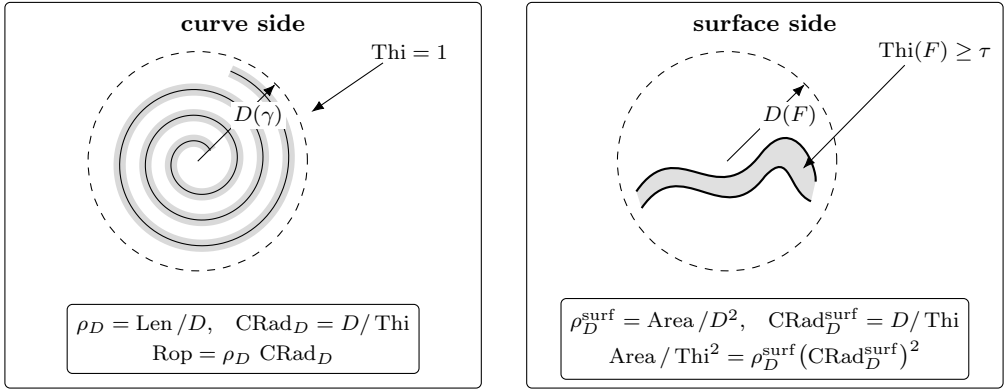

\subsection{The surface-side factorization}
\label{subsec:surface-density}

The surface side is parallel, with area in the role of length and the squared
spatial scale in the role of the linear scale.  Let \(D(F)\) be a size
functional for surfaces, for instance \(\operatorname{diam}(F)\) or the minimal
enclosing radius \(R_{\min}(F)\).  Define
\[
  \rho_D^{\mathrm{surf}}(F)=\frac{\Area(F)}{D(F)^2},
  \qquad
  \CRad_D^{\mathrm{surf}}(F)=\frac{D(F)}{\Thi(F)},
  \qquad
  \operatorname{Pack}_D^{\mathrm{surf}}(F)=\frac{\Thi(F)}{D(F)}.
\]
The representative-level decomposition is
\[
  \frac{\Area(F)}{\Thi(F)^2}
  =\rho_D^{\mathrm{surf}}(F)\bigl(\CRad_D^{\mathrm{surf}}(F)\bigr)^2 ,
\]
so the normalized area budget is \(\Area(F)/\Thi(F)^2\leq\Delta/\tau^2\).  Fix a
representative \(\gamma\) and an essential class \([F]\) in \(E(\gamma)\), and
minimize over its \(\tau\)-thick members:
\[
  \rho_D^{\mathrm{surf}}([F];\gamma)=\inf\rho_D^{\mathrm{surf}}(F'),
  \qquad
  \CRad_D^{\mathrm{surf}}([F];\gamma)=\inf\CRad_D^{\mathrm{surf}}(F'),
\]
\[
  \mathfrak a_\tau([F];\gamma)=\inf\frac{\Area(F')}{\Thi(F')^2},
\]
the infima over essential \(F'\in[F]\) with \(\Thi(F')\geq\tau\).

\begin{proposition}[Optimized surface inequality]
\label{prop:surface-optimized-inequality}
For every essential class \([F]\) with nonempty \(\tau\)-admissible subclass,
\[
  \mathfrak a_\tau([F];\gamma)
  \geq
  \rho_D^{\mathrm{surf}}([F];\gamma)\bigl(\CRad_D^{\mathrm{surf}}([F];\gamma)\bigr)^2,
\]
with equality if and only if there is a \(\tau\)-thick minimizing sequence in
\([F]\) that simultaneously optimizes the normalized area, the surface density,
and the surface compression radius.
\end{proposition}

\begin{proof}
For each admissible \(F'\) the representative-level decomposition gives
\[
  \frac{\Area(F')}{\Thi(F')^2}
  =\rho_D^{\mathrm{surf}}(F')\bigl(\CRad_D^{\mathrm{surf}}(F')\bigr)^2
  \geq\rho_D^{\mathrm{surf}}([F];\gamma)\bigl(\CRad_D^{\mathrm{surf}}([F];\gamma)\bigr)^2 ;
\]
take the infimum over \(F'\).  The equality clause is the surface analogue of the
curve equality criterion \cite{OzawaDensityCompression}: if the product tends to
the product of the two positive infima, neither positive factor can diverge, so
both tend to their infima.
\end{proof}

We call the minimizers of \(\rho_D^{\mathrm{surf}}\),
\(\CRad_D^{\mathrm{surf}}\), and \(\Area/\Thi^2\) the surface density,
compression, and area ideal strata of \([F]\) relative to \(\gamma\); as on the
curve side they need not coincide.

\subsection{Attainment via positive reach}
\label{subsec:surface-attainment}

On the curve side, existence of density and compression minimizers, and their
polygonal approximation, require analytic hypotheses verified only for
particular size functionals \cite{OzawaDensityCompression}.  On the surface
side positive reach supplies the compactness; what it does not supply by
itself is a uniform area bound along a minimizing sequence, which must be
assumed for the ratio functionals.

\begin{proposition}[Attainment of the surface density and compression ideals]
\label{prop:surface-attainment}
Let \(D\) be a size functional continuous under \(C^1\)-convergence of thick
surfaces (for example \(\operatorname{diam}\) or \(R_{\min}\), which are
continuous under Hausdorff convergence).  Fix \(\gamma\) and an
essential class \([F]\) with nonempty \(\tau\)-admissible subclass.
\begin{enumerate}[label=(\roman*),leftmargin=2em]
\item If the infimum defining the normalized area
\(\mathfrak a_\tau([F];\gamma)\) admits a minimizing sequence with uniformly
bounded area, then it is attained.  Equivalently, its area-truncated version,
obtained by restricting the competition to \(\Area\leq\Delta\), is attained
for every \(\Delta\).
\item The same conclusion holds for the infimum defining
\(\rho_D^{\mathrm{surf}}([F];\gamma)\) or
\(\CRad_D^{\mathrm{surf}}([F];\gamma)\): a uniformly area-bounded minimizing
sequence gives an attained infimum, and every area-truncated ideal is attained.
\end{enumerate}
The area-boundedness proviso in (ii) is not removable in general: a
\(\tau\)-thick representative may grow a long finger of extrinsic diameter
\(D\) at area cost of order \(\tau D\), so a minimizing sequence for the ratio
\(\Area/D^2\) can have unbounded area, and the compactness lemma does not
apply to it.
\end{proposition}

\begin{proof}
For either clause, take a minimizing sequence \(F_n'\in[F]\) with
\(\Thi(F_n')\geq\tau\) and \(\Area(F_n')\leq\Delta\) for some fixed
\(\Delta\), as provided by the hypothesis or by the area truncation.  The
surfaces are
essential, hence anchored by \Cref{lem:anchoring}, so by
\Cref{lem:compactness-thick-surfaces} a subsequence converges in \(C^1\) to a
\(\tau\)-thick limit \(F_\infty'\), isotopic to the tail and hence in \([F]\)
and essential.  Since \(D\) is \(C^1\)-continuous, area is lower
semicontinuous, and the thickness constraint passes to the limit in the
favorable direction (\(\Thi(F_\infty')\geq\tau\) by
\Cref{lem:reach-hausdorff-closed} and the chart limits), \(F_\infty'\)
realizes the infimum.
\end{proof}

\begin{remark}[Surface analogue of polygonal approximation]
\label{rem:surface-vs-polygonal}
\Cref{prop:surface-attainment} is the surface counterpart of the polygonal
approximation theorem for compression radii \cite{OzawaDensityCompression},
which on the curve side holds only under approximation, compactness, and
lower-semicontinuity hypotheses established there for \(D=\operatorname{diam}\)
and \(D=R_{\min}\).  On the surface side the compactness and
lower-semicontinuity hypotheses are subsumed by the
single positive-reach compactness of \Cref{lem:compactness-thick-surfaces}, so
the area-truncated surface density and compression ideals exist for every
\(C^1\)-continuous size functional; only the uniform area bound along a
minimizing sequence remains a genuine hypothesis for the untruncated ratio
functionals, as the finger construction in \Cref{prop:surface-attainment}
shows.  This is otherwise the same phenomenon as in the recognition theory:
reach makes the two-dimensional object directly compact.
\end{remark}

\subsection{Coupled profiles and the trade-off frontier}
\label{subsec:coupled-profiles}

A pair \((\gamma,F)\) carries two density--compression profiles at once, the
curve profile \((\rho_D(\gamma),\CRad_D(\gamma))\) and the surface profile
\((\rho_D^{\mathrm{surf}}(F),\CRad_D^{\mathrm{surf}}(F))\), and the normalized
budget splits as
\[
  \frac{\Len(\gamma)}{\Thi(\gamma)}+\frac{\Area(F)}{\Thi(F)^2}
  =\rho_D(\gamma)\CRad_D(\gamma)
  +\rho_D^{\mathrm{surf}}(F)\bigl(\CRad_D^{\mathrm{surf}}(F)\bigr)^2
  \leq\Lambda+\frac{\Delta}{\tau^2}.
\]
The knot--surface trade-off frontier \(\Delta_{\min}^{\tau}(\Lambda;[F])\) of
\Cref{def:trade-off-frontier} is the Pareto frontier of this coupled
optimization: it records how the available room is divided between curve
density, curve compression, surface density, and surface compression, and sharp
changes in it reflect changes in these four profiles, as anticipated in
\Cref{rem:frontier-density-compression}.

\subsection{Compatibility with finite-resolution bounds}
\label{subsec:dc-finite-resolution}

\begin{proposition}[Degree of freedom through density--compression]
\label{prop:dof-density-compression-surface}
Let \((\gamma,F)\) be a thick pair with \(\Thi(\gamma)=1\) and
\(\Thi(F)\geq\tau\).  For \(0<\varepsilon\leq c\min\{1,\tau\}\),
\[
  \DoF_\varepsilon(\gamma,F)
  \leq
  C\left(
  \frac{\rho_D(\gamma)\CRad_D(\gamma)}{\varepsilon}
  +
  \frac{\rho_D^{\mathrm{surf}}(F)
  \left(\CRad_D^{\mathrm{surf}}(F)\right)^2\Thi(F)^2}{\varepsilon^2}
  \right).
\]
In particular, if \((\gamma,F)\in\calZ_{\Lambda,\Delta,\tau}(K)\), then
\[
  \DoF_\varepsilon(\gamma,F)
  \leq
  C\left(
  \frac{\Lambda}{\varepsilon}+\frac{\Delta}{\varepsilon^2}
  \right).
\]
\end{proposition}

\begin{proof}
Use the finite-resolution estimate of \Cref{sec:dof},
\[
  \DoF_\varepsilon(\gamma,F)
  \leq C\left(\frac{\Len(\gamma)}{\varepsilon}
  +\frac{\Area(F)}{\varepsilon^2}\right),
\]
and substitute the two density--compression factorizations.  The final bound
follows from \(\Len(\gamma)\leq\Lambda\) and \(\Area(F)\leq\Delta\).
\end{proof}

\begin{remark}[Packing interpretation and scope]
\label{rem:dc-scope}
The packing ratios \(\operatorname{Pack}_D=\Thi/D\) and
\(\operatorname{Pack}_D^{\mathrm{surf}}=\Thi/D\) record how efficiently a tube or
a thick sheet is stored in its enclosing scale; larger packing means smaller
compression.  As on the curve side, the surface factorization proves no new area
lower bound by itself: it isolates factors, and a genuine bound requires
independent control of surface density or surface compression.  A natural
benchmark question, parallel to the crossing-number benchmarks of
\cite{OzawaDensityCompression}, is which topological complexity of \([F]\) ---
genus, number of boundary components, or normal-coordinate size
(\Cref{prop:conditional-normal-coefficient-bound}) --- forces large surface
density versus large surface compression.  The density--compression viewpoint
does not enter the hypotheses of the core theorems, which use only
\((\Lambda,\Delta,\tau,\varepsilon)\).
\end{remark}

\section{Finite recognition and essential surface data}
\label{sec:finite-recognition}

Finite knot recognition uses the distinction between a lifted multigraph,
which records projection-fibre components and geometric movies, and a
monotone diagram image, which carries finite recognition certificates
\cite{OzawaFiniteRecognition}. We adapt that distinction to
knot--surface pairs.  The resulting two-layer recognition scheme is shown in
\Cref{fig:two-witnesses}.

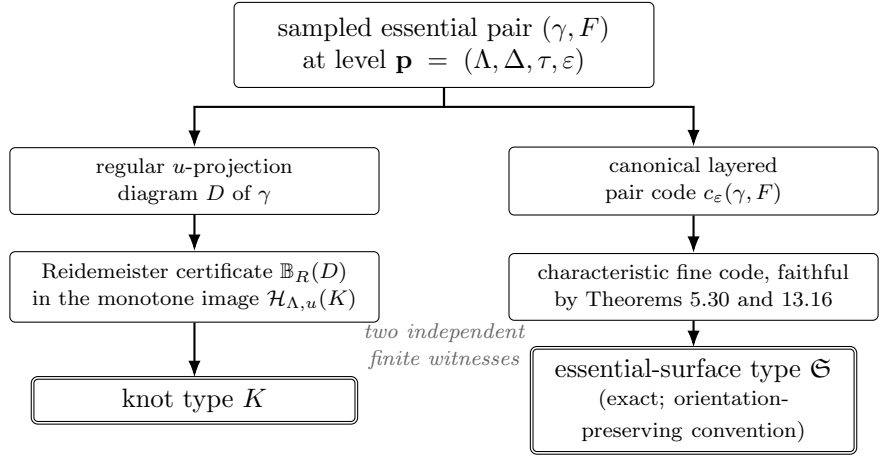
\begin{figure}[t]
\centering
\begin{tikzpicture}[x=1cm,y=1cm,
  box/.style={draw, rounded corners=2pt, align=center, inner sep=4pt,
              text width=46mm, font=\scriptsize},
  src/.style={draw, rounded corners=2pt, align=center, inner sep=5pt,
              text width=52mm, font=\small},
  result/.style={draw, double, rounded corners=2pt, align=center,
                 inner sep=4pt, text width=40mm, font=\small},
  lbl/.style={font=\scriptsize\itshape, text=black!60, align=center},
  arrow/.style={-{Latex[length=2.2mm]}, thick}]
  \node[src] (pair) at (6.6,5.2)
    {sampled essential pair $(\gamma,F)$ at level
     $\mathbf p=(\Lambda,\Delta,\tau,\varepsilon)$};
  \node[box] (diag) at (3.3,3.4)
    {regular $u$-projection diagram $D$ of $\gamma$};
  \node[box] (cert) at (3.3,2.0)
    {Reidemeister certificate $\mathbb B_R(D)$ in the monotone image
     $\calH_{\Lambda,u}(K)$};
  \node[result] (knot) at (3.3,0.5) {knot type $K$};
  \node[box] (code) at (9.9,3.4)
    {canonical layered pair code $c_\varepsilon(\gamma,F)$};
  \node[box] (char) at (9.9,2.0)
    {characteristic fine code, faithful by
     \Cref{thm:faithfulness,thm:surface-recognition}};
  \node[result] (surf) at (9.9,0.5)
    {essential-surface type $\mathfrak S$\\[-0.5mm]
     {\scriptsize(exact; orientation-preserving convention)}};
  \draw[arrow] (pair.south) -- ++(0,-0.25) -| (diag.north);
  \draw[arrow] (pair.south) -- ++(0,-0.25) -| (code.north);
  \draw[arrow] (diag) -- (cert);
  \draw[arrow] (cert) -- (knot);
  \draw[arrow] (code) -- (char);
  \draw[arrow] (char) -- (surf);
  \node[lbl] at (6.6,1.22) {two independent\\finite witnesses};
\end{tikzpicture}
\caption{The two unconditional recognition layers of
\Cref{thm:two-unconditional-layers}.  From one sampled bounded-geometry pair,
the projection diagram yields a finite Reidemeister certificate that
recognizes the knot type in the monotone diagram image, while the layered
pair code is characteristic for the carried essential-surface type.  The two
witnesses are logically independent: the first lives in diagram
combinatorics, the second in the faithful finite-resolution geometry of the
pair.  Lifted multigraphs, which retain projection- and code-fiber
components, refine both layers but are not needed for recognition.}
\label{fig:two-witnesses}
\end{figure}

\subsection{Classical verification layers}
\label{subsec:classical-verification-layers}

The recognition theorem uses the sampled geometric pair itself as the
certificate.  After that pair has been reconstructed, classical algorithms
can be applied as independent verification layers.  A controlled subdivision
gives a triangulated exterior in which the surface may be normalized;
Jaco--Oertel type procedures can then test incompressibility and
boundary-incompressibility, while efficient triangulations and crushing can
simplify the ambient input \cite{JacoOertel,JacoRubinstein0Efficient}.  Dually,
the same finite cell data determine a decorated spine, so a classical
spine-move certificate is a possible alternative discrete witness
\cite{MatveevBook}.

These procedures are complementary to, rather than ingredients of,
\Cref{thm:faithfulness}.  Positive-reach reconstruction proves that equal fine
codes determine the same embedded pair even while the knot exterior moves.
Normal-surface and spine procedures begin only after a discrete ambient
structure has been chosen and address topological verification of the
reconstructed object.

\subsection{Projection-framed normalization for pairs}

Fix a direction \(u\in S^2\).  Let \(\Sim_u^+(\R^3)\) be the group of
orientation-preserving similarities whose rotational part fixes \(u\).
Projection to \(u^\perp\) is well defined on the quotient by this group, unlike
on a quotient by all rotations.  Following \cite{OzawaFiniteRecognition}, set
\[
  \calX_{\Lambda,u}(K)
  =
  \{\gamma\text{ of type }K:\Rop(\gamma)\leq\Lambda\}/\Sim_u^+(\R^3),
\]
and let \(\calX^<_{\Lambda,u}(K)\) denote the strict sublevel
\(\Rop(\gamma)<\Lambda\).  Normalizing \(\Thi(\gamma)=1\) identifies these
spaces with unit-thickness representatives modulo translations and rotations
about the \(u\)-axis.

For an unnormalized pair define the scale-invariant surface quantities
\[
  \widehat{\Area}_\gamma(F)
  =\frac{\Area(F)}{\Thi(\gamma)^2},
  \qquad
  \widehat{\Thi}_\gamma(F)
  =\frac{\Thi(F)}{\Thi(\gamma)}.
\]
They agree with \(\Area(F)\) and \(\Thi(F)\) on the unit-thickness slice.

\begin{definition}[Projection-framed filtered pair space]
For fixed \(u\), define
\[
  \calZ^{\mathrm{ess}}_{\Lambda,\Delta,\tau,u}(K)
  =
  \left\{
  (\gamma,F)\ \middle|\
  \begin{array}{l}
  \gamma\text{ represents }K,\quad \Rop(\gamma)\leq\Lambda,\\
  F\subset E(\gamma)\text{ is essential},\\
  \widehat{\Area}_\gamma(F)\leq\Delta,\quad
  \widehat{\Thi}_\gamma(F)\geq\tau
  \end{array}
  \right\}/\Sim_u^+(\R^3).
\]
The fully strict space
\(\calZ^{\mathrm{ess},<}_{\Lambda,\Delta,\tau,u}(K)\) is obtained by replacing
all three geometric inequalities by strict inequalities.  Forgetting the
projection frame maps this space to
\(\calZ^{\mathrm{ess}}_{\Lambda,\Delta,\tau}(K)\).
\end{definition}

\begin{remark}[Why strict levels are retained]
A path on the boundary of a closed geometric sublevel need not admit a generic
perturbation within exactly the same closed level.  On a fully strict level, a
compact admissible family has uniform slack in ropelength, area, and relative
thickness.  Relative smoothing and transversality may therefore be performed
without leaving the level.  Closed-level information is recovered by
arbitrarily small right relaxation.  This is the pair-space counterpart of the
strict-sublevel method used in \cite{OzawaFiniteRecognition}; no point-set
claim at an individual critical closed level is needed below.
\end{remark}

\subsection{The lifted graph and the monotone diagram--code image}

Let \(G_S(K)\) be the classical \(S^2\)-Reidemeister multigraph.  Parallel
edges are retained and each edge is decorated by its Reidemeister type.  A
Barbensi--Celoria certificate is a complete rooted radius ball
\(\mathbb B_R(D)\), including move types, multiplicities, and the ambient
valence of every vertex \cite{BarbensiCeloria,OzawaFiniteRecognition}.
Certificates are compared as abstract decorated
rooted multigraphs, with the diagram labels forgotten.

For the knot-only projection-framed space, \cite{OzawaFiniteRecognition}
defines the lifted multigraph
\(\calG^{\mathrm{lift}}_{\Lambda,u}(K)\), whose vertices remember connected
components of regular-projection fibers, and its diagram image
\(\calH_{\Lambda,u}(K)\subset G_S(K)\).  The lifted graphs need not be monotone
because fiber components can merge, whereas
\[
  \Lambda\leq\Lambda'
  \quad\Longrightarrow\quad
  \calH_{\Lambda,u}(K)\subset\calH_{\Lambda',u}(K).
\]
Recognition is therefore formulated in \(\calH\), while component persistence
is retained in \(\calG^{\mathrm{lift}}\).

Fix a concrete pair encoding at resolution \(\varepsilon\), write
\(c_\varepsilon(\gamma,F)\) for its layered code, and abbreviate
\[
  \mathbf p=(\Lambda,\Delta,\tau,\varepsilon).
\]

\begin{definition}[Lifted surface-decorated multigraph]
The lifted surface-decorated multigraph
\[
  \calG^{\mathrm{surf},\mathrm{lift}}_{\mathbf p,u}(K)
\]
has vertices \((D,c,C)\), where \(D\) is the regular \(u\)-projection of
\(\gamma\), \(c=c_\varepsilon(\gamma,F)\), and \(C\) is a connected component
of the corresponding projection--code fiber in
\(\calZ^{\mathrm{ess}}_{\Lambda,\Delta,\tau,u}(K)\).  Edges are represented by
admissible pair movies which are projection-generic on the knot coordinate and
which realize one specified Reidemeister wall crossing or one elementary
surface-code transition.
\end{definition}

\begin{definition}[Monotone surface diagram--code image]
The graph
\[
  \calH^{\mathrm{surf}}_{\mathbf p,u}(K)
\]
is obtained from
\(\calG^{\mathrm{surf},\mathrm{lift}}_{\mathbf p,u}(K)\)
by forgetting the fiber component \(C\) and retaining every realized decorated
vertex \((D,c)\) and every realized edge.  It is an increasing graph-valued
filtration in \(\Lambda\) and \(\Delta\), and a decreasing filtration in the
lower thickness cutoff \(\tau\).
\end{definition}

The finiteness of the vertex set of
\(\calH^{\mathrm{surf}}_{\mathbf p,u}(K)\) requires care.  The surface-code
coordinate is finite by
\Cref{thm:finite-resolution-finiteness-filtered-pairs}, but the diagram
coordinate is not automatically finite for a \emph{fixed} projection
direction: the Buck--Simon estimate bounds the \emph{average} crossing number
over all directions \cite{BuckSimon}, and an average bound does not control
one fixed direction.  Indeed, at fixed thickness and length the crossing
number in a fixed direction is genuinely unbounded; see
\Cref{rem:fixed-direction-unbounded}.  We therefore impose a quantitative
transversality margin on the projection.

\begin{definition}[Transversality margin]
\label{def:transversality-margin}
Fix \(\theta_0\in(0,\pi/2]\).  A direction \(u\in S^2\) is
\emph{\(\theta_0\)-transverse} for a unit-thickness representative \(\gamma\)
if
\begin{enumerate}[label=(\roman*),leftmargin=2em]
\item \(|\langle\gamma'(s),u\rangle|\leq\cos\theta_0\) for almost every \(s\)
(no near-\(u\) tangencies), and
\item the projection \(\pi_u(\gamma)\) is regular and every double point of
\(\pi_u(\gamma)\) has crossing angle at least \(\theta_0\).
\end{enumerate}
\end{definition}

\begin{lemma}[Fixed-direction crossing bound under a margin]
\label{lem:fixed-direction-crossing}
There is a universal constant \(C>0\) such that, for every unit-thickness
representative \(\gamma\) with \(\Len(\gamma)\leq\Lambda\) and every
\(\theta_0\)-transverse direction \(u\), the diagram \(D=\pi_u(\gamma)\)
satisfies
\[
  \operatorname{Cr}(D)\ \leq\ C\,\frac{\Lambda^2}{\theta_0^6}.
\]
\end{lemma}

\begin{proof}
Write \(p=\pi_u\circ\gamma\).  By margin clause (i) the planar speed satisfies
\(|p'(s)|\geq\sin\theta_0\), and since \(\Thi(\gamma)=1\) gives
\(|\gamma''|\leq1\) almost everywhere, the planar curve \(p\), reparametrized
by its arclength, has curvature at most
\(\kappa_0=\sin^{-2}\theta_0\) almost everywhere; its total length is at most
\(\Lambda\).

Set \(r_0=\theta_0/(4\kappa_0)\) and subdivide \(p\) into
\(M\leq\lceil2\Lambda/r_0\rceil\) consecutive arcs of planar length at most
\(r_0/2\).  We claim two such arcs \(\alpha,\beta\) can produce at most one
crossing of \(D\).  Suppose \(q_1,q_2\) are two double points on
\(\alpha\cap\beta\).  Parametrize both arcs by signed arclength from
\(q_1\); along
each arc the tangent direction varies by at most
\(\kappa_0\cdot r_0/2\leq\theta_0/8\).  Let \(n\) be a unit normal to the
tangent of \(\alpha\) at \(q_1\).  Then, for signed parameters \(s,t\) with
\(|s|,|t|\leq r_0\),
\[
  \bigl|\langle\alpha(s)-q_1,\ n\rangle\bigr|\leq \tfrac{\theta_0}{8}\,|s|,
  \qquad
  \bigl|\langle\beta(t)-q_1,\ n\rangle\bigr|
  \geq |t|\sin\theta_0-\tfrac{\theta_0}{8}\,|t|
  \geq \tfrac{\theta_0}{4}\,|t|,
\]
using the crossing-angle bound \(\geq\theta_0\) of margin clause (ii) at
\(q_1\) and \(\sin\theta_0\geq\theta_0/2\) on \((0,\pi/2]\).  Evaluating at
\(q_2=\alpha(s^*)=\beta(t^*)\) gives
\(\tfrac{\theta_0}{4}|t^*|\leq\tfrac{\theta_0}{8}|s^*|\), so
\(|t^*|\leq|s^*|/2\);
exchanging the roles of \(\alpha\) and \(\beta\) gives \(|s^*|\leq|t^*|/2\).
Hence \(s^*=t^*=0\) and \(q_2=q_1\), proving the claim.

Moreover a single arc cannot cross itself at all: its tangent direction varies
by at most \(\theta_0/8<\theta_0\), while a self-crossing would require two
branches meeting at angle at least \(\theta_0\).  Consequently
\[
  \operatorname{Cr}(D)\leq\binom{M}{2}\leq M^2
  \leq\Bigl(\frac{4\Lambda}{r_0}+1\Bigr)^{\!2}
  =\Bigl(16\Lambda\,\frac{\kappa_0}{\theta_0}+1\Bigr)^{\!2}
  \leq C\,\frac{\Lambda^2}{\theta_0^6},
\]
where \(\kappa_0/\theta_0=\theta_0^{-1}
\sin^{-2}\theta_0\leq C'\theta_0^{-3}\) on \((0,\pi/2]\).  No sharpness is
claimed.
\end{proof}

\begin{proposition}[Finite vertex set at fixed budget and margin]
\label{prop:finite-vertex-margin}
Fix \((\Lambda,\Delta,\tau,\varepsilon)\) and \(\theta_0\in(0,\pi/2]\), and
restrict the vertex set of \(\calH^{\mathrm{surf}}_{\mathbf p,u}(K)\) to
decorated vertices \((D,c)\) whose knot coordinate is represented by some
\(\gamma\) for which \(u\) is \(\theta_0\)-transverse.  Then this restricted
graph \(\calH^{\mathrm{surf}}_{\mathbf p,u,\theta_0}(K)\) has only finitely
many vertices, provided the finite-resolution encoding scheme is one of the
bounded schemes in \Cref{sec:dof}.
\end{proposition}

\begin{proof}
The pair codes form a finite set by
\Cref{thm:finite-resolution-finiteness-filtered-pairs}.  By
\Cref{lem:fixed-direction-crossing}, every diagram represented by a
\(\theta_0\)-transverse projection at the fixed ropelength level has crossing
number at most \(C\Lambda^2/\theta_0^6\).  There
are only finitely many spherical diagram types with bounded crossing number.
Thus only finitely many decorated pairs \((D,c)\) can occur.
\end{proof}

\begin{remark}[Why the margin cannot be dropped]
\label{rem:fixed-direction-unbounded}
Without a transversality margin, the fixed-direction vertex set can be
infinite even at a fixed budget.  The obstruction is geometric: two locally
parallel unit-thickness strands at vertical distance \(2\), one straight and
one oscillating transversally to \(u\) with amplitude \(\eta\) and wavelength
of order \(\sqrt\eta\), keep curvature at most one and add only \(O(\eta)\)
length per oscillation, yet each oscillation adds a fixed number of crossings
in the direction \(u\); letting \(\eta\to0\) produces representatives of
bounded length and thickness whose fixed-direction crossing numbers diverge,
while the crossing angles degenerate like \(\sqrt\eta\).  The average
crossing number stays bounded --- consistent with Buck--Simon
\cite{BuckSimon} --- because the directions in which the oscillation is
visible form a small solid angle.  For the same reason, a Markov-inequality
argument shows that for each representative a positive-measure set of
directions has crossing number at most twice the average, so finiteness in a
\emph{well-chosen} direction is automatic; what fails is uniformity over
representatives in one \emph{fixed} direction, and that is exactly what the
margin restores.  Structurally, the restriction to
\(\theta_0\)-transverse vertices is a filtration by an additional geometric
genericity parameter: \(\calH^{\mathrm{surf}}_{\mathbf p,u,\theta_0}(K)\)
increases as \(\theta_0\) decreases, and its union over \(\theta_0>0\) is the
full fixed-direction graph, which need not have a finite vertex set at any
fixed budget.
\end{remark}

\begin{remark}[Division of labor]
The lifted decorated graph is the appropriate object for strict-level
component and merge questions.  The monotone diagram--code image is the
appropriate object for births, visibility, and finite certificates.  Using a
single graph for both purposes obscures monotonicity in exactly the same way as
in the knot-only theory.
\end{remark}

\subsection{Unconditional knot certificates at finite ropelength}

We record the v2 theorem that supplies the knot-level input.

\begin{theorem}[Finite visibility and finite knot recognition
\cite{OzawaFiniteRecognition}]
\label{thm:knot-finite-visibility-v2}
Fix a knot type \(K\) and a projection direction \(u\).
\begin{enumerate}[label=(\roman*),leftmargin=2em]
\item Every finite specified submultigraph \(Q\subset G_S(K)\) is contained in
\(\calH_{\Lambda_Q,u}(K)\) for some finite \(\Lambda_Q\).
\item There is a universal computable nondecreasing function \(A(n)\) such
that a rooted radius-\(R\) ball centered at a \(c\)-crossing diagram is visible
by level \(A(c+2R)\).
\item Define
\[
  L_{\mathrm{char},u}(K)
  =\inf_Q\inf\{\Lambda:Q\subset\calH_{\Lambda,u}(K)\},
\]
where \(Q\) ranges over saturated Barbensi--Celoria characteristic balls.
Then \(L_{\mathrm{char},u}(K)<\infty\).  A saturated occurrence of such a
certificate in another knot's filtered diagram graph forces that knot to be
\(K\) or its mirror.
\end{enumerate}
\end{theorem}

\begin{remark}[Visibility versus coherent lifting]
Finite visibility in the monotone graph \(\calH\) requires no endpoint-coherent
choice of one spatial lift at each shared diagram.  If a coherent lift in
\(\calG^{\mathrm{lift}}\) is desired, v2 identifies the exact remaining
condition: the finite pattern must admit a face-consistent planar labeling.
This condition is automatic for trees and separated cube systems, but it plays
no role in the unconditional finiteness of
\(L_{\mathrm{char},u}(K)\).
\end{remark}

\begin{definition}[Finite knot--surface theory at a bounded scale]
At scale \((\Lambda,\Delta,\tau,\varepsilon,u)\), the available finite data
consist of
\begin{enumerate}[label=(\roman*),leftmargin=2em]
\item saturated Barbensi--Celoria rooted-ball certificates contained in
\(\calH_{\Lambda,u}(K)\); and
\item realized fine pair codes, together with their diagram--code incidences in
\(\calH^{\mathrm{surf}}_{\mathbf p,u}(K)\).
\end{enumerate}
This terminology asserts finiteness of the witnesses used for recognition, not
finiteness of either complete graph.
\end{definition}

The surface data may additionally record boundary slopes, JSJ incidence,
cutting pieces, taut Seifert surfaces, Kakimizu adjacency, position invariants,
and merge-scale information.  The novelty is that these labels are organized
inside one geometric filtration and can be attached to a knot certificate
whose finite visibility is already guaranteed by
\Cref{thm:knot-finite-visibility-v2}.

\subsection{Unconditional finite recognition of essential-surface types}
\label{subsec:surface-finite-recognition}

The knot and surface mechanisms are different.  Knot recognition uses a
sufficiently large saturated rooted ball in a complete Reidemeister multigraph.
Surface recognition uses direct geometric reconstruction: positive reach makes
a sufficiently fine pair code a separating certificate for the sampled ambient
pair.  Both conclusions are unconditional.

\begin{definition}[Essential-surface type]
\label{def:surface-type}
An \emph{essential-surface type} is an equivalence class of pairs \((K,F)\),
where \(K\subset S^3\) is a knot and \(F\subset E(K)\) is a properly embedded
essential surface or finite essential surface system, under ambient isotopy of
\(S^3\): two pairs \((K,F)\) and \((K',F')\) have the same type if some ambient
isotopy of \(S^3\) carries \(K\) to \(K'\) and \(F\) to \(F'\).  Allowing
orientation-reversing homeomorphisms as well gives the type up to mirror.  We
write \(\type(\gamma,F)\) for the essential-surface type of a geometric pair
\((\gamma,F)\).
\end{definition}

The faithfulness theorem was stated for two pairs of the same knot type, but its
proof never used that hypothesis: equality of codes gives controlled closeness
of both curves and both surfaces and hence an ambient isotopy of pairs.

\begin{corollary}[Cross-type faithfulness]
\label{cor:cross-type-faithfulness}
Fix \(\Delta,\tau>0\) and \(0<\varepsilon\leq c_1\min\{1,\tau\}\), and use the concrete
encoding scheme of \Cref{con:concrete-encoding} with
\(\delta_{\gamma}\leq\varepsilon\) and
\(\delta_F\leq\delta_{\mathrm{st}}/(8C_*)\), where
\(C_*:=\max\{4,2C_{\mathrm{enc}}\}\) and \(C_{\mathrm{enc}}\) is the constant
in \Cref{lem:code-proximity}.  Let
\((\gamma,F)\in\calZ^{\mathrm{ess}}_{\Lambda,\Delta,\tau}(K)\) and
\((\gamma',F')\in\calZ^{\mathrm{ess}}_{\Lambda',\Delta,\tau}(K')\) be any two
pairs, of possibly different knot types, with the same encoded
\(\varepsilon\)-type.  Then there is an ambient isotopy of \(S^3\) carrying
\((\gamma,F)\) to \((\gamma',F')\); in particular \(K'=K\) and
\(\type(\gamma,F)=\type(\gamma',F')\); with the canonical-code convention
of \Cref{rem:code-well-defined} the code is itself rigid-motion invariant, so
no residual normalization ambiguity remains.
\end{corollary}

\begin{proof}
Equal codes force
\[
  d_H(\gamma,\gamma')\leq C\varepsilon,
  \qquad
  d_H(F,F')\leq C\varepsilon,
\]
with matching quantized tangents and tangent planes, by
\Cref{lem:code-proximity}, after choosing raw realizations of the common canonical word exactly as in the
proof of \Cref{thm:faithfulness}.  The global lattice makes this comparison
literal even though the budgets differ.  With the same choices of \(c_1\), \(\delta_{\gamma}\), and \(\delta_F\),
these are the positional and angular
hypotheses of \Cref{lem:pair-stability}; its constants involve neither length
budget, so \(\Lambda\) and \(\Lambda'\) are irrelevant.  The lemma produces an ambient
isotopy of
\(S^3\) taking \((\gamma,F)\) to \((\gamma',F')\) through proper pairs.
Neither step used a common
knot type; the equality \(K'=K\) is an output, not an input.
\end{proof}

\begin{definition}[Characteristic surface code]
\label{def:characteristic-surface-code}
Fix a window \((\Delta,\tau,\varepsilon)\) with
\(\varepsilon\leq c_1\min\{1,\tau\}\).  An encoded \(\varepsilon\)-type \(c\) is
\emph{characteristic} for an essential-surface type \(T\) if every essential
geometric pair with \(\Thi(F)\geq\tau\), \(\Area(F)\leq\Delta\),
\(\Thi(\gamma)=1\), and encoded \(\varepsilon\)-type \(c\) has
\(\type(\gamma,F)=T\).  If one additionally identifies codes under ambient
reflections, the corresponding statement is recognition up to mirror.  By the
canonical-code convention of \Cref{rem:code-well-defined}, no normalization or
discretization ambiguity remains.
\end{definition}

Unlike the knot case, where the finite witness is a complete saturated
Reidemeister ball, here every realized sufficiently fine code is already a
complete witness.

\begin{theorem}[Every fine code is characteristic: unconditional surface recognition]
\label{thm:surface-recognition}
Fix \(\Delta,\tau>0\) and \(0<\varepsilon\leq c_1\min\{1,\tau\}\).  Then every encoded
\(\varepsilon\)-type realized by a pair in some
\(\calZ^{\mathrm{ess}}_{\Lambda,\Delta,\tau}(K)\) is characteristic for the
essential-surface type of that pair.  Consequently, within the window
\((\Delta,\tau,\varepsilon)\), distinct essential-surface types never share a
code, and each visible type is recognized by any one of its realized codes.
\end{theorem}

\begin{proof}
Let \(c\) be realized by \((\gamma,F)\), of type \(T\).  If
\((\gamma',F')\) is any pair in the window with the same code \(c\), then
\Cref{cor:cross-type-faithfulness} makes it ambient isotopic to
\((\gamma,F)\), so it also has type \(T\).  Distinct visible types therefore
have disjoint code sets.
\end{proof}

\begin{definition}[Surface finite recognition length]
\label{def:surface-recognition-length}
For an essential-surface type \(T\) with underlying knot type \(K\), and a
window \((\Delta,\tau,\varepsilon)\) with
\(\varepsilon\leq c_1\min\{1,\tau\}\), define
\[
  L^{\mathrm{surf}}_{\mathrm{char}}(T;\Delta,\tau)
  =
  \inf\bigl\{\Lambda\geq\Rop(K)\mid
  \text{some pair of type }T\text{ lies in }
  \calZ^{\mathrm{ess}}_{\Lambda,\Delta,\tau}(K)\bigr\}.
\]
This is the first ropelength threshold at which a bounded-geometry
representative of the surface type, and hence a characteristic code for it,
becomes visible.
\end{definition}

\begin{proposition}[Recognition length and compactified attainment]
\label{prop:surface-recognition-length}
Let \(T\) be an essential-surface type with underlying knot \(K\), and let
\(\mathfrak S_T\) be the corresponding filtered surface type.  Then
\[
  L^{\mathrm{surf}}_{\mathrm{char}}(T;\Delta,\tau)
  =\beta_{\mathfrak S_T}(K;\Delta,\tau).
\]
It is finite exactly when \(T\) has a representative satisfying the specified
area and relative-thickness bounds; when finite, the associated compactified
level \(\overline\beta_{\mathfrak S_T}\leq\beta_{\mathfrak S_T}\) is attained
in the sense of \Cref{cor:attainment-visibility}, with exact-slice attainment
in the situations listed there.
The type is recognized at every larger ropelength level.
\end{proposition}

\begin{proof}
The two infima are over the same levels.  Attainment is
\Cref{cor:attainment-visibility}; recognition at all larger levels follows
from \Cref{thm:surface-recognition} and monotonicity of visibility.
\end{proof}

\begin{definition}[Joint knot--surface recognition threshold]
For a surface type \(T=(K,F)\), define
\[
  L^{\mathrm{joint}}_{\mathrm{char},u}(T;\Delta,\tau)
  =
  \max\left\{
  L_{\mathrm{char},u}(K),
  L^{\mathrm{surf}}_{\mathrm{char}}(T;\Delta,\tau)
  \right\}.
\]
The first term measures the visibility of a finite saturated diagrammatic
certificate for \(K\); the second measures the visibility of a direct
positive-reach code for the pair type.
\end{definition}

\begin{theorem}[Two unconditional recognition layers]
\label{thm:two-unconditional-layers}
Let \(T=(K,F)\) be an essential-surface type satisfying the bounded-geometry
window \((\Delta,\tau)\).  Then
\(L^{\mathrm{joint}}_{\mathrm{char},u}(T;\Delta,\tau)<\infty\).  For every
\(\Lambda>L^{\mathrm{joint}}_{\mathrm{char},u}(T;\Delta,\tau)\), the scale
\((\Lambda,\Delta,\tau,\varepsilon,u)\), with
\(0<\varepsilon\leq c_1\min\{1,\tau\}\), contains both
\begin{enumerate}[label=(\roman*),leftmargin=2em]
\item a saturated Barbensi--Celoria certificate recognizing \(K\) up to
mirror; and
\item a characteristic pair code recognizing the exact type \(T\), with no
normalization or mirror ambiguity under the orientation-preserving
minimal-code convention of
\Cref{rem:code-well-defined}; recognition is up to mirror only if codes are
additionally identified under ambient reflections
(\Cref{def:characteristic-surface-code}).
\end{enumerate}
No geometric liftability or projection--Cerf tameness hypothesis is required
for either conclusion.
\end{theorem}

\begin{proof}
The knot term is finite by
\Cref{thm:knot-finite-visibility-v2}.  The surface term is finite by the
hypothesis and \Cref{prop:surface-recognition-length}.  Above their maximum,
monotonicity supplies the knot certificate and the visible pair; the two
recognition conclusions are respectively
\Cref{thm:knot-finite-visibility-v2} and
\Cref{thm:surface-recognition}.
\end{proof}

\begin{remark}[Two unconditional mechanisms]
The Barbensi--Celoria theorem supplies a finite local certificate inside a
complete combinatorial invariant, while finite visibility supplies a bounded
ropelength realization of that certificate.  Surface recognition instead
samples the geometric pair directly and invokes positive-reach reconstruction.
Thus the surface theorem is not ``more unconditional'' than the knot theorem;
it uses a smaller witness because its code already contains the embedded
object rather than only its projected move neighbourhood.
\end{remark}

\begin{remark}[Recognition up to symmetry and complete witnesses]
By \Cref{cor:cross-type-faithfulness}, a shared fine pair code forces
equality of essential-surface types \emph{exactly}: the canonical code uses
orientation-preserving normalizations only, so no normalization or mirror
ambiguity remains.  Recognition is up to mirror precisely when codes are, in
addition, deliberately identified under ambient reflections, as in
\Cref{def:characteristic-surface-code}.  On the knot side, the analogous conclusion requires a
\emph{saturated} complete rooted ball: arbitrary subgraph occurrence is too
weak, and suppressing parallel Reidemeister edges loses reconstruction data.
This is the precise distinction between the two finite witnesses.
\end{remark}

\begin{remark}[The finite surface principle as a sliced theorem]
Once \((\Delta,\tau,\varepsilon)\) is fixed with
\(\varepsilon\leq c_1\min\{1,\tau\}\), an essential-surface type visible in the window
is determined by any one realized code.  Its boundary slopes, JSJ incidence,
taut and Kakimizu adjacency, and cutting decomposition are therefore functions
of that finite code.  The residual infinity is exactly the one isolated
throughout the paper: letting \(\Delta\to\infty\) or \(\tau\to0\) restores
infinitely many types, one finite slice at a time.
\end{remark}


\section{Outlook and open problems}
\label{sec:further}

We record the main problems left open by the present finite-resolution
framework.  They are grouped here to avoid interrupting the proof-oriented
parts of the paper.  Some material which is useful but not central to the main
line--for example bridge-sphere cutting and low-bridge evidence--is included
here rather than as a separate section.  \Cref{fig:program-map} locates the
principal open directions around the established core.

\begin{figure}[t]
\centering
\begin{tikzpicture}[x=1cm,y=1cm,
  core/.style={draw, double, rounded corners=2pt, align=center, inner sep=5pt,
               text width=46mm, font=\small},
  open/.style={draw, rounded corners=2pt, align=center, inner sep=4pt,
               text width=38mm, font=\scriptsize},
  darrow/.style={-{Latex[length=2mm]}, thick, dashed}]
  \node[core] (core) at (6.6,2.5)
    {\textbf{established core}\\[0.5mm]
     bounded-geometry pair spaces:\\ smooth finiteness, faithful codes,\\
     visible complexes, writhe window};
  \node[open] (transverse) at (2.1,5.1)
    {stratified transverse systems and\\ Gordon--Luecke intersection graphs};
  \node[open] (sutured) at (6.6,5.1)
    {geometric sutured hierarchies:\\ total area and hierarchy length};
  \node[open] (ghs) at (11.1,5.1)
    {generalized Heegaard splittings\\ and Hempel-distance bounds};
  \node[open] (carrying) at (1.65,-0.1)
    {quantitative branched-surface\\ carrying and normal bounds};
  \node[open] (graphics) at (6.6,-0.1)
    {bounded-geometry sweepouts and\\ Rubinstein--Scharlemann graphics};
  \node[open] (algo) at (11.55,-0.1)
    {finite-resolution algorithms;\\ rigidity and minimal decoration};
  \draw[darrow] (core) -- (transverse);
  \draw[darrow] (core) -- (sutured);
  \draw[darrow] (core) -- (ghs);
  \draw[darrow] (core) -- (carrying);
  \draw[darrow] (core) -- (graphics);
  \draw[darrow] (core) -- (algo);
\end{tikzpicture}
\caption{The program map of \Cref{sec:further}.  The established core is the
bounded-geometry theory proved in this paper.  Dashed arrows indicate the
open extensions listed below: enlarging the pair space to transverse systems,
hierarchies, generalized splittings, or sweepouts; making carrying and
normal-coordinate interfaces quantitative; and turning the finite search
spaces into algorithms with a resolved rigidity theory.  Each extension
requires new compactness input beyond disjoint positive-reach systems.}
\label{fig:program-map}
\end{figure}
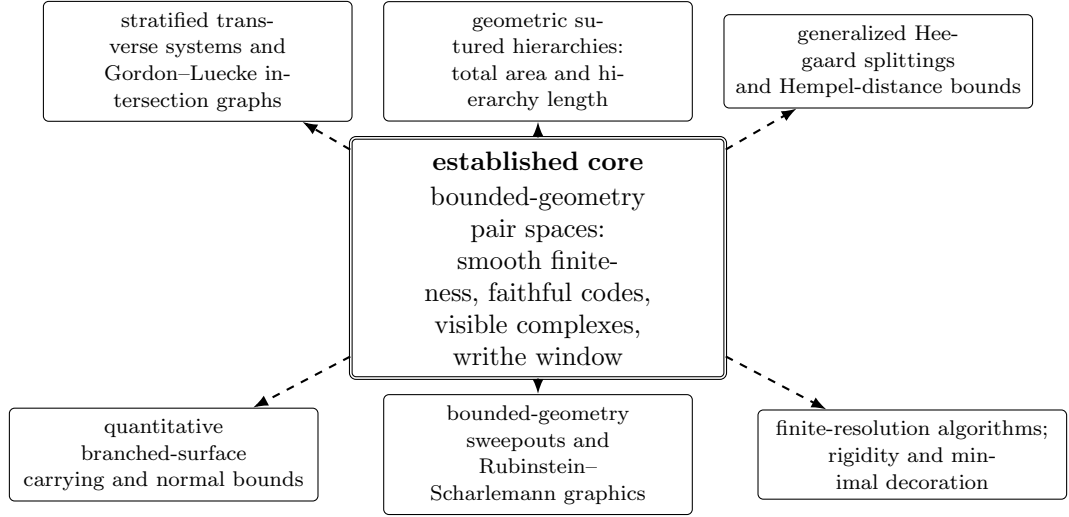

\begin{remark}[Low-bridge evidence]
If \(K\) is in \(n\)-bridge position with bridge sphere \(P\), then
\(P\cap E(K)\) is a \(2n\)-punctured sphere and cutting along it decomposes the
exterior into two trivial tangle exteriors.  For 2-bridge knots
\cite{Schubert}, the Hatcher--Thurston classification
\cite{HatcherThurston} implies that the closed essential layer is
empty.  For 3-bridge knots and links, the author's classification of
genus-two closed incompressible surfaces \cite{Ozawa3bridge} gives finite
low-bridge evidence for
the framework: the closed genus-two essential layer is controlled by finitely
many surface-piece types in the two 3-string trivial tangle exteriors.  Thus
low bridge number supplies useful test cases for the filtered theory, but it
is not part of the core construction.
\end{remark}

\begin{enumerate}[label=(\roman*)]
\item \emph{Sharp ropelength--area--thickness inequalities.}  The packing
estimates above are deliberately coarse.  A central problem is to find sharp
hypotheses under which \(\Len(\gamma)\), \(\Area(F)\), and \(\Thi(F)\) impose
nontrivial constraints on one another.

\item \emph{Normal coefficient bounds.}  The thickness interpretation of Haken
sums suggests explicit estimates for normal coordinates under area and
thickness constraints.  Such estimates would make the bridge between normal
surface theory and finite-resolution geometry quantitative.

\item \emph{Compression scales and birth levels.}  Compressing length,
compressing area, and area-decreasing compressions should be related to the
birth function on the compression poset.  In particular, one would like
effective criteria for when a compression strictly lowers the
\((\Lambda,\Delta,\tau)\)-level.

\item \emph{Boundary torus metric control.}  Boundary length estimates follow
from collar hypotheses, but slope-height estimates require control of the
meridian--longitude lattice on \(\partial E(\gamma)\).  The writhe window
(\Cref{thm:writhe-window}) closes this unconditionally for numerical slopes;
the intersection-number refinement, which requires a stable longitudinal-systole
lower bound (\Cref{thm:conditional-slope-height}), remains conditional.

\item \emph{Boundary twisting versus ropelength.}  The boundary twisting
diameter \(\operatorname{Tw}_{\partial}(K)\) measures the largest intersection
distance between boundary slopes of essential surfaces.  Its finite-length
refinements \(\Theta^{\partial}_{p,\ell}\), \(\nu_{\lambda,\ell}\), and
\(\operatorname{Tw}^{\mathrm{rate}}_{\partial}(K)\) ask where this twisting is
forced to occur on the boundary of a thick tube.
\Cref{cor:unconditional-slope-height} gives the unconditional \(3/4\)-power
lower bound \(\Lambda\gtrsim(|r|-C(\tau)\Delta)^{3/4}\) for numerical slope
heights.  The remaining quantitative problems are to determine the sharp
exponent, to prove intersection-number versions in controlled families such as
two-bridge, Montesinos, alternating, or adequate knots, and to decide whether
the \(4/3\) in the writhe bound can be improved for essential-surface-carrying
representatives.

\item \emph{Kakimizu and admissible-component persistence.}  The quantities
\(h_{\Delta,\tau}\), \(\Lambda_{\mathrm{first}}\),
\(\Lambda_{\mathrm{conn}}\), \(\Lambda_{\mathrm{MS}}\), and the corresponding
merge scales should be computed for standard families such as torus knots,
low-bridge knots, and satellite knots.

\item \emph{Least-area representatives and prescribed thickness.}  Least-area
theory gives small-\(\tau\) ideal surfaces once a compact embedded representative
has positive relative thickness.  A key problem is to estimate this thickness
from the geometry of the knot exterior and the surface class.

\item \emph{Filtered rigidity and stabilization.}  Compute the image and kernel
of the action on
\(\mathbf{ES}^{\mathrm{geom},\mathfrak D}(E,\mu)\) for standard knot families,
and determine whether finite decorated stages eventually eliminate all
nongeometric automorphisms.  Even when no single level is canonical, one may
ask whether the cofinal action stabilizes and whether its kernel agrees with
the kernel of the action on the full essential-surface complex.

\item \emph{Frontier-crossing compatibility and minimal decoration.}  Determine
which finite data on JSJ support, frontier slopes, and annular or toral twist
coordinates are necessary and sufficient to glue piecewise realizations.
Identify the weakest natural decoration for which the characteristic frontier
and peripheral meridian are intrinsically reconstructible.

\item \emph{Finite-resolution algorithms.}  The DoF estimate suggests a finite
search space at fixed \((\Lambda,\Delta,\tau,\varepsilon)\).  Turning this into
an implementable algorithm requires explicit finite encodings, transition rules
for admissible deformations, computable normal-coordinate bounds, and finite
procedures for constructing the visible complexes and their automorphism
groups.
\item \emph{Geometric branched-surface carrying.}
\label{prob:geometric-carrying}
Let \(B\) be an incompressible branched surface in a fixed knot exterior.
Determine whether fixed bounds \(\Area(F)\leq\Delta\) and
\(\Thi(F)\geq\tau\) restrict the integral branch weights realized by carried
surfaces to a finite, effectively bounded set.  This would convert the
finite-carrier theorem of Floyd--Oertel into a quantitative geometric
carrying theorem.

\item \emph{Stratified thickness and Gordon--Luecke graphs.}
Develop a compactness theory for transverse systems
\(F_1\cup\cdots\cup F_m\) in which each surface has positive reach, all
intersection angles have a uniform positive lower bound, and double curves,
vertices, and peripheral strata have controlled separation.  Does bounded
total area then imply finitely many stratified pair-isotopy types and finitely
many intersection graphs in each bounded window?

\item \emph{Geometric sutured hierarchies.}
For a taut hierarchy
\[
 (M_0,\gamma_0)\xrightarrow{F_1}\cdots
 \xrightarrow{F_n}(M_n,\gamma_n),
\]
study the total area \(\sum_i\Area(F_i)\), the minimum thickness
\(\min_i\Thi(F_i)\), and the hierarchy length.  Under fixed bounds, are there
only finitely many hierarchy types, and is a least geometric hierarchy cost
attained in a prescribed Thurston-norm class?

\item \emph{Higher topological index and finite completion.}
The visible disk complexes exhaust \(\mathcal D(H)\), but higher homotopy of
finite stages need not stabilize monotonically.  Determine geometric
hypotheses under which Bachman's topological index \(n\geq2\) is detected at a
finite stage.  Even for index zero and one, find quantitative completion bounds
that turn persistent nonappearance into a finite certificate, uniformly over a
prescribed bounded surface window.

\item \emph{Generalized Heegaard splittings and Hempel distance.}
Enlarge the finite pair space from the disk complex of one fixed surface to
include all thin and thick levels together with compression-body incidence and
compatible disk-set data.  Does a bounded length--area--thickness window
contain only finitely many generalized Heegaard-splitting types?  For fixed
genus, can the window give an effective upper bound on the Hempel distances of
the thick levels?

\item \emph{Geometric graphics and decorated spines.}
Determine whether a generic bounded-geometry two-parameter family of
sweepouts has a finite encoding whose discriminant recovers the
Rubinstein--Scharlemann graphic.  In parallel, determine whether layered pair
codes admit functorial decorated-spine duals for which admissible deformations
correspond to uniformly controlled Matveev--Piergallini move sequences.

\end{enumerate}

The most direct algebraic-geometric connection is through Culler--Shalen
surfaces and boundary slopes arising from character varieties
\cite{CullerShalen,CooperCullerGilletLongShalen}.  Comparing such
surfaces with the ideal and near-ideal surface layers introduced here is a
natural direction, but no algebraic-geometric machinery is used in the present
paper.

\section{Interfaces with variational and degeneration theories}
\label{sec:variational-interfaces}
The finiteness theory concerns bounded geometric windows. The following
statements identify what further hypotheses are needed for critical-point
and degeneration arguments; they do not turn compactness into such a theory.

\subsection{Resources and transport}
Work with actual pair configurations before dividing by pair isotopy.
For a deformation $H=(\gamma_t,F_t)$ use the resource vector
\[
 b(H)=\bigl(\sup_t\Len(\gamma_t),\ \sup_t\Area(F_t),
                  \ \sup_t\Thi(F_t)^{-1}\bigr).
\]
The set of attainable upper budgets is upward closed; concatenation takes
the coordinatewise maximum. The coordinatewise infimum need not itself be
attainable: upper sets generated by $(1,10)$ and $(10,1)$ have infimum
$(1,1)$ but do not contain it. Thus separate visibility numbers need not
be realized by one common pair or deformation.

If a pair-space projection were a Hurewicz fibration over a path-connected
base, its total-space path components would be the orbits of the base
fundamental group acting on fibre path components. Path lifting proves
surjectivity of the fibre-to-total component map, and lifting projected
paths proves the orbit relation. This statement is conditional: bounded
length, area, thickness and collar constraints do not by themselves make
the projection a fibration.

\subsection{A conditional contact multiplier theorem}
\begin{proposition}[Abstract contact certificate]
Let $V$ be a Banach variation space, $A$ a compact Hausdorff active set,
$L:V\to\mathbb R$ continuous linear, and $G:V\to C(A)$ continuous linear.
Assume that $L(v)<0$ and $Gv>0$ pointwise would produce a feasible local
decrease at the configuration in question. At a constrained local minimum
there are $\lambda\ge0$ and a finite positive Radon measure $\mu$ on $A$,
not both zero, such that
\[
 \lambda L(v)=\int_A Gv\,d\mu\qquad(v\in V).
\]
If some $v_*$ satisfies $Gv_*\ge c>0$, then $\lambda>0$; normalizing
$\lambda=1$ gives $\mu(A)\le L(v_*)/c$.
\end{proposition}
\begin{proof}
The linear subspace $(L,G)(V)$ is disjoint from the open convex cone
$\{(r,g):r<0,\ g>0\}$. The separating-functional theorem gives a nonzero
continuous functional vanishing on the subspace and nonpositive on that
cone; closedness of the subspace is unnecessary because the other set is
open. Varying $r$ and positive $g$ shows it has the form
$(r,g)\mapsto\lambda r-\int g\,d\mu$ with $\lambda\ge0$ and $\mu\ge0$,
by the Riesz representation theorem. If $\lambda=0$, evaluation at $v_*$
forces $\mu=0$, a contradiction. The mass bound follows by the same evaluation.
\end{proof}

This is the Fritz John form of the Lagrange multiplier rule for a continuum of inequality constraints, obtained from the Hahn--Banach separation theorem; see \cite{IoffeTihomirov}. It is recorded only to fix the form in which a geometric contact certificate would be stated.
To apply this theorem geometrically one must construct the variation space
and the operator $G$, prove the asserted feasible-decrease implication, and
include every active constraint. Strut pairs alone may omit focal contacts,
multiple principal directions, both normal orientations and boundary-collar
conditions. Compactness does not provide differentiability or a constraint
qualification. On a free patch, minimality follows only if arbitrary
compactly supported two-sided variations there are feasible; a free-boundary
condition similarly requires justified boundary variations. These are
additional assertions, not automatic consequences of the theorem above.

\subsection{What weak limits do and do not give}
On a fixed compact exterior, normalized area measures of nonempty surfaces
are probability measures and have weakly convergent subsequences. The same
holds for tangent-plane distributions on the compact Grassmann bundle.
This elementary conclusion requires no thickness-area product estimate.
It does not imply that the limit is an integral or rectifiable varifold,
or a measured essential lamination. In varying exteriors one first needs
an explicit identification with a common compact ambient space and estimates
for that identification.

Even ordered graphical sheets can approach a non-lamination configuration:
$z=0$ and $z=x^2+1/n$ are disjoint, but the limiting graphs touch along
$x=0$ without agreeing nearby. Elliptic equations and a touching principle,
along with boundary, holonomy and essentiality control, would be needed to
exclude such behaviour in a geometric application.

On a single fixed finite carrier, nonnegative integer weights satisfying
switch equations $Aw=0$ have normalized subsequences in
$\{w\ge0:Aw=0,\ \|w\|_1=1\}$, a compact polytope. This is an algebraic
projective limit statement. It does not produce a common carrier for all
surfaces, preserve essentiality of a limiting lamination, or prove an
asymptotic relation between area and thickness. Such geometric conclusions
remain separate from finite carrier algebra.

Criticality, a deformation lemma, and a degeneration theorem are therefore
three distinct inputs. In particular, no mountain-pass surface or Morse
theory of the bounded pair space is asserted solely from geometric finiteness.

\subsection{Visible complexes and finite support}
A simplex in a resource-bounded surface complex must be witnessed by a
simultaneous disjoint surface system satisfying the total area and union
separation/collar conditions. Individual visibility of its vertices is
insufficient. If the simultaneous conditions weaken on taking subsystems,
these simplices form a complex. For a directed increasing family of such
complexes $\mathcal S_{\mathbf a}$ in which every finite set of simplices
is captured at a common level, finite support of cycles and bounding chains
gives
\[
 \varinjlim_{\mathbf a}H_i(\mathcal S_{\mathbf a};\mathbf k)
   \cong H_i\bigl(\bigcup_{\mathbf a}\mathcal S_{\mathbf a};\mathbf k\bigr).
\]
This algebraic argument does not establish geometric visibility of every
simplex of a desired terminal surface complex. Finally, a compression disk
visible in one window is a positive certificate of compressibility; failure
to find one in a finite window is not a certificate of incompressibility
unless an effective bound captures all possible witnesses.

\section{Conclusion: bounded geometry and finite topology}

The results of this paper establish three complementary facts about essential
surfaces in knot exteriors.  First, geometry constrains topology.  Once
ropelength, area, and relative thickness are controlled, the possible
pair-isotopy types form a finite set, and peripheral geometric data impose
explicit restrictions on boundary slopes.  Second, finite geometry determines
topology.  At a resolution fine relative to the thickness scales, layered
finite codes distinguish every pair-isotopy type in the bounded slice.  Third,
finite geometry organizes rigidity.  In a fixed exterior, the bounded slices
form finite visible essential-surface complexes which exhaust the full complex
and carry the exterior symmetry action with controlled parameter change.
There is also a local compression-theoretic extension: for every fixed
compact two-sided surface, bounded disk area and relative disk thickness give
finite visible disk complexes exhausting the full disk complex.  The resulting
birth/merge signature supplies finite witnesses for compressibility and weak
reducibility and persistent signatures for incompressibility and strong
irreducibility.  \Cref{fig:conclusion-triad} assembles the three global
passages from geometry to topology in one picture.

\begin{figure}[t]
\centering
\begin{tikzpicture}[x=1cm,y=1cm,
  hub/.style={draw, rounded corners=2pt, align=center, inner sep=5pt,
              text width=34mm, font=\small},
  outc/.style={draw, double, rounded corners=2pt, align=center, inner sep=5pt,
              text width=36mm, font=\small},
  mech/.style={font=\scriptsize\itshape, text=black!60, align=center},
  arrow/.style={-{Latex[length=2.2mm]}, thick}]
  \node[hub] (win) at (6.6,2.1)
    {bounded geometric window\\[0.5mm]
     $\Len\leq\Lambda$, $\Area\leq\Delta$, $\Thi\geq\tau$};
  \node[outc] (top) at (6.6,4.9)
    {finite topology\\{\scriptsize finitely many pair-isotopy types}};
  \node[outc] (left) at (1.95,-0.4)
    {finite description\\{\scriptsize faithful layered codes}};
  \node[outc] (right) at (11.25,-0.4)
    {finite symmetry stage\\{\scriptsize visible complexes
      $\ES_{\Delta,\tau}(E)$}};
  \draw[arrow] (win) -- node[mech, right=1mm] {compactness} (top);
  \draw[arrow] (win) -- node[mech, above left=0mm and -3mm] {reconstruction} (left);
  \draw[arrow] (win) -- node[mech, above right=0mm and -3mm] {functoriality} (right);
  \draw[arrow, dashed] (win.south) -- (6.6,0.42);
  \node[mech, fill=white, inner sep=1pt] at (6.6,0.05)
    {writhe window:\\quantitative peripheral rigidity};
\end{tikzpicture}
\caption{The passage from geometry to topology, in summary.  A single bounded
geometric window yields a finite set of pair-isotopy types through
positive-reach compactness
(\Cref{prin:geometry-constrains-topology}), a faithful finite description
through layered codes and reconstruction
(\Cref{prin:finite-geometry-determines-topology}), and a finite filtered
stage for exterior symmetries through the visible complexes
(\Cref{prin:finite-geometry-organizes-rigidity}).  The writhe window is the
quantitative form of the same passage on the peripheral torus.}
\label{fig:conclusion-triad}
\end{figure}
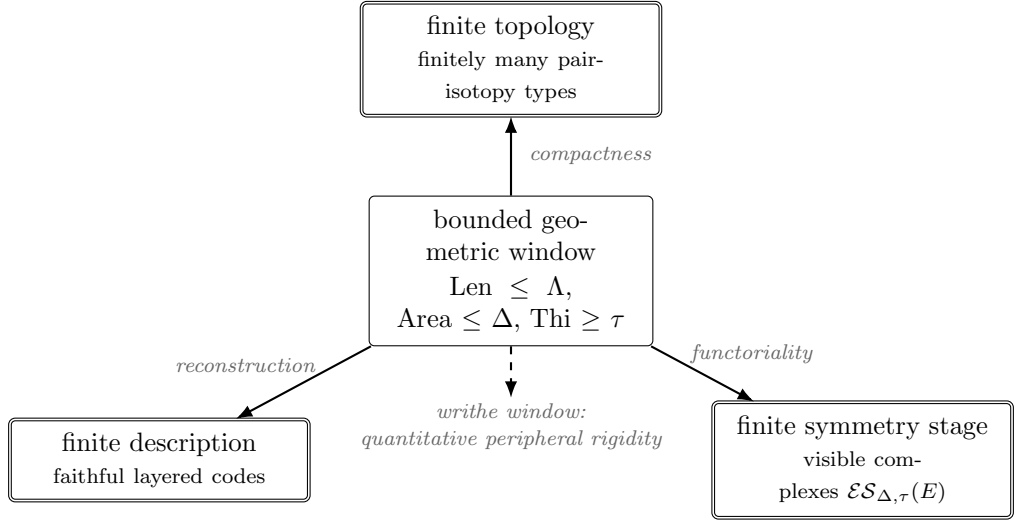

The first statement is not a disguised normal-surface argument.  Its mechanism
is compactness of anchored positive-reach surfaces, including control of the
boundary stratum while the ambient knot exterior moves.  The second statement
is not merely a counting result.  Its mechanism is a reconstruction theorem:
equal sufficiently fine codes force ambient pair-isotopy.  The third statement
does not identify every combinatorial automorphism with a homeomorphism; it
places that separate rigidity problem into a directed system of finite
objects.  Together, these results turn a bounded geometric window into a
finite topological world, a finite description of that world, and a finite
stage on which its symmetries can be tested.

The writhe window makes the same philosophy quantitative.  Writhe is a
geometric quantity attached to a representative, whereas numerical boundary
slope is topological data carried by an essential surface.  The window shows
that area and relative thickness control the discrepancy between them, and
the Buck--Simon estimate then converts this into an unconditional
ropelength--area--slope inequality.  Thus compactness, reconstruction, and
peripheral rigidity are three forms of the same passage from geometry to
topology.

The relation with the classical theories is therefore complementary rather
than competitive.  Normal surfaces, efficient triangulations, branched
surfaces, and spines provide combinatorial coordinates and carriers; sutured
manifolds and generalized Heegaard splittings organize decompositions and
thin--thick structure; Hempel distance, intersection graphs, and
Rubinstein--Scharlemann graphics measure separation or singular behaviour in
curve complexes and parameter spaces.  The contribution of the present
framework is to place geometric visibility, finite recovery, and quantitative
birth levels across these kinds of objects.  Only the interfaces explicitly
proved above belong to the established core; the fixed-surface disk-complex
filtration is part of that core, while higher-index stabilization, entire
hierarchies, transverse systems, and sweepout extensions remain open.

The framework is deliberately bounded.  It does not claim that a knot exterior
has only finitely many essential surfaces, that finite-resolution quotient
persistence coincides automatically with smooth admissible-component
persistence, that every automorphism of a visible complex is geometric, that absence of a
compression or weak reduction at one arbitrary finite disk-complex stage is a
negative certificate, or that the resulting finite search spaces already come
with optimal algorithms.
Rather, it identifies the hypotheses under which infinite-dimensional smooth
data become topologically finite and exactly recoverable, and it isolates the
additional image--kernel--reconstruction questions required for rigidity.

This suggests a broader question.  For which classes of embedded or immersed
objects does bounded geometry make topology finite, when can that finite
topology be reconstructed from finite geometric data, and when do the
resulting finite stages recover the symmetries of the ambient object?  The
present work answers the finiteness and reconstruction questions for
bounded-geometry essential surfaces carried by thick knot exteriors and
provides the filtered setting for the symmetry question.  Extending the same
three principles to wider classes of three-manifolds, spatial graphs,
higher-dimensional embeddings, or other geometric variational problems is a
natural direction for future work.

\section*{Acknowledgements}

\paragraph{Use of generative AI}
The author used ChatGPT (OpenAI) and Claude Fable 5 (Anthropic) as interactive
aids during the preparation of this manuscript, including for mathematical
discussion, consideration of alternative formulations and possible proof
strategies, preliminary consistency checks, and improvement of the exposition.
No AI system is an author of, or bears responsibility for, any result in this
paper.  The author reviewed and edited all AI-assisted material, checked the mathematical definitions, statements, proofs, computations and references, and takes full responsibility for the content of this manuscript.


\end{document}